 \documentclass[final,3p,times]{elsarticle}

\journal{Expositiones Mathematicae}

\usepackage{amssymb}
\usepackage{amsmath}

\usepackage{mathtools} 
\usepackage{amssymb}
\usepackage{enumitem,mathrsfs}
\usepackage{stackrel}

\usepackage{graphicx}

\usepackage{upgreek}

\usepackage[colorlinks=true]{hyperref}
\hypersetup{urlcolor=blue, citecolor=red}

\usepackage{placeins} 
\usepackage{flafter} 

\usepackage[usenames,dvipsnames]{xcolor}

\usepackage[normalem]{ulem}

\DeclareFontFamily{OT1}{pzc}{}
\DeclareFontShape{OT1}{pzc}{m}{it}{<-> s * [1.10] pzcmi7t}{}
\DeclareMathAlphabet{\mathpzc}{OT1}{pzc}{m}{it}

\setlist{noitemsep,topsep=0pt,partopsep=0pt}

\newcommand{\ZZ}{\ensuremath{{\mathbb Z}}}
\newcommand{\CC}{\ensuremath{{\mathbb C}}}
\newcommand{\CW}{\ensuremath{{\widehat{\mathbb C}}}}
\newcommand{\RR}{\ensuremath{{\mathbb R}}}
\newcommand{\NN}{\ensuremath{{\mathbb N}}}

\newcommand{\HH}{\ensuremath{{\mathbb H}}}

\newcommand{\esf}{\ensuremath{{\mathbb S}^2}}

\newcommand{\R}{\ensuremath{{\mathcal R}}}

\font\myfont=cmr10 at 12pt
\newcommand{\e}{{\text{\myfont e}}}

\newcommand{\Ai}{{\text{Ai}}}
\newcommand{\Bi}{{\text{Bi}}}

\newcommand{\abs}[1]{\left\lvert#1\right\rvert}

\renewcommand{\Re}[1]{{\mathfrak{Re}\left(#1\right)}}
\renewcommand{\Im}[1]{{\mathfrak{Im}\left(#1\right)}}

\newcommand{\del}[2]{\frac{\partial #1}{\partial #2}}

\newcommand{\ent}[2]{{}_{#1}\hspace{-1pt}\mathscr{E}_{#2}}
\newcommand{\htan}[2]{{}_{#1}\hspace{-2pt}\mathpzc{Th}_{#2}}    

\newcommand{\dominio}[3]{{}_{#2}{#1}_{#3}}

\usepackage{tikz,etoolbox}
\newcommand{\circled}[1]{ \text{\tikz[baseline=(char.base)]{
    \node[shape=circle,draw,inner sep=1pt] (char) {$#1$};}} }
\robustify{\circled}
\newcommand{\raiz}[1]{ \text{\tikz[baseline=(char.base)]{
    \node[shape=circle,draw,inner sep=0.001pt] (char) {
    \text{\tikz[baseline=(char.base)]{
    \node[shape=circle,draw,inner sep=1pt] (char) {$#1$};}}
    };}} }
\robustify{\raiz}

\newcommand{\msigma}{\ensuremath{{\mathfrak{a}}}}
\newcommand{\mrho}{\ensuremath{{\mathfrak{r}}}}

\usepackage{graphicx}
\makeatletter
\newcommand*\bigcdot{\mathpalette\bigcdot@{.5}}
\newcommand*\bigcdot@[2]{\mathbin{\vcenter{\hbox{\scalebox{#2}{$\m@th#1\bullet$}}}}}
\makeatother

\newlength\replength
\newcommand\repfrac{.33}

\newcommand\rulewidth{.6pt}
\newcommand\tdashfill[1][\repfrac]{\cleaders\hbox to \replength{%
  \smash{\rule[\arraystretch\ht\strutbox]{\repfrac\replength}{\rulewidth}}}\hfill}

\newcommand\tdotfill[1][\repfrac]{\cleaders\hbox to \replength{%
  \smash{\raisebox{\arraystretch\dimexpr\ht\strutbox-.1ex\relax}{.}}}\hfill}

\usepackage{amsthm}

\newtheorem{theorem}{Theorem}[section]

\newtheorem{corollary}[theorem]{Corollary}

\newtheorem{lemma}[theorem]{Lemma}
\newtheorem{proposition}[theorem]{Proposition}
\newtheorem{conjecture}[theorem]{Conjecture}

\theoremstyle{definition}
\newtheorem{definition}[theorem]{Definition}
\newtheorem{remark}[theorem]{Remark}

\newtheorem{example}{Example}[section]
\newtheorem*{theorem*}{Theorem}

\begin{document}

\begin{frontmatter}



\title
{Tessellations and Speiser graphs arising from
meromorphic functions on simply connected Riemann surfaces}

\author[1]{Alvaro Alvarez--Parrilla} 
\ead{alvaro.uabc@gmail.com} 
\affiliation[1]{organization={Grupo Alximia SA de CV},
            addressline={Ryerson 1268}, 
            city={Ensenada},
            postcode={22800}, 
            state={Baja California},
            country={M\'exico}}

\author[2]{Jes\'us Muci\~no--Raymundo\corref{cor2} } 
\ead{muciray@matmor.unam.mx} 
\affiliation[2]{organization={Centro de Ciencias Matem\'aticas},
             addressline={Universidad Nacional Aut\'onoma de M\'exico},
             city={Morelia},
             state={Michoac\'an},
             country={M\'exico}}

\begin{abstract}
Motivated by W. P. Thurston, we ask: What is the shape of a meromorphic function on a simply connected
Riemann surface $\Omega_z$? 
We consider Speiser functions, \emph{i.e.}\ meromorphic functions on 
a simply connected Riemann
surface, that have a finite number ${\tt q} \geq 2$ of singular (critical or asymptotic) values. 
As a first result, we make
precise the correspondence between: Speiser functions $w(z)$, Speiser Riemann surfaces 
$\R_{w(z)}$, Speiser $\tt q$--tessellation, and analytic Speiser graphs of index $\tt q$. 
As the second main result, we characterize tessellations with alternating colors 
(equivalently abstract pre--Speiser graphs) that are realized by Speiser functions on 
$\Omega_z$. 
The characterization is in terms of the $\tt q$--regular extension problem of 
bipartite planar graphs. 
As third main results, the Speiser Riemann surface $\R_{w(z)}$ 
can be constructed by isometric glueing of a finite number of types of sheets, 
where each sheet is a maximal domain of single--valuedness of $w^{-1}(z)$. 
Furthermore, a unique decomposition of $\R_{w(z)}$ into maximal logarithmic towers and a soul is provided. 
Using vector fields we recognize that 
logarithmic towers come in two flavors: 
exponential or $h$--tangent blocks, 
directly related to the exponential or the 
hyperbolic tangent functions on the upper half plane. 
The surface $\R_{w(z)}$ of a finite Speiser function is characterized 
by surgery of a rational block and a finite number of exponential or 
$h$--tangent blocks.
\end{abstract}




\end{frontmatter}




\tableofcontents

\section{Introduction}
\subsection{Brief statement of the results}
Let $w(z):\Omega_z \longrightarrow \CW_w$ be a meromorphic function 
on the simply connected Riemann surface
$\Omega_z$, \emph{i.e.}
the Riemann sphere $\CW_z$,
the complex plane $\CC_z$, 
or the Poincar\'e unit disk $\Delta_z$.
Allow us first to summarize the main results.

\begin{enumerate}[label=\Alph*),leftmargin=*]

\item
On $\Omega_z$,
we provide a bijective correspondence between:

\begin{enumerate}[label=\arabic*),leftmargin=*]
\item
\emph{Speiser functions}
with ${\tt q} \geq 2$ singular values.

\item 
\emph{Speiser Riemann surfaces}.

\item
\emph{Speiser $\tt q$--tessellations}.

\item
\emph{Analytic Speiser graphs of index $\tt q$}.

\end{enumerate}

\smallskip

\item
We answer the question
\begin{equation*}
\text{\emph{What is the shape of a Speiser function
on $\Omega_z$?}}
\end{equation*}

\noindent
by characterizing when 
an abstract \emph{pre--Speiser graph} represents 
a Speiser function. 
The solution presented in terms of Hall--type inequalities 
arises from the \emph{bipartite transportation problem}.

\item
We provide a 
decomposition of Speiser Riemann surfaces 
into \emph{maximal logarithmic towers} and a \emph{soul}.
This decomposition is unique 
and provides another anwser to Question (B): 

\centerline{
\emph{the shape of a Speiser function is its soul.}
}
 
\noindent 
Moreover, 

{\it 
\begin{center}
finite Speiser functions 
are those that can be constructed by surgery of
\\
maximal logarithmic towers and
a soul arising from a rational function. 
\end{center}
}

\end{enumerate}

\noindent
A few words on the above is in order.

Most of (A) is classical, we fill in the details to provide a structured and modern 
approach that allows us to prove (B) and (C).

Question (B) was first answered by W.\,P.\,Thurston for the 
generic\footnote
{Generic means that it only has simple critical points.} 
rational functions on $\CW_z$, we present an answer 
valid for all 
Speiser functions (which of course include rational functions).

Decomposition (C) is reminiscent of the dichotomy between 
Fatou and Julia sets; 
in the sense of presenting regions where a 
function behaves 
\emph{tamely} $\longleftrightarrow$ maximal logarithmic towers, or
\emph{wildly} $\longleftrightarrow$ soul.
It is to be noted that maximal logarithmic towers arise from 
considering the accurate behaviour of exponential or hyperbolic 
tangent functions near the
essential singularity at $\infty\in \CW_z$.

\subsection{Accurate results and comments}
Recalling W.\,P.\,Thurston's question on rational functions
on $\CW_z$, see 
\cite{Thurston2}, 
\cite{Koch-Lei}, it is natural to extend it to:
\begin{equation*}
\text{\emph{What is the shape of a meromorphic function
on $\Omega_z$?}}
\end{equation*}

\noindent
In order to answer the above question, we restrict ourselves to the family of 
\emph{Speiser functions}, which are meromorphic functions $w(z)$ on 
$\Omega_z$ with a finite set of
${\tt q} \geq 2$ distinct singular values in $\CW_w$. 
This is a large family 
that includes rational functions on $\CW_z$, and
many transcendental functions on $\Omega_z= \CC_z$ or $\Delta_z$.

As an appropriate first answer to the ``shape of a Speiser function'', 
we propose a 
\emph{Speiser $\tt q$--tessellation}, 
which is the output of the 
\emph{Schwarz--Klein--Speiser's algorithm} 
(see \S\ref{sec:Schwarz-Klein-Speiser-algorithm}),
with roots in the works of 
H.\,A.\,Schwarz \cite{Schwarz}, 
F.\,Klein \cite{Klein}, 
A.\,Speiser \cite{Speiser}; 
as far as we known, studied by
R.\,Nevanlinna \cite{Nevanlinna2}.
We briefly describe the algorithm as follows.
Let $\gamma$ be a Jordan path through the cyclically ordered 
singular values
$\mathcal{W}_{\tt q} \doteq [{\tt w}_1, \ldots , {\tt w}_{\tt q}]$, and consider the pullback graph 

\centerline{$w^*(\gamma) = \widehat{\Gamma}_{\tt q}$. }

\noindent 
Then, the underlying Speiser $\tt q$--tessellation, 
Definition \ref{def:Speiser-tessellation}, is
\begin{equation*}
\mathscr{T}_\gamma(w(z))
=
(\Omega_z \cup \partial_{\mathcal{I}} \Omega_z )
\backslash \widehat{\Gamma}_{\tt q} 
=
\underbrace{
T_1 \cup \ldots \cup T_\alpha \cup \ldots  }_{ n\text{ blue tiles} }
\cup 
\underbrace{ T^\prime_1 \cup  
\ldots  \cup T^{\prime}_\alpha  \cup \ldots }_{n \text{ grey tiles}}\, , 
\ \ \ 
2 \leq n \leq \infty,
\end{equation*}

\noindent
with a 
\emph{consistent $\tt q$--labelling 
$\mathcal{L}_{\mathcal{W}_{\tt q}}$ of the vertices of
the graph $\widehat{\Gamma}_{\tt q}$}, 
where  
$\partial_{\mathcal{I}} \Omega_z $ denotes the 
\emph{ideal boundary of $\Omega_z$ depending on $w(z)$}, namely
see Proposition \ref{prop:dos-tipos-singularidades}.
Summarizing,

{\it
\begin{center}
given a Speiser function $w(z)$,
provided with a cyclic order  
$\mathcal{W}_{\tt q}$,
\\
on its singular values,
the \emph{Schwarz--Klein--Speiser's algorithm}
determines 
\\
a Speiser $\tt q$--tessellation \
$
\big( \underbrace{ \mathscr{T}_\gamma(w(z)) }_{\text{tessellation}}, 
\underbrace{ \mathcal{L}_{\mathcal{W}_{\tt q}} }_{\substack{ \text{consistent}\\
{\tt q}-\text{labelling}} }  \big) 
$.
\end{center}
}

\noindent
The tessellation consists of tiles that are topological 
$\tt q$--gons with alternating colors, see
\cite{Chislenko-Tschinkel}, 
\cite{Koch-Lei}, 
\cite{GonzalezMucino} and
\cite{AlvarezGutierrezMucino}. 
It provides a simple and
straightforward visual description of the Speiser function
$w(z)$;
in particular,
if $\gamma=\RR\cup\{\infty\}$, then
it is clear that set theoretically
$\widehat{\Gamma}_{\tt q} $ 
is a real analytic curve simply given by $\{\Im{w(z)}=0\}$.

\noindent
We call the naturally associated underlying graph 
$\widehat{\Gamma}_{\tt q}$ an \emph{$\tt A$--map}.
Thus, the Speiser $\tt q$--tessellation is equivalent to 

\centerline{
$(\widehat{\Gamma}_{\tt q}, \mathcal{L}_{\mathcal{W}_{\tt q}})$.
}

\noindent
A clear understanding of this structure, 
naturally leads, through duality, to
\emph{analytic Speiser graphs of index $\tt q$},

\centerline{
$(
\underbrace{\mathfrak{S}_{w(z)} }_{\text{graph}}, 
\underbrace{\  \mathcal{L}_{\mathcal{W}_{\tt q}} \ }_{
\substack{ \text{consistent}\\
{\tt q}-\text{labelling}}
} )$.}

\noindent
In plain words, 
an analytical Speiser graph of index ${\tt q}$
in $\Omega_z$, is a countable, connected, bipartite,
planar multigraph\footnote{
A \emph{multigraph} is a graph that admits multiple edges 
between the same pair of vertices.} 
with vertices $\times$ and $\circ$, each with valence $\tt q$, whose ${\tt w}_{\tt j}$--faces in $\Omega_z$ are
labelled cyclically with $\mathcal{L}_{\mathcal{W}_{\tt q}}$, 
so that the labelling follows a clockwise order around $\times$ and 
anticlockwise order around $\circ$; 
see 
\cite{Nevanlinna2} and 
\cite{GoldbergOstrovskii}\,ch.\,4
for examples.
Furthermore, it is required that for each ${\tt w}_{\tt j}\in\mathcal{W}_{\tt q}$ 
at least one ${\tt w}_{\tt j}$--face is not a digon.
See Definitions \ref{def:Speiser-graph} and \ref{def:analytic-Speiser-graph} for full details.

\noindent
The analytical Speiser graph $\mathfrak{S}_{\tt q}$
of index ${\tt q}$ structure provides
an equivalent answer to the ``shape of a Speiser function''.
The Speiser $3$--tessellations on $\CW_z$
with $\mathcal{W}_3=[0,1,\infty]$ are naturally 
equivalent to the celebrated dessins d'enfants, see
\cite{Jones-Wolfart} for that theory. 

Furthermore, the \emph{Speiser Riemann surface} 

\centerline{$\R_{w(z)} \doteq 
\{(z,w(z))\} \subset\Omega_z\times\CW_w$, }

\noindent 
associated to a Speiser function $w(z)$,
see Definition \ref{def:Speiser-q-function}.2, 
is a powerful tool towards understanding the ``shape of a Speiser function''.  All this allows us to prove
the correspondence (A) previously announced.

\theoremstyle{plain}
\newtheorem*{theoMain}{Theorem \ref{th:main-theorem}}
\begin{theoMain}[Main Correspondence]
Let $\Omega_z$ be a simply connected Riemann surface, 
and let ${\tt q} \geq 2$.
There exists a one to one correspondence between:
\begin{enumerate}[label=\arabic*),leftmargin=*]

\item
Speiser functions 

\centerline{$w(z): \Omega_z \longrightarrow \CW_w$,}

\noindent  
provided with a cyclic order $\mathcal{W}_{\tt q}$ 
for its ${\tt q}$ singular values.

\item 
Speiser Riemann surfaces 

\centerline{$\R_{w(z)} \subset  \Omega_z \times \CW_w$, }

\noindent 
provided with a cyclic order $\mathcal{W}_{\tt q}$ 
for the $\tt q$ projections of its branch points on $\CW_w$.

\item
Speiser $\tt q$--tessellations 

\centerline{
$
\big( \underbrace{ \mathscr{T}_\gamma(w(z)) }_{\text{tessellation}}, 
\underbrace{ \mathcal{L}_{\mathcal{W}_{\tt q}} }_{\substack{ \text{consistent}\\
{\tt q}-\text{labelling}} }  \big) 
$.
}

\item
Analytic Speiser graphs of index $\tt q$

\centerline{
$(
\underbrace{\mathfrak{S}_{w(z)} }_{\text{graph}}, 
\underbrace{\  \mathcal{L}_{\mathcal{W}_{\tt q}} \ }_{
\substack{ \text{consistent}\\
{\tt q}-\text{labelling}}
} )$.}

\end{enumerate}
\end{theoMain}

As a useful consequence of the Main Correspondence, 
a tessellation or Speiser graph with a 
chosen consistent $\tt q$--labelling
$\mathcal{L}_{\mathcal{W}_{\tt q}}$ produces a family 
of Speiser functions parametrized by 
$Aut(\Omega_z)\times \text{Stab}(\mathcal{W}_{\tt q})$, 
where $\text{Stab}(\mathcal{W}_{\tt q}) \subset Aut(\CW_w)$ 
is the isotropy group of the cyclic order $\mathcal{W}_{\tt q}$.
Very roughly speaking, $\mathcal{W}_{\tt q}$ 
provides the complex analytic information for
tessellations and Speiser graphs. 
See Lemma \ref{lem:no-unicidad-funciones-teselaciones}.

The proof of the Main Correspondence follows by showing 
that rows two and three of the following diagram commute

\begin{center}
\begin{picture}(180,215)(-30, -85)

\put(-173,0){\vbox{\begin{equation}\label{dia:correspondencia-completa}\end{equation}}}

\put(-10,112){$w(z)$}

\put(120,112){\ \ $\R_{w(z)}$}

\put(30,113){\vector(1,0){85}}
\put(115,113){\vector(-1,0){85}}
\put(55,120){classical}

\put(0,105){\vector(0,-1){23}}
\put(140,105){\vector(0,-1){23}}

\put(140,92){$\begin{array}{l}\text{adding cyclic} \\ \text{order } \mathcal{W}_{\tt q} \end{array}$ }

\put(-25,70){$(w(z), \mathcal{W}_{\tt q})$}

\put(120,70){$(\R_{w(z)} , \mathcal{W}_{\tt q} )$}

\put(30,73){\vector(1,0){85}}
\put(115,73){\vector(-1,0){85}}

\put(-40,10){$(\mathscr{T}_\gamma(w(z)),\mathcal{L}_{\mathcal{W}_{\tt q}})$}

\put(120,10){$(\mathfrak{S}_{w(z)}, \mathcal{L}_{\mathcal{W}_{\tt q}} )$} 

\put(135,-30){$\mathfrak{S}_{\tt q}$}
\put(-5,-30){$\widehat{\Gamma}_{\tt q}$}

\put(137,-69){$\mathfrak{S}$}
\put(-3,-69){$\Gamma$}

\put(-3,-59){\vector(0,1){22}}
\put(3,-37){\vector(0,-1){22}}
\put(140,-59){\vector(0,1){22}}

\put(0,-19){\vector(0,1){22}}
\put(140,-19){\vector(0,1){22}}

\put(0,60){\vector(0,-1){36}}
\put(-120,35){ $\begin{array}{r}\text{Schwarz--Klein--Speiser's}\\ \text{algorithm}\end{array}$ }

\put(40,24){\vector(2,1){75}}

\put(40,12){\vector(1,0){75}}
\put(115,12){\vector(-1,0){80}}
\put(40,-1){\text{Proposition } \ref{prop:bijection-extends-to-action} }

\put(140,-12){$\begin{array}{l}\text{adding consistent} \\ 
{\tt q}\text{--labelling } \mathcal{L}_{\mathcal{W}_{\tt q}} \end{array}$ }

\put(30,-27){\vector(1,0){85}}
\put(115,-27){\vector(-1,0){85}} 

\put(-82,-50){$\begin{array}{r}\text{adding vertices} \\ 
\text{of valence 2}  \end{array}$ }
\put(140,-50){$\begin{array}{l}\text{adding edges} 
\\ \text{to form digons.}  \end{array}$ }

\put(7,-50){$\begin{array}{l}\text{forgetting vertices} 
\\ \text{of valence 2}  \end{array}$ }

\put(30,-65){\vector(1,0){85}}
\put(115,-65){\vector(-1,0){85}}
\put(57,-80){duality}

\put(66,37){$\begin{array}{r}\text{lifting}\\ \text{the tessellation}\end{array}$}

\end{picture}
\end{center}

\noindent
Thus, from our point of view/perspective, the notion of 
``shape of a Speiser function''
is given by 
the third row of the above diagram:
a Speiser $\tt q$--tessellation or equivalently
an analytic Speiser graph of index $\tt q$.

It is interesting to note that in Diagram \ref{dia:correspondencia-completa}.

\noindent
$\bigcdot$ 
The top two rows contain analytical objects/information.

\noindent
$\bigcdot$ 
Rows two and three contain pairs, whose second entry is essentially the cyclic order $\mathcal{W}_{\tt q}$
of the singular values of $w(z)$.

\noindent
$\bigcdot$ 
In order to complete the whole picture and gain a better understanding of 
``shape of a Speiser function'', we introduce the two last rows,
containing topological and combinatorial objects, 
and information related to 
\emph{topological branched coverings of $\Omega_z$}.

\noindent
$\bigcdot$ 
In the fifth row of the above diagram, we have the
the most basic objects which can be completed with some structure
so that they characterize a Speiser function:
\emph{$\tt t$--graphs $\Gamma$} and their duals 
\emph{pre--Speiser graphs $\mathfrak{S}$}.
Consider a 
tessellation of $\Omega_z$ with alternating colors 
whose tiles are topological $\rho$--gons, 
where $\rho \leq {\tt q}$ depends on the tile.
The boundary of its $\rho$--gons, is by definition 
a graph, called a $\tt t$--map $\Gamma$.
See Definitions \ref{def:de-t-graph} 
and \ref{def:pre-Speiser-graph} for details.

\medskip
Formally, Question (B) can be restated as
the following inverse problem:

{\it
\begin{center}
Characterize tessellations 
$(\Omega_z \cup \partial_{\mathcal{I}} \Omega_z )\backslash \Gamma$
with alternating colors and 
\\
not necessarily homogeneous tiles, 
\\
that are realized by topological branched 
coverings of $\Omega_z$, 
\\
hence by Speiser functions.
\end{center}
}

\noindent
In these terms, the above inverse problem 
can be translated in our language as follows.
\begin{equation}
\label{eq:pregunta-mosaico-speiser}
\begin{array}{c}
\text{Question: \it is it possible to characterize whether a 
{\tt t}--graph }\Gamma, 
\\
\text{\it or equivalently a pre--Speiser graph }\mathfrak{S},
\text{\it represents a Speiser function?}
\end{array}
\end{equation}

\noindent
$\bigcdot$ 
Going from the bottom to the fourth row
in Diagram \ref{dia:correspondencia-completa}, 
a certain homogenization procedure is 
required: the $\tt A$--map $\widehat{\Gamma}$ is homogeneous
(each polygon 
of the tessellation $\mathscr{T}_\gamma(w(z))$
has  ${\tt q}$ edges, 
\emph{i.e.} it is a $\tt q$--gon) and
the Speiser graph $\mathfrak{S}_{\tt q}$ is regular
(all its vertices have valence $\tt q$); however
the $\tt t$--graph $\Gamma$ is not necessarily
homogeneous and the pre--Speiser graph 
not necessarily regular.

\smallskip

\noindent
As an advantage of Speiser graphs, the study of
Question \eqref{eq:pregunta-mosaico-speiser}
is more lucid using them, this motivates
our notion of  pre--Speiser graph. 
The solution arises from the equivalent problem:
the \emph{bipartite transportation problem} associated to
planar graphs, in our case
the pre--Speiser graph $\mathfrak{S}$, 
see \S\ref{sec:bipartite-transportation-problem}.

\smallskip

\noindent
\textbf{Theorem \ref{prop:balance-conditions-Speiser-graphs}},
{\it
provides the solution, in terms of Hall--type inequalities,
it is valid for finite and infinite pre--Speiser graphs; 
hence it solves 

\noindent 
$\bigcdot$ 
the
elliptic case, when $\mathfrak{S}$ is finite, $\Omega_z=\CW_z$, and 
also

\noindent 
$\bigcdot$ 
the parabolic and hyperbolic cases 
when $\mathfrak{S}$ is infinite, $\Omega_z=\CC_z$ or $\Delta_z$.
}

\medskip

Finally in \S\ref{sec:Soul-tower-decomposition},
the decomposition (C),
an answer to the ``shape of a Speiser function'', is provided by the \emph{soul}\footnote{
To fix ideas, the usual plane polygons in $\CC$
are examples of souls.}.
The motivation comes from Speiser graphs: the notions of \emph{logarithmic ends} 
and their complement the \emph{nucleus}; 
Definition \ref{def:Log-end-for-Speiser-graph}.1 
and \ref{def:Log-end-for-Speiser-graph}.2.
In the context of Riemann surfaces, the above
gives rise to \emph{maximal logarithmic towers} and 
their complement the \emph{soul};
Definitions \ref{def:logarithmic-tower-for-Riemann-surface} and 
\ref{def:soul-of-surface}.

\smallskip

\noindent
\textbf{Theorem \ref{theo:decomposition-soul-logtowers}},
{\it
provides
a unique decomposition of $\R_{w(z)}$ into

\noindent 
$\bigcdot$ 
$0 \leq {\tt p} \leq \infty$ 
\emph{maximal logarithmic towers}
and 
 
\noindent 
$\bigcdot$
their complement, the \emph{soul}.  

\noindent 
Conversely, since the soul can be geometrically recognized as a flat $\tt p$--gon, then we can glue (maximal) logarithmic towers to it to ``recover'' the Riemann surface $\R_{w(z)}$.
}

\smallskip 

\noindent 
Furthermore, logarithmic towers come in two flavors
exponential or $h$--tangent blocks. They are
directly related to the exponential $\exp(z)$
or the hyperbolic tangent $\tanh(z)$ on the upper half plane
$\HH$. 
Since the behavior of $w(z)$ is tame on the towers, 
the soul 
carries the essential information of $w(z)$. 
The proof of Theorem \ref{theo:decomposition-soul-logtowers}
uses the Main Correspondence, Theorem \ref{th:main-theorem}, 
and a decomposition of Riemann surfaces into 
maximal domains of single--valuedness, 
Proposition \ref{prop:Rw-is-a-union-of-sheets}.

\smallskip

As a corollary, we 
provide 
a constructive characterization of \emph{finite\footnote{
Functions $w(z)$ whose Riemann surface $\R_{w(z)}$ only have a finite number of branch points.} 
Speiser functions}.
Note 
that, for finite Speiser functions,
the hyperbolic case, 
$\Omega_z= \Delta_z$, does not appear;
moreover, the only finite Speiser functions in the elliptic case,
$\Omega_z= \CW_z$, are
the rational functions.
This leaves the parabolic case 
$\Omega_z= \CC_z$ as the only one left to 
consider.

\noindent
In Definitions \ref{def:rational-block} and \ref{def:piezas-elementales}, 
we introduce the elementary blocks arising from 
the soul and the maximal logarithmic towers:

\begin{enumerate}[label=\alph*),leftmargin=*]
\item 
\emph{rational--block}, 
$R(z): \overline{\mathscr{P}}\subset\CW_z\longrightarrow \CW_w$, 
for a Jordan domain $\mathscr{P}$,

\item 
\emph{exponential block}, 
$\exp(z): \overline{\HH}\subset\CW_z\longrightarrow\CW_w$,
Figure \ref{fig:piezas-elementales}.a,

\item 
\emph{$h$--tangent block}, 
$\tanh(z): \overline{\HH}\subset\CW_z \longrightarrow\CW_w$,
Figure \ref{fig:piezas-elementales}.b.
\end{enumerate}

\noindent 
With the above blocks we obtain the following characterization,
see Corollary \ref{Cor:caract-Speiser-finitas}. 

\smallskip

{\it 
\begin{center}
Finite Speiser functions 
are those that can be constructed by surgery of
\\
a rational block with 
$2\leq {\tt p} < \infty$
exponential and $h$--tangent blocks.
\end{center}
}

\smallskip

\noindent
A well studied subfamily of the finite Speiser functions are the 
Nevanlinna functions $w(z)$,
denoted in \cite{EremenkoMerenkov2} as $N$--functions;
functions that have $2\leq {\tt p} < \infty$ logarithmic singularities 
and no algebraic singularities of the inverse function\footnote{
Equivalently, that the Riemann surface $\R_{w(z)}$ associated to $w(z)$ only has $\tt p$ 
infinitely ramified branch points.}
$w^{-1}(z)$, 
see \cite{Nevanlinna1} \S8, \cite{Nevanlinna2} p.\ 301, \cite{BergweilerEremenko}, 
\cite{AlvarezMucino4}.
An immediate consequence of Corollary \ref{Cor:caract-Speiser-finitas} is that, 

\smallskip

{\it 
\begin{center}
$N$--functions 
are those that can be 
constructed by 
surgery of a rational block, 
without interior singular points, 
\\
with $2 \leq {\tt p} < \infty $ 
exponential and $h$--tangent blocks.
\end{center}
}
\smallskip

\noindent 
Figure \ref{fig:hexagono3} illustrates the construction.
The $N$--functions coincide with the meromorphic functions
on $\CW_z$ with exactly one essential singularity at $\infty$,
having
$2\leq {\tt p} < \infty$ logarithmic singularities
and no other transcendental singularities of $w^{-1}(z)$.
In plain words, $N$--functions are the simplest 
meromorphic functions on $\CW_z$ with one essential singularity.

The introduction
of the $h$--tangent elementary blocks extends the previous 
work of M.\,Taniguchi 
\cite{Taniguchi1},
\cite{Taniguchi2} 
to a natural/larger framework. 
As valuable and advantageous mechanisms, we recognize
the rational, exponential and $h$--tangent blocks,
arising from the \emph{sharp tools of 
singular complex analytic vector fields} canonically associated to 
meromorphic functions $w(z)$,
\begin{equation}
\label{eq:vector-field-from-function}
X_{w(z)}(z) \doteq \frac{1}{w^\prime(z)} \del{}{z}, 
\end{equation}

\noindent
see the ``Dictionary'' \cite{AlvarezMucino3}\, prop.\, 2.5.

\noindent
$\bigcdot$
The first is a tool that allows easy glueing and pasting of Riemann surfaces and functions, 
as in \S\ref{sec:Gluing Riemann surfaces} and
\S\ref{sec:Soul-tower-decomposition}.

\noindent
$\bigcdot$
The second is that, 
visualizing the phase portraits of the 
$X_{w(z)}(z)$ associated to $w(z)$, 
improves the global understanding of $w(z)$. 
The behaviour which is lost in the 
tessellations can readily be observed with the 
visualization\footnote{Throughout 
this entire work, the phase portrait of $X_{w(z)}(z)$
means the phase portrait of the
real vector field $\Re{X_{w(z)}}(z)$.
}  
of $X_{w(z)} (z)$, 
\emph{e.g.}\ simple poles of $w(z)$ can be clearly described 
as dipoles of the vector fields.

\noindent
$\bigcdot$
In particular, tessellations or Speiser graphs, because of their topological nature, 
can not distinguish between the (holomorphic) exponential block  
and the strictly meromorphic $h$--tangent block.
An advantage of vector fields is that it allows us to easily distinguish between them,
as can be observed in Remark \ref{rem:vec-fields-distinguish-elementary-blocks} and
in Figure \ref{fig:piezas-elementales}.
See \cite{Alvarez-Mucino-Solorza-Yee}, \cite{AlvarezMucino1}, \cite{AlvarezMucino2}, 
\cite{AlvarezMucino3}, \cite{AlvarezMucino4} for further details, references and 
applications.

\medskip
In \S\ref{sec:examples} 
we have collected a number of examples of the 
Main Correspondence (A):
Speiser functions with a cyclic order $\mathcal{W}_{\tt q}$,
the decomposition of $\R_{w(z)}$ into

\noindent
$\bigcdot$ 
maximal domains of single--valuedness 
(Proposition \ref{prop:Rw-is-a-union-of-sheets}),

\noindent
$\bigcdot$ 
maximal logarithmic towers and soul (Corollary \ref{Cor:caract-Speiser-finitas}),

\noindent 
Speiser $\tt q$--tessellations, and
analytic Speiser graphs of index $\tt q$.

\subsection{Epilogue} 
The underlying theme of several of our previous works has been the
study of essential singularities of functions and 
vector fields, see \cite{AlvarezMucino1}, \cite{Alvarez-Mucino-Solorza-Yee}, 
\cite{AlvarezMucino2}, \cite{AlvarezMucino3},
\cite{AlvarezMucino4}.
As a concluding remark of this introduction, we would like to point out that,

\begin{center}
{\it
the simplest functions $w(z)$ with a unique essential singularity on the Riemann sphere 
are those whose soul has no algebraic singularities, 
and has $ 2 \leq  {\tt q} < \infty$ exponential and $h$--tangent blocks 
(\emph{i.e.}\ $N$--functions).
}
\end{center}
As an immediate consequence of Equation \eqref{eq:vector-field-from-function}, 
\begin{center}
{\it
the simplest complex analytic vector fields $X_{w(z)}(z)$ with a unique essential singularity 
on the Riemann sphere are those whose distinguished parameter $w(z)$ 
is a single--valued $N$--function as above.
}
\end{center}

\section{Singularities of the inverse for meromorphic functions}

Let $w(z): \Omega_z \longrightarrow \CW_w$ be a  
meromorphic function. 

\begin{remark}[Natural boundary of $w(z)$]
\label{rem:natural-boundaries}
Throughout this work,  
$\Omega_z$ is either $\CW_z$, $\CC_z$ or $\Delta_z\doteq \{\abs{z}<1\}$. 
In the cases  $\CC_z$ or $\Delta_z$, 
we assume that 
$\infty$ or $\{ \abs{z}=1 \}$ are natural boundaries of $w(z)$, 
\emph{i.e.}\ $w(z)$
can not be analytically extended as a meromorphic
function across these boundaries.
\end{remark}

\begin{definition}[Singularities of $w^{-1}(z)$; \cite{Iversen}, 
\cite{BergweilerEremenko},
\cite{EremenkoReview}]
\label{def:eremenko1}
Take ${\tt w}\in\CW_{w}$ and denote by 
$D({\tt w},\rho) \subset \CW_w$
the disk of radius $\rho > 0$ (in the spherical metric) 
centered at $\tt w$. 
For every 
$\rho > 0$, choose a component 
$U_{\tt w}(\rho)\subset \Omega_z $ of  
$w^{-1}(D({\tt w},\rho))$ in such a way that 
$\rho_{1} < \rho_{2}$ implies $U_{\tt w}(\rho_1) \subset U_{\tt w}(\rho_{2})$. 
Note that the function $U_{\tt w} : \rho \to U_{\tt w}(\rho)$ 
is completely determined by its germ at 0. 

\noindent
The two possibilities below can occur for the germ of $U_{\tt w}$.
\begin{enumerate}[label=\arabic*),leftmargin=*] 
\item 
$\cap_{\rho>0} U_{\tt w}(\rho)=
\{z_k \},
\, z_k\in \Omega_z $. 
In this case, ${\tt w}=w(z_k )$. 

\noindent
Moreover, if ${\tt w}\in\CC_w$ and $w'( z_k )\neq0$, 
or 
${\tt w} = \infty$ and $ z_k $ is a simple pole of $w(z)$, 
then $z_k$ is called an \emph{ordinary point}. 

\noindent
On the other hand, 
if ${\tt w}\in\CC_w$ and $w'(z_k) = 0$, 
or 
if ${\tt w} = \infty$ and $z_k$ is a multiple pole of $w(z)$, 
then $z_k$ is called a \emph{critical point} and $\tt w$ 
is called a \emph{critical value} 
of $w(z)$. 
We also say that the critical point $z_k$ \emph{lies over $\tt w$}.
In this case, 
$U_{\tt w} : \rho \to U_{\tt w}(\rho)$ defines an \emph{algebraic singularity of $w^{-1}(z)$}.

\item $\cap_{\rho>0}U_{\tt w}(\rho) = \varnothing$. 
We then say that our choice $\rho \to
U_{\tt w}(\rho)$ defines a \emph{transcendental 
singularity of $w^{-1}(z)$}
and that the transcendental singularity 
$U_{\tt w}$ \emph{lies over $\tt w$}. 
\end{enumerate}

\noindent 
In both cases, 
the open set $U_{\tt w}(\rho) \subset \Omega_z$ is called a 
\emph{neighbourhood of the  singularity $U_{\tt w}$}. 
Therefore, 
when $\zeta_{m} \in \Omega_z$, 
we say that $\zeta_{m} \to U_{\tt w}$ 
if for every $\rho>0$ 
there exists $m_{0} \in \NN$ such that 
$\zeta_{m} \in U_{\tt w}(\rho)$, for $m \geq m_{0}$.
\end{definition}

A transcendental singularity $U_{\tt w}$, \emph{i.e.}\ the
germ in Definition \ref{def:eremenko1} case (2), 
can be understood as
the addition, to $\Omega_z$, of an ideal point $U_{\tt w}$, 
together with its corresponding family of neighbourhoods 
$\{ U_{\tt w}(\rho) \} \subset \Omega_z$.
If we perform the above for all the transcendental singularities
of $w^{-1}(z)$,
then a completion/compactification of 
$\Omega_z$ is constructed. See \cite{Ahlfors-Sario} Ch.\,I\,\S\,6, for the general
construction.

\begin{definition}
\label{def:punto-ideal}
\
\begin{enumerate}[label=\arabic*),leftmargin=*] 
\item
An \emph{ideal point $U_{\tt w}$} of 
$w(z)$ is a transcendental singularity of 
$w^{-1}(z)$.

\item
The set of ideal points 
is the \emph{ideal boundary of $\Omega_z$},
denoted as 

\centerline{$\partial_{\mathcal{I}}\Omega_z$.}
\end{enumerate}
\end{definition}

The ideal boundary of $\Omega_z$ is totally disconnected, separable 
and compact, see for instance 
\cite{Ahlfors-Sario} Ch.\,I\,\S\,6, or
\cite{Richards} proposition 3.

\begin{definition}\label{defasymptoticvaluepath}
\
\begin{enumerate}[label=\arabic*),leftmargin=*] 
\item 
Let $U_{\tt w}$ be a transcendental singularity of $w^{-1}(z)$.
An \emph{asymptotic value ${\tt w}\in\CW_w$ of $w(z)$}
means that, 
for sufficiently small $\rho>0$,
there exists a $C^1$ 
\emph{asymptotic path}
$\alpha_{\tt w} ( \tau ): [0, \infty) \longrightarrow 
U_{\tt w}(\rho)\subset\Omega_z$,
$\alpha_{\tt w}(0)=z_{\tt o}  \in \Omega_z \backslash \mathcal{S}$,
tending to 
$ z_{\iota} \in \partial_{\mathcal{I}} \Omega_z$ with well defined slope
at the limit $\tau\to\infty$, such that  
\begin{equation}
\label{valasintPsi}
{\tt w} =
\lim_{
\tau  \to \infty
}
w ( \alpha_{\tt w} ( \tau ) )
\in\CW_w .
\end{equation}
\noindent

\noindent 
We shall not distinguish between 
individual members $\alpha_{\tt w}$ of
the class of asymptotic
paths $[\alpha_{\tt w}]$ giving rise to the 
same transcendental singularity $U_{\tt w}$ over $\tt w$  
of $w^{-1}(z)$.
By Equation \eqref{valasintPsi}, the asymptotic path 
$\alpha_{\tt w} ( \tau )$
ends at the transcendental singularity $z_{\iota}=U_{\tt w}$.

\item
A pair $(\alpha_{\tt w}, {\tt w})$ is a 
\emph{branch point of 
the Riemann surface $\R_{w(z)}$ of $w(z)$}.
\end{enumerate}
\end{definition}

\begin{remark}
\label{rem:correspondencia-caminos-asintoticos}
There is a bijective correspondence between 
the following:

\noindent
i) 
classes $[\alpha_{{\tt w}}( \tau )]$ 
of asymptotic\footnote{
A slight abuse of notation is made here, when $U_{{\tt w} }$ is algebraic, 
the path $\alpha_{{\tt w} }(\tau)\to z_{\iota}$ is not an asymptotic path, 
it is just a path arriving to the critical point $z_{\iota}$.
}  paths $\alpha( \tau )$, with asymptotic value ${\tt w}$,

\noindent
ii) transcendental singularities $U_{{\tt w} }$ 
of $w^{-1}(z)$ over ${\tt w}$, and

\noindent
iii) ideal points $U_{{\tt w} }\in\partial_{\mathcal{I}}\Omega_z$ of $w(z)$,

\noindent
iv) branch points $(\alpha_{{\tt w} }, {{\tt w} })$ of $\R_{w(z)}$.

\end{remark}

According to Definition \ref{def:eremenko1}, throughout
all this work the points and singularities of $w(z)$ are of the
following kinds:

\noindent 
$\bigcdot$
simple zeros and poles are ordinary points in $\Omega_z$,

\noindent 
$\bigcdot$
critical points (in particular 
zeros and poles of order at least two),  
are algebraic singularities of $w^{-1}(z)$, in $\Omega_z$,

\noindent 
$\bigcdot$
transcendental singularities of $w^{-1}(z)$,
in $\partial_{\mathcal{I}} \Omega_z$.

\begin{definition}
\label{def:singular-value-point-cosingular}
\
\begin{enumerate}[label=\arabic*),leftmargin=*] 
\item
A \emph{singular value ${\tt w}_{\tt j}\in \mathcal{SV}_w \subset \CW_w$ of $w(z)$} is either a critical value or an asymptotic value.

\item
A \emph{singular point }
$z_\iota \in \mathcal{SP}_w
\subset \Omega_z \cup \partial_{\mathcal{I}} \Omega_z$ 
\emph{of $w(z)$} is 
either 

\noindent
$\bigcdot$ 
a critical point $z_\iota \in \Omega_z$ of $w(z)$ that lies over the critical value ${\tt w}_{\tt j}$, or 

\noindent
$\bigcdot$ 
a transcendental singularity $z_\iota \in \partial_{\mathcal{I}} \Omega_z$ of $w^{-1}(z)$ that lies over the asymptotic value ${\tt w}_{\tt j}$.

\item 
The \emph{cosingular points of  $w(z)$} are

\centerline{
$\mathcal{CS}_w \doteq  w^{-1}(\mathcal{SV}_w)  \backslash \mathcal{SP}_w \subset \Omega_z$, 
}

\noindent
\emph{i.e.}\ 
the points in the preimage of $\mathcal{SV}_w$ that are not 
singular points of $w(z)$.
\end{enumerate}
\end{definition}

In all that follows, we assume that $w(z)$ has
non empty singular value set $\mathcal{SV}_w$.
  
Note that if $z_\iota$ is a critical point and ${\tt w}_{\tt j}$ is its corresponding critical value, 
then of course ${\tt w}_{\tt j}= w(z_\iota)=\lim_{\tau  \to \infty} w ( \alpha_{{\tt w}_{{\tt j}(\iota)}} ( \tau ) )$,
for any path $\alpha_{{\tt w}_{{\tt j}(\iota)}} ( \tau ) \to z_\iota$.
Thus, using Definition \ref{def:singular-value-point-cosingular}.2 and abusing notation, 
we shall sometimes write 
$w(z_\iota)$ for the singular value associated to the singular point $z_\iota$, 
instead of the more cumbersome
$\lim_{\tau  \to \infty}
w ( \alpha_{{\tt w}_{{\tt j}(\iota)}} ( \tau ) )$.

\section{Speiser functions}
We now introduce the family of functions that will be the main subject in this work.
\begin{definition}
\label{def:Speiser-q-function}
\
\begin{enumerate}[label=\arabic*),leftmargin=*] 
\item
A  meromorphic function $w(z): \Omega_z \longrightarrow \CW_w$
with a finite set of
distinct singular values

\centerline{$
\mathcal{SV}_w=
\{{\tt w}_1, \ldots, {\tt w}_{\tt j}, \ldots, {\tt w}_{\tt q} \},
\ \ \ 
{\tt q} \geq 2 ,
$}

\noindent 
is a \emph{Speiser function with $\tt q$ singular values},
also know as a \emph{Speiser function of index $\tt q$}.

\item
The corresponding Riemann surface

\centerline{$
\R_{w(z)} = \{ (z, w(z))\} \subset \Omega_z \times \CW_w
$}

\noindent
is a \emph{Speiser Riemann surface with $\tt q$ singular values}. 

\item A meromorphic function 
$w(z)$ with a finite number of singularities of $w^{-1}(z)$ is a
\emph{finite Speiser function}\footnote{
In this case $\Omega_z$ is either $\CW_z$ or $\CC_z$.
}.
\end{enumerate}
\end{definition}

\begin{example}[Speiser functions and finite Speiser functions]
\label{ex:examples-finite-non-finite-Speiser-fns}
Speiser functions (of the appropriate index $\tt q$) 
comprise a large family of useful functions. Some examples are:
\begin{enumerate}[label=\arabic*.,leftmargin=*]
\item
Rational functions. Since rational functions of degree 
$2\leq n < \infty$ have a finite set of 
$2 \leq {\tt r} \leq 2n-2$
critical points and 
$2 \leq {\tt q} \leq 2n-2$
critical values, then
they belong
to both the Speiser and finite Speiser class.

\item 
Functions with an infinite number of critical points 
and no transcendental singularities of $w^{-1}(z)$. 
The Weirstrass $\wp(z)$ function, 
see \cite{AlvarezGutierrezMucino} example 5.1, 
is a Speiser function with (generically) $4$ critical values 
and zero asymptotic values; 
however it has an infinite number of critical points, 
hence it is not a finite Speiser function.

\item 
Functions with zero critical values and a finite number ${\tt p}<\infty$ of asymptotic values.
These are called $N$--functions, in honor of the Nevanlinna brothers. As examples we mention:
\begin{enumerate}[label=\alph*),leftmargin=*]

\item 
The simplest cases are $w(z)= \e^z$, with asymptotic values $\{0,\infty \}$, 
and $w(z)=\tanh(z)$ with asymptotic values $\{ -1, 1 \}$. 
Each has two logarithmic singularities of $w^{-1}(z)$.
See Example \ref{example:N-exp-tanh} for full details.

\item 
Quotients of Airy functions, say $w(z)=\frac{\Bi(z)}{\Ai(z)}$, 
that has $3$ logarithmic singularities 
over 3 distinct asymptotic values.
See Example \ref{example:tessellation-Airy} for full details.
\end{enumerate}

\item 
The function $w(z)=\exp(\exp(z))$ is a Speiser function with $3$ singular values, 
namely $\{0, 1, \infty\}$; 
however it is not a finite Speiser function since it has an infinite number of logarithmic singularities 
of $w^{-1}(z)$. Further details can be found in Example \ref{example:exp-exp}.
See Figure \ref{fig:mosaico-exp-exp}.a for the Speiser 3--tessellation and 
Figure \ref{fig:mosaico-exp-exp}.b for the Speiser graph of index 3.

\item 
Functions with both critical and asymptotic values. 
\begin{enumerate}[label=\alph*),leftmargin=*]
\item 
Consider $w(z)=\int_0^z P(\zeta) \e^{E(\zeta)} d\zeta$
with $P, E\in\CC[z]$ polynomials of degree $0\leq{\tt r}<\infty$ and 
$0< {\tt p} /2 < \infty$ 
respectively. 
These functions have (generically) ${\tt r}$ critical values and ${\tt p}$ asymptotic values,
thus they are in the Speiser class. Moreover they also are in the finite Speiser class since 
they have a finite number of singular points. 
See for instance \cite{AlvarezMucino2} and \cite{AlvarezMucino3}.

\item
The function $w(z)= \cos{\sqrt{z}}$ is in the Speiser family since its 
singular values are $\{ -1, 1, \infty \}$:
the critical values are $\{ -1,1 \}$ 
associated to an infinite number of critical points, and it has one transcendental singularity of 
$w^{-1}(z)$ over the asymptotic value $\{ \infty \}$. Clearly it is not a finite Speiser function.
See also \cite{GoldbergOstrovskii}\,p.\,360.

\item
Another example is $w(z)=\e^{\sin(z)}$, in this case the critical values are $\{ \e,\, \e^{-1} \}$ 
and the asymptotic values are $\{ 0, \infty \}$; 
thus it is a Speiser function with 4 distinct singular values. 
On the other hand, it has an infinite number of critical points and an
infinite number of transcendental singularities of $w^{-1}(z)$ 
over each of the the asymptotic values $\{ 0, \infty \}$; thus it is not a finite Speiser function.
See Example \ref{example:tessellation-expsin} for more details.
\end{enumerate}
\end{enumerate}
\end{example}

Speiser functions are simple in the following aspect:
the transcendental singularities of the inverse 
belong to the simplest kind.

\begin{definition}
\label{def:logaritmica}
A transcendental singularity $U_{{\tt w}_{\tt j}}$ of $w^{-1}(z)$ over ${\tt w}_{\tt j}$ is a
\emph{logarithmic singularity 
over ${\tt w}_{\tt j}$} if 

\centerline{
$w(z) : U_{{\tt w}_{\tt j}} (\rho) \subset \Omega_z \longrightarrow
D({{\tt w}_{\tt j}}, \rho) \backslash\{{\tt w}_{\tt j}\} \subset \CW_w$} 

\noindent 
is a universal covering for small enough
$\rho $. 
\end{definition}

\begin{proposition}
\label{prop:dos-tipos-singularidades}
Let $w(z)$ be an Speiser function. 

\begin{enumerate}[label=\arabic*),leftmargin=*]
\item
The singular values of $w(z)$ are isolated, 
hence the singularities of $w^{-1}(z)$ are either
\begin{enumerate}[label=\roman*),leftmargin=*]

\item
logarithmic or 

\item
algebraic.
\end{enumerate}

\item
Depending on $w(z)$, the following cases for $\Omega_z \cup \partial_{\mathcal{I}} \Omega_z$
appear:

\begin{enumerate}[label=\roman*),leftmargin=*]

\item
the Riemann sphere $\CW_z$,

\item
a non Hausdorff compactification

\centerline{
$\CC_z \cup \{\infty_{1}, \ldots , \infty_{\tt p}\}$,
}

\noindent
with $2\leq {\tt p} \leq \infty$ ideal points,

\item 
a Hausdorff compactification with an infinite number of ideal points 

\centerline{
$\Delta_z \cup_{\sigma =1}^\infty  \{ \e^{i \theta_{\sigma} } \}$.
}
\end{enumerate}

\end{enumerate}
\end{proposition} 

\begin{proof}
Assertion (1) follows from the classical theorem of Nevanlinna
on isolated asymptotic values, 
\cite{Nevanlinna2} Ch.\, XI, \S1.3, also see
\cite{AlvarezMucino4} theorem 4.4.

The general construction for (2) is in \cite{Ahlfors-Sario}
Ch.\,I\,\S\,6. 
In particular, regarding case (ii), the ad hoc construction
of the non--Hausdorff compactification of $\CC_z$, for $N$--functions
(see \S\ref{sec:N-functions}) and
for the family of functions 
$\{ w(z)= \int^z P(\zeta)\e^{E(\zeta)} d\zeta \}$,
appeared in \cite{AlvarezMucino3}\,p.\,12. 
For (iii) the ideal points originate from 
classes $[\alpha_{{\tt w}_{\tt j}}(\tau)]$ of asymptotic paths
as usual in the boundary of the hyperbolic disk $\Delta_z$. 
\end{proof}

\begin{corollary}
For $w(z)$ a Speiser function, 
its domain is
$\Omega_z= \CW_z$ if and only if $w(z)$ is a rational function.
\hfill\qed
\end{corollary}

Because of Proposition \ref{prop:dos-tipos-singularidades}.1.i, 
from this point forward, 
we convene that whenever we are considering Speiser functions with transcendental singularities,
the terms,

\centerline{
$
\begin{array}{c} \text{logarithmic singularity} \\ \text{of } w^{-1}(z) \end{array} 
\longleftrightarrow 
\begin{array}{c} \text{transcendental singularity} \\ \text{of } w^{-1}(z) \end{array} 
\longleftrightarrow 
\begin{array}{c} \text{ideal point} \\ \text{of } w(z) \end{array}
$
}

\noindent
are referring to the same object.

The Riemann surfaces $\R_{w(z)}$ are valuable tools, they have natural projections 
$\pi_1$ and $\pi_2$ as in the following commutative diagram.
Note that $\pi_1$ is in fact a biholomorphism.

\begin{center}
\begin{picture}(180,80)(0,10)

\put(-150,40){\vbox{\begin{equation}
\label{diagramaRX}\end{equation}}}

\put(30,75){$\Omega_z$}

\put(119,75){$\R_{w(z)} \subset \Omega_z\times\CW_w$}

\put(108,78){\vector(-1,0){60}}
\put(75,85){$\pi_1$}

\put(133,65){\vector(0,-1){30}}
\put(138,47){$ \pi_2 $}

\put(45,68){\vector(2,-1){75}}
\put(59,38){$w(z) $}

\put(125,20){ $\CW_w$.} 

\end{picture}
\end{center}

If ${\tt w}\in\CW_w$ 
is an asymptotic value of $w(z)$, then 
there is at least one logarithmic singularity $U_{\tt w}$ 
of the inverse 
$w^{-1}(z)$ over ${\tt w}$.
Certainly, there can be many
(finite or even infinite) 
different logarithmic singularities 
as well as critical and ordinary points over the same 
singular value ${\tt w}$. 

\begin{definition}
\label{def:multiplicity-of-sing-point}
Let $w(z)$ be a Speiser function, and consider the singularities of the inverse $w^{-1}(z)$.
\begin{enumerate}[label=\arabic*),leftmargin=*] 
\item 
The 
\emph{multiplicity $m$ of an ordinary point point $z \in \Omega_z$} is $1$. 

\item
The 
\emph{multiplicity $m_\iota$ of an algebraic singularity or critical point 
$z_\iota \in \Omega_z$}
is the number $ 2 \leq m_\iota < \infty$ such that 
$w(z)$ is locally equivalent to $\{z \mapsto z^{m_\iota} \}$.

\item
The \emph{multiplicity $m_\iota$ of a logarithmic singularity 
$z_\iota\in \partial_{\mathcal{I}} \Omega_z$ of $w^{-1}(z)$}
is $\infty$.

\item
In all the cases the multiplicity is 
also known as the \emph{ramification index}.
\end{enumerate}
\end{definition}

\begin{definition}
\label{def:multiplicity-of-sing-values}
The 
\emph{multiplicity $\mu_{\tt j} \in \NN \cup \{ \infty\}$
of a singular value ${\tt w}_{\tt j}$},
is the number of branch points of 
the Riemann surface $\R_{w(z)}$ that 
project via $\pi_2$ to ${\tt w}_{\tt j}$.

\noindent 
The \emph{total number of branch points of $\R_{w(z)}$} is
\vspace{0.1cm}

\centerline{$\delta = \mu_{1 }+  \cdots + \mu_{\tt q}
\in \NN \cup \{ \infty \}$.}
\end{definition}

\begin{remark}[The existence of non--trivial multiplicities of the singular values 
makes enumerating singular points a non--trivial issue]
\label{rem:Branch-points-enumeration}
In order to enumerate singular points and values, 
consider the following.

\noindent
1.\ Choose a singularity of $w^{-1}(z)$, thus we either have: 

$\bigcdot$
a critical point (algebraic singularity of $w^{-1}(z)$),
$z_\iota\in\Omega_z$, or 

$\bigcdot$
an ideal point (logarithmic singularity of $w^{-1}(z)$),
$z_\iota\in \partial_{\mathcal{I}} \Omega_z$.

\noindent
2.\ We can then obtain its corresponding singular value ${\tt w}_{{\tt j}(\iota)}=w(z_\iota)$.

\noindent
3.\ So, a branch point of $\R_{w(z)}$, namely the pair 
$(\alpha_{{\tt w}_{{\tt j}(\iota)} }, {\tt w}_{{\tt j}(\iota)})$,
can be naturally identified with the pair
\begin{equation*}
(z_\iota, {\tt w}_{{\tt j}(\iota)} ) =
\Bigg( 
\underbrace{z_\iota = \lim_{\tau \to \infty} \alpha_{{\tt w}_{{\tt j}(\iota)} } 
(\tau ) }_{\in \,\Omega_z \cup \partial_{\mathcal{I}}\Omega_z } \, , 
\underbrace{ {\tt w}_{{\tt j}(\iota)} = w(z_\iota)}_{\in \,\CW_w} 
\Bigg) .
\end{equation*}

\noindent 
4.\ Finally, a reordering of the singular points 
$\{ z_\iota \} \subset \Omega_z \cup \partial_{\mathcal{I}}\Omega_z$ 
so that the singular values ${\tt w}_{{\tt j}(\iota)}$ are grouped together, provides 
the relationship
\begin{equation}
{\tt j}(\iota) = \begin{cases}
\ 1 & \text{ for } \iota=1,\ldots,\mu_1, \\
\ 2 & \text{ for } \iota=\mu_1 +1, \ldots, \mu_1 + \mu_2, \\
\ \, \vdots \\
\ {\tt q} & \text{ for } \iota =
\mu_1+ \cdots  +  \mu_{{\tt q}- 1} +
1,\ldots,\delta ,
\end{cases}
\end{equation}

\noindent
between the singular values, identified by ${\tt j}(\iota)$, and the 
corresponding branch point identified by the unique index 
$\iota\in\{1,\ldots,\delta\}$.

\noindent
5.
Additionally, 
letting $m_{\iota} \in \NN \cup \{ \infty\}$ denote the ramification index of $w(z)$ 
at $(z_\iota, {\tt w}_{{\tt j}(\iota)} )$, we can use the triplet

\centerline{
$( z_\iota ,  {\tt w}_{{\tt j}(\iota)} , m_{\iota} )$,
}

\noindent
to easily distinguish between algebraic and logarithmic singularities of $w^{-1}(z)$.

\noindent
The triplet $(z_\iota, {\tt w}_{{\tt j}(\iota)}, m_{\iota})$ represents an
\begin{equation}
\label{eq:multiplicities-of-branch-points}
\begin{array}{rcccl}
\begin{array}{c}
\text{algebraic}
\\
\text{singularity of $w^{-1}(z)$}  
\end{array} 
 &  \iff  &
2\leq m_{\iota}<\infty 
& \iff & 
\begin{array}{c}
z_\iota \in \{ z_\kappa\} \subset \Omega_z
\\
\text{for some } \kappa \in \{ 1,\ldots,{\tt r}\},
\end{array}
\\
&&&&
\\
\begin{array}{c}
\text{logarithmic}
\\
\text{singularity of $w^{-1}(z)$}  
\end{array}
 &   \iff &  
m_{\iota}=\infty
& \iff &  
\begin{array}{c} 
z_\iota \in \{  z_\sigma \} \subset \partial_{\mathcal{I}}\Omega_z
\\ 
\text{for some }\sigma \in \{ 1,\ldots,{\tt p}\}.
\end{array}
\end{array}
\end{equation}

\end{remark}

\subsection{Speiser functions: notation for singular values and singular points}
Summarizing, 
the distinct singular values of $w(z)$ 
shall be denoted as
\begin{equation}
\label{eq:singular-values-details}
\mathcal{SV}_w=
\sum\limits_{{\tt j}=1}^{{\tt q}} \mu_{\tt j} {\tt w}_{\tt j} =
\big \{
({\tt w}_1, \mu_1 ), \ldots , 
({\tt w}_{\tt j} ,\mu_{\tt j} ), \ldots, 
({\tt w}_{\tt q}, \mu_{\tt q} )
\big \},
\end{equation}

\noindent
where $1\leq\mu_{\tt j}\leq\infty$ indicates the multiplicity of the singular value ${\tt w}_{\tt j}$.

\noindent
The singular points of $w(z)$ are 
\begin{equation}
\mathcal{SP}_w = \{ z_1, \ldots, 
z_\iota , \ldots ,
z_\delta \} 
\subset \Omega_z \cup \partial_{\mathcal{I}} \Omega_z,
\text{ where }
2 \leq \delta \leq \infty.
\end{equation}

\noindent
The corresponding branch points of $\R_{w(z)}$
are as follows.
\begin{multline}
\label{eq:enumerating-branch-points}
\mathcal{BP}_w  = 
\Big\{
\underbrace{(z_1,{\tt w}_1,m_{1}), 
(z_2, {\tt w}_1,m_{2}), 
\ldots (z_{\mu_1}, {\tt w}_1, m_{\mu_1})}_{\mu_1},
\hfill
\\
\underbrace{(z_{\mu_1 + 1},{\tt w}_2,m_{\mu_1 + 1}), 
(z_{\mu_1+2}, {\tt w}_2,m_{\mu_1+2}), 
\ldots (z_{\mu_1+\mu_2}, {\tt w}_2, m_{\mu_1+\mu_2})}_{\mu_2},
\\
\vdots
\\
\underbrace{(z_{\mu_1 +\ldots+\mu_{{\tt j}-1} + 1},{\tt w}_{\tt j} ,m_{\mu_1 + \ldots+\mu_{{\tt j}-1} + 1}), 
(z_{\mu_1 + \ldots + \mu_{{\tt j}-1}+2}, {\tt w}_{\tt j} ,m_{\mu_1 + \ldots+\mu_{{\tt j}-1}+2}), 
\ldots (z_{\mu_1 + \ldots+\mu_{\tt j}}, {\tt w}_{\tt j}, m_{\mu_1 + \ldots+\mu_{\tt j} })}_{\mu_{\tt j}},
\\
\vdots
\\
\underbrace{(z_{\delta-\mu_{\tt q} + 1}, {\tt w}_{\tt q}, m_{\delta-\mu_{\tt q}+ 1 }), 
(z_{\delta-\mu_{\tt q} + 2}, {\tt w}_{\tt q}, m_{\delta-\mu_{\tt q}+2 }),
\ldots (z_\delta, {\tt w}_{\tt q}, m_{\delta })}_{\mu_{\tt q}} 
\Big\}
\\
= 
\big\{
( z_1, {\tt w}_1, m_{1} ), \ldots, ( z_\iota, {\tt w}_{ {\tt j}(\iota) }, m_{\iota} ),\ldots, ( z_\delta, {\tt w}_{\tt q}, m_{\delta} )
\big\}
=\sum_{\iota=1}^{\delta} ( z_\iota, {\tt w}_{{\tt j}(\iota)}, m_{\iota} )\, .
\end{multline}

\noindent
Each type of branch point is identified by the value of its ramification index $m_{\iota}$,
as in \eqref{eq:multiplicities-of-branch-points}.
Our notation is:

\noindent
${\tt q}$ is the number of distinct singular values, 
\\
${\tt p}= \#\{m_{\iota}=\infty\}$ is the total number of logarithmic singularities of $w^{-1}(z)$,
\\
${\tt r} = \#\{2\leq m_{\iota} < \infty\}$ is the total number of algebraic singularities (counted with multiplicity), and hence 
\\
$\delta = {\tt p}+ {\tt r}$ 
is the total number of singularities of $w^{-1}(z)$. 

\noindent
Note that $2\leq {\tt q} <\infty$, 
$0\leq{\tt p}\leq \infty$, 
and $0\leq {\tt r} \leq \infty$.

\begin{remark}
\label{rem:val-asint-con-as}
As is usual in the literature,
we shall denote a singular value by ${\tt a}_{{\tt j}(\iota)}$ when we want to 
emphasize that it is an asymptotic value. 
Otherwise it will be denoted by ${\tt w}_{{\tt j}(\iota)}$.
\end{remark}

We provide some features for the simplest 
families of Speiser functions.

\begin{remark}[Rational functions]
For rational $R(z)$ only a finite number of algebraic singularities appear and no logarithmic singularities of $R^{-1}(z)$.
Of course, the algebraic singularities of the inverse are the critical points of the function.
\end{remark}

\subsection{$N$--functions: only a finite number of logarithmic singularities and no algebraic singularities}
\label{sec:N-functions}
The original definition of an $N$--function is due to R. Nevanlinna who considered functions $w(z)$ 
on $\CC_z$, that are solutions
to the Schwarzian differential equation 
\begin{equation}
\label{eq:Schwarzian-diff-eqn}
Sw\{ w, z \} = P(z),
\end{equation}

\noindent
 where $P(z)$ is a polynomial, and 

\centerline{
$Sw\{ f, z \} \doteq 
\frac{f'''(z)}{f'(z)} - 
\frac{3}{2} \left(\frac{f''(z)}{f'(z)} \right)^2$,
}

\noindent 
is the usual Schwarzian derivative.
It is a deep, and classical, result that the above is equivalent to having only a finite number of logarithmic singularities and no algebraic singularities, see 
\cite{Nevanlinna1}, \cite{Nevanlinna2} and
\cite{EremenkoMerenkov2}.
Since there are no critical values, then according to Remark \ref{rem:val-asint-con-as},
the $\tt q$ distinct singular values are 
asymptotic values denoted by:

\centerline{
$
\mathcal{AV}_w =
\sum\limits_{{\tt j} =1}^{\tt q} \mu_{\tt j}  {\tt a}_{\tt j}=
\{({\tt a}_1, \mu_1 ), \ldots , 
({\tt a}_{\tt j} ,\mu_{\tt j} ), \ldots, 
({\tt a}_{\tt q}, \mu_{\tt q}  ) \},
\ \ \
{\tt a}_{\tt j} \in \CW_w, \
{\tt q} \geq 2
$.
}

\noindent
The singular points are now asymptotic points denoted by

\centerline{
$\mathcal{AP}_w = \{ \infty_1,\ldots,\infty_\sigma,\ldots,\infty_{\tt p} \}
\subset \CC_z\backslash \{ \infty_1,\ldots,\infty_\sigma,\ldots,\infty_{\tt p} \}
$.
}

\noindent
The corresponding branch points are all infinitely ramified, so $m_{\sigma}= \infty$ for $\sigma \in \{1,\ldots,{\tt p} \}$:
\begin{equation*}
\mathcal{BP}_w =
\big\{
( \infty_1, {\tt a}_1, \infty ), \ldots, ( \infty_{\mu_1}, {\tt a}_1, \infty ),
\ldots, ( \infty_\sigma, {\tt a}_{ {\tt j}(\sigma)}, \infty ),
\ldots, 
( \infty_{{\tt p}-\mu_{\tt p}}, {\tt a}_{\tt q}, \infty ),\ldots, ( \infty_{\tt p}, {\tt a}_{\tt q}, \infty )
\big\}
=\sum_{\sigma=1}^{\tt p} 
( z_\sigma, {\tt a}_{{\tt j}(\sigma)}, \infty )\, ,
\end{equation*}

\noindent  
thus there are finitely many, namely
${\tt p} = \delta = \sum_{{\tt j}=1}^{\tt q} \mu_{\tt j}$
infinitely ramified branch points (logarithmic singularities of $w^{-1}(z)$).
Once again ${\tt q} < {\tt p}$ if and only if at least one $\mu_{\tt j} \geq 2$.

See Examples \ref{example:N-exp-tanh}.a, \ref{example:tessellation-Airy}.a, 
\ref{example:function-wM4alt}.a for $N$--functions with index ${\tt q}=2, 3, 3$ respectively.
For an example of a (non finite) Speiser function of index ${\tt q}=4$, that is
not an $N$--function,
consider Example \ref{example:tessellation-expsin}.a.
 
\section{Speiser Riemann surfaces}
\label{sec:Speiser-Riemann-surfaces}

As in Definition \ref{def:Speiser-q-function}, a \emph{Speiser Riemann surface} is 

\centerline{$
\R_{w(z)}= 
\left\{
\big(z,w(z) \big) \ \vert \  z \in \Omega_z \right\} 
\subset 
\Omega_z \times\CW_{w}$,}

\noindent
where $w(z)$ is a Speiser function with ${\tt q}$ singular values.

\noindent
Each  $\R_{w(z)}$ is simply connected 
with branch points as previously 
described in Diagram \eqref{diagramaRX}.
Roughly speaking, $\R_{w(z)}$ is the domain where the inverse function 
$w^{-1}(z)$ is single--valued.
In fact, considering Diagram \eqref{diagramaRX}, 
$\pi_1$ ``is the inverse'' of the Speiser function $w(z)$.

Recall that in $\R_{w(z)}$, the $2\leq\delta\leq\infty$ branch points can be described, when 
displayed in the notation of a ``divisor'', as
\begin{equation}
\sum_{\iota=1}^{\delta} ( z_\iota, {\tt w}_{{\tt j}(\iota)}, m_{\iota} ),
\quad
\text{ with } {\tt j}(\iota)\in\{1,\ldots,{\tt q} \}\, ,
\ {\tt q} \geq 2 \, , 
\end{equation}
so that ${\tt w}_{{\tt j}(\iota)}\in\CW_w$ indicates over which 
singular value the branch point lies over, and $2\leq m_{\iota} \leq \infty$, 
as in \eqref{eq:multiplicities-of-branch-points}, 
indicates the ramification index of the corresponding branch point.

\begin{remark}
\label{rem:ramindex-subindex}
1.\ In order to specify the location (in $\R_{w(z)}$) of a branch point, 
the ramification index $m_{\iota}$ is not needed, 
thus it will sometimes be omitted.

\noindent
2.\ 
Note that only one subindex, namely $\iota$, is needed to identify the branch points, however the
other indices are sometimes convenient for what follows. 
A shorthand notation for the branch point shall be 

\centerline{
$\circled{\iota} \doteq ( z_\iota, {\tt w}_{{\tt j}(\iota)}, m_{\iota} )$.}
\end{remark}

A natural question to ask about Riemann surfaces of meromorphic functions $w(z)$ is:

\begin{center}
{\it
Can $\R_{w(z)}$ be expressed in terms of maximal domains\\ 
of single--valued branches of $w^{-1}(z)$?
}
\end{center}

\noindent
In order to answer this we shall need the following.
\subsection{Surgery of Riemann surfaces}
\label{sec:Gluing Riemann surfaces}
In the Riemann surfaces
category, surgery tools are widely used, {\it v.g.} 
\cite{Strebel}\,p.\,56 ``welding of surfaces'', 
\cite{MR}, 
or \cite{Thurston} \S 3.2.--3.3 for general 
discussion on geometric structures.
Let $w(z)$ be a Speiser function, 
the singular complex analytic vector field 

\centerline{$X_{w(z)}(z) \doteq 
\dfrac{1}{w^\prime(z)} \dfrac{\partial }{ \partial z}$ }

\noindent 
is canonically associated to it;
see the ``Dictionary'' \cite{AlvarezMucino3}\,prop.\,2.5.
Moreover,
a complex analytic vector field $X$ on a Riemann 
surface $M$ has an associated singular flat metric $g_X$
on $M^0$, the surface minus the singular points of the metric. 
The real trajectories of $\Re{X}$ are unitary geodesics
on $(M^0, g_X)$; see 
\cite{AlvarezMucino1}\,lemma\,2.6 
and
the singular complex analytic dictionary
\cite{AlvarezMucino3}\,prop.\,2.5.  
Throughout the entire work 

\centerline{$(\CW_w, \del{}{w})$}

\noindent  
denotes the Riemann sphere furnished with the 
holomorphic vector field $\del{}{w}$. 
Equivalently, this pair denotes the flat Riemannian metric with a 
singularity at $\infty$ on $\CW_w$. 
The concepts of unitary geodesics,
euclidean segments and trajectories of 
the real vector field 
$\Re{\e^{i \theta} \del{}{w}} \doteq 
\cos(\theta)\del{}{x}+
\sin(\theta)\del{}{y}$ 
(which are circles through $\infty$ in $\CW_w$),
are used in the same way.

The use of vector fields allows us to isometrically glue Riemann surfaces, 
as in the following Corollary, whose proof can be found in the above references.

\begin{corollary}[Surgery of flat surfaces]
\label{cor:pegado-isometrico}  
Let $(M^0, g_X)$, $(N^0, g_Y)$ be two flat surfaces arising 
from two singular complex analytic vector fields $X$ and $Y$. 
Assume that both spaces $M^0$, $N^0$
have as geodesic boundary components of the same length: the trajectories
$\sigma_1 (\tau)$, $\sigma_2(\tau)$ of $\Re{X}$ and $\Re{Y}$, 
$\tau \in I \subset \RR$,
respectively. Then,
the isometric glueing of them 
along these geodesic boundary,  
is well defined, and provides a new flat surface 
on $M^0 \cup N^0$ arising from a new complex analytic vector field
$Z$ that extends $X$ and $Y$.
\hfill $\Box$
\end{corollary}

The notion of a segment of $(\CW_z, \del{}{w})$
passing through the singular point $\infty$ will be useful.

\begin{definition}
\label{def:geodesic-segments}
Given two distinct points ${\tt w}_\alpha, {\tt w}_\beta \in \CW_w$, 
a \emph{geodesic segment  
in $(\CW_w, \del{}{w})$} is defined as follows.

\begin{enumerate}[label=\arabic*),leftmargin=*] 
\item 
If ${\tt w}_\alpha, {\tt w}_\beta \in \CC_w$, as
\begin{enumerate}[label=\roman*),leftmargin=*]

\item 
the oriented straight line segment $\overline{{\tt w}_\alpha {\tt w}_\beta} \subset \CC_w$, 
or 

\item
the oriented arc of a circle in $\CW_w$, 
starting at ${\tt w}_\alpha$, 
passing through $\infty$ and 
ending at ${\tt w}_\beta$;
it is denoted
by $\overline{{\tt w}_\alpha\, \infty\, {\tt w}_\beta}$.
\end{enumerate}

\noindent
Note that 
$\overline{{\tt w}_\alpha\, \infty\,  {\tt w}_\beta}
\cup \overline{{\tt w}_\beta {\tt w}_\alpha} $
is the unique circle in $\CW_w$ passing through 
${\tt w}_\alpha$, $\infty$ and ${\tt w}_\beta$.

\item
If ${\tt w}_\alpha = \infty$ and ${\tt w}_\beta \in\CC_w$, as 
one of 
the oriented arcs of a circle in $\CW_w$ with $\Im{w}=\Im{{\tt w}_\beta}$,
denoted by ${}_{\pm}\overline{\infty {\tt w}_\beta}$.

\item
If ${\tt w}_\alpha \in \CC_w$ and ${\tt w}_\beta = \infty$, as 
one of 
the oriented arcs of a circle in $\CW_w$ with 
$\Im{w}=\Im{{\tt w}_\alpha}$, 
denoted by ${}_{\pm}\overline{ {\tt w}_\alpha \infty}$.
\end{enumerate}
\end{definition}

Note that for any pair
${\tt w}_{\alpha}, \, {\tt w}_{\beta}$  
there are two choices of geodesics segments between them.

As usual, a \emph{branch cut}
is the operation of removing from $\CW_w$ 
a geodesic segment 
$\overline{ {\tt w}_{{\tt j}( \msigma )} 
{\tt w}_{{\tt j}( \mrho )} }$ in $(\CW_w, \del{}{w})$. 

\begin{definition}
\label{def:sheet-branchcut}
A \emph{sheet with branch cuts} is 
\begin{equation}
\label{eq:sheet}
\mathfrak{L}_\Upxi = 
\CW_w \big\backslash \, 
\Bigg(\bigcup\limits_{
{\tt j}( \msigma ) , {\tt j}( \mrho ) \in  \Upxi } 
\overline{ {\tt w}_{{\tt j}( \msigma )} 
{\tt w}_{{\tt j}( \mrho )} } \Bigg)  \, ,
\end{equation}

\noindent 
such that 

\begin{enumerate}[label=\roman*),leftmargin=*] 
\item
the subindex
$\Upxi=\{
\overline{ {\tt w}_{{\tt j}( \msigma )} 
{\tt w}_{{\tt j}( \mrho )} } 
\ \vert \
{\tt j}( \msigma ) \neq {\tt j}( \mrho )
\} 
\neq \emptyset$ 
enumerates the particular collection of branch cuts,

\item
the geodesic segments
$\overline{ {\tt w}_{{\tt j}( \msigma )} {\tt w}_{{\tt j}( \mrho )} }$,
$\overline{ {\tt w}_{{\tt j}( \msigma^\prime )} {\tt w}_{{\tt j}( \mrho^\prime )} }$ 
intersect at most at their endpoints,

\item
$\mathfrak{L}_\Upxi$ is simply connected.
\end{enumerate}
\end{definition}

In particular if both ${\tt w}_{{\tt j}( \msigma )}, {\tt w}_{{\tt j}( \mrho )} \in\CC_w$ we use the usual geodesic segment with length
$\abs{ {\tt w}_{{\tt j}( \msigma )} - {\tt w}_{{\tt j}( \mrho )} }$.
In case either 
${\tt w}_{{\tt j}( \msigma )}$ or ${\tt w}_{{\tt j}( \mrho )}$ is $\infty\in\CW_w$, 
we shall choose $\overline{ {\tt w}_{{\tt j}( \msigma )} {\tt w}_{{\tt j}( \mrho )} }$ 
as an appropriate geodesic segment in
$(\CW_w, \del{}{w})$ that ensures that condition (iii) of Definition \ref{def:sheet-branchcut} is satisfied.

\begin{lemma}
Given a distinct set $\{{\tt w}_1,\ldots,{\tt w}_{\tt q} \}\subset\CW_w$ of ${\tt q}$ singular values, 
there are a finite number of \emph{types} $\mathfrak{L}_\Upxi$
(\emph{i.e.}\ non isometric sheets with branch cuts, 
understood as translation surfaces) 
that can be formed.
\end{lemma}

\begin{proof}
Let $K_{\tt q}$ be the complete graph with $\tt q$ vertices formed by the set $\{{\tt w}_1,\ldots,{\tt w}_{\tt q} \}\subset\CW_w$ of distinct $\tt q$ values.
Clearly the number of possible sheets is less than the number of subgraphs of $K_{\tt q}$, which is finite.
\end{proof}

\begin{definition}
\label{def:diagonal}
\
\begin{enumerate}[label=\arabic*),leftmargin=*] 
\item
A segment $\Delta_{\vartheta \msigma \mrho}\subset\R_{w(z)}$ 
\emph{is a diagonal of $\R_{w(z)}$} \,
when 

\begin{enumerate}[label=\roman*),leftmargin=*] 
\item 
the projection 

\centerline{
$\pi_2(\Delta_{\vartheta \msigma \mrho}) = 
\overline{{\tt w}_{{\tt j}(\msigma)} {\tt w}_{{\tt j}(\mrho)}}$
}

\noindent 
is a geodesic segment in $(\CW_w, \del{}{w})$,

\item
the interior 
of $\pi_{1}(\Delta_{\vartheta \msigma \mrho})$ is in 
$\Omega_{z}$, and 

\item
the endpoints, $z_{\msigma}$ and $z_{\mrho}$, of
$\pi_{1}(\Delta_{\vartheta \msigma \mrho})$
are algebraic or logarithmic singularities of $w^{-1}(z)$.
\end{enumerate} 

\item
An oriented diagonal $\Delta_{\vartheta \msigma \mrho}$
\emph{starts at the branch point
$\circled{\msigma}=(z_{\msigma}, {\tt w}_{{\tt j}(\msigma)}, m_{\msigma})$ 
and ends at the branch point
$\circled{\mrho}=(z_{\mrho}, {\tt w}_{{\tt j}(\mrho)}, m_{\mrho})$}.
In this case, we shall say that the branch points,
$\circled{\msigma}$ and 
$\circled{\mrho}$,
\emph{share the sheet identified\footnote{
In this case, the index $\vartheta$ enumerates the 
sheets that share the branch points $\circled{\msigma}$ and 
$\circled{\mrho}$. 
Also if two diagonals have the same $\vartheta$ then their 
corresponding branch points share the same sheet.
} 
by the diagonal $\Delta_{\vartheta \msigma \mrho}$ in $\R_{w(z)}$.}
\end{enumerate}
\end{definition}

From the above definitions, notation, 
and the repeated use of isometric glueing
(\emph{i.e.}\ Corollary \ref{cor:pegado-isometrico}) between sheets along their branch cuts,
the following result is clear.

\begin{proposition}[Decomposition of $\R_{w(z)}$ into maximal domains of single--valuedness]
\label{prop:Rw-is-a-union-of-sheets}
Let $w(z)$ be a Speiser function with $\tt q \geq 2$ 
distinct singular values.

\begin{enumerate}[label=\arabic*),leftmargin=*]

\item 
The Riemann surface $\R_{w(z)}$ 
associated to $w(z)$
can be constructed by 
isometric glueing of sheets, denoted by $\sim$,  as follows
\begin{equation}
\label{eq:Rw-decomposition-in-sheets}
\R_{w(z)} 
= \left[ \bigcup_\vartheta \mathfrak{L}_{\Upxi,\vartheta} \right]
\Big/ \sim \ = 
\left[
\bigcup_\vartheta \left( \CW_w \big\backslash \, 
\bigg(\bigcup\limits_{
{\tt j}( \msigma ) , {\tt j}( \mrho ) \in  \Upxi 
} 
\overline{ {\tt w}_{{\tt j}( \msigma )} 
{\tt w}_{{\tt j}( \mrho )} } \bigg)
\right)_\vartheta 
\right]  \, \Bigg/ \sim ,
\end{equation}

\noindent
In the above expression the following conventions are observed.

\noindent
$\bigcdot$
The singular values 
of $w(z)$
are denoted by $ \{{\tt w}_{{\tt j} (\iota) } \}_{\iota=1}^\delta $,
recall Equation \eqref{eq:enumerating-branch-points}.

\noindent
$\bigcdot$
The index $\Upxi$
indicates  the type of sheet, {\it i.e.}\ 
the particular collection  

\centerline{
$\{ \overline{ {\tt w}_{{\tt j}( \msigma ) } 
{\tt w}_{{\tt j}( \mrho )} } \ \vert\ 
{\tt j}( \msigma ) \neq {\tt j}( \mrho )\}$ 
}

\noindent
that is considered on each sheet $\vartheta$,
a finite number of 
types of sheets 
$\mathfrak{L}_\Upxi$ appears.

\item 
The number of sheets in 
\eqref{eq:Rw-decomposition-in-sheets} is 
$n=\abs{ \{\vartheta\} }$, with $2 \leq n \leq \infty$.
Note that $n < \infty$ if and only if 
$w(z)$ is a rational function of 
degree $n \geq 2$. 
\end{enumerate}
\hfill 
\qed
\end{proposition}

Note that the decomposition is in no way unique.
Also note that this answers the question posed at the end of last subsection. 

\begin{corollary}
\label{cor:sheet-is-maximal-domain}
Each sheet $\mathfrak{L}_{\Upxi}$ is a maximal domain in which $w^{-1}(z)$ admits a single--valued branch.
\qed
\end{corollary}

\begin{remark}[The relevance of the sheets and of the decomposition of $\R_{w(z)}$]
Even though the sheets $\mathfrak{L}_{\Upxi}$, as in \eqref{def:sheet-branchcut}, 
appear to be very innocent (they are only copies of the Riemann sphere with certain 
branch cuts),
when they are ``mounted'' on the Riemann surface $\R_{w(z)}$ they gain relevance
in the sense that: 

\noindent
$\bigcdot$
they are now maximal domains where $w^{-1}(z)$ admits a 
single--valued branch,

\noindent
$\bigcdot$
they also inherit the vector field structure of $\R_{w(z)}$ provided by $\pi_2^* \del{}{w}$ 
(see Diagram \eqref{diagramaRX}).
\end{remark}

For examples of decomposition of $\R_{w(z)}$ in sheets, 
see Example \ref{example:N-exp-tanh}.b for the simplest case of an $N$--function, 
Examples \ref{example:tessellation-Airy}.b and \ref{example:function-wM4alt}.b 
for a presentation of two non--trivial cases of $N$--functions when ${\tt q} = 3$;
Examples \ref{example:tessellation-expsin}.b and \ref{example:sinexpsin}.b 
show Speiser functions of index ${\tt q}=4$ that are not $N$--functions.

\smallskip
\noindent
The Riemann surface $\R_{w(z)}$ of a Speiser--function $w(z)$, has:
\begin{itemize}[label=$\bigcdot$ ,leftmargin=*]
\item 
$0\leq{\tt r}\leq\infty$ finitely ramified branch points, and 

\item 
$0\leq{\tt p}\leq\infty$ infinitely ramified branch points 
\end{itemize}

\noindent
over ${\tt q}$ distinct singular values (with 
$2\leq {\tt q} \leq {\tt r} + {\tt p}$).

\begin{remark}[Features of some families of functions]
1.
Let $w(z)=R(z)$ be a rational function of degree $n\geq 2$.
The surface $\R_{w(z)}$ has 
${\tt r}<\infty$ 
finitely ramified branch points 
over $2\leq {\tt q} \leq 2n-2$ 
distinct critical values.

\noindent 
2.
Consider a polynomial $w(z)=P(z)$ of degree $n\geq 2$, then
$\R_{w(z)}$ has a finitely ramified branch point of ramification 
index $n-1$ over $\infty\in\CW_w$ and whose projection 
via $\pi_1$ is $\infty\in\CW_z$, 
and up to $n-1$ finitely ramified branch points over 
${\tt q}-1$ distinct finite critical values.

\noindent 
3.
Given an $N$--function $w(z)$, its surface
$\R_{w(z)}$ has:
\begin{itemize}[label=$\bigcdot$ ,leftmargin=*]

\item
${\tt p}<\infty$ 
infinitely ramified branch points (logarithmic singularities of $w^{-1}(z)$) over 
${\tt q} \leq {\tt p}<\infty$
distinct asymptotic values, and 

\item
no finitely ramified branch points (algebraic singularities of $w^{-1}(z)$). 
\end{itemize}
\end{remark}

\section{Schwarz--Klein--Speiser tessellations}   
\label{sec:tessellations}

We roughly follow the classical works of 
H. A. Schwarz \cite{Schwarz}, 
F. Klein \cite{Klein}, \cite{Chislenko-Tschinkel},
R. Nevanlinna \cite{Nevanlinna2}\,ch.\,XI\,\S2, and 
A. Speiser \cite{Speiser}.
However, we make some precisions\footnote{
Regarding the vertices of infinite valence.}
that we consider improve the presentation and our understanding.
We develop, in an axiomatic way, tessellations, their associated graphs and certain labellings associated to them.

Let $\Omega_z$ be a simply connected Riemann surface.
For our purposes, we shall require the compactifications of $\Omega_z$
provided by Proposition \ref{prop:dos-tipos-singularidades}.2.ii--iii. 
Namely,

\centerline{
$(\Omega_z \cup \partial_{\mathcal{I}} \Omega_z) 
\doteq 
\begin{cases}
\CW_z 
\\
\CC_z\cup\{\infty_1,\ldots,\infty_{\tt p} \} 
\\
\Delta_z \cup_{\sigma =1}^\infty  \{ \e^{i \theta_{\sigma} } \}
\end{cases}$}

\noindent
as the case may arise.

\begin{definition}
\label{def:de-teselacion}
\
\begin{enumerate}[label=\arabic*),leftmargin=*] 
\item
A  \emph{tessellation} of a
surface $\Omega_z$ is a collection 
of alternating colored tiles
\begin{equation}
\label{teselacion}
\mathscr{T}
=
\underbrace{
T_1 \cup \ldots \cup T_\alpha \cup \ldots  }_{ n\text{ blue tiles} }
\cup 
\underbrace{ T^\prime_1 \cup  
\ldots  \cup T^{\prime}_\alpha  \cup \ldots }_{n \text{ grey tiles}} \subset 
\Omega_z \cup \partial_{\mathcal{I}} \Omega_z, 
\ \ \ 
2 \leq n \leq \infty,
\end{equation}

\noindent
where the \emph{tiles}
$\{ T_{\alpha}, \, T_{\alpha}^\prime \}$
are open Jordan domains, 
such that: 

\begin{enumerate}[label=\roman*),leftmargin=*]
\item
The union of their closures 
$\cup_{{\alpha}}
\big( \overline{T_{\alpha}} 
\cup
\overline{T_{\alpha}^\prime } \big)$
is $\Omega_z \cup \partial_{\mathcal{I}} \Omega_z$.

\item 
The boundary of the closure of each tile 
$\partial\overline{T_\alpha}$ (resp. 
$\partial\overline{ T_{\alpha}^\prime}$) has 
$\rho$ vertices and $\rho$ edges, where $2 \leq \rho \leq {\tt q}$
and $\rho$ depends on the particular tile.

\item
If the intersection of the closures of any
two tiles is non--empty, 
then it consists of a finite number of 
simple paths (edges) and their extreme points (vertices).
\end{enumerate}

\item
If all the tiles $\{ T_{\alpha} \}$
(resp. $\{T_{\alpha}^\prime \}$) have 
the same number of vertices and edges,
we shall say that the tessellation is \emph{homogeneous}.
\end{enumerate}
\end{definition}

Note that, if $n<\infty$, then 
a  tessellation $\mathscr{T}$
has $n$ blue tiles and $n$ grey tiles; this is 
called the {\it global balance condition} in
\cite{Koch-Lei}, see \S \ref{sec:planar-graphs-Thrurston} for
further discussion. 
In the case of $n=\infty$, we shall say that the tessellation $\mathscr{T}$ satisfies
the global balance condition if the cardinality of the blue
and gray tiles are equal.
By looking at the boundaries of the tiles,
say $\partial \overline{T_\alpha}, \,  
\partial \overline{T_{\alpha}^\prime}$,
a tessellation $\mathscr{T}$ 
determines an underlying graph $\Gamma$.

\begin{definition}
\label{def:de-t-graph}
A {\it ${\tt t}$--graph} $\Gamma$ is an oriented 
connected graph 
embedded in $(\Omega_z \cup \partial_{\mathcal{I}} \Omega_z)$, 
with vertices $V(\Gamma)$ of infinite or
even valence 
greater than or equal to $4$  
and edges $E (\Gamma)$, such that:

\begin{enumerate}[label=\roman*),leftmargin=*]
\item
\begin{equation}
\label{eq:tessellation-t-map}
\begin{array}{rl}
\mathscr{T}(\Gamma) \doteq & 
(\Omega_z \cup \partial_{\mathcal{I}} \Omega_z) \backslash \Gamma
\\
& \vspace{-.3cm}
\\ 
= &
\underbrace{
T_1 \cup \ldots \cup T_\alpha \cup \ldots  }_{ n\text{ blue tiles} }
\cup 
\underbrace{ T^\prime_1 \cup  
\ldots  \cup T^{\prime}_\alpha  \cup \ldots }_{n \text{ grey tiles}} \subset 
\Omega_z \cup \partial_{\mathcal{I}} \Omega_z, 
\ \ \ 
2 \leq n \leq \infty, 
\end{array}
\end{equation}

\noindent
is a tessellation,   
as in Definition \ref{def:de-teselacion}. 

\item
Each blue tile $T_\alpha$ is on
the left side of the oriented edges of $\Gamma$.

\item
If there are vertices $V(\Gamma)$ with 
infinite valence, they are on the ideal boundary 
$\{\infty_1,\ldots,\infty_{\tt p} \}$ or 
$\partial_{\mathcal{I}} \Delta_z$. 
Moreover, there are no finite valence vertices on the ideal boundary.
\end{enumerate}
\end{definition}

With the above in mind, a tessellation $\mathscr{T}$ and a 
${\tt t}$--graph $\Gamma$ are essentially 
equivalent objects, where the alternating 
colouring in Equation 
\eqref{teselacion} corresponds 
to the orientation of the edges in Definition 
\ref{def:de-t-graph}. 
In simple words, 
a ${\tt t}$--graph must be understood as the simplest
oriented graph describing a tessellation. 

The tessellations arising from complex analytic functions
are homogeneous and require a more accurate notion, as follows. 

\begin{definition}
\label{definicion-de-A-map}
\begin{enumerate}[label=\arabic*),leftmargin=*] 
\item
An {\it ${\tt A}$--map} $\widehat{\Gamma}_{\tt q}$ is an 
oriented, connected graph 
embedded in  $\Omega_z \cup \partial_{\mathcal{I}} \Omega_z$, 
with vertices $V(\widehat{\Gamma}_{\tt q})$ of
infinite or 
even valence  greater than or equal to $2$
and edges $E (\widehat{\Gamma}_{\tt q})$, such that:

\begin{enumerate}[label=\roman*),leftmargin=*]
\item
The subset of vertices of valence greater
than or equal to 4 is non empty. 

\item
If we forget all the vertices
of valence 2 of $\widehat{\Gamma}_{\tt q}$, then
we obtain a ${\tt t}$--graph $\Gamma$ such that:
\begin{equation}
\label{eq:tessellation}
\mathscr{T}(\widehat{\Gamma}_{\tt q}) \doteq 
\mathscr{T}(\Gamma)
\end{equation}

\noindent
is, set theoretically, a tessellation as in Definition 
\ref{def:de-teselacion}.
\end{enumerate}

\item 
The boundary 
$\partial \overline{T_\alpha}$ 
(resp. $\partial \overline{T_{\alpha}^\prime}$) of  
each tile consists of exactly $\tt q$ vertices and $\tt q$ edges 
of $\widehat{\Gamma}_{\tt q}$,
{\it i.e.} the tessellation $\mathscr{T}(\widehat{\Gamma}_{\tt q})$ is homogeneous.

\item 
We shall say that $\Gamma$ and $\widehat{\Gamma}_{\tt q}$, as in (ii) above, are \emph{compatible}.
\end{enumerate}
\end{definition}

The {\it forgetting vertices operation}
in part (ii) above is as follows. 
We consider a vertex, say $z_1=0$ of valence 2 
in $\widehat{\Gamma}_{\tt q}$ and its two adjacent edges, say $(-1,0)$ and $(0,1)$, 
thus we have
$(-1,0) \cup \{0\} \cup (0,1)$.
The operation of forgetting
the vertex $0$ replaces the above by a unique
edge $(-1,1)$. 

\begin{example}
In Figure \ref{fig:contraejemplos} we show three tessellations corresponding to $\tt t$--graphs
$\Gamma$ that do not represent Speiser functions on $\CW_z$.
For Figure \ref{fig:contraejemplos}.a this follows by 
observing that it only 
has one branch point. 
For Figures \ref{fig:contraejemplos}.b--c this will be shown 
in \S\ref{sec:planar-graphs-Thrurston}.

\begin{figure}[h!]
\begin{center}
\includegraphics[width=0.7\textwidth]{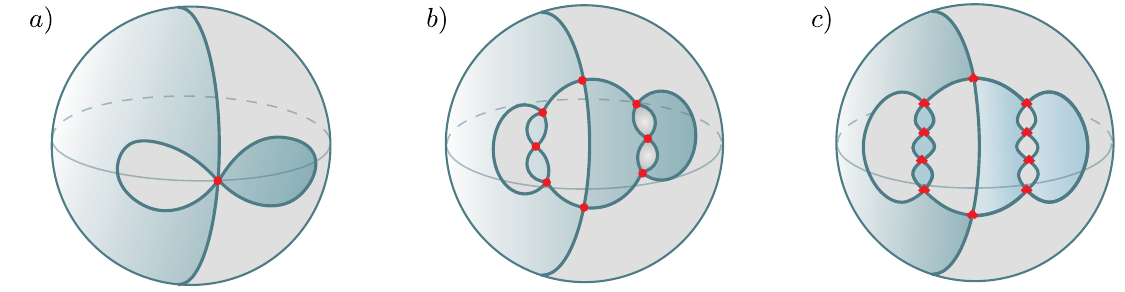}
\caption{
Examples of $\tt t$--graphs $\Gamma$, whose corresponding tessellations 
$\mathscr{T}(\Gamma)$ do not represent Speiser functions.}
\label{fig:contraejemplos}
\end{center}
\end{figure}
\end{example}

\begin{example}
Consider the non generic rational function 

\centerline{
$R(z)=\dfrac{z(z^2-1)(z^2-4)}{(z-3)}$
}

\noindent
of degree 5. 
It has 6 critical points (5 simple ones located on the plane 
and an multiplicity 4 critical point at
$\infty\in\CW_z$), and 6 critical values
$\mathcal{SV}_R=\{ {\tt w}_1, {\tt w}_2, {\tt w}_3, {\tt w}_4, {\tt w}_5, {\tt w}_6=\infty \}$ 
lying on $\gamma=\RR\cup\{\infty\}$.
The inverse image of $\gamma$ under $R(z)$, namely
$\Gamma= R^{-1} (\RR \cup \{ \infty\})$ is a $\tt t$--graph with tessellation 
$\mathscr{T}(\Gamma)$ as in Figure \ref{fig:tres-mosaicos-etiquetas-h}.a--b.
Figure \ref{fig:tres-mosaicos-etiquetas-h}.c shows an $\tt A$--map $\widehat{\Gamma}_6$,
constructed by edge subdivision of $\Gamma$ (hence they are compatible), 
and its corresponding homogeneous tessellation
$\mathscr{T}(\widehat{\Gamma}_6) = \CW_z \backslash \ \widehat{\Gamma}_6$ whose
tiles are 6--gons with two types of vertices: red vertices of valence greater than or equal to 4
corresponding to the critical points of $R(z)$, and green vertices of valence 2 corresponding 
to the cocritical points of $R(z)$.
All tiles are $6$--gons having labelled vertices 
with cyclic order $\mathcal{W}_{6}=[{\tt w}_1, \ldots , {\tt w}_6]$.

\begin{figure}[h!]
\begin{center}
\includegraphics[width=0.9\textwidth]{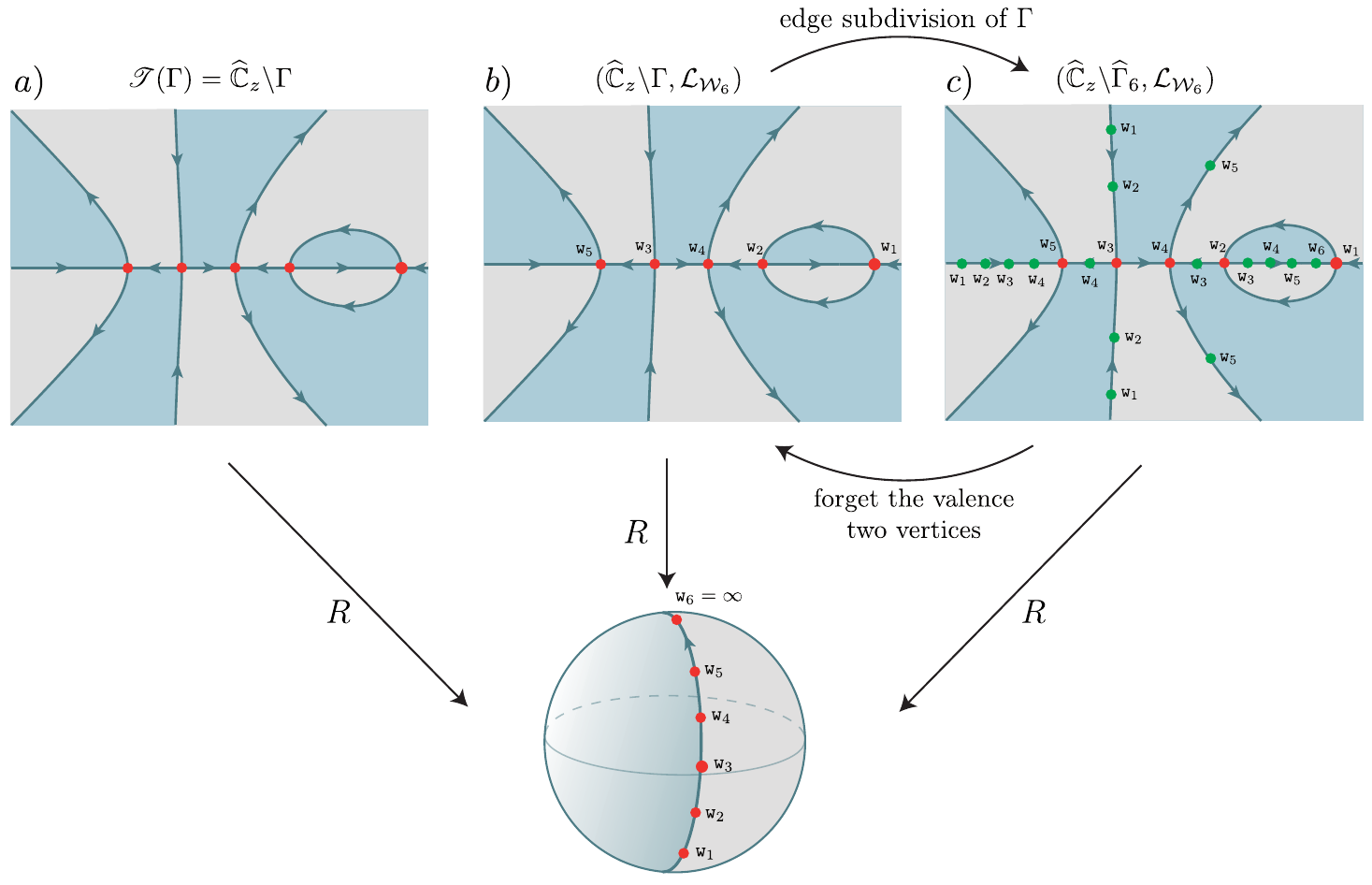}
\caption{
Affine view of the tessellation of the 
non generic rational function $R(z)={z(z^2-1)(z^2-4)}/{(z-3)}$ of degree 5. 
It has 6 critical points, one of them being $\infty\in\CW_z$ and 6 critical values
$\mathcal{SP}_R=\{ {\tt w}_1, {\tt w}_2, {\tt w}_3, {\tt w}_4, {\tt w}_5, {\tt w}_6=\infty \}$ 
lying on $\gamma=\RR\cup\{\infty\}$.
a) 
The ${\tt t}$--graph $\Gamma= R^{-1} (\RR \cup \{ \infty\})$ and 
its non homogeneous tessellation
$\mathscr{T}(\Gamma)$. 
b) 
The ${\tt t}$--graph $\Gamma$ with consistent $6$--labelling 
$\mathcal{L}_{\mathcal{W}_6}: V(\Gamma) 
\longrightarrow \mathcal{W}_6$, where $\mathcal{L}_{\mathcal{W}_6}(\infty)= {\tt w}_6$. 
c) 
The ${\tt A}$--map $\widehat{\Gamma}= R^* \gamma$, 
its homogeneous tessellation
$\mathscr{T}(\widehat{\Gamma}_6)$,
and its
consistent $6$--labelling 
$\mathcal{L}_{\mathcal{W}_6}$; each tile 
is a $6$--gon, with
vertices at the (red) critical points $\mathcal{SP}_R$, 
the point $\infty\in\CW_z$ (which has label ${\tt w}_6$), and
the (green) cocritical points $\mathcal{CS}_R$.
This figure appears as figure 1 of \cite{AlvarezGutierrezMucino} with slightly different notation.
}
\label{fig:tres-mosaicos-etiquetas-h}
\end{center}
\end{figure} 
\end{example}

\begin{example}
In Figure \ref{fig:SinExpSin}.a we observe a tessellation 
$\mathscr{T}(\widehat{\Gamma}_4)$ 
corresponding to the Speiser function of index 4

\centerline{$w(z) = \sin (z) \exp (\sin (z))$.}

\noindent
The $\tt t$--graph $\Gamma$ consists of: the black edges, 
an infinite number of red vertices of valence 4 or 8 (corresponding to the real critical points of $w(z)$
with critical values $\e$ and $-\e^{-1}$), 
and an infinite number of vertices ``at infinity'' of infinite valence
(ideal points $\{\infty_1, \infty_2,\ldots,\infty_{\tt p},\ldots \}$ in the non--Hausdorff 
compactification $\CC_z \cup \{\infty_1,\infty_2,\ldots,\infty_{\tt p},\ldots \}$ 
corresponding to the 
logarithmic singularities of $w^{-1}(z)$ with asymptotic values $0$ and $\infty$).
A compatible $\tt A$--map is obtained by edge subdivision consisting 
of adding an infinite number of green vertices of valence 2 at the cosingular points (with cosingular 
values $0$, $\e$ and $-\e^{-1}$).
See Example \ref{example:sinexpsin} for further details.
\end{example}

\begin{definition}
\label{def:orden-ciclico-y-losetas}
\
\begin{enumerate}[label=\arabic*),leftmargin=*] 
\item
Consider ${\tt q}\geq 2$ distinct values $\{ {\tt w}_\ell \}_{ \ell=1}^{\tt q} \subset \CW_w$, 
and assign them a \emph{cyclic order}, say  

\centerline{$
\mathcal{W}_{\tt q} =
[{\tt w}_1, \ldots, 
{\tt w}_{\tt j}, \ldots,
{\tt w}_{\tt q} ] $;
}

\noindent 
further, consider a representative 
$\gamma$ of the isotopy class of Jordan paths relative to 
the $\tt q$ distinct values traversed in the order given above. Thus

\centerline{ 
$\gamma \subset \CW_w$ 
\ runs through \ 
$\mathcal{W}_{\tt q}$ \ (in the chosen order).
}

\noindent 
The isotopy class $[\gamma]$ realizes the above 
\emph{cyclic order $\mathcal{L}_\gamma$ for the $\tt q$ distinct values}.

\item
The path $\gamma$ determines a \emph{trivial tessellation} of the sphere

\centerline{
$\mathscr{T}(\gamma)  = 
\CW_w \backslash \gamma = T \cup T^\prime$}

\noindent
with two tiles (which are topological $\tt q$--gons), 
the \emph{blue tile $T$} is on the left side of $\gamma$, 
the \emph{grey tile $T^\prime$} is on the right side of $\gamma$.
\end{enumerate}
\end{definition}

\begin{remark}[Cyclic order]
By definition, the cyclic order $\mathcal{L}_\gamma$ 
and the cyclic order of the $\tt q$ distinct values 

\centerline{$\mathcal{W}_{\tt q}=[ {\tt w}_{1}, \ldots, {\tt w}_{\tt j}, \ldots, {\tt w}_{\tt q} ] $}

\noindent
are to be thought of as equivalent.
\end{remark}

\begin{remark}[Graph and geodesic structures on $\gamma$]
\label{rem:gamma-poligonal-y-grafica}
1. 
As a graph, $\gamma$ is a cyclic graph with ${\tt q}$ ordered vertices,
namely 
$[{\tt w}_{1}, \ldots, {\tt w}_{\tt j}, \ldots, 
{\tt w}_{\tt q} ] \subset\CW_w$, 
and the respective segments 
$\{ \overline{{\tt w}_{\tt j} {\tt w}_{{\tt j}+1}} \} \subset \gamma $
as edges. 

\noindent 2.
Moreover, 
when it is convenient, one may choose
$\gamma$ as a polygonal with ${\tt q}$
geodesic segments 

\centerline{ 
$\overline {{\tt w}_1 {\tt w}_2 } \cup
\ldots \cup
\overline {{\tt w}_{{\tt q}-1} {\tt w}_{\tt q} } \cup
\overline {{\tt w}_{\tt q} {\tt w}_1 }
$.}
\end{remark}

The orientation of $\widehat{\Gamma}_{\tt q}$ is 
inherited by the cyclic order $\mathcal{L}_\gamma$, \emph{i.e.}\ anticlockwise
for the blue tiles $T_\alpha$ of the tessellation.

After Theorem \ref{teo:principal}, the
name ${\tt A}$--map for $\widehat{\Gamma}_{\tt q}$ 
must be understood as a 
coarse abbreviation of
``complex analytic function''.

By condition (iii) of Definition \ref{def:de-t-graph}, 
$\Gamma$ or $\widehat{\Gamma}_{\tt q}$ are on 
$\Omega_z=\CW$ if and only if  
they do not have vertices of infinite valence.

\begin{definition}
\label{def:etiquetado-consistente}
A {\it consistent $q$--labelling 

\centerline{$\mathcal{L}_{\mathcal{W}_{\tt q}}: V(\Gamma) \longrightarrow \mathcal{W}_{\tt q}, 
\ \ \
{\tt q} \geq 2$, } 

\noindent 
for a ${\tt t}$--graph $\Gamma$} satisfies the following conditions:

\begin{enumerate}[label=\roman*),leftmargin=*]
\item
For each blue tile $T_\alpha$ of the tessellation
$\mathscr{T}(\Gamma)$, 
if $\{ z_\iota \}$ are the vertices of its
boundary $\partial \overline{T_\alpha}$, ordered 
with cyclic anti--clockwise sense, then the 
labels (values) $\{ \mathcal{L}_{\mathcal{W}_{\tt q}} ( z_\iota ) \} \subset \mathcal{W}_{\tt q}$
appear exactly once and with the same cyclic 
order provided by $\mathcal{L}_\gamma$.

\item
Each label (value) ${\tt w}_{j} \in \mathcal{W}_{\tt q}$ appears under
$\mathcal{L}_{\mathcal{W}_{\tt q}}$ for at least one
vertex $z_\iota \in V(\Gamma)$ of
$\Gamma$, which by definition have
valence greater than or equal to $4$.
\end{enumerate}  
\end{definition}

\begin{remark}[Consistent $\tt q$--labelling for ${\tt A}$--maps]
\noindent 1.\ The notion of consistent $\tt q$--labelling for 
a ${\tt t}$--graph $\Gamma$ extends to any 
compatible ${\tt A}$--map
$\widehat{\Gamma}_{\tt q}$, with a notable distinction:

\noindent
$\bigcdot$ all the labels 
of $\mathcal{W}_{\tt q}$ appear on the vertices  
of each blue tile $T_\alpha$, since all the tiles of 
$\mathscr{T}(\widehat{\Gamma}_{\tt q})$ are $\tt q$--gons.

\noindent 2.\ On the other hand, for 
a $\tt t$--graph $\Gamma$ some labels
of $\mathcal{W}_{\tt q}$ are usually hidden  in the boundary
of each blue tile, since the tiles of the tessellation 
$\mathscr{T}(\Gamma)$ can be 
$\rho$--gons, for $2 \leq \rho \leq {\tt q} $,
see example 2 and figure 1 of \cite{AlvarezGutierrezMucino}.
By abuse of notation, we use the notion of consistent 
$\tt q$--labelling for ${\tt t}$--graphs and ${\tt A}$--maps.
\end{remark}

A precise statement for the above remark is as follows.

\begin{lemma}[Consistent $\tt q$--labellings for $\Gamma$ and $\widehat{\Gamma}_{\tt q}$]
\label{lem:consistent-labelings-Gamma-hatGamma}
Let $\Gamma$ be a $\tt t$--graph and $\widehat{\Gamma}_{\tt q}$ an $\tt A$--map that
are compatible 
($\Gamma$ can be obtained from $\widehat{\Gamma}_{\tt q}$ by forgetting vertices of valence 2).

\noindent
1.\ If $\Gamma$ supports a consistent $\tt q$--labelling $\mathcal{L}_{\mathcal{W}_{\tt q}}$, 
then $\mathcal{L}_{\mathcal{W}_{\tt q}}$ extends to a consistent $\tt q$--labelling on
$\widehat{\Gamma}_{\tt q}$.

\noindent
2.\ If $\widehat{\Gamma}_{\tt q}$ supports a consistent $\tt q$--labelling 
$\mathcal{L}_{\mathcal{W}_{\tt q}}$, 
then $\mathcal{L}_{\mathcal{W}_{\tt q}}$ restricts to a consistent $\tt q$--labelling on
$\Gamma$.
\end{lemma}
\begin{proof}
Statement 1 follows by edge subdivision that 
adds valence two vertices to $\Gamma$ to obtain $\widehat{\Gamma}_{\tt q}$.
Statement 2, the converse, is also true by the forgetting vertices of valence 2 operation.
\end{proof}

Note that a given $\tt t$--graph $\Gamma$ can support several consistent $\tt q$--labellings:

\noindent
$\cdot$ 
Figure \ref{fig:tres-mosaicos-etiquetas-h}.a
shows a $\tt t$--graph $\Gamma$ that supports the 
consistent $6$--labelling shown in 
Figure \ref{fig:tres-mosaicos-etiquetas-h}.b.
The consistent $6$--labelling of a compatible 
$\tt A$--map $\widehat{\Gamma}_{6}$ 
is shown in Figure \ref{fig:tres-mosaicos-etiquetas-h}.c.

\noindent
$\cdot$ 
Figure 2.a of \cite{AlvarezGutierrezMucino}, shows a consistent $4$--labelling for another
$\tt A$--map obtained from the same $\tt t$--graph $\Gamma$ 
of Figure \ref{fig:tres-mosaicos-etiquetas-h}.a. 
Figures 2.b and 2.c 
of \cite{AlvarezGutierrezMucino}
show two different
consistent $5$--labellings for two different $\tt A$--maps, whose subjacent $\tt t$--graphs are 
also the same $\Gamma$.

\noindent
The above shows that, for a given $\tt t$--graph, there might be several different consistent 
$\tt q$--labellings for fixed $\tt q$ and for different $\tt q$'s.

\begin{definition}
\label{def:Speiser-tessellation}
A \emph{Speiser $\tt q$--tessellation}
is a pair $\big(\mathscr{T}(\widehat{\Gamma}_{\tt q}),\, \mathcal{L}_{\mathcal{W}_{\tt q}} \big)$
where
\begin{enumerate}[label=\roman*),leftmargin=*] 
\item
$\mathscr{T}(\widehat{\Gamma}_{\tt q})$ is a tessellation on 
$\Omega_z$, 
arising from an $\tt A$--map $\widehat{\Gamma}_{\tt q}$
as in Definition \ref{definicion-de-A-map}, 
and 

\item
$\mathcal{L}_{\mathcal{W}_{\tt q}}$ is a consistent $\tt q$--labelling of $\widehat{\Gamma}_{\tt q}$.
\end{enumerate}
\end{definition}

\subsection{Schwarz--Klein--Speiser's algorithm}
\label{sec:Schwarz-Klein-Speiser-algorithm}
 
\smallskip 

\noindent 
Let $w(z):\Omega_z \longrightarrow \CW_w$ be a 
Speiser function. 
Recall that $\Omega_z$ is either $\CW_z$, $\CC_z$, or 
$\Delta_z$.

\noindent 
{\bf Step 1.}
Choose a cyclic order for the $\tt q$ singular values of $w(z)$

\centerline{$
\mathcal{W}_{\tt q} =
[{\tt w}_1, \ldots, 
{\tt w}_{\tt j}, \ldots,
{\tt w}_{\tt q} ] \subset \CW_w$,
}

\noindent 
and
consider
a Jordan path $\gamma$ realizing the cyclic order $\mathcal{L}_\gamma$ of the
singular values, as in Definition \ref{def:orden-ciclico-y-losetas}.1.

\noindent 
{\bf Step 2.} 
Compute the inverse image of $\gamma$,

\centerline{
$w^{-1}(\gamma) \subset \Omega_z$.
}

\noindent
and complete it to $\Gamma\subset ( \Omega_z \cup \partial_{\mathcal{I}}\Omega_z )$ 
by adding the ideal points (logarithmic singularities, 
see \eqref{eq:multiplicities-of-branch-points}),
$\{ z_\sigma \} \subset \partial_{\mathcal{I}}\Omega_z$.
Note that $\partial_{\mathcal{I}}\Omega_z $ can be 
$\varnothing$, $\{ \infty_1,\ldots,\infty_{\tt p} \}$
or a subset of $\{ \abs{z}  = 1\}$, 
according to whether the Riemann surface $\R_{w(z)}$ has elliptic, parabolic or 
hyperbolic conformal type, respectively.

\noindent 
In graph theory, the pullback graph 
$$
\widehat{\Gamma}_{\tt q}=w(z)^* \gamma  
\subset ( \Omega_z \cup \partial_{\mathcal{I}}\Omega_z )
$$
is well defined.
It has the singular points $\mathcal{SP}_w$
and the cosingular points $\mathcal{CS}_w$ of $w(z)$ as vertices $V(\widehat{\Gamma}_{\tt q})$,
and the respective segments in $\widehat{\Gamma}_{\tt q}$ as edges.
Moreover, the logarithmic singularities are vertices located in $\partial_{\mathcal{I}}\Omega_z$ 
and with infinite valence.

\noindent 
Set theoretically $\Gamma = \widehat{\Gamma}_{\tt q}$,  
however they are isomorphic as graphs if and only 
the cosingular point set $\mathcal{CS}_w$ of $w(z)$ 
is empty.

\noindent 
{\bf Step 3.} 
The tessellation 
determined by $w(z)$ and $\gamma$ is 
$$
( \Omega_z \cup \partial_{\mathcal{I}}\Omega_z ) \backslash \widehat{\Gamma}_{\tt q}
=
\underbrace{
T_1 \cup \ldots \cup T_\alpha \ldots  }_{n \text{ blue tiles} }
\cup 
\underbrace{ T^\prime_1 \cup  
\ldots  \cup T^\prime_\beta \ldots  }_{ n\text{ grey tiles}} \ , 
\ \ \ 
2 \leq n \leq \infty .
$$

\noindent 
{\bf Step 4.} 
The cosingular points $\mathcal{CS}_w$ play a crucial role
on the boundary of the tiles; they the are vertices of 
$\widehat{\Gamma}_{\tt q}$ of valence 2. 
In fact, the tiles of the tessellation 
are also topological $\tt q$--gons.
For each tile $T_\alpha$, the pullback of the 
cyclic order $\mathcal{L}_\gamma$
\begin{equation}
\label{eq:pullback-etiquetado-de-gamma}
\begin{array}{cccccc}
w^{-1}({\tt w}_1) & \ldots & w^{-1}({\tt w}_{\tt j}) & \ldots & w^{-1}({\tt w}_{\tt q}) & \in \partial \overline{T_\alpha} \subset \mathcal{SP}_w \cup \mathcal{CS}_w
\\
\downarrow & & \downarrow & & \downarrow & 
\\
{\tt w}_1 & \ldots & {\tt w}_{\tt j} & \ldots & {\tt w}_{\tt q} & \in \mathcal{W}_{\tt q} \, ,
\end{array}
\end{equation}

\noindent
determines 
a consistent\footnote{
Since $\{ {\tt w}_{\tt j} \}$ is the set of singular values of $w(z)$, 
condition (ii) of Definition \ref{def:etiquetado-consistente} is trivially satisfied by $w(z)^*\mathcal{L}_{\gamma}$.
} 
$\tt q$--labelling, $w(z)^*\mathcal{L}_\gamma$, for $\widehat{\Gamma}_{\tt q}$. 
We have constructed the analytic
Speiser $\tt q$--tessellation 
\begin{equation}\label{eq:tesellation-of-w}
(\mathscr{T}_\gamma(w(z)),w(z)^*\mathcal{L}_\gamma ) =
\big( \underbrace{(\Omega_z \cup \partial_{\mathcal{I}}\Omega_z ) \backslash w(z)^* \gamma}_{
\text{tessellation}}, 
\underbrace{w(z)^*\mathcal{L}_\gamma}_{\substack{\text{consistent}\\ {\tt q}-\text{labelling}}} 
\big),
\end{equation}

\noindent
that is the output of the algorithm.

\begin{example}[Speiser $\tt q$--tessellations for some functions]
On $\CW_z$, 
for Speiser $\tt q$--tessellations of rational functions, 
see \cite{AlvarezGutierrezMucino}, figures 1, 2, 5, and 6. 
Examples of consistent $\tt q$--labellings are in 
\cite{GonzalezMucino}, figures 4, 6.
Obviously, Speiser $\tt q$--tessellations are natural for
meromorphic functions on compact Riemann surfaces
of genus $g\geq 1$, 
see \cite{AlvarezGutierrezMucino} figure 7.

\noindent
On $\CC_z$, for Speiser $\tt q$--tessellations arising from transcendental Speiser functions $w(z)$ 
(with an essential singularity at $\infty\in\CW_z$), 
see Figures 
\ref{fig:Tessellation-Exp}.a,
\ref{fig:Tessellation-AiBi}.a, 
\ref{fig:Tessellation-p=4}.a,
\ref{fig:mosaico-exp-exp}.a, 
\ref{figExpSin}.a, 
and
\ref{fig:SinExpSin}.a.
\end{example}

\begin{example}[Not all labelled ${\tt A}$--maps are consistent 
$\tt q$--labelled $\tt A$--maps]
\label{example-not-all-Amaps-are-consistent}
We provide two examples of this issue. 

\noindent
1.\ The first one is for a finite ${\tt A}$--map and is based upon
figure 10 of \cite{Koch-Lei}, 
which we have reproduced as 
Figure \ref{fig:Non-consistent-6-labbeled} 
to make things easier for the reader. 
The figure depicts an ${\tt A}$--map with 4 tiles of each color, 
where each tile is a 6--gon,
and a labelling with labels $\mathcal{W}_6=[1,2,3,4,5,6]$.  
However, all the vertices labelled 5 have valence 2, thus,

\noindent
$\bigcdot$ 
the labelling is not a consistent $6$--labelling,

\noindent
$\bigcdot$ 
the vertices labelled 5 are fake 
cocritical points, equivalently, 
5 is a fake critical value.

\noindent
However, 
by forgetting the vertices labelled 5, we obtain another 
${\tt A}$--map, presumably corresponding to  a rational function of degree $4$,
whose tiles are $5$--gons,
in fact a consistent $5$--labelling
exists, 
see Example \ref{example:non-uniqueness-extension-Speiser-graph}, 
case ${\tt q}=5$, particularly Figure \ref{fig:MosaicoSpeiser-q5}.b.

\begin{figure}[h!tbp]
\begin{center} 
\includegraphics[width=0.3\textwidth]{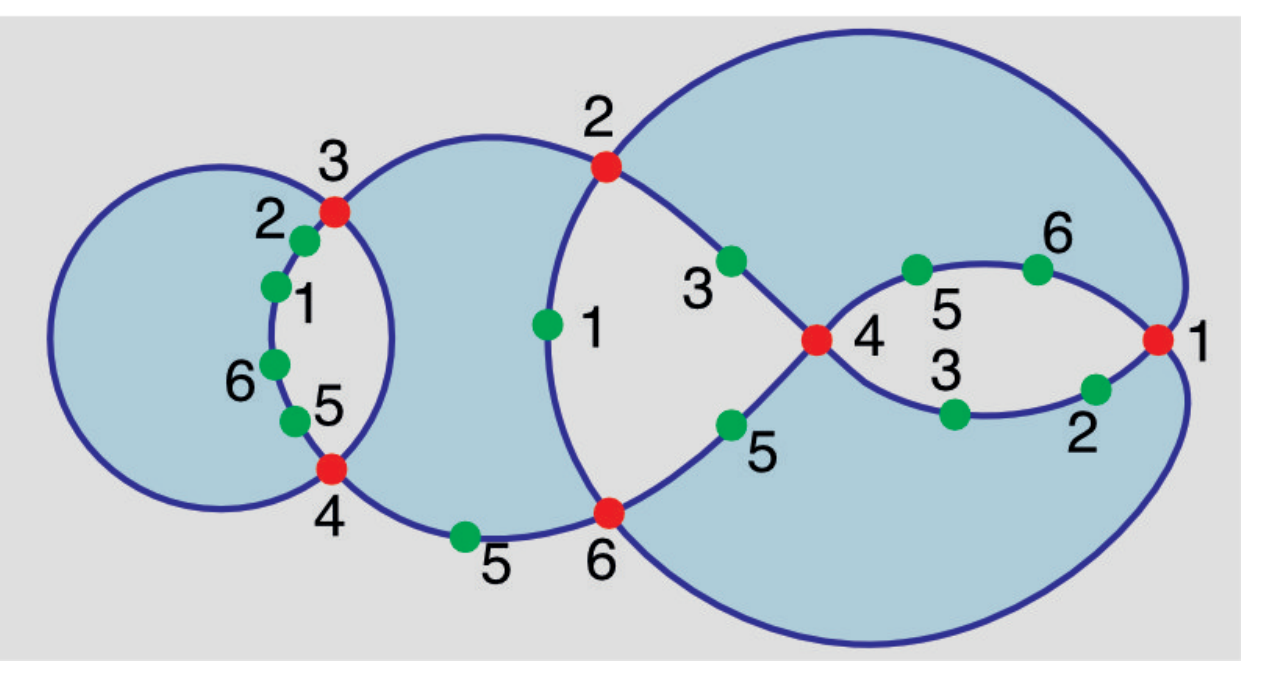}
\caption{
${\tt A}$--map with 4 tiles of each color, 
where each tile is a 6--gon,
and a non consistent 6--labelling with labels $\mathcal{W}_6=[1,2,3,4,5,6]$; 
label `5' is only assigned to vertices of valence 2.
This is figure 10 of \cite{Koch-Lei} and is attributed to 
W.\,P.\,Thurston.
}
\label{fig:Non-consistent-6-labbeled}
\end{center}
\end{figure}

\noindent
2.\ The second example is for an infinite $\tt A$--map, shown in Figure \ref{fig-4-Speiser}.f;
in this case the $\tt A$--map depicted 
has a labelling in $\mathcal{W}_4=[1,2,3,4]$. Once again, there is a label (in this case 4)
that does not appear as a vertex of valence greater than or equal to 4, thus 

\noindent
$\bigcdot$ 
the labelling is not a consistent $4$--labelling,

\noindent
$\bigcdot$ 
the vertices labelled 4 are fake cosingular 
points, equivalently, 4 is a fake singular value.

\noindent
However, 
by forgetting the vertices labelled 4, we obtain another ${\tt A}$--map 
with ${\tt q}=3$, that is
tiles which are 3--gons and a consistent 
$3$--labelling; Figure \ref{fig-4-Speiser}.d shows the dual graph\footnote{
A Speiser graph of index 3, as will be seen in \S\ref{sec:Speiser-graphs}.
} $\mathfrak{S}_3$ 
of the new ${\tt A}$--map.

\noindent
For a more general discussion and further examples see
\S\ref{sec:planar-graphs-Thrurston}.
\end{example}

\begin{remark}[Properties of Speiser $\tt q$--tessellations]
1.
The conformal type of $\R_{w(z)}$ determines 
the ambient space 
$\Omega_z \cup \partial_{\mathcal{I}}\Omega_z $
of the oriented graph $\widehat{\Gamma}_{\tt q}$ in
Definition \ref{definicion-de-A-map}, namely 
$\CW_z$, $\CC_z\cup \{\infty_1,\ldots,\infty_{\tt p} \}$ 
or $\Delta_z \cup \partial_{\mathcal{I} } \Delta_z $.

\noindent
2.\ 
The tiles of the Speiser $\tt q$--tessellation \eqref{eq:tesellation-of-w}
are $\tt q$--gons; blue and grey tiles 
corresponding to the inverse image under $w^{-1}(z)$ of 
the blue and grey tiles
$T$, $T^\prime$ in Step 3, respectively.

\noindent
3.\ Considering the graph $\widehat{\Gamma}_{\tt q}=w(z)^*\gamma$,
the vertices of 
\begin{itemize}[label=$\bigcdot$ ,leftmargin=*]
\item
valence 2 are cosingular points of $w(z)$ in $\Omega_z$,

\item 
finite valence greater than or equal to 4 are algebraic singularities of $w^{-1}(z)$ (critical points of $w(z)$) in $\Omega_z$, and

\item
infinite valence are logarithmic singularities of $w^{-1}(z)$ in 
$\partial_{\mathcal{I}}\Omega_z$.
\end{itemize}
\end{remark}

We provide some features for the simplest 
families of Speiser functions.

\begin{remark}[Tessellations for rational functions]
Let 
$R(z)$ be a rational function of degree $n \geq 2$,
due to Remark \ref{rem:natural-boundaries}, 
we have  $\Omega_z = \CW_z$.
The set of asymptotic values of $R(z)$ is empty 
and 
the set of singular values consists exclusively of 
critical values.
The finite number of critical points is $ 2\leq {\tt r} \leq 2n-2$,
the cosingular points are called cocritical points.
Obviously,
$\widehat{\Gamma}_{\tt q}\subset \CW_z$ 
is a finite graph\footnote{
Thus condition (iii) of Definition \ref{def:de-t-graph} is satisfied.};  
with vertices of 
valence 2 at the cocritical points of $R(z)$, 
and even valence greater than
or equal to 4 at the critical points of $R(z)$.

\noindent
Furthermore, 
if the distinct ${\tt q}$ critical values of $R(z)$ lie in $\RR$,
the computation of the topological tessellation is readily available. 
In fact, $\mathscr{T}_\gamma(R(z)) = \CW_z\backslash R^{-1}(\RR\cup\{\infty\})$,
where $R^{-1}(\RR\cup\{\infty\})$ is a real algebraic curve.
\end{remark}

\begin{remark}[Tessellations for $N$--functions]
Let $w(z):\CC_z\longrightarrow\CW_w$ be an $N$--function.
The set of critical values is empty and
the set of singular values consists exclusively of 
asymptotic values:

\centerline{
$\mathcal{AV}_w=\{(a_1, \mu_1), (a_2,\mu_2), \ldots, (a_{\tt q},\mu_{\tt q}) \}$.
}

\noindent
Moreover, the cosingular points are called coasymptotic points, and they can be defined by

\centerline{$\{w^{-1} (a_{\tt j})\}_{{\tt j}=1}^{\tt q} \cap \Omega_z$. }

\noindent
In this case $\widehat{\Gamma}_{\tt q}$ is an infinite graph  with 
vertices of valence 2 at the coasymptotic points of $w(z)$. 

\noindent
Furthermore, 
as was shown by 
R. Nevanlinna \cite{Nevanlinna1} \S 8, 
\cite{Nevanlinna2} \S XI.3.5, 
the fact that $w(z)$ satisfies 
the Schwarzian differential equation 
\eqref{eq:Schwarzian-diff-eqn}, 
implies that $\Omega_z=\CC_z$,
and that the set of logarithmic singularities is finite, \emph{i.e.}\ ${\tt p}<\infty$.
Thus, the ${\tt p}$ vertices of infinite valence $\{\infty_1,\ldots,\infty_{\tt p} \}$ of
$\widehat{\Gamma}_{\tt q}$,
are the logarithmic singularities of $w(z)$.
The compactification 
$\CC_{z} \cup \{\infty_1, \ldots , \infty_{\tt p} \}$
is non Haussdorff. 

\noindent 
Example \ref{example:N-exp-tanh}.c illustrates 
the simplest case of the Speiser $\tt q$--tessellations for  
an $N$--function. 
Additionally, 
two non--trivial cases of $N$--functions when ${\tt q} = 3$ are in
Examples \ref{example:tessellation-Airy}.c and 
\ref{example:function-wM4alt}.c. 
Examples \ref{example:exp-exp}.c, \ref{example:tessellation-expsin}.c,
and \ref{example:sinexpsin}.c show $\tt q$--tessellations 
for Speiser functions that are not $N$--functions, 
the first for ${\tt q}=3$ and the last two for ${\tt q}=4$.
\end{remark}

\begin{remark}[Bounds on $\tt q$ for a consistent $\tt q$--labelling of a $\tt t$--graph, 
depending on its tessellation]
\label{rem:bounds-for-q}
Given a $\tt t$--graph $\Gamma$, 
there are natural upper and lower bounds on the 
positive integer $\tt q$ of a consistent $\tt q$--labelling $\mathcal{L}_{\mathcal{W}_{\tt q}}$ that can be
assigned to $\Gamma$. 
Recall from Definition \ref{def:de-teselacion}.ii that for each tile $T_\alpha$, its boundary
$\partial \overline{T_\alpha}$ consists of $2\leq \rho_\alpha \leq {\tt q}$ vertices, 
thus taking the largest $\rho_\alpha$ in $\Gamma$ provides a lower bound for $\tt q$.
On the other hand, from Definition \ref{def:etiquetado-consistente}.ii it follows that 
$\tt q$ is bounded above by the number of vertices of valence greater than or equal to 4.
Summarizing,
\begin{equation}
\label{cota-q-mosaicos}
{\tt q}_{\text{min}} \doteq 
\max \# \left \{ \begin{array}{c}
\text{vertices on } \partial \overline{T_\alpha},  \\ 
\text{ for }
T_\alpha \in \mathscr{T}(\Gamma)
\end{array}
\right \}
\leq  \, {\tt q} \leq 
{\tt q}_{\text{max}} \doteq
\# \left \{ \begin{array}{c}
\text{vertices of } \Gamma \text{ with}\\
\text{valence } \geq 4
\end{array}
\right\} .
\end{equation}

\noindent
Note that ${\tt q}_{\text{max}}$ can be infinite.
\end{remark}

The following is a surprising but useful result.

\begin{lemma}
\label{lem:A-mapa-admite-etiquetados}
An $\tt A$--map $\widehat{\Gamma}_{\tt q}$ supports at least one consistent 
${\tt q}_{\rm o}$--labelling $\mathcal{L}_{\mathcal{W}_{{\tt q}_{\rm o}}}$, for 
$2 \leq {\tt q}_{\text{min}} \leq {\tt q}_{\rm o} \leq {\tt q} \leq {\tt q}_{\text{max}}$.
\end{lemma}

\begin{proof}
Given the $\tt A$--map $\widehat{\Gamma}_{\tt q}$, 
let $\Gamma$ be the $\tt t$--graph, alluded to in Definition \ref{definicion-de-A-map}.ii
(in other words $\Gamma$ and $\widehat{\Gamma}_{\tt q}$ are compatible).
Note that ${\tt q}_{\text{min}} \leq {\tt q} \leq {\tt q}_{\text{max}}$ as in the above Remark.
The existence of a consistent ${\tt q}$--labelling is as follows:
\begin{enumerate}[label=\arabic*),leftmargin=*]
\item Choose ${\tt q}_{\rm o} = {\tt q}$,

\item 
assign the labels $\mathcal{W}_{{\tt q}_{\rm o}}$ to the ${\tt q}_{\rm o}$ vertices of any 
blue tile $T_\alpha$ of $\widehat{\Gamma}_{\tt q}$, in an anticlockwise order,

\item
propagate the labelling to all the neighbor grey tiles, 
this can be done since
the tessellation $\mathcal{T}(\widehat{\Gamma}_{\tt q})$ is homogeneous, \emph{i.e.}\, 
its tiles are topological ${\tt q}_{\rm o}$--gons,

\item
continue as above to all the tiles of $\mathcal{T}(\widehat{\Gamma}_{{\tt q}_{\rm o}})$,
using the fact that $\Omega_z$ is simply connected.
\end{enumerate}
This provides a labelling 

\centerline{$
\mathcal{L}_{\mathcal{W}_{{\tt q}_{\rm o}}}: V(\widehat{\Gamma}_{{\tt q}_{\rm o}}) \longrightarrow \mathcal{W}_{{\tt q}_{\rm o}}, 
$}

\noindent
to $\widehat{\Gamma}_{{\tt q}_{\rm o}}$.
Clearly, $\mathcal{L}_{\mathcal{W}_{{\tt q}_{\rm o}}}$ satisfies condition (i) of a consistent ${\tt q}_{\rm o}$--labelling,
see Definition \ref{def:etiquetado-consistente}. 

\noindent
If the choice of ${\tt q}_{\rm o}$ satisfies condition (ii) then we are done.

\noindent
Otherwise, there is a label (value), say ${\tt w}_{{\tt j}_{\rm o}}$, that does not appear under 
$\mathcal{L}_{\mathcal{W}_{{\tt q}_{\rm o}}}$ for a vertex of $\widehat{\Gamma}_{{\tt q}_{\rm o}}$
of valence greater than or equal to 4. 
We can erase this label from the cyclic order $\mathcal{W}_{{\tt q}_{\rm o}}$, 
obtaining a new cyclic order with ${\tt q}_{\rm o} - 1$ distinct values.
Now, go back to (2) with a new ${\tt q}_{\rm o}$ being one less than the previous one and repeat the process.
Since the original $\tt A$--map $\widehat{\Gamma}_{\tt q}$ has a non empty subset of vertices of valence greater than or equal to 4, this process eventually stops at a ${\tt q}_{\rm o} \geq {\tt q}_{\text{min}}$.
\end{proof}

The reader is invited to keep in mind the above result when considering 
Remark \ref{rem:minimality-condition} (an interpretation of the
consistent $\tt q$--labeling as a minimality condition).

\begin{theorem}[From Speiser functions to tessellations and back]
\label{teo:principal}
\hfill\\
Let $\Omega_z$ be a simply connected 
Riemann surface.
\begin{enumerate}[label=\arabic*),leftmargin=*]
\item 
A Speiser function $w(z): \Omega_z \longrightarrow \CW_w$ of
index $2\leq {\tt q} < \infty$, 
provided with a cyclic order $\mathcal{W}_{\tt q}$ 
and a path $\gamma$ realizing it, determine
a  homogeneous tessellation 
$$
\begin{array}{rl}
\mathscr{T}_\gamma(w(z)) = &
( \Omega_z \cup \partial_{\mathcal{I}}\Omega_z ) \backslash \Gamma
= 
\begin{cases}
\CW_z \backslash \Gamma
\\
(\CC_z\cup\{\infty_1,\ldots,\infty_{\tt p} \}) 
\backslash \Gamma
\\
(\Delta \cup \partial_{\mathcal{I}} \Delta ) \backslash \Gamma
\end{cases}
\\
& \vspace{-.20cm}
\\
= &
\underbrace{
T_1 \cup \ldots \cup T_\alpha \cup \ldots  }_{ n\text{ blue tiles} }
\cup 
\underbrace{ T^\prime_1 \cup  
\ldots  \cup T^{\prime}_\alpha  \cup \ldots }_{n \text{ grey tiles}} \subset 
\Omega_z \cup \partial_{\mathcal{I}} \Omega_z, 
\end{array}
$$

\noindent  
whose tiles are topological $\tt q$--gons 
with alternating colors, and
a consistent $\tt q$--labelling $w(z)^*\mathcal{L}_\gamma$.

\smallskip 

\item
Let $\mathscr{T}$ be a possibly non homogeneous 
tessellation of $\Omega_z$.
Assume in addition that $\mathscr{T}$ is 
provided with a consistent $\tt q$--labelling $\mathcal{L}_{\mathcal{W}_{\tt q}}$.
Then, they determine 
a Riemann surface $\Omega_z$,
a non unique Speiser function

\centerline{$w(z): \Omega_z \longrightarrow \CW_w$,} 

\noindent 
and a Jordan path $\gamma$ 
satisfying that the tessellation

\centerline{
$(\mathscr{T}_\gamma(w(z)) , w(z)^* \mathcal{L}_\gamma )$ \quad is \quad
$(\mathscr{T}, \mathcal{L}_{\mathcal{W}_{\tt q}} )$, 
}

\noindent
up to
orientation preserving homeomorphism of $\Omega_z$.

\end{enumerate}
\end{theorem}

\begin{remark}
Note that in statement (2), the fact that the consistent $\tt q$--labelling 
$\mathcal{L}_{\mathcal{W}_{\tt q}}$ is provided to $\mathscr{T}$, 
ensures that the choice and ordering of the values $\mathcal{W}_{\tt q}$ 
are an essential part of the hypothesis.

\noindent 
Clearly, the number of tiles $2n$ of $\mathscr{T}$ is finite 
if and only if $w(z)$ is rational function of degree $n$ on $\CW_z$.
\end{remark}

\begin{proof}
Statement (1) follows directly from the Schwarz--Klein--Speiser's algorithm.

\medskip
For statement (2), we proceed with the following steps:

\noindent
Step 1.\ Recall that the tessellation $\mathscr{T}$ is equivalent to a $\tt t$--graph 
$\Gamma \subset \Omega_z \cup \partial_{\mathcal{I}} \Omega_z$, thus in fact we have $(\Gamma,\mathcal{L}_{\mathcal{W}_{\tt q}})$.

\smallskip
\noindent 
Step 2.\ By using the consistent $q$--labelling 
$\mathcal{L}_{\mathcal{W}_{\tt q}}$ and edge subdivision operation 
for $\Gamma$, we get an associated ${\tt A}$--map 
$\widehat{\Gamma}_{\tt q} \subset \Omega_z \cup \partial_{\mathcal{I}} \Omega_z$ with a 
consistent $q$--labelling 

\centerline{
$\mathcal{L}_{\mathcal{W}_{\tt q}}: V(\widehat{\Gamma}_{\tt q}) \longrightarrow \mathcal{W}_{\tt q}$,}

\noindent 
as follows.

\noindent
{\it Edge subdivision operation.}
Let $\overline{z_\iota z_\sigma}$ be an edge of $\Gamma$ with 
labels, say 
$\mathcal{L}_{\mathcal{W}_{\tt q}}(z_\iota)={\tt w}_{\tt h}$ and 
$\mathcal{L}_{\mathcal{W}_{\tt q}}(z_\sigma)={\tt w}_{\tt j}$. 

\noindent 
If $ {\tt j} - {\tt h} =(z_\iota) = 
1  \pmod{\tt q} $, 
then $\overline{z_\iota z_\sigma}$ is an edge of 
$\widehat{\Gamma}_{\tt q}$.

\noindent 
If $ {\tt j} - {\tt h} =
\nu +1 \geq 2  \pmod{\tt q} $, 
then we consider $\nu$ new vertices,
$\zeta_1, \ldots, \zeta_\nu$,
in the original edge 
$\overline{z_\iota z_\sigma}$, 
which is replaced by $\nu +1$ new edges 

\centerline{$ 
\overline{z_\iota \zeta_1}, \,
\overline{\zeta_1 \zeta_2}, \ldots ,
\overline{\zeta_\nu z_\sigma}
$}

\noindent 
of $\widehat{\Gamma}_{\tt q}$. Moreover, the labels of these new 
vertices of valence 2 of $\widehat{\Gamma}_{\tt q}$ are

\centerline{$
\mathcal{L}_{\mathcal{W}_{\tt q}}(z_\iota) = {\tt w}_{\tt h} ,\ \ 
\mathcal{L}_{\mathcal{W}_{\tt q}}(\zeta_1) = {\tt w}_{{\tt h}+1} , \ \ 
\ldots \ ,  \ \ 
\mathcal{L}_{\mathcal{W}_{\tt q}}(\zeta_\nu) = {\tt w}_{{\tt h} + \nu} , \ \ 
\mathcal{L}_{\mathcal{W}_{\tt q}}(z_\sigma) = {\tt w}_{\tt j} ;
$}

\noindent
with arithmetic \hspace{-10pt} $\mod{\tt q} $ in the subindices.

\smallskip
\noindent
Step 3.\ Since $\widehat{\Gamma}_{\tt q}$ is homogeneous, we can recognize
that $\mathscr{T}(\widehat{\Gamma}_{\tt q})$ inherits a natural conformal structure 
from the glueing of the contiguous tiles $T_\alpha$ and $T_\alpha^\prime$ according
to the consistent $\tt q$--labelling $\mathcal{L}_{\mathcal{W}_{\tt q}}$.
In fact, we recognize that there is a Speiser Riemann surface 
$\R(\widehat{\Gamma}_{\tt q},\mathcal{L}_{\mathcal{W}_{\tt q}}) \subset \Omega_z \times \CW_w$, 
with a tessellation as above, that projects via $\pi_1$, see \eqref{diagramaRX}, 
to $\mathscr{T}(\widehat{\Gamma}_{\tt q})$.

\smallskip
\noindent
Step 4.\ Finally, the Speiser Riemann surface $\R(\widehat{\Gamma}_{\tt q},\mathcal{L}_{\mathcal{W}_{\tt q}})$ provides 
the Speiser function $w(z)$.

The non--uniqueness of the Speiser function $w(z)$ arises from the following.
\begin{definition}\label{def:isotropy-Wq}
Let \emph{$\text{Stab}(\mathcal{W}_{\tt q})$ be the isotropy group of $\mathcal{W}_{\tt q}$}, 
that is the subgroup of $Aut(\CW_w)$ that leaves invariant the set $\{ {\tt w}_{\tt j} \}_{{\tt j}=1}^{\tt q}$ 
and also preserves the chosen cyclic order on them.
\end{definition}

\begin{lemma}[Non uniqueness of Speiser functions arising from tessellations]
\label{lem:no-unicidad-funciones-teselaciones}
Let $w(z)$ be a Speiser function provided with a cyclic order 
$\mathcal{W}_{\tt q}$ for its $\tt q$ singular values.
Consider the action 
\begin{align*}
Aut(\Omega_z) \times \text{Stab}(\mathcal{W}_{\tt q}) \times (\Omega_z \times \CW_w )
& \longrightarrow  \Omega_z \times \CW_w \\
(g,h,z,w) & \longmapsto  (g(z), h(w)).
\end{align*}

\noindent 
Each non--trivial element in 
$Aut(\Omega_z) \times \text{Stab}(\mathcal{W}_{\tt q})$ provides a
different function with the same $\mathcal{W}_{\tt q}$.
\begin{enumerate}[label=\roman*),leftmargin=*]
\item 
Since $Aut(\Omega_z)$ is a Lie group, it gives rise to an infinite 
number of functions.

\item
For the group $\text{Stab}(\mathcal{W}_{\tt q})$ we have the following 
(up to conjugation in $Aut(\CW_w)$) cases.

\noindent
$\bigcdot$ If $\tt q=2$ and $\text{Stab}(\mathcal{W}_{\tt q})\neq Id$, then 
$\text{Stab}(\mathcal{W}_{\tt q})\cong\CC^*$ are the homotheties.

\noindent
$\bigcdot$ If $\tt q\geq 3$ and $\text{Stab}(\mathcal{W}_{\tt q})\neq Id$, then 
$\text{Stab}(\mathcal{W}_{\tt q})$ is one of the finite subgroups of $PSL(2,\CC)$.  
\end{enumerate}
\end{lemma}
\begin{proof}[Proof of Lemma]
The action of $Aut(\Omega_z) \times \text{Stab}(\mathcal{W}_{\tt q})$ extends to functions as 

\centerline{$(g,h,w(z)) \mapsto h(w(g(z)))$.}

\noindent
Note that the cyclic order of the $\tt q$ singular
values $\mathcal{W}_{\tt q}$, may have non--trivial isotropy

\centerline{
$Id\neq\text{Stab}(\mathcal{W}_{\tt q})\subset Aut(\CW_w)$.}

\noindent
The finite subgroups of $PSL(2,\CC)$ are:
the rotations $\ZZ_n$ for $n\geq2$, 
the dihedral group $\mathbb{D}_n$ for $n\geq2$ 
(generated by the rotations, and the inversion $w\mapsto 1/w$), and
the groups $H_{p,q,r}$ with $p,q,r\geq 2$, associated to the symmetries of the 
regular polyhedra inscribed in $\CW_w$.
See \cite{Sullivan-notes} for more details.
\end{proof}

\begin{example} 1. 
A simple family is 
$\{ h(\sin (g(z))) 
\ \vert \ 
g \in Aut(\CW_z), \ h \in \text{Stab} ([-1,1,\infty]) \}$. 
Obviously $\cos(z)$ is an element of it.

\noindent 
2. Let $G \subset Aut(\CW_z)$ be a finite group. For the classical rational $G$--invariant functions $R(z)$ the Lemma applies, giving
origin to different explicit expressions for $R(z)$ in 
the literature. Compare with \cite{Schwarz} and \cite{Pakovich}.  
\end{example}

Theorem \ref{teo:principal} is proved.
\end{proof}

\section{Speiser graphs}
\label{sec:Speiser-graphs} 
Recalling Equation \eqref{eq:tesellation-of-w}, Speiser $\tt q$--tessellations
$\big( (\Omega_z \cup \partial_{\mathcal{I}}\Omega_z ) \backslash w(z)^* \gamma, 
w(z)^*\mathcal{L}_\gamma \big)$
 arising from Speiser functions $w(z)$, 
as in Theorem \ref{teo:principal}.1, 
are very concrete objects.
The dual graph $\mathfrak{S}_{w(z)}$ of the 
$\tt A$--map $\widehat{\Gamma}_{\tt q} = w(z)^*\gamma$ 
is the Speiser graph of index $\tt q$, of $w(z)$.
Here,
we develop Speiser graphs in an ad hoc axiomatic way.
As departure point, our Speiser graphs are embedded in 
$\esf$  or $B(0,1)$; because of the uniformization 
theorem this will be enough to completely specify
the conformal type of the domain of the associated Speiser 
functions.

The original concept is in  \cite{Speiser}.
We roughly follow 
\cite{Nevanlinna2},
\cite{GoldbergOstrovskii} p.\,355, and \cite{Masoero} p.\,54.
Once again, we make some precisions that we consider improve
the presentation and our understanding.

\begin{definition}[\cite{GoldbergOstrovskii} p.\,355]
\label{def:Speiser-graph}
A \emph{Speiser graph of index ${\tt q} \geq 2$} 
(or \emph{line complex of index $\tt q$}), 
$$
\mathfrak{S}_{\tt q} =\Big( V(\mathfrak{S}_{\tt q})= 
\underbrace{ \{\times_\alpha, \ \circ_\beta \} }_{\text{vertices}}, 
\,
E(\mathfrak{S}_{\tt q}) = 
\underbrace{\{ \overline{\times_\alpha \circ_\beta} \}
}_{\text{edges}}    \Big) \, ,
$$

\noindent
is a connected, locally finite\footnote{
“Locally finite’’ means every vertex has finite valence 
and each compact subset of $\esf$ or $B(0,1)$ meets only finitely 
many edges.}, multigraph satisfying the following:

\begin{enumerate}[label=\roman*),leftmargin=*]
\item
The graph
$\mathfrak{S}_{\tt q}$ is
properly embedded in $\esf$ when it is finite, 
or in $B(0,1)$ when it is infinite.

\item 
The set of vertices $V(\mathfrak{S}_{\tt q})$ is a finite or countable set.

\item 
The graph $\mathfrak{S}_{\tt q}$ is bipartite, with vertices in $\{\times, \circ\}$.

\item
Every vertex has valence $\tt q$.
\end{enumerate}
\end{definition} 

Note that a Speiser graph of index $\tt q$ is the dual of 
an $\tt A$--map $\widehat{\Gamma}_{\tt q}$.
Of course, the dual of a $\tt t$--graph $\Gamma$ also exists.

\begin{definition}
\label{def:pre-Speiser-graph}
A \emph{pre--Speiser graph $\mathfrak{S}$} is a graph satisfying Definition \ref{def:Speiser-graph}, 
with (iv) replaced by:

\smallskip
\noindent
iv$^\prime$) each vertex has a valence $2\leq\rho\leq {\tt q}$; the valence of each vertex 
is allowed to differ.
\end{definition}

\begin{remark}[Regular graph / homogeneous tessellation]
\label{rem:regular-description}
If a graph satisfies condition (iv) of Definition 
\ref{def:Speiser-graph}, it is said to be 
\emph{$\tt q$--regular}, or just \emph{regular}.

\noindent
1.\ Through duality, the fact that the Speiser graph 
$\mathfrak{S}_{\tt q}$ 
of index $\tt q$ is regular is equivalent  to the fact that the tessellation 
$\mathscr{T}(\widehat{\Gamma}_{\tt q})$ is 
homogeneous, recall Definition \ref{def:de-teselacion}.2.

\noindent
2.\ 
A priori,  pre--Speiser graphs $\mathfrak{S}$ are not regular, similarly
the tessellations $\mathscr{T}(\Gamma)$ arising from a $\tt t$--graph $\Gamma$
are usually not homogeneous.
\end{remark}

The concept of consistent $\tt q$--labelling $\mathcal{L}_{\mathcal{W}_{\tt q}}$, for
$\tt t$--graphs $\Gamma$ and $\tt A$--maps $\widehat{\Gamma}_{\tt q}$, 
has its corresponding dual for pre--Speiser graphs $\mathfrak{S}$ and 
Speiser graphs $\mathfrak{S}_{\tt q}$ of index $\tt q$. 
We shall convene on using the same name and symbol $\mathcal{L}_{\mathcal{W}_{\tt q}}$
when applied to $\mathfrak{S}$ or $\mathfrak{S}_{\tt q}$.

\begin{definition}
\label{def:etiquetado-consistente-Speiser}
Given a cyclic order $\mathcal{W}_{\tt q}$, 
a {\it consistent $q$--labelling 

\centerline{$\mathcal{L}_{\mathcal{W}_{\tt q}}: E(\mathfrak{S}) \longrightarrow \mathcal{W}_{\tt q}, 
\ \ \
{\tt q} \geq 2$, } 

\noindent 
for a pre--Speiser graph $\mathfrak{S}$ (Speiser graph $\mathfrak{S}_{\tt q}$ of index $\tt q$)} 
satisfies the following conditions:
\begin{enumerate}[label=\roman*),leftmargin=*]
\item 
The edges have labels 
in $\mathcal{W}_{\tt q}=[{\tt w}_1, \ldots, {\tt w}_{\tt q} ]$,
with no label repeated around each vertex,
the ordering of the edges around a vertex is according to their labels, 
cyclic clockwise for a $\times$-vertex, cyclic anticlockwise for a $\circ$--vertex.

\item
For ${\tt w}_{\tt j} \in \mathcal{W}_{\tt q}$, 
a \emph{${\tt w}_{\tt j}$--face of $\mathfrak{S}$} is a component of 
$\esf \backslash \mathfrak{S}$ (when $V(\mathfrak{S})$ is finite), or 
of $B(0,1) \backslash \mathfrak{S}$ (when $V(\mathfrak{S})$ is infinite),
with alternating edges labeled ${\tt w}_{j-1}$ and ${\tt w}_{\tt j}$.
We require that, for each ${\tt w}_{\tt j}\in\mathcal{W}_{\tt q}$, 
there is at least one ${\tt w}_{\tt j}$--face of $\mathfrak{S}$ 
that is not a digon. 
\end{enumerate}

\noindent 
The same applies to a Speiser graph 
$\mathfrak{S}_{\tt q}$ of index $\tt q$.
\end{definition}

\begin{definition}
\label{def:analytic-Speiser-graph}
An \emph{analytic Speiser graph of index $\tt q$} is a pair

\centerline{$(\mathfrak{S}_{\tt q}, \mathcal{L}_{\mathcal{W}_{\tt q}})$,}

\noindent
where $\mathfrak{S}_{\tt q}$ is a Speiser graph of index $\tt q$ and
$\mathcal{L}_{\mathcal{W}_{\tt q}}$ is a consistent 
$\tt q$--labelling.
\end{definition}

\begin{remark}[On the notation for the labels]
\label{rem:notation-of-labels}
As is usual in the literature, unless explicitly stated, 
we shall consider the labels to be 
$\mathcal{W}_{\tt q}=[{1},\ldots,{\tt q}]\subset\CW_w$ 
to make the discussion simpler.
In general,
  
\centerline{$\mathcal{W}_{\tt q}=
[{\tt w}_1, \ldots , {\tt w}_{\tt q}]
\doteq
[1, \ldots, {\tt q}]\subset\CW_w$, }

\noindent 
according to Definition \ref{def:orden-ciclico-y-losetas}.1.
\end{remark}

\begin{remark}[Requirement (ii) of $\mathcal{L}_{\mathcal{W}_{\tt q}}$ 
is a \emph{minimality condition}]
\label{rem:minimality-condition}
Suppose a labelling for a Speiser graph $\mathfrak{S}_{\tt q}$ fails to
satisfy (ii) 
for exactly one label, say ${\tt w}_{j_0}\in\mathcal{W}_{\tt q}$. 
Then all  
${\tt w}_{j_0}$--faces are digons.
Forgetting\footnote{
The forgetting edge operation for $\mathfrak{S}_{\tt q}$ is the analogue of the forgetting vertex operation
for the corresponding $\tt A$--map $\widehat{\Gamma}_{\tt q}$.
} 
the edges labelled ${\tt w}_{j_0}$, also forgets the 
${\tt w}_{j_0}$--face, and the resulting graph satisfies all the requirements of a Speiser graph
of index ${\tt q}-1$.

\noindent
In other words, when considering analytic Speiser graphs and relaxing the labelling 
so that it only satisfies (i) but not (ii), then several such Speiser graphs of 
different indices $\tt q$
give origin to the same function $w(z)$. 

\noindent
For instance, Figure \ref{fig-4-Speiser}.e is a graph, with ${\tt q}=4$, with a labelling 
satisfying the requirements of 
Definition \ref{def:etiquetado-consistente-Speiser} except for condition (ii); 
Figure \ref{fig-4-Speiser}.d is a graph, with ${\tt q}=3$, satisfying all the requirements: 
both represent the same
function. Furthermore, by replacing each digon labelled $4$ in Figure \ref{fig-4-Speiser}.e
with two digons labelled $4$ and $5$, we obtain another graph, with ${\tt q}=5$, that satisfies 
Definition \ref{def:etiquetado-consistente-Speiser} except for condition (ii).
Clearly this can be continued to an arbitrary ${\tt q}>3$.
Compare with \cite{GoldbergOstrovskii}\,p.\,355, where 
the minimality condition is not included.

\smallskip
\noindent
An alternate description of condition (ii) of Definition \ref{def:etiquetado-consistente-Speiser},
appears in \cite{Masoero}\,p.\,54 in terms of the monodromy:

\smallskip 

\begin{enumerate}[label=ii$^\prime$),leftmargin=*]
\item
For any ${\tt j}\in\ZZ_{\tt q}$,  define a map 
$\nu_{\tt j}$
from 
$V(\mathfrak{S}_{\tt q})$ to itself as follows
$\nu_{\tt j}(v)$ is the vertex adjacent to $v$ 
with respect to the edge ${\tt j}$. 
The composition $\Sigma_{\tt j}(v)\doteq \nu_{\tt j} \circ \nu_{{\tt j}-1}(v)$ of
two maps is a permutation of the vertices $\circ$ and $\times$. 
Require that for each ${\tt j}\in\ZZ_{\tt q}$, $\Sigma_{\tt j}(v)\not = v$ for some $v\in V(\mathfrak{S}_{\tt q})$,
{\it i.e.} none of the maps $\Sigma_{\tt j}$ are the identity.
\end{enumerate}
\end{remark}

\begin{remark}\label{rem:despues-de-def-Speiser-graph}
1.\ 
An analytic Speiser graph $(\mathfrak{S}_{\tt q}, \mathcal{L}_{\mathcal{W}_{\tt q}})$ 
naturally induces a cell decomposition 

\centerline{
$\esf \backslash \mathfrak{S}_{\tt q}$  
\ or \
$B(0,1)\backslash \mathfrak{S}_{\tt q}$,
}

\noindent 
depending on whether $\mathfrak{S}_{\tt q}$ is finite or infinite, respectively.

\noindent
$\bigcdot$ 
The cells of dimension 0 correspond to $V(\mathfrak{S}_{\tt q})$.

\noindent 
$\bigcdot$ The cells of dimension 1 correspond to the edges 
$E(\mathfrak{S}_{\tt q})$.

\noindent 
$\bigcdot$ 
The cells of dimension 2 are the connected components of 
the decomposition 
of $\esf \backslash \mathfrak{S}_{\tt q}$ or
$B(0,1) \backslash \mathfrak{S}_{\tt q}$,
which are called \emph{faces}.

\noindent
2.\ 
The faces inherit the cyclic order of the edges;
the faces have cyclic clockwise order, around each
$\times$--vertex, and the cyclic anticlockwise order 
around each $\circ$--vertex. 
Thus, the order of the faces coincides with 
the order of the edges.  As matter of record:

\centerline{
\emph{labelling edges of $\mathfrak{S}_{\tt q}$ or labelling faces of 
the cell decomposition is equivalent.}
}

\noindent
3.
It is easy to see that when we go around the boundary of a face, 
the edges have labels ${\tt w}_{j-1}$ and ${\tt w}_{\tt j}$ (for some ${\tt w}_{\tt j}\in\mathcal{W}_q$) 
and the labels alternate,
{\it i.e.} each face is a ${\tt w}_{\tt j}$--face for some label ${\tt w}_{\tt j}\in\mathcal{W}_q$.
 
\noindent
4.
Each face is bounded by either  

\noindent 
$\bigcdot$
an finite even set of edges, 
a \emph{bounded face} 
(\emph{algebraic elementary region} according to \cite{GoldbergOstrovskii}),
or 

\noindent
$\bigcdot$ 
by an infinite set of edges, 
an \emph{unbounded face}  
(\emph{logarithmic elementary region} according to \cite{GoldbergOstrovskii}).

\noindent
5.
Several edges with consecutive labels, 
having common vertices $\circ$ and $\times$, 
form a so--called \emph{edge bundle}. 
Clearly, two edges belonging
to the same edge bundle and having labels ${\tt j}-1$ and ${\tt j}$ form a boundary of a face, which is a digon.
In graph theory language, a multigraph admits edge bundles,
thus our Speiser graphs of index $\tt q$ are 
(generically) multigraphs.

\noindent
6.\ 
If a face is not a digon, 
its label will be written inside it. 
Because of condition (ii) of Definition \ref{def:etiquetado-consistente-Speiser}, 
it is not necessary to write labels inside digons.
\end{remark}

\begin{example}[Speiser graphs of index $\tt q$ for some functions]
For Speiser graphs of index $\tt q$ arising from transcendental Speiser functions $w(z)$ 
on $\CC_z$, with an essential singularity at $\infty\in\CW_z$,
see
\cite{GoldbergOstrovskii}\,ch.\,4  and 
our Figures 
\ref{fig:Tessellation-Exp}.c,
\ref{fig:Tessellation-AiBi}.c, 
\ref{fig:Tessellation-p=4}.c,
\ref{fig:mosaico-exp-exp}.b, 
\ref{figExpSin}.b, 
\ref{fig:SinExpSin}.b.
\end{example}

\subsection{Duality: Tessellations and Speiser graphs}
\label{sec:graph-results}
The duality between the $\tt A$--maps $\widehat{\Gamma}_{\tt q}$ 
and Speiser graphs $\mathfrak{S}_{\tt q}$ of index $\tt q$ provides the following bijections.

\begin{proposition}[Bijection between Speiser tessellations and analytic Speiser graphs]
\label{prop:bijection-extends-to-action}
Let $\mathcal{W}_{\tt q}$ be fixed (that is the set of $\tt q$ distinct values 
$\{ {\tt w}_{\tt j} \}_{{\tt j}=1}^{\tt q}$ and the cyclic order on them are fixed).
\begin{enumerate}[label=\arabic*),leftmargin=*]
\item
There is a bijection between Speiser $\tt q$--tessellations and 
analytic Speiser graphs of index ${\tt q}$,

\centerline{$
\big( \mathscr{T}(\widehat{\Gamma}_{\tt q}), 
\mathcal{L}_{\mathcal{W}_{\tt q}} \big) 
\longleftrightarrow 
(\mathfrak{S}_{\tt q}, \mathcal{L}_{\mathcal{W}_{\tt q}} )
$.}

\item
The above bijection extends to a bijection that includes the action of 
$Aut(\Omega_z) \times \text{Stab}(\mathcal{W}_{\tt q})$,
\emph{i.e.}\ 
%
\begin{equation*}
Aut(\Omega_z) \times \text{Stab}(\mathcal{W}_{\tt q}) \times
\big( (\Omega_z \cup \partial_{\mathcal{I}}\Omega_z ) \backslash \widehat{\Gamma}_{\tt q}, 
\mathcal{L}_{\mathcal{W}_{\tt q}} \big) 
\quad
\longleftrightarrow \quad
Aut(\Omega_z) \times \text{Stab}(\mathcal{W}_{\tt q}) \times
(\mathfrak{S}_{\tt q}, \mathcal{L}_{\mathcal{W}_{\tt q}} ).
\end{equation*}
\end{enumerate}
\end{proposition}
\begin{proof}
Because of the duality between the $\tt A$--maps $\widehat{\Gamma}_{\tt q}$ 
and the Speiser graphs $\mathfrak{S}_{\tt q}$ of index $\tt q$, 
the bijection should follow immediately.

\noindent
However, since 

\centerline{$
\mathscr{T}(\widehat{\Gamma}_{\tt q} )= 
(\Omega_z \cup \partial_{\mathcal{I}}\Omega_z ) \backslash \widehat{\Gamma}_{\tt q}
$,}

\noindent
care must be taken with the ideal boundary.

\noindent
\emph{Finite $\widehat{\Gamma}_{\tt q}$ case.} 
Note that $\widehat{\Gamma}_{\tt q} \subset \CW_z$.
Since the ideal boundary is empty, 
$\partial_{\mathcal{I}}\Omega_z=\varnothing$,
the cell decomposition $\CW_z \backslash\widehat{\Gamma}_{\tt q}$, 
is the dual of $\CW_z \backslash\mathfrak{S}_{\tt q}$. 

\noindent
\emph{Infinite $\widehat{\Gamma}_{\tt q}$ case.} 
Recalling condition (iii) of Definition \ref{def:de-t-graph}, 
it follows that the ambient space for $\widehat{\Gamma}_{\tt q}$ is either 
$\CC_z\cup \{\infty_1,\ldots,\infty_{\tt p} \}$ or $\Delta_z \cup \partial_{\mathcal{I}} \Delta_z$.
Let $\widehat{\Gamma}_0=\widehat{\Gamma}_{\tt q}\backslash V_\infty$, where
$V_\infty $ denotes the vertices with infinite valence of
$V(\widehat{\Gamma}_{\tt q})$. 
Considering the cell decomposition 
$\CC_z \backslash \widehat{\Gamma}_0$ or $\Delta_z\backslash\widehat{\Gamma}_0$, 
note that within the $C^1$--category $\CC_z \cong \Delta_z \cong B(0,1)$, hence 
the cell decomposition is $B(0,1)\backslash\widehat{\Gamma}_0$.
Its dual, is $B(0,1)\backslash\mathfrak{S}_{\tt q}$.

\noindent
This finishes the proof of statement 1.

For statement 2,
note that the action commutes with the duality.
\end{proof}

Speiser graphs have been used extensively in various contexts, one of the most
common is for studying Speiser functions $w(z)$ from a combinatorial perspective. 
From Theorem \ref{teo:principal} and the above bijection 
we immediately obtain.

\begin{corollary}[An analytic Speiser graph determines a family of Speiser functions]
\label{cor:Speiser-graph-a-Speiser-function}
An analytic Speiser graph $(\mathfrak{S}_{\tt q},\mathcal{L}_{\mathcal{W}_{\tt q}})$ 
of index $\tt q$, determines a 
non unique Speiser function $w(z)$ 
provided with $\tt q$ distinct singular values $\mathcal{W}_{\tt q}$.
\hfill\qed
\end{corollary}

In fact, because of the complete duality/bijection 
given by Proposition \ref{prop:bijection-extends-to-action},
once a consistent $\tt q$--labelling is chosen, working with either of the following  pairs is the same

\centerline{$
\underbrace{
\big( 
\widehat{\Gamma}_{\tt q}, 
\mathcal{L}_{\mathcal{W}_{\tt q}} \big) 
}_{\substack{ \tt A-\text{map \&} \\ 
\text{consistent }{\tt q}-\text{labelling} } } 
\quad
\longleftrightarrow 
\quad
\underbrace{
\big( \mathscr{T}_\gamma(w(z)), 
w(z)^*\mathcal{L}_\gamma \big)
}_{\substack{ \text{Tessellation arising from} \\
\text{Speiser function } w(z) }}
\quad
\longleftrightarrow 
\underbrace{
(\mathfrak{S}_{\tt q}, \mathcal{L}_{\mathcal{W}_{\tt q}} ) 
}_{\substack{ \text{analytic Speiser graph } \& \\
\text{consistent }{\tt q}-\text{labelling} } } .
$}

\begin{proposition}[The faces of  
$(\mathfrak{S}_{\tt q}, \mathcal{L}_{\mathcal{W}_{\tt q}} )$ 
and 
their relation to singularities of the inverse  $w^{-1}(z)$]
\label{prop:speiser-graph-implies-speiser-function}
The cell decomposition provides the following relationships.

\begin{enumerate}[label=\arabic*),leftmargin=*]
\item
If a ${\tt w}_{\tt j}$--face is a digon, 
then 
the corresponding point $\zeta_\iota \in \Omega_z$ 
(a vertex of valence 2 of the $\tt A$--map
$\widehat{\Gamma}_{\tt q}$)
is an ordinary point,
equivalently a cosingular point 
with cosingular value $w(\zeta_\iota)={\tt w}_{\tt j}$. 

\item
If a ${\tt w}_{\tt j}$--face is a $2m$--gon, for $2 \leq m < \infty$,
then the corresponding point $z_\iota \in \Omega_z$ 
(a vertex of finite even valence greater than or equal to 4 of 
the $\tt A$--map $\widehat{\Gamma}_{\tt q}$)
is an algebraic singularity of the inverse 
$w^{-1}(z)$, with critical value $w(z_\iota)={\tt w}_{\tt j}$.
Moreover, $m_\iota$ is the ramification index of $z_\iota$, 
equivalently the multiplicity of the critical point $z_\iota$.

\item 
If a ${\tt w}_{\tt j}$--face is unbounded 
(an $\infty$--gon), then it 
corresponds to  a  logarithmic
singularity of $w^{-1}(z)$ over the asymptotic value 
${\tt a}_{\tt j} \doteq {\tt w}_{\tt j}$.
\end{enumerate}

\end{proposition}

\begin{proof}
Follows from Definition 
\ref{def:analytic-Speiser-graph} and 
Remark \ref{rem:despues-de-def-Speiser-graph}.4.
\end{proof}

\begin{remark}[Recognizing conformal type of $(B(0,1),J)$
directly from the infinite Speiser graphs $\mathfrak{S}_{\tt q}$]
\label{rem:Conformal-type-Speiser-graph}
It is clear that the Riemann surface $\R_{w(z)}$ associated to a finite Speiser graph 
$\mathfrak{S}_{\tt q}$ has parabolic type.
The recognition of the conformal type of $\R_{w(z)}$ can be done in several different ways:
all of which are equivalent on \emph{any} infinite, finitely--ended, locally finite planar graph, 
in particular on the Speiser graph of a Speiser function $w(z)$. 

\noindent
1.\  
Random--walk criterion; see \cite{Doyle-Snell}, \cite{Woess}, \cite{Lyons-Peres}.

\noindent
2.\  
Resistance (Nash–Williams) criterion; see \cite{Nash-Williams}, \cite{Soardi}.

\noindent
3.\  
Modulus (extremal length) criterion; see \cite{Duffin}, \cite{Benjamini-Schramm}, \cite{Heinonen-Koskela}. 

\noindent
4.\  
Circle-packing criterion; see \cite{He-Schramm}, \cite{Stephenson}. 

\noindent
5.\ 
Isoperimetric or Cheeger Constant $h(\mathfrak{S}_{\tt q})$ criterion; see \cite{Dodziuk}.
\end{remark}

Recalling Remark \ref{rem:bounds-for-q} and using duality, we have the following bounds, in terms of the pre--Speiser graph $\mathfrak{S}$, for $\tt q$ to be a 
consistent $\tt q$--labeling of $\mathfrak{S}$.
\begin{equation}
\label{cota-q-graficas}
{\tt q}_{\text{min}} \doteq 
\max \left\{ 
\begin{array}{c}
\text{valence of the}\\ 
\text{vertices of } \mathfrak{S}
\end{array}
\right \} 
\leq  \, {\tt q} \leq 
{\tt q}_{\text{max}} \doteq
\# \left \{ \begin{array}{c}
\text{faces of } \mathfrak{S}, \\
\text{that are not digons}
\end{array}
\right\} .
\end{equation}

\noindent
Once again, note that ${\tt q}_{\text{max}}$ can be infinite.

\begin{lemma}
\label{lem:Speiser-graph-admite-etiquetados}
A Speiser graph $\mathfrak{S}_{\tt q}$ of index $\tt q$ 
supports at least one consistent 
${\tt q}_{\rm o}$--labelling $\mathcal{L}_{\mathcal{W}_{{\tt q}_{\rm o}}}$, for 
$2 \leq {\tt q}_{\text{min}} \leq {\tt q}_{\rm o} \leq {\tt q} \leq {\tt q}_{\text{max}}$.
\end{lemma}

\begin{proof}
Follows directly from duality and Lemma \ref{lem:A-mapa-admite-etiquetados}.
\end{proof}

\begin{remark}[Non uniqueness of the functions]
\label{rem:no-unicidad}
\hfill

\noindent 
1.
Let us consider in more detail the case when $\text{Stab}(\mathcal{W}_{\tt q})\neq Id$.

\noindent
For ${\tt q}=2$. The subcase of two algebraic singularities
of $w^{-1}(z)$ leads to $w(z)=\lambda(z-a)^n/(z-b)^n$, for
$n \geq 2$, $a\neq b$. 
The subcase of two logarithmic singularities of
$w^{-1}(z)$, leads to 

\centerline{
$w(z)=\ent{}{} (z)=\e^{z}$ \ and \ $w(z)=\htan{}{} (z)=\tanh(z)$,
}

\noindent 
which will give rise to ``elementary blocks'' as in Definition \ref{def:piezas-elementales}.  
See also Example \ref{example:N-exp-tanh}.

\noindent
For ${\tt q}=3$. Up to $Aut(\CC_w)$ the choice of the 
singular values
is $\{0,\, 1,\, \infty \}$, that is $w(z)$ is an algebraic
or transcendental Bely{\u \i}'s function. The theory 
of dessins d' enfants applies for $w(z)^*\gamma$ in this case, 
see \cite{Jones-Wolfart}.

\noindent
2.\ 
E. Drape \cite{Drape}, and C.\ Blanc \cite{Blanc} studied the classification of the Speiser graphs 
associated to $N$--functions with only one branch point over each asymptotic value, 
later W.\ Lotz \cite{Lotz}, in his thesis, dropped the assumption of only one branch point 
over each asymptotic value. 

\noindent
3.\ Given a pair ($\tt A$--map, consistent $\tt q$--labelling), say 
$(\widehat{\Gamma}, \mathcal{L}_{\mathcal{W}_{\tt q}})$, 
note that choosing any representative $\gamma$ of the isotopy class of simple closed paths
relative to the $\tt q$ distinct values $\{ {\tt w}_{1},\ldots,{\tt w}_{\tt q} \}$ does not change
the cyclic order $\mathcal{L}_\gamma = \mathcal{W}_{\tt q}$ and thus does not change the 
consistent $\tt q$--labelling $\mathcal{L}_{\mathcal{W}_{\tt q}}$. 
However, by relaxing the condition of isotopy relative to the $\tt q$ distinct values 
$\{ {\tt w}_{1},\ldots,{\tt w}_{\tt q} \}$, that is by choosing $\widetilde{\gamma}\notin [\gamma]$
but still requiring that $\widetilde{\gamma}$ visit the $\tt q$ distinct values, 
the cyclic order $\mathcal{L}_{\widetilde{\gamma}} \neq \mathcal{L}_\gamma$ changes. Thus
the corresponding consistent $\tt q$--labelling also changes, say to 
$\mathcal{L}_{ \widetilde{\mathcal{W}}_{\tt q} }$, 
and the new pair 
$(\widehat{\Gamma}, \mathcal{L}_{ \widetilde{\mathcal{W}}_{\tt q} } ) \neq 
(\widehat{\Gamma}, \mathcal{L}_{\mathcal{W}_{\tt q}} )$. 
Compare with \cite{Habsch}, \cite{Tairova1}, \cite{Tairova2}.

\end{remark}

We provide some features for the simplest 
families of Speiser functions.

\begin{remark}[Speiser graphs for rational functions]
Let $(\mathfrak{S}_{w(z)}, \mathcal{L}_{\mathcal{W}_{\tt q}} )$
be the analytic Speiser graph of a rational function $w(z)=R(z)$ 
of degree $n \geq 2$.
The dual of $\mathfrak{S}_{w(z)}$ is an 
$\tt A$--map
$\widehat{\Gamma}_{\tt q}$ embedded in $\CW_z$. 
In particular, there are no unbounded faces of $\CW_z \backslash \mathfrak{S}_{w(z)}$.
Furthermore, 
for each ${\tt w}_{\tt j}\in\mathcal{W}_{\tt q}$, 
at least one ${\tt w}_{\tt j}$--face is a 
$2m$--gon for some $m \geq 2$ 
(where $m$ is the multiplicity of the corresponding critical point, 
see Definition \ref{def:multiplicity-of-sing-point}).
In simple words, each label 
${\tt w}_{\tt j}$ comes from a critical point of $w(z)$,
\emph{i.e.}\ the labels are exactly the critical values.

\noindent 
In the case of polynomials of degree $r \geq 2$, once again 
$\mathfrak{S}_{w(z)}$ is finite and embedded in $\CW_z$.
Furthermore, the ${\tt w}_{\tt j}$--face containing $\infty \in \CW_z$ 
has $2r$ edges, 
and in fact ${\tt w}_{\tt j}=\infty\in\CW_w$.
\end{remark}

\subsection{Speiser graphs for $N$--functions}
\label{subsec:Speiser-for-N-functions}
From Proposition \ref{prop:speiser-graph-implies-speiser-function} and the definition of
$N$--function, it follows immediately that an analytic Speiser graph of index $\tt q$
$(\mathfrak{S}_{w(z)}, \mathcal{L}_{\mathcal{W}_{\tt q}})$ for an $N$--function 
$w(z):\Omega_z\longrightarrow \CW_w$, requires that:
\begin{enumerate}[label=\roman*),leftmargin=*]

\item 
its conformal type is parabolic, so $\Omega_z = \CC_z$,

\item
the only bounded faces of $\mathfrak{S}_{w(z)}$ are digons,

\item
the labels $\mathcal{W}_{\tt q}$
are exactly the asymptotic values of $w(z)$,

\item
there are $2\leq{\tt p}<\infty$ unbounded faces of 
$\mathfrak{S}_{w(z)}$:
for $\iota=1,\ldots,\tt p$, the unbounded face with label 
${\tt w}_{{\tt j}(\iota)}={\tt a}_{{\tt j}(\iota)}\in\mathcal{W}_{\tt q}$
corresponds to the class of asymptotic paths $[\alpha_{{\tt j}(\iota)}]$ 
associated to the asymptotic value ${\tt a}_{{\tt j}(\iota)}$ 
(recall \eqref{eq:enumerating-branch-points} for notation),

\item
$\mathfrak{S}_{w(z)}$ 
has $\tt p$ ``logarithmic ends'' and no other ends.
\end{enumerate}

\begin{remark}[Historical origin and remarks on logarithmic ends]
\label{rem:historical-origin-log-ends}
In the literature, the structure appearing in (v) above can be found with different names 
and also for different objects:
the term ``logarithmic end'' appears in 
\cite{Nevanlinna2}\,p.\,292 (who attributes it to A.\ Speiser), 
also \cite{GoldbergOstrovskii}\,p.\,379--380 uses it for
the combinatorial and analytic objects. 
The terms ``logarithmic tower'', ``helicoid'', ``half--logarithmic spiral'' appear in  
\cite{AlvarezMucino1}\,p.\,152,\,194, and
\cite{AlvarezMucino3}\,p.\,23, once again for the combinatorial and analytic objects.
The term ``logarithmic staircase'' is used in \cite{Geyer-Merenkov}\,p.\,362 for the analytic object. 
We shall use the term ``logarithmic end'' for the combinatorial objects,
and ``logarithmic tower'' for the analytic objects
(see Definition \ref{def:logarithmic-tower-for-Riemann-surface} 
and Remark \ref{nomenclatura-logarithmic-tower}).
\end{remark}

\noindent
To make a precise definition in the combinatorial case,
we shall need the following.

\begin{definition}
\label{def:neighbors-of-sets}
Let $v\in V(\mathfrak{S}_{\tt q})$ be a vertex of a Speiser graph $\mathfrak{S}_{\tt q}$,
and $S\subset V(\mathfrak{S}_{\tt q})$ be a subset of vertices a Speiser graph $\mathfrak{S}_{\tt q}$.

\noindent
1.\ 
An \emph{open neighborhood $N(v)$ of the vertex $v$} consists 
of vertices directly adjacent to $v$ by an edge of $\mathfrak{S}_{\tt q}$.

\noindent
2.\ 
The \emph{open neighborhood complex 
of $S$, denoted\footnote{
As far as we know, there is no commonly used notation for the
union of the open neighborhoods the vertices of the set $S$.
} 
by $N^{\rm o}(S)$},
is the union of the open neighborhoods of each vertex in $S$,
that is 

\centerline{
$N^{\rm o}(S) = \bigcup\limits_{v\in S} N(v)$.
}

\noindent
3.\ 
The \emph{open neighborhood of the set $S$, denoted, 
$N(S)$} is $N^{\rm o}(S)\backslash S$.
\end{definition}

\begin{example}
Note that $N^{\rm o}(S)$ may or may not contain $S$. 
For instance $v\notin N(v)$,
since there are no loops in $\mathfrak{S}_{\tt q}$;
that is why $N(v)$ is called an \emph{open} neighborhood.
However, if the subgraph $\mathfrak{S}_{\tt q}[S]$ spanned by $S$ is connected 
and contains more than one vertex, then 
$S\subset N^{\rm o}(S)$.
As is usual in graph theory, saying ``the neighborhood of $S$'' 
should be understood as ``the open neighborhood of $S$''.
\end{example}

\begin{definition}
\label{def:Log-end-for-Speiser-graph}
1.\ A \emph{logarithmic end
$\mathcal{T}$, 
of a Speiser graph $\mathfrak{S}_{\tt q}$ of index 
${\tt q} \geq 3$}, 
is a subset $\mathcal{T}\subset \mathfrak{S}_{\tt q}$
such that:
\begin{enumerate}[label=\roman*),leftmargin=*]

\item
It has an infinite number of ordered vertices 
$v_{2\uptau-1}, v_{2\uptau} \in \{ \times, \circ \}$ with 
$\uptau \in\NN$.

\item
All even vertices $v_{2\uptau}$ are of the same type
($\times $ or $\circ$), 
and all odd vertices $v_{2\uptau-1}$ are of the other type. 

\item 
There are $1\leq\rho_1 < {\tt q}$ edges connecting 
$v_{2\uptau-1}$ to $v_{2\uptau}$ and 
$1 \leq \rho_2 < {\tt q}$ edges connecting $v_{2\uptau}$ to $v_{2 \uptau+1}$,
where ${\tt q} = \rho_1 + \rho_2$.
In other words, $\mathcal{T}$ is formed by a sequence of edge bundles
with alternating number of edges $\rho_1$ and $\rho_2$.

\item The open neighborhood $N(\mathcal{T})$ 
consists of only one vertex.

\item
$\mathcal{T}$ is maximal in $\mathfrak{S}_{\tt q}$, 
that is;
if given any $\mathcal{T}^\prime$ satisfying (i)--(iv) such that
$\mathcal{T} \subset \mathcal{T}^\prime$,
then $\mathcal{T}= \mathcal{T}^\prime$.
\end{enumerate}

\noindent 
2.\ The \emph{nucleus\footnote{
This concept 
appears as ``nucleus'' in \cite{Nevanlinna2} p.\,299,
and ``soul'' in 
\cite{AlvarezMucino1}\,p.\,196 and 
\cite{AlvarezMucino3}\,p.\,56.} 
$\mathfrak{N}_\mathfrak{S}$,  
of a Speiser graph $\mathfrak{S}_{\tt q}$ of index $\tt q$},  
is the subset obtained as the complement of the
logarithmic ends in $\mathfrak{S}_{\tt q}$.
\end{definition}

\begin{remark}
\label{rem:On-logarithmic-ends}
1.\
Note that logarithmic ends $\mathcal{T}$ of a 
Speiser graph $\mathfrak{S}_{\tt q}$ are not Speiser graphs in themselves.
In fact, they are pre--Speiser graphs
because
the first vertex $v_1$ of each logarithmic end 
of $\mathfrak{S}_{\tt q}$
has valence $\rho_1<{\tt q}$, instead of the required ${\tt q}$.
On the other hand, the nucleus $\mathfrak{N}_\mathfrak{S}$ of 
$\mathfrak{S}_{\tt q}$ has ``loose edges'' 
(\emph{i.e.}\ homeomorphic to $[0,1)$, see Figure \ref{fig-4-Speiser})
where the logarithmic ends used to be attached to, so it is not even a pre--Speiser graph.

\noindent
2.\
Condition (iv) of Definition \ref{def:Log-end-for-Speiser-graph}.1 is a condition
that allows for the nucleus to be well defined 
and unique. 

\noindent
3.\
Logarithmic ends, and hence the nucleus, are defined for arbitrary Speiser graphs, not only
for those associated to $N$--functions.
See for instance, Figures \ref{fig-4-Speiser}, \ref{fig:Tessellation-AiBi}--\ref{fig:Speiser3-types},
where the nucleus is colored red and the logarithmic ends are black.
\end{remark}

\begin{example}[Speiser graphs following Nevanlinna brothers]
\label{ejemplo-graficas-Speiser-con-etiquetas}
In Figure \ref{fig-4-Speiser},
we illustrate four Speiser graphs $(\mathfrak{S}_{\tt q}, \mathcal{L}_{\mathcal{W}_{\tt q}} )$
of index ${\tt q}=3,\, 4$.
The cyclic order for the labels 
of the faces of $\mathfrak{S}_{\tt q}$
is $\mathcal{W}_{\tt q} = [1,\ldots,{\tt q} ]$;  
in accordance to Remark \ref{rem:notation-of-labels}.
Usually, the labelling is not shown on digons of the Speiser graph,
however throughout Figure \ref{fig-4-Speiser} they are shown for pedagogical reasons.
The corresponding nucleus are red. 

\noindent
Figure \ref{fig-4-Speiser}.a--c are 
Speiser graphs with ${\tt p}={\tt q}=4$ and nuclei consisting of 1, 2, and 3 vertices respectively.
Clearly, there are Speiser graphs with ${\tt p}={\tt q}=4$ and any number $n\in\NN$ 
of vertices in the nucleus.

\noindent
Figure \ref{fig-4-Speiser}.a. and \ref{fig-4-Speiser}.c appear in \cite{Nevanlinna2} p.\,298,
whereas, Figure \ref{fig-4-Speiser}.d.
appears in \cite{Nevanlinna2} p.\,300 as the corresponding 
Speiser $3$--tessellation.

\noindent
Figure \ref{fig-4-Speiser}.e is an example of a planar graph that does not
satisfy the minimality condition (ii) of Definition \ref{def:etiquetado-consistente-Speiser}, 
hence is not a Speiser graph of index $4$.
However, by forgeting the edges/faces labelled 4, 
it reduces to $\mathfrak{S}_3$ the Speiser graph of index 3 shown in
Figure \ref{fig-4-Speiser}.d.

\begin{figure}[h!tbp]
\begin{center}
\includegraphics[width=0.7\textwidth]{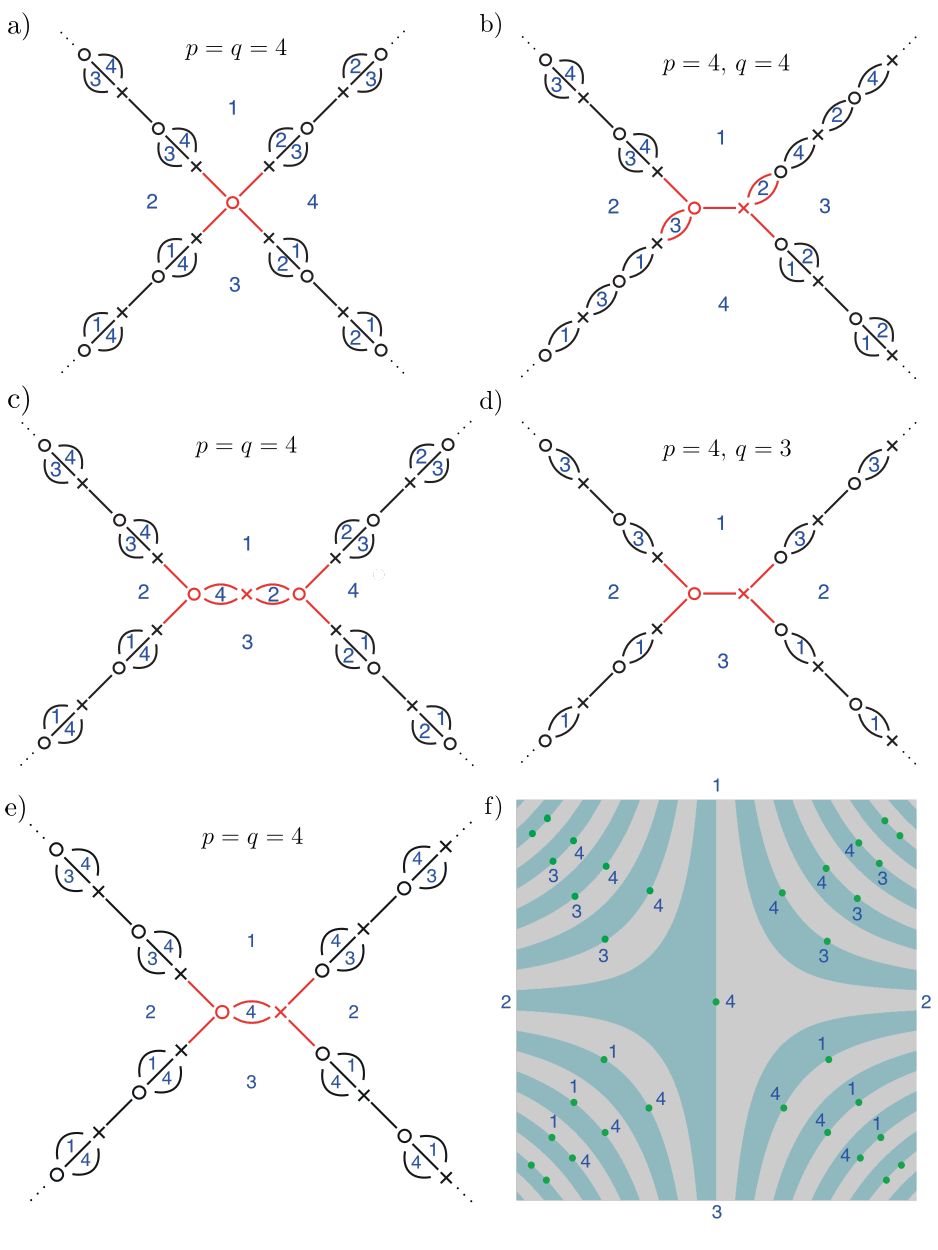}
\caption{ 
Analytic Speiser graphs of index ${\tt q}$ that represent  
$N$--functions $w(z)$; 
(a)--(c)
with ${\tt p}={\tt q}=4$,
(d) with ${\tt p}=4, {\tt q}=3$.
The nuclei are colored red and the ${\tt p}$
logarithmic ends are colored black.
(e) is a labelled Speiser graph that does not have
a consistent 4--labelling,
since every $4$--face is a digon.
(f) The tessellation corresponding to the 
Speiser graph of (e), does not have a
consistent 4--labelling since the label `4' 
only appears on vertices of valence 2.
The labels follow the convention of Remark 
\ref{rem:notation-of-labels}.
}
\label{fig-4-Speiser}
\end{center}
\end{figure}

\end{example}

\begin{lemma}[Nevanlinna \cite{Nevanlinna2}\, ch.\,XI,\,\S12--13]
\label{lemma-q-2-3}
Let $(\mathfrak{S}_{\tt q}, \mathcal{L}_{\mathcal{W}_{\tt q}} )$ be an
analytic Speiser graph 
of index $\tt q$, with $\tt p$ logarithmic ends and whose bounded 
faces are digons.

\begin{enumerate}[label=\arabic*),leftmargin=*]
\item
For each fixed ${\tt p}=2,\, 3$, 
there exist only one Speiser graph $\mathfrak{S}_{\tt q}$ of index $\tt q$.

\item 
For each ${\tt p} \geq 4$, 
there exist infinitely many Speiser graphs.
\end{enumerate}
\end{lemma}

\noindent 
The above Speiser graphs $\mathfrak{S}_{\tt q}$,
characterize families of $N$--functions $\{w(z)\}$, as we will see
in Corollary  \ref{cor:Speiser-graph-characterization}.3.

\begin{proof}
First note that if ${\tt p}=2,3$, then the index ${\tt q} = {\tt p}$.

Case ${\tt p}=2$, 
is in Figure \ref{fig:Tessellation-Exp}.c,
the vertices have valence 2, it is infinite in both 
directions. 
It corresponds to the exponential function (up to post--composition
with a M\"obius transformation), see \cite{Nevanlinna-Paatero}\,\S\,6.16.  

Case ${\tt p}=3$ is in Figure  
\ref{fig:Tessellation-AiBi}.c.

For assertion (2), we consider ${\tt p} = 4$:
the existence of an infinite number of Speiser graphs of index 4,
follows from Figure \ref{fig-4-Speiser}.a--c
by noting that the nuclei can contain an arbitrary number of vertices.
\end{proof}

\begin{remark}[On the multiplicity of asymptotic values]
Note that for the Speiser graph of an $N$--function, 
if some asymptotic values have multiplicity $\geq 1$, then necessarily the number of 
logarithmic ends, namely $\tt p$, is strictly greater than the index $\tt q$. 
This can be observed for instance in Figure \ref{fig-4-Speiser}.d.
\end{remark}

For more examples of Speiser graphs of index $\tt q$, the reader is invited to
consult Example \ref{example:N-exp-tanh}.d for the simplest case of an $N$--function, 
Examples \ref{example:tessellation-Airy}.d and \ref{example:function-wM4alt}.d 
for two non--trivial cases of $N$--functions when ${\tt q} = 3$;
Examples \ref{example:exp-exp}.d, \ref{example:tessellation-expsin}.d, 
and \ref{example:sinexpsin}.d show Speiser graphs of index $\tt q$
for Speiser functions that are not $N$--functions, 
the first for ${\tt q}=3$ and the last two for ${\tt q}=4$;
Example \ref{example:Speiser3-types} shows a Speiser graph of index ${\tt q}=4$ with
one logarithmic end, an infinite number of unbounded faces and 
an infinite number of bounded $4$--gons.

\begin{example}[There are  ``ends'' 
of a Speiser graph $\mathfrak{S}_{\tt q}$ of index $\tt q$, which are not logarithmic ends]
\label{example:ends-not-logarithmic}
Consider 
$w(z)=\sin(z^2)$, it is a Speiser function with $\mathcal{SV}_w=\{ -1,1, \infty \}$.
It has an infinite number of algebraic singularities and 4 logarithmic singularities.
However, its Speiser graph of index 3 has no logarithmic ends; 
the unbounded faces are separated by an
infinite collection of 4--gons and digons, 
see figure 14 in \cite{GoldbergOstrovskii}\,p.~360.
\end{example}

Moreover, it is easy to see that
between two contiguous ``ends'' there must be an unbounded face of $\mathfrak{S}_{\tt q}$. 
This in turn implies that there are the same number of ``ends'' as unbounded faces.

\section{A complete correspondence}

\begin{theorem}
\label{th:main-theorem}
Let $\Omega_z$ be a simply connected Riemann surface, 
and let
${\tt q} \geq 2$.
There exists a one to one correspondence between:

\begin{enumerate}[label=\arabic*),leftmargin=*]

\item
Speiser functions 

\centerline{$w(z): \Omega_z \longrightarrow \CW_w$,}

\noindent  
provided with a cyclic order $\mathcal{W}_{\tt q}$ 
for its ${\tt q}$ singular values.

\item 
Speiser Riemann surfaces 

\centerline{$\R_{w(z)} \subset  \Omega_z \times \CW_w$, }

\noindent 
provided with a cyclic order $\mathcal{W}_{\tt q}$ 
for its ${\tt q}$ singular values.

\item
Speiser $\tt q$--tessellations 

\centerline{
$
\big( 
\underbrace{(\Omega_z \cup \partial_{\mathcal{I}} \Omega_z ) \backslash w(z)^*\gamma}_{\text{tessellation}},  
\underbrace{w(z)^*\mathcal{L}_\gamma}_{\substack{ \text{consistent}\\
{\tt q}-\text{labelling}} } 
\big)
$.
}

\item
Analytic Speiser graphs of index $\tt q$

\centerline{
$(
\underbrace{\mathfrak{S}_{w(z)} }_{\substack{ \text{Speiser} \\ \text{graph} }}, 
\underbrace{ w(z)^*\mathcal{L}_\gamma }_{
\substack{ \text{consistent}\\
{\tt q}-\text{labelling}}
} )$.}

\end{enumerate}
\end{theorem}

\begin{proof}
The proof of Theorem \ref{th:main-theorem} proceeds as in 
rows two and three of
Diagram \ref{dia:correspondencia-completa} in the Introduction.

\noindent
(1) $\Longleftrightarrow$ (2) is 
Definition \ref{def:Speiser-q-function}.2 (classical).

\noindent
(1) $\Longrightarrow$ (3) is Theorem \ref{teo:principal}.1
(Schwarz--Klein--Speiser's algorithm)

\noindent
(3) $\Longrightarrow$ (1) is Theorem \ref{teo:principal}.2
 (lifting the tessellation $(\mathscr{T},\mathcal{L}_{\mathcal{W}_{\tt q}})$ to $\R_{w(z)}$
 followed by (1) $\Longleftrightarrow$ (2)).

\noindent
(3) $\Longleftrightarrow$ (4) is given by Proposition \ref{prop:bijection-extends-to-action}
(duality of the $\tt A$--map and the analytic Speiser graph).

\end{proof}

\begin{table}[htp]
\caption{Some relationships between the different objects, 
note that
${\tt w}_{{\tt j}(\cdot )}, \ {\tt a}_{{\tt j}(\cdot )} 
\in \mathcal{W}_{\tt q} 
\subset \CW_w $:
}
\begin{center}
\begin{tabular}{|c|c|c|c|}
\hline
&&& \vspace{-.3cm}
\\
Speiser function & Speiser Riemann & Speiser $\tt q$--tessellation & Analytic Speiser graph 
\\
$w(z)$ & surface $\R_{w(z)}$ & $ (\Omega_z \cup \partial_{\mathcal{I}} \Omega_z )\backslash \widehat{\Gamma}_{\tt q}$ & $(\mathfrak{S}_{\tt q}, \mathcal{L}_{\mathcal{W}_{\tt q}} )$ of index $\tt q$ 
\\
&&& \vspace{-.3cm}
\\
\hline
\hline
singular value & & vertex label & 
 \\
${\tt w}_{{\tt j}(\iota)} \in \CW_w$ & $(z_\iota, {\tt w}_{{\tt j}(\iota)}, m_{\iota} )$ & for a vertex of $\widehat{\Gamma}_{\tt q}$ & face label 
\\
 & & with valence $\geq 4$ & 
\\
\hline
&&& \vspace{-.35cm}
\\
critical point $z_\kappa \in \Omega_z$ & finitely ramified  & vertex of $\widehat{\Gamma}_{\tt q}$ & bounded  
\\
of order 
& branch point & with valence & $\tt w_{{\tt j}(\kappa)}$--region 
\\
$2\leq m_\kappa < \infty$ 
& $(z_\kappa, {\tt w}_{{\tt j}(\kappa)}, m_{\kappa} )$ & $2 m_{\kappa}$ & is a $2 m_{\kappa}$--gon
\\ 
&&& \vspace{-.35cm}
\\
\hline
&&&  \vspace{-.35cm}
\\
$z_\sigma = U_{{\tt a}_\sigma} \in \partial_{\mathcal{I}} \Omega_z$ & infinitely ramified & vertex of $\widehat{\Gamma}_{\tt q}$ & unbounded  \\
a logarithmic & branch point & with $\infty$ valence & ${\tt a}_{{\tt j}(\sigma)}$--region \\
singularity of $w^{-1}(z)$ & $(z_\sigma, {\tt a}_{{\tt j}(\sigma)}, \infty )$ & on $\partial_\mathcal{I} \Omega_z \subset$ & $\infty$--sided polygon \\
over ${\tt a}_{{\tt j}(\sigma)} \in \CW_w$ &  & $\partial B(0,1)\cong \partial \RR^2$,& 
\\
&&& \vspace{-.35cm}
\\
\hline
cosingular point & regular point & ``hollow'' vertex of & digon 
\\
$z \in \Omega_z$  & $(z, {\tt w}_{{\tt j}}, 1 )$
 & $\widehat{\Gamma}_{\tt q}$ with valence 2 &
\\ 
\hline
\end{tabular}
\end{center}
\label{table:tabla-relaciones}
\end{table}%

\begin{corollary}[Speiser graph characterization of 
rational functions and $N$--functions]
\label{cor:Speiser-graph-characterization}
Consider an analytical Speiser graph 
$(\mathfrak{S}_{\tt q}, \mathcal{W}_{\tt q})$ of index ${\tt q}$ in $\Omega_z$.
For each case (1)--(3), the following statements are equivalent.
\begin{enumerate}[label=\arabic*.,leftmargin=*]
\item 
\begin{enumerate}[label=\roman*),leftmargin=*]
\item 
The associated function $w(z)$ is rational.
	
\item 
The Speiser graph $\mathfrak{S}_{\tt q}$ of index $\tt q$ is finite.

\item 
The Speiser graph 
$\mathfrak{S}_{\tt q}$ of index $\tt q$ is properly embedded in $\esf$.

\item 
Considering the cell decomposition 
$\Omega_z \backslash \mathfrak{S}_{\tt q}$, 
the number of unbounded faces
is zero and the number of bounded faces,
that are not digons, is finite 
(equal to the number of critical points of 
the associated function $w(z)$).

\end{enumerate}
	
\smallskip	
\item	
\begin{enumerate}[label=\roman*),leftmargin=*]
\item The associated function $w(z)$ is, up to M\"obius transformation,
a polynomial of degree $n\geq 2$.

\item 
The Speiser graph $\mathfrak{S}_{\tt q}$ of index $\tt q$ is finite and,
considering the cell decomposition	
$\Omega_z \backslash \mathfrak{S}_{\tt q}$,
there is an ${\tt w}_{\tt j}$--face which is a 
$2n$--gon, $n\geq 2$.

\end{enumerate}

\smallskip	
\item 
\begin{enumerate}[label=\roman*),leftmargin=*]	
\item 
The associated $w(z)$ is an $N$--function.

\item 
The Speiser graph $\mathfrak{S}_{\tt q}$ of index $\tt q$ is infinite and,
considering the cell decomposition	
$\Omega_z \backslash \mathfrak{S}_{\tt q}$,
there are a finite number ${\tt p}\geq 2$ of unbounded faces 
(equal to the number of singular values of the associated function $w(z)$, counted with multiplicity),
and all bounded faces are digons.	
Note that the index of the Speiser graph is ${\tt q}\leq{\tt p}$.
	
\end{enumerate}
\end{enumerate}
\end{corollary}

\begin{proof}
Both Theorem \ref{th:main-theorem} and 
Proposition \ref{prop:speiser-graph-implies-speiser-function}.2 
play key roles in all statements.

Statement (1) now follows directly from Definition \ref{def:Speiser-graph}.i.

Statement (2),
up to M\"obius transformations: 
the label assigned to  
the $\tt w_j$--face, that is a $2n$--gon, is the critical
value ${\tt w}_{\tt j} = \infty\in\mathcal{W}_{\tt q}$ of 
the critical point $\infty \in \CW_z$.

Statement (3)
follows from the definition of $N$--function in 
\S \ref{sec:N-functions}. Note that the singular values
are in fact asymptotic values. 
\end{proof}

The analogous Speiser tessellation characterization of Speiser functions is left for the interested reader.

\section{When does a pre–Speiser graph represent a Speiser function?}
\label{sec:planar-graphs-Thrurston}
The original question 
\emph{what is the shape of a rational function?}, 
was posed in 2010 by  W.\,P.\,Thurston in MathOverflow \cite{Thurston2}.  
This can be translated in terms of Speiser functions,
$\tt t$--graphs and pre--Speiser graphs.
\begin{equation}
\tag{\ref{eq:pregunta-mosaico-speiser}}
\begin{array}{c}
\text{Question: \it is it possible to characterize whether a {\tt t}--graph }\Gamma, 
\\
\text{\it or equivalently a pre--Speiser graph }\mathfrak{S},
\text{\it represents a Speiser function?}
\end{array}
\end{equation}

\noindent 
As far as we known, 
the problem of characterizing when a 
$\tt t$--graph $\Gamma$
arises from a rational function was considered by 
Speiser \cite{Speiser}.
In 2020, 
a report of the results of W.\,P.\,Thurston, S.\,Koch and T.\,Lei
for generic rational functions
$R(z)$ appeared in \cite{Koch-Lei},
\emph{generic} of degree $n$ means that $R(z)$ has $2n-2$ distinct critical values.
The report provides negative examples and 
states conditions under which a planar
tessellation arises from generic rational functions $R(z)$ and suitable  paths $\gamma$.
In 2015, J.\,Tomasini \cite{Tomasini}
proved a characterization for rational functions
in the general case,
with a different presentation. 
A constructive method
for $\tt t$--graphs $\Gamma$ 
originating from generic polynomials 
was studied in L.\,Gonz\'alez--Cely \emph{et al.}
in \cite{GonzalezMucino};
they provide a different characterization by 
showing an explicit construction of a consistent $\tt q$--labelling. 

In \S\ref{sec:Certain-contrains} we explore the suitable values
of $\tt q$ for a pre--Speiser graph.
\S\ref{sec:bipartite-transportation-problem} provides the necessary and sufficient conditions for a pre--Speiser 
graph to be extendable to a Speiser graph of index $\tt q$. 
These subsections consider pre--Speiser Graphs with an arbitrary number of faces. 
For the rational case; in
\S\ref{sec:Koch-Lei} we review W.P.\,Thurston's \emph{et al.}  
approach, and in
\S \ref{sec:Tomasini} a comparison of Tomasini's method is considered.

\subsection{Certain constraints on the extension of pre--Speiser graphs to Speiser graphs.}
\label{sec:Certain-contrains}

Recalling that a cyclic order
$\mathcal{W}_{\tt q}$ is 
equivalent to the isotopy class $[\gamma]$
of paths $\gamma$ relative to the values $\{ {\tt w}_\ell \}_{\ell=1}^{\tt q}$, 
we can now answer Question \eqref{eq:pregunta-mosaico-speiser}.

\begin{corollary}[What is the shape of a Speiser function?]
\label{cor:shape-of-Speiser-function}
A $\tt t$--graph $\Gamma$, or equivalently a pre--Speiser graph $\mathfrak{S}$, 
supports a consistent $\tt q$--labelling $\mathcal{L}_{\mathcal{W}_{\tt q}}$
if and only if 
there exist Speiser functions $w(z)$ with
cyclic orders $\mathcal{W}_{\tt q}$ on their singular values such that 
$\mathscr{T}_\gamma(w(z))=\mathscr{T}(\Gamma)$.
\end{corollary}

\begin{proof}
Follows directly from Theorem \ref{teo:principal}, Lemma 
\ref{lem:no-unicidad-funciones-teselaciones} and Remark \ref{rem:minimality-condition}.
\end{proof}
Note that the consistent $\tt q$--labelling associated to the Speiser function with 
cyclic order $\mathcal{W}_{\tt q}$ is given by 
$\mathcal{L}_{\mathcal{W}_{\tt q}} = w(z)^*\mathcal{L}_\gamma$.

\noindent
From the theory developed up to this point, it is clear that
Question \eqref{eq:pregunta-mosaico-speiser}
is equivalent to 
\emph{finding necessary and sufficient conditions for 
when a $\tt t$--graph $\Gamma$ 
(or its dual the pre--Speiser graph $\mathfrak{S}$)
can be extended to 
at least one 
$\tt A$--map $\widehat{\Gamma}_{\tt q}$
(or their duals Speiser graphs $\mathfrak{S}_{\tt q}$ of 
index $\tt q$).}

\smallskip
As Example \ref{example:non-uniqueness-extension-Speiser-graph} below shows, 
both a $\tt t$--graph $\Gamma$ and a pre--Speiser graph $\mathfrak{S}$ can be extended
to $\tt A$--maps $\widehat{\Gamma}_{\tt q}$ and Speiser graphs $\mathfrak{S}_{\tt q}$
of index $\tt q$ for distinct values of ${\tt q}<\infty$.

\begin{example}[Non uniqueness of the extended Speiser graph of index $\tt q$]
\label{example:non-uniqueness-extension-Speiser-graph}
Consider the planar tessellation 
$\mathscr{T}(\Gamma)$, with $\tt t$--graph $\Gamma$, 
attributed to W.\,P.\, Thurston, 
that appears as figure 10 in \cite{Koch-Lei}, and that we
reproduce here in Figure \ref{fig:Thurston-Koch-TanLei}.a \emph{without labels}. 
In Figure \ref{fig:Thurston-Koch-TanLei}.b is the corresponding dual; the pre--Speiser graph $\mathfrak{S}$.

\begin{figure}[h!tbp]
\begin{center}
\includegraphics[width=0.6\textwidth]{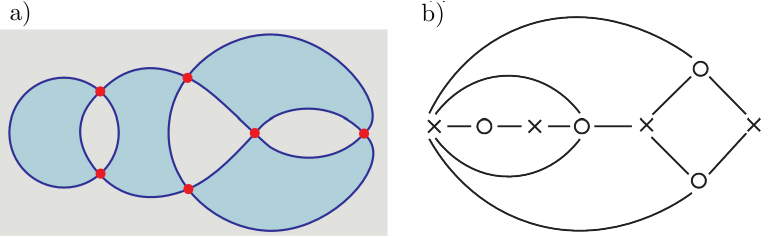}
\caption{
(a) Tessellation $\mathscr{T}(\Gamma)$
of the sphere $\esf$, its $\tt t$--graph $\Gamma$
has $k$--gons, $k=2,3,4,5$, as tiles.
(b) The dual of (a) is a planar bipartite
graph on $\esf$, the pre--Speiser graph $\mathfrak{S}$, 
with vertices of valence $k$.
}
\label{fig:Thurston-Koch-TanLei}
\end{center}
\end{figure}

\noindent
From Figure \ref{fig:Thurston-Koch-TanLei}.a we see that, since the largest tiling is a 5--gon 
(equivalently in Figure \ref{fig:Thurston-Koch-TanLei}.b the largest valence for a vertex is 5), 
then the minimum number of labels that are needed to specify the Speiser tessellation 
(equivalently the corresponding $\tt A$--map $\widehat{\Gamma}_{\tt q}$ 
and Speiser graph $\mathfrak{S}_{\tt q}$) is ${\tt q}=5$. 
Moreover, since there are 6 vertices of the $\tt t$--graph $\Gamma$ of 
Figure \ref{fig:Thurston-Koch-TanLei}.a
(equivalently in Figure \ref{fig:Thurston-Koch-TanLei}.b there are 6 faces 
on the pre--Speiser graph $\mathfrak{S}_{\tt q}$),
then the maximum number of labels has to be ${\tt q}=6$.  
Thus there are two possibilities for ${\tt q}$, namely 5 and 6.

Case ${\tt q}=5$. 
Consider Figure \ref{fig:MosaicoSpeiser-q5};
in (a) labels 
$\mathcal{W}_5 \doteq [1,2,3,4,6]$ 
are added to the 6 vertices 
of the $\tt t$--graph $\Gamma$ (thus necessarily at least one label is repeated); 
recall Remark \ref{rem:notation-of-labels}.
In (b), by edge subdivision, vertices of valence two are added 
so as to make each tile a 
$5$--gon with a consistent 
5--labelling $\mathcal{L}_{\mathcal{W}_5}$
as in Definition \ref{def:etiquetado-consistente}; 
thus a Speiser $5$--tessellation $(\mathscr{T}(\widehat{\Gamma}_5),\mathcal{L}_5)$.
In (c) the corresponding Speiser graph $\mathfrak{S}_5$ of index ${\tt q}=5$ is shown.
The corresponding rational function $R(z)$, 
obtained from the specific choice of $\mathcal{W}_5$,
has  simple critical points two of
which lie over the same critical value ${\tt w}_4$. 
Thus, $R(z)$ is not a generic rational function.

\begin{figure}[h!tbp]
\begin{center} 
\includegraphics[width=0.6\textwidth]{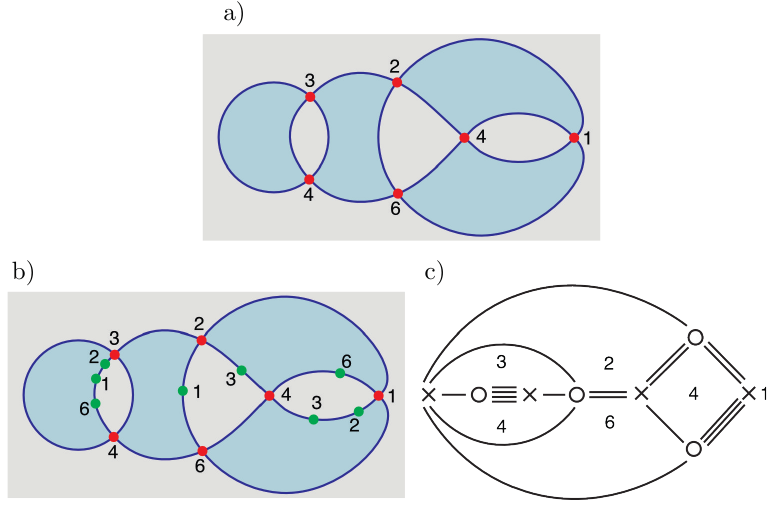}
\caption{
(a) Labelled $\tt t$--graph with labels 
$\mathcal{W}_5 \doteq [1,2,3,4,6]$.
(b) Vertices of valence 2 are added so as to have a consistent 
labelling $\mathcal{L}_{\mathcal{W}_5}$, 
thus obtaining a Speiser $5$--tessellation.
(c) The dual of (b), \emph{i.e.}\ a Speiser graph of index $5$.
(b)--(c) represent a rational function 
$R(z)$
with 6 simple critical points and
5 distinct critical values;
the critical value of multiplicity two is the one corresponding to the label ${\tt w}_4=4$. 
Thus, $R(z)$
is not a generic rational function.
}
\label{fig:MosaicoSpeiser-q5}
\end{center}
\end{figure}

Case ${\tt q}=6$. 
Consider Figure \ref{fig:MosaicoSpeiser-q6}; 
in (a) labels $\mathcal{W}_6 \doteq [1,2,3,4,5,6]$ 
are added to the 6 vertices 
of the $\tt t$--graph $\Gamma$;
recall Remark \ref{rem:notation-of-labels}.
In (b), by edge subdivision, vertices of valence two are added so as to make each tile a 
$6$--gon with
a consistent 6--labelling $\mathcal{L}_{\mathcal{W}_6}$ as in Definition \ref{def:etiquetado-consistente}; 
thus a Speiser $6$--tessellation $(\mathscr{T}(\widehat{\Gamma}_6),\mathcal{L}_6)$. 
In (c) the corresponding Speiser graph $(\mathfrak{S}_6,\mathcal{L}_6)$ of index $6$ is shown.
The corresponding rational function $R(z)$, obtained from the specific choice of the 
6 distinct ordered critical values $\mathcal{W}_6$, has 6 simple critical points, one over
each of the 6 distinct critical values. Thus, a generic rational function 
$R(z)$.
This corresponds to the the two colored tilling of \cite{Koch-Lei} that appears in 
figure 5 (a).

\begin{figure}[h!tbp]
\begin{center}
\includegraphics[width=0.6\textwidth]{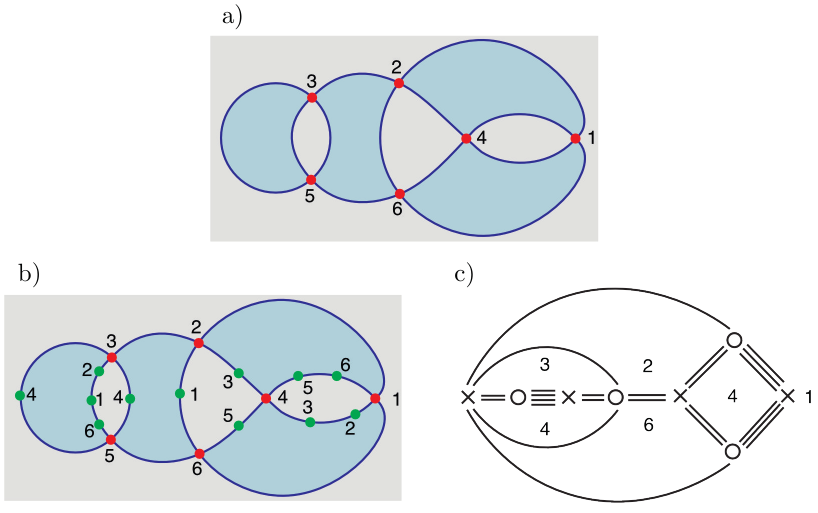}
\caption{
(a) Labelled $\tt t$--graph with labels 
$\mathcal{W}_6 \doteq [1,2,3,4,5,6]$.
(b) Vertices of valence 2 are added so as to have a consistent 
labelling $\mathcal{L}_{\mathcal{W}_6}$, 
thus obtaining a Speiser $6$--tessellation.
(c) The dual of (b), \emph{i.e.}\ a Speiser graph of index $6$.
(b)--(c) represent a rational function 
$R(z)$
with 6 simple critical points and 
exactly 6 distinct critical values. 
Thus, a generic rational function 
$R(z)$ of degree 4.
}
\label{fig:MosaicoSpeiser-q6}
\end{center}
\end{figure}

\end{example}

With this in mind, the following result is useful and follows immediately from 
Table \ref{table:tabla-relaciones}, 
Proposition \ref{prop:speiser-graph-implies-speiser-function}, Theorem \ref{th:main-theorem}, 
and the bounds \eqref{cota-q-mosaicos} and \eqref{cota-q-graficas}.

\begin{lemma}
\label{lem:cotas-q}
Given a $\tt t$--graph $\Gamma$ or a pre--Speiser graph $\mathfrak{S}$, 
upper and lower bounds for ${\tt q}\geq 2$, 
in order that they extend to a 
Speiser $\tt q$--tessellation 
$\big(\mathscr{T}(\widehat{\Gamma}_{\tt q}),\, \mathcal{L}_{\mathcal{W}_{\tt q}} \big)$
or an analytic Speiser graph 
$(\mathfrak{S}_{\tt q}, \mathcal{L}_{\mathcal{W}_{\tt q}})$
of index $\tt q$, 
are as follow
\begin{align}
\tag{\ref{cota-q-mosaicos}}
\max \# \left \{ \begin{array}{c}
\text{vertices on } \partial \overline{T_\alpha},  \\ 
\text{ for }
T_\alpha \in \mathscr{T}( \Gamma)
\end{array}
\right \}
\leq  & \, {\tt q} \leq 
\# \left \{ 
\begin{array}{c}
\text{vertices of } \Gamma \text{ with}\\
\text{valence } \geq 4
\end{array}
\right\}  ,
\\
\tag{\ref{cota-q-graficas}}
\max \left\{ 
\begin{array}{c}
\text{valence of the}\\ 
\text{vertices of } \mathfrak{S}
\end{array}
\right \} 
\leq & \, {\tt q} \leq 
\# \left \{ \begin{array}{c}
\text{faces of } \mathfrak{S}, \\
\text{that are not digons}
\end{array}
\right\} .
\end{align}
\hfill\qed
\end{lemma}

In the case of $\tt t$--graphs with $2n<\infty$ tiles 
(presumably corresponding to rational functions, 
say of degree $n\geq 2$), 
the Riemann--Hurwitz formula implies that
the right hand side 
of Equation \eqref{cota-q-mosaicos} 
is also bounded by $2n-2$.
Since the upper bounds
\eqref{cota-q-mosaicos}--\eqref{cota-q-graficas} are specific to the 
particular $\tt t$--graph or pre--Speiser graph, it is usually better than that 
given by the Riemann--Hurwitz formula. 
Moreover, 
\eqref{cota-q-mosaicos}--\eqref{cota-q-graficas}
also apply to infinite $\tt t$--graphs and pre--Speiser graphs.

\subsection{Necessary and sufficient conditions for a pre--Speiser graph to be extendable to a Speiser graph of index $\tt q$}
\label{sec:bipartite-transportation-problem}
Recalling Definition \ref{def:pre-Speiser-graph}.1,
we shall now proceed to answer  
Question \eqref{eq:pregunta-mosaico-speiser}
in terms of pre--Speiser graphs and Speiser graphs.
The setup is as follows.

\noindent
Suppose you have a connected bipartite 
pre--Speiser 
graph $\mathfrak{S}=(V_\circ \cup V_\times, E)$ embedded in the 
plane with the sets of vertices
$V_\circ$ and $V_\times$ having the same cardinality, say 
$n$ (possibly infinite). 
The set $E=E(\mathfrak{S})$ denotes the edges of $\mathfrak{S}$.
Given a vertex $v\in V_\circ \cup V_\times$, denote by $\deg(v) = \rho_v$ the valence of $v$. 
Further assume that $2 \leq \rho_v \leq {\tt q}$ for all vertices and that 
${\tt q}<\infty$ lies within the bounds \eqref{cota-q-graficas} given by Lemma \ref{lem:cotas-q}.

We are looking for necessary and sufficient conditions 
such that $\mathfrak{S}$ can be extended to a 
planar $\tt q$--regular bipartite 
multigraph\footnote{Recall that
a multigraph admits edge bundles, see 
Remark \ref{rem:despues-de-def-Speiser-graph}.5.
}

\centerline{
$\mathfrak{S}_{\tt q} = (V_\circ \cup V_\times, E\cup E_{new})$,
}

\noindent
by just adding copies of the edges $E$ 
(\emph{i.e.}\ $E_{new}$ are copies of $E$) to $\mathfrak{S}$.

\begin{definition}
\label{def:deficiency-neighbors-of-sets-within-part}
Given a vertex $v\in V_\circ \cup V_\times$, with valence $\deg{v}=\rho_v$, 
we say that $d_v \doteq {\tt q}-\rho_v$ is the \emph{deficiency of vertex $v$}.
\end{definition}
 
Recalling Definition \ref{def:neighbors-of-sets}.2, we 
see that.

\begin{lemma}
Given a finite set $S\subset V_\circ$, it follows that 
the neighborhood\footnote{
To be completely clear, this is the open neighborhood as in 
Definition \ref{def:neighbors-of-sets}.3.
}
of $S$ is a subset of $V_\times$, in other words
$N(S)\subset V_\times$. 
Similarly for a finite set $T\subset V_\times$, the neighborhood of $T$ is 
a subset of $V_\circ$, in other words
$N(T)\subset V_\circ$.
\hfill\qed
\end{lemma}

We now have a \emph{bipartite transportation problem:}
a classical optimization problem in operation research that models the distribution of resources from multiple supply sources to multiple demand destinations. It can be naturally formulated using bipartite graphs, where supply nodes form one partition and demand nodes form another, with edges representing possible routes and their associated costs 
(in our case the cost is the same for each edge).

The vertices $v\in V_\circ$ have supply $d_v$
and vertices $w\in V_\times$ have demand $d_w$, 
and it is possible to ship only along existing 
edges with unlimited capacity.

The question now is, 
whether we can choose non--negative integers 
$x(e)$ for each edge 
$e = \overline{\times_\alpha \circ_\beta}  \in E$
(how many parallel copies of $e$ to add, in order to 
form edge bundles),
such that every vertex $v \in V_\circ \cup V_\times $ reaches valence $\tt q$.

\noindent 
For a given edge $e\in E(\mathfrak{S})$, 
the  corresponding edge in the extended graph 
$\mathfrak{S}_{\tt q}$ 
will have multiplicity $1 + x(e)$; 
recall rows 4 and 5 in Diagram \ref{dia:correspondencia-completa}.

The solution for this bipartite transportation
problem is classic and well known. 
See for instance \cite{Schrijver}, \cite{Lovasz-Plummer}, \cite{Ford-Fulkerson} 
and references therein; for foundational resources see \cite{Gale}, \cite{Ryser}. 

There exist non--negative integers $x(e)$ solving the above question if and only if 
the following conditions are satisfied.
\begin{enumerate}[label=\arabic*),leftmargin=*]	
\item 
Global balance   

\centerline{$
\sum\limits_{v\in V_\circ } d_v = \sum\limits_{w\in V_\times } d_w .
$} 

\item 
Local balance / Hall type neighborhood inequalities;
for every finite sets $S\subset V_\circ$ and $T\subset V_\times$, the following inequalities hold

\centerline{$
\sum\limits_{v\in S } d_v \leq \sum\limits_{w\in N(S) } d_w ,
$} 
 
\centerline{$
\sum\limits_{w\in T } d_w \leq \sum\limits_{v\in N(T) } d_v .
$} 
\end{enumerate}
These are the max--flow / min--cut conditions on the bipartite transportation network with
capacities $d_v$ on the vertex arcs and infinite capacities on edge arcs; 
by unimodularity, a feasible real solution is integral.

For countably infinite graphs with finite $\tt q$, the same conditions, required for all 
finite subsets $S\subset V_\circ$ and $T\subset V_\times$ are necessary and sufficient
(sufficiency follows by an exhaustion by finite induced subgraphs and a compactness/limit
argument).

The above proves the result below, which completes
Diagram \ref{dia:correspondencia-completa}.

\begin{theorem}[Pre--Speiser graph extension to Speiser graph of index $\tt q$]
\label{prop:balance-conditions-Speiser-graphs}
A pre--Speiser graph
$\mathfrak{S}=(V_\circ \cup V_\times, E)$ embedded in the 
plane
extends to a Speiser graph $\mathfrak{S}_{\tt q}$ of index $\tt q$ 
if and only if 
$\mathfrak{S}$ satisfies the following conditions:

\begin{enumerate}[label=\arabic*),leftmargin=*]	
\item 
Global balance

\centerline{$
\sum\limits_{v\in V_\circ} ({\tt q} - \rho_v) = \sum\limits_{w\in V_\times} ({\tt q} - \rho_w) .
$} 

\item Local balance / Hall neighborhood inequalities:
for every finite sets $S\subset V_\circ$ and $T\subset V_\times$, 
\begin{eqnarray}
\label{eq:Hall-S}
\sum\limits_{v\in S} ({\tt q} - \rho_v) &\leq \sum\limits_{w\in N(S)} ({\tt q} - \rho_w) ,
\\
\label{eq:Hall-T}
\sum\limits_{w\in T} ({\tt q} - \rho_w) &\leq \sum\limits_{v\in N(T)} ({\tt q} - \rho_v) .
\end{eqnarray}

\end{enumerate}

\hfill\qed
\end{theorem}

Finally a use of Lemma \ref{lem:Speiser-graph-admite-etiquetados} proves.

\begin{corollary}
A pre--Speiser graph $\mathfrak{S}$ represents a Speiser function
if and only if it satisfies conditions (1) and (2) of 
Theorem \ref{prop:balance-conditions-Speiser-graphs}.
\hfill\qed
\end{corollary}

\subsection{W.\,P.\,Thurston \emph{et al.}'s approach}
\label{sec:Koch-Lei}
In \cite{Koch-Lei}, they consider a planar tessellation with alternating colors, 
which in our language corresponds to
$\mathscr{T}(\Gamma)$ arising from a $\tt t$--graph 
$\Gamma$.
Three conditions are required 
in order to 
characterize the planar
tessellations that represent generic
rational functions.

\noindent 
i) The tiles/faces of $\mathscr{T}(\Gamma)$ 
are Jordan regions.

\noindent 
ii) Global balance. 
For finite graphs $\Gamma$, 
with an alternating blue--gray 
colouring of the faces of $\mathscr{T}(\Gamma)$, 
there are the same number of blue faces as there are of gray faces.

\noindent 
iii) Local balance. 
For any oriented simple closed path 
in $\Gamma$, say $\Upsilon$, that is bordered by blue faces on the left and grey on the right 
(except at the corners), 
there are strictly more blue faces than grey faces on the left side of $\Upsilon$.

In figure 3 of \cite{Koch-Lei},
a tessellation that is globally balanced but not locally 
balanced is shown.
The next example illustrates that lack of local balance is actually 
very easy to obtain.

\begin{example}[Every $\tt t$--graph can be modified to one
without local balance]
\label{example:non-locally-balanced}
Let $\mathscr{T}(\Gamma)$ 
be a globally
and locally balanced tessellation, finite or infinite.
Consider 
any edge of $\Gamma$
as in Figure \ref{fig:mosaico-speiser-graph-No-localbalance}.a.
Replace the edge with the graph shown in  
Figure \ref{fig:mosaico-speiser-graph-No-localbalance}.c, 
to obtain the graph of Figure \ref{fig:mosaico-speiser-graph-No-localbalance}.b.
This new $\tt t$--graph $\Gamma^\prime$ is still globally balanced, 
but it is not locally balanced;
the path $\alpha$, that does not satisfy the 
requirements for local balance, 
is colored green in the figures. 
Note that one could also use Figure \ref{fig:mosaico-speiser-graph-No-localbalance}.d, 
instead of Figure \ref{fig:mosaico-speiser-graph-No-localbalance}.c.

\noindent
By duality, the corresponding statement for the pre--Speiser graph is: replace 
the single horizontal edge by the dual graph of 
Figure \ref{fig:mosaico-speiser-graph-No-localbalance}.c, as indicated
in Figures \ref{fig:mosaico-speiser-graph-No-localbalance}.e--f.

We shall use 
Theorem \ref{prop:balance-conditions-Speiser-graphs}, 
particularly \eqref{eq:Hall-T}, to show that 
it is not possible to add edges to the 
pre--Spesiser graph $\mathfrak{S}^\prime$ to
make it a regular graph (thus a Speiser graph $\mathfrak{S}_{\tt q}$ of index $\tt q$). 
The problem lies with the vertices inside the green area.
We shall work with Figure \ref{fig:mosaico-speiser-graph-No-localbalance}.f: 
a subset of the pre-Speiser graph $\mathfrak{S}$, dual to $\Gamma^\prime$.
The vertices of the subset of $\mathfrak{S}$, have been labelled as 
$\{ \circ_1, \ldots, \circ_5\}\subset V_\circ$ and 
$\{ \times_1, \ldots, \times_5\}\subset V_\times$.
Consider the set
$T=\{\times_2, \times_3\}$, thus $N(T)=\{\circ_4, \circ_5\}$, 
and observe that 

\centerline{
$\sum\limits_{w\in T} ({\tt q}-\rho_w) > \sum\limits_{v\in N(T)} ({\tt q}-\rho_v)$,
\quad for $5\leq {\tt q} \leq 8$.
}

\noindent
Thus, $\mathfrak{S}$ can not be extended to a Speiser graph of index $\tt q$, 
for $5\leq {\tt q} \leq 8$ (and consequently by Lemma \ref{lem:cotas-q}, 
for any $\tt q$).
Equivalently, $\Gamma^\prime$ can not be extended to an 
$\tt A$--map 
$\widehat{\Gamma_{\tt q}^\prime}$. 

\begin{figure}[h!tbp]
\begin{center}
\includegraphics[width=0.6\textwidth]{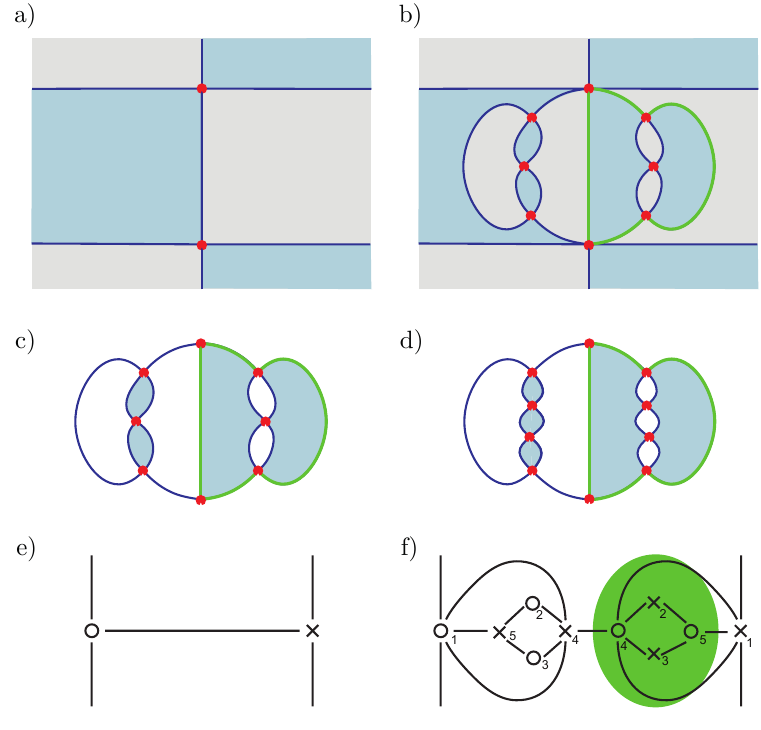}
\caption{
Modifying a $\tt t$--graph to make it non locally balanced.
(a) An edge $e$ 
of an arbitrary $\tt t$--graph $\Gamma$, 
that is globally and locally balanced.
(b) By replacing the edge $e$ with the graph in (c), we obtain a globally balanced
but not locally balanced $\tt t$--graph. 
Of course one could also use (d) to make another different $\tt t$--graph that is not locally balanced.
In all cases, the green path is the one that does not satisfy the local balance requirement.
By duality, in (e) is the pre--Speiser graph, corresponding to (a) 
and (f) is the modified pre--Speiser graph, corresponding to (b).
In (f) it is impossible to make the subgraph, enclosed by the green area, a regular subgraph.
}
\label{fig:mosaico-speiser-graph-No-localbalance}
\end{center}
\end{figure}
\end{example}

Recalling Definition \ref{def:de-teselacion} of tessellation,
the above example shows that:

\begin{corollary}
\label{cor:existence-of-tessellations-dont-represent-function}
There are finite and infinite 
tessellations $\mathscr{T}(\Gamma)$  
that do not represent any Speiser function.
\hfill\qed
\end{corollary}

\subsection{J.\,Tomasini's approach}
\label{sec:Tomasini}
In \cite{Tomasini}, an intermediate approach between considering tessellations or 
their duals the Speiser graphs, is taken.

Instead of considering the pullback of a Jordan path that traverses the labelled singular values,
J.\,Tomasini considers the pullback via the rational function 
$w(z)=R(z)$ of a ``spider'', $T_{\tt q}$,
consisting of a central black vertex (corresponding to a regular point), 
and simple edges to labelled red vertices (the singular values). 
From this \emph{increasing bipartite map} (a planar labelled bipartite graph) 
$w(z)^*T_{\tt q}$, he erases the labels and the 
valence 1 (red) vertices, together with their incident edges; 
obtaining a \emph{skeleton of $w(z)$}.
See Figure \ref{fig:arana-de-tomasini}, 
specifically the top row and first two left hand figures.

\begin{figure}[h!tbp]
\begin{center}
\includegraphics[width=0.95\textwidth]{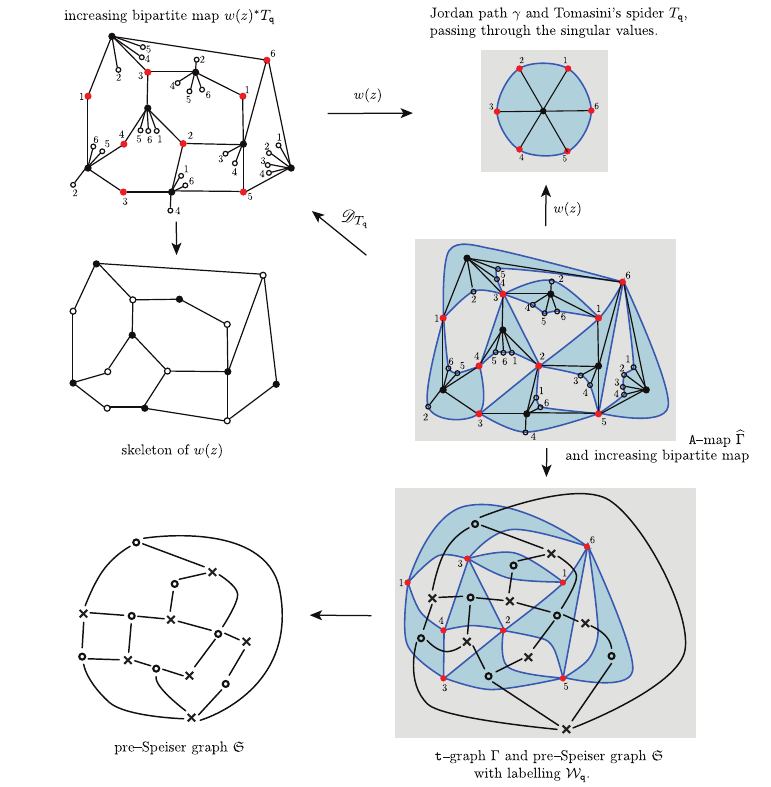}
\caption{
Starting on the top left is Tomasini's increasing bipartite map 
$w(z)^*T_{\tt q}$;
below it is the skeleton of $w(z)$.
On the top right are  
the Jordan path $\gamma$ and the spider $T_{\tt q}$, 
passing through the singular values.
Middle right is the $\tt A$--map $\widehat{\Gamma}_{\tt q}$
superimposed with the increasing bipartite map 
$w(z)^*T_{\tt q}$ to exemplify the deformation retract $\mathscr{D}_{T_{\tt q}}$.
Bottom row contains the pre--Speiser graph and the $\tt t$--graph $\Gamma$ for 
comparison with Tomasini's combinatorial objects.
}
\label{fig:arana-de-tomasini}
\end{center}
\end{figure}

Using the above he proves.

\begin{theorem*}[Tomasini's characterization of finite branched self covers of $\esf$ \cite{Tomasini}]
A finite planar bipartite  graph $G$ is realized as a skeleton of a 
branched self cover of $\esf$ if and only if
$G$ is 

\begin{enumerate}[label=\arabic*),leftmargin=*]	
\item 
Globally balanced;
there are the same number of black vertices as faces of $G$.

\item Locally balanced;
for any subgraph $H$ of $G$ with more than one black vertex,
the number of black vertices of $H$ is greater than or equal to the number of faces of $H$.
\end{enumerate}
\end{theorem*}

Roughly speaking, the relationship between the tessellation scheme of \S\ref{sec:tessellations}
and Tomasini's planar bipartite graphs 
is through a deformation retract $\mathscr{D}_{T_{\tt q}}$ of the Speiser $\tt q$--tessellation 
$\mathscr{T}(w(z)^*\gamma)$ 
(equivalently the $\tt A$--map $\widehat{\Gamma}_{\tt q} = w(z)^*\gamma$)
to the increasing bipartite map $w(z)^* T_{\tt q}$.
Thus,  $\mathscr{D}_{T_{\tt q}}$ retracts 
each blue tile/face 
to a regular point in a blue face 
(obtaining a black vertex for each blue face), 
fixing the vertices and coloring them red. 
The same deformation retract $\mathscr{D}_{T_{\tt q}}$ can be applied to the $\tt t$--graph
$\Gamma$ to obtain the skeleton $G$.
Of course the labelling of the increasing bipartite map $w(z)^*T_{\tt q}$ coincides with the
consistent $\tt q$--labelling $\mathcal{L}_{\mathcal{W}_{\tt q}}$ of the $\tt A$--map 
$\widehat{\Gamma}_{\tt q} = w(z)^*\gamma$ 
and its dual the Speiser graph $\mathfrak{S}_{\tt q}$ of index $\tt q$.
See Figure \ref{fig:arana-de-tomasini}.

\smallskip
Corollary \ref{cor:existence-of-tessellations-dont-represent-function} together with all of the above suggest the following.

\begin{conjecture}
\label{prop:shape-of-Speiser-function}
The three local balance conditions of 
Theorem \ref{prop:balance-conditions-Speiser-graphs}, section \S\ref{sec:Koch-Lei}, and
section \S\ref{sec:Tomasini}, are equivalent.
\end{conjecture}


\section{Geometrical decomposition of Speiser functions}
\label{sec:Soul-tower-decomposition}
Let $w(z)$ be a Speiser function 
with Riemann surface $\R_{w(z)}$. 
In this section we revisit Speiser Riemann surfaces, 
as in \S\ref{sec:Speiser-Riemann-surfaces}, in order
to obtain a unique geometrical decomposition of $\R_{w(z)}$ into 
\emph{maximal logarithmic towers} and a \emph{soul},
closely related to its Speiser graph $\mathfrak{S}_{\tt q}$.
The notation differs slightly from that of \S\ref{sec:Speiser-Riemann-surfaces},
where we were interested in obtaining a decomposition of $\R_{w(z)}$ in
maximal domains of single--valuedness of $w^{-1}(z)$. 

\noindent
As a first step, recall notation previous to Remark \ref{rem:val-asint-con-as};
$\delta = {\tt p}+{\tt r}$ indicates the number\footnote{
Both, ${\tt p}$ and ${\tt r}$, can be infinite or zero.
} 
of singularities of $w^{-1}(z)$,
\emph{i.e.}\ the number of branch points of $\R_{w(z)}$.

\noindent
If ${\tt p}=0$, then there are no logarithmic singularities; thus $w(z)$ only has $\tt r$
algebraic singularities. 

\noindent
Furthermore, since logarithmic singularities occur only on the ideal boundary of $\Omega_z$,
it follows that $\Omega_z = \CW_z$ does not support meromorphic functions $w(z)$ with logarithmic singularities.
Thus the case ${\tt p}\neq 0$, does not occur for $\Omega_z = \CW_z$.

\subsection{The pieces: flat $\tt p$--gons, maximal logarithmic towers, the soul}
The following definitions, describing 
flat $\tt p$--gon, logarithmic tower, soul, 
rational block, exponential and $h$--tangent blocks,
are the elementary pieces of our decomposition 
of the Riemann surface $\R_{w(z)}$.

\begin{definition}
\label{def:poligono}
Let ${\tt p}\geq 2$, a \emph{flat $\tt p$--gon  
$\big(\overline{\mathcal{P}}, \, \mathpzc{w} (\zeta) \big)$} 
is a pair consisting of a Riemann surface with boundary $\overline{\mathcal{P}}$ 
furnished with a meromorphic function 
$\mathpzc{w} (\zeta): \mathcal{P} \longrightarrow \CW_w$,
with 
$0 \leq \rho \leq \infty$ critical points,
satisfying the following.

\begin{enumerate}[label=\roman*),leftmargin=*] 
\item
The interior $\mathcal{P}$ of $\overline{\mathcal{P}}$ 
is an open Jordan domain, with oriented
boundary $\partial \overline{\mathcal{P}}$ homeomorphic to $\mathbb{S}^1$.

\item 
$\mathcal{P}$ is on the left side of the boundary.

\item
The boundary $\partial \overline{\mathcal{P}}$ has 

\centerline{
vertices $\{  \zeta _\beta 
\ \vert \ \beta \in 1, \ldots , {\tt p} \}$
and sides
$\{ S_{\beta} = 
\overline{\zeta_\beta \zeta_{\beta+1}}  
\ \vert \ \beta \in 1, \ldots , {\tt p} \}$, }

\noindent
cyclically enumerated.

\item 
The directional derivative of $ \mathpzc{w}(\zeta)$
in the interior of the sides $\{ S_{\beta} \}$
is non zero.

\item 
The image $ \mathpzc{w}(S_{\beta})$
is a geodesic segment in $(\CW_w, \del{}{w})$
with extreme points 
$$ 
\mathpzc{w}(\zeta_\beta) = {\tt a}_\beta, \
\mathpzc{w}(\zeta_{\beta+1}) = {\tt a}_{\beta+1}, 
\ \ \ 
{\tt a}_{\beta}\neq {\tt a}_{\beta+1}.
$$
\end{enumerate}
\end{definition}

Consider the sides
$\{ S_\beta \}_{\beta =1}^{\tt p}$ 
of a flat ${\tt p}$--gon 
$\overline{\mathcal{P}}$, 
with extreme points $\{ \zeta_\beta, \, \zeta_{\beta +1} \} \subset 
S_\beta$. 
By identifying $\sim$
all the vertices  $\{ \zeta_\beta\} $
to one point, say $\infty_\sim$, we obtain
a Riemann surface $\overline{\mathcal{P}}/\hspace{-.1cm}\sim$, which is 
homeomorphic to a sphere $\esf$ with $\tt p$ open discs
$U_\beta$ removed; 
the closure of the disks
share only one common point. 
This common point is also denoted as $\infty_\sim$ 
in $\overline{\mathcal{P}}/\hspace{-.1cm}\sim$.

Note that a flat $\rho$--gon contains $\rho$ critical points; 
the case when $\rho < \infty$ will be relevant. 
To see this, consider the following equivalence relation.

\begin{definition}
\label{def:holo-right-left-equiv}
Two meromorphic functions  
\emph{$w_\ell(z): V_\ell \longrightarrow\CW_w$, 
$\ell=1,2$,
are right--left equivalent} 
when there exist biholomorphisms 
$\phi:V_1 \longrightarrow V_2$, and
$\varphi: w_2(V_2) \longrightarrow w_1 (V_1)$ such that
\begin{equation}
\label{eq:right-equiv}
 w_1 = \varphi \circ w_2 \circ \phi .
\end{equation}
\end{definition}

\begin{lemma}
\label{lem:polygon-rational-block}
Let $\big(\overline{\mathcal{P}}, \, \mathpzc{w} (\zeta) \big)$ 
be a flat $\tt p$--gon
with $\rho<\infty$ critical points. 
Then $\mathpzc{w}(\zeta):
\mathcal{P}\longrightarrow\CW_w$ is 
right--left equivalent to
a rational function 

\centerline{$
R(z): 
\mathscr{P} \subset \CW_z \longrightarrow\CW_w$,}

\noindent 
where $\mathscr{P}$ is 
an appropriate Jordan domain.
\end{lemma}

\begin{proof}
Recall that the sides $\overline{\zeta_\beta \zeta_{\beta+1}}$ of $\mathcal{P}$ 
are straight line segments.

There exists an embedding $\varphi(z):\mathcal{P}\hookrightarrow \CW_z \times \CW_w$
such that $\mathpzc{w}(z) = \pi_2\circ\varphi(z)$, according to
Diagram \eqref{diagramaRX}.
Moreover, the critical points of $\mathpzc{w}(z)$ correspond to finitely ramified branch points
of the Riemann surface with boundary $\varphi(\mathcal{P})$.
Because of (v) of Definition \ref{def:poligono}, it is clear that $\varphi( \overline{\zeta_\beta \zeta_{\beta+1}} ) = \overline{{\tt a}_\beta {\tt a}_{\beta+1}}$.
Thus the Riemann surface with boundary $\varphi(\mathcal{P})$ can be extended, 
in $\CW_z\times\CW_w$, to a Riemann surface $\widehat{\mathcal{P}}$ 
without boundary such that $\pi_1(\widehat{\mathcal{P}})= \CW_z$
which has at most $\rho$
finitely ramified branch points
and no infinitely ramified branch points.
Since $\rho<\infty$, 
this proves the existence of a rational function $R(z)$ on $\CW_z$ such that 
on $\mathscr{P}= \pi_1(\mathcal{P})$,
and $R(z)$ is right--left equivalent to $\mathpzc{w}(z)$ 
on $\mathcal{P}$.
\end{proof}

\begin{definition}
\label{def:rational-block}
A function
$\mathpzc{w}(\zeta):\mathcal{P}\longrightarrow\CW_w$ 
arising from a flat $\tt p$--gon, with $\rho<\infty$ critical points,
is a \emph{rational block $R(z): \mathscr{P}\subset\CW_z \longrightarrow\CW_w$}.
\end{definition}

\begin{lemma}[Surgery of a flat $\tt p$--gon to 
pairs of logarithmic singularities]
\label{lem:extention-to-petal}
Let $\big(\overline{\mathcal{P}}, \, \mathpzc{w} (\zeta) \big)$ 
be a flat $\tt p$--gon, and
let $S_\beta$ be a side on the boundary of $\overline{\mathcal{P}}$,
where its extreme points 
has values 
\begin{enumerate}[label=\roman*),leftmargin=*] 

\item
$\mathpzc{w}( \zeta_\beta )= {\tt a}_\beta \in \CC_w, \, 
\mathpzc{w}(\zeta_{\beta +1})= {\tt a}_{\beta +1}= \infty \in \CW_w$, or

\item
$\mathpzc{w}( \zeta_\beta)={\tt a}_\beta, \, 
\mathpzc{w}(\zeta_{\beta +1} )={\tt a}_{\beta +1} \in \CC_w$.
\end{enumerate}

\noindent 
Then, there exist an extension of
$\mathpzc{w}(\zeta)$, in $\mathcal{P}$, to the interior of 
the corresponding open disk $U_\beta$, 
the $\beta$--th component of the boundary
in $\mathcal{P}/\hspace{-.1cm}\sim$,
such that $\mathpzc{w}(\zeta) \vert_{U_\beta}$ is 
right--left equivalent to 

\begin{enumerate}[label=\roman*),leftmargin=*] 
\item
$\exp(\zeta)$ or

\item
$\tanh(\zeta)$, 

\end{enumerate}
\noindent   
respectively.
\end{lemma}

Figure \ref{fig:hexagono3} illustrates the Lemma.

\begin{figure}[h!tbp]
\begin{center} 
\includegraphics[width=0.8\textwidth]{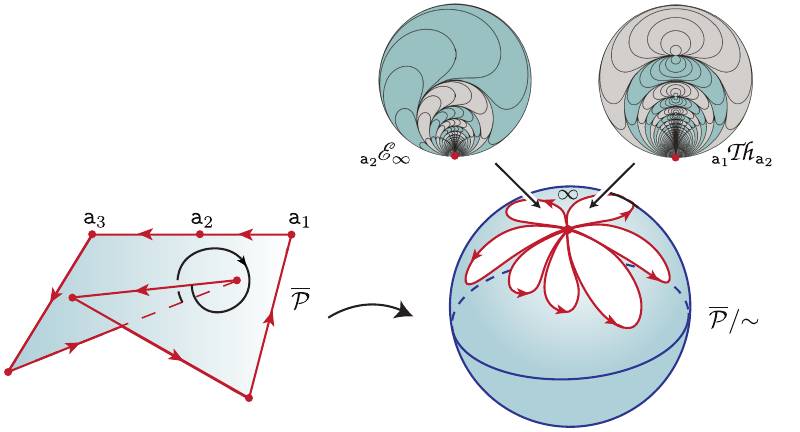}
\caption{
A flat ${\tt p}$--gon $\big(\overline{\mathcal{P}}, \, \mathpzc{w} (\zeta) \big)$,
whose vertices have been identified to one point $\infty_\sim$, 
obtaining a Riemann surface 
$\overline{\mathcal{P}}/\hspace{-.1cm}\sim$ homeomorphic to a sphere
$\esf$ with $\tt p$ open disks removed. 
As second step, we glue an
$h$--tangent block $\htan{{\tt a}_1}{{\tt a}_2}$
or
an exponential block $\ent{{\tt a}_2}{\infty}$
on each boundary component.
We sketch a 7--gon with a vertex 
${\tt a}_5$ of cone angle $> 2\pi$.
Note that ${\tt a}_3=\infty\in\CW_w$.
}
\label{fig:hexagono3}
\end{center}
\end{figure}

\begin{proof}
By (v) in Definition \ref{def:poligono}, we have that 
${\tt a}_\beta \neq  {\tt a}_{\beta+1}$. 

\noindent 
When ${\tt a}_\beta, \, {\tt a}_{\beta+1} \in \CC_w$, 
then there is an affine biholomorphism $\varphi_\beta$
that takes the oriented segment $[-1,1]$ to 
the oriented side 
$S_\beta = \overline{\zeta_\beta \zeta_{\beta+1}}$, 
with
$\varphi_\beta^{-1}({\tt a}_\beta) = -1$ and  
$\varphi_\beta^{-1}({\tt a}_{\beta+1}) = 1$.
The required analytic extension of $\mathpzc{w}(\zeta)$ is
\begin{equation*}
\label{eq:extension-a-un-hueco-de-disco}
\begin{array}{rcl}
\mathpzc{w}(\zeta): 
\big( (\overline{\mathcal{P}} /\hspace{-.1cm}\sim ) \cup 
\overline{U}_\beta 
\big)
\backslash \{ \infty_\sim \} 
\subset \esf 
& \longrightarrow   &  \CW_w
\\
\zeta & \longmapsto &
\hspace{-5pt}
\left\{
\hspace{-5pt}
\begin{array}{ll}
\mathpzc{w}(\zeta) &  \zeta \in \overline{\mathcal{P}} \backslash \{ \infty_\sim \}
\\
& \vspace{-.3cm}
\\
(\varphi_\beta \circ \tanh)(\zeta)
& 
\zeta \in 
\overline{U}_\beta 
\backslash \{ \infty_\sim \}.
\end{array}
\right. 
\end{array}
\end{equation*}

\noindent 
The argument is as follows. 
Since $\mathcal{P}$ is a Riemann surface, let $J$
denote its complex structure. 
By using  the affine map $\varphi_\beta$, 
it is enough to perform the extension of 
$J$ to a complex structure $J_\beta$ on 
$ \overline{U}_\beta 
\backslash \{ \infty_\sim \}$. 
We can recognize that $( U_\beta, J_\beta)$ is 
biholomorphic to the lower half plane 
$\HH_- = \{ \Im{\zeta} < 0\}$. 
Hence the function $\tanh(\zeta)$ makes sense.

When $\mathpzc{w}({\tt a}_\beta) \in \CC_w, \, \mathpzc{w}({\tt a}_{\beta+1})=\infty \in \CW_w$, 
then there is a biholomorphism $\varphi_\beta$
taking $[0,\, +\infty ]$ to 
the oriented side 
$S_\beta = \overline{{\tt a}_\beta {\tt a}_{\beta+1}}$.
By an analogous argument, 
the required extension of 
$\mathpzc{w}(\zeta)$ 
is $\big(\varphi_\beta \circ \exp\big)(\zeta)$. 
\end{proof}

In view of Lemma \ref{lem:extention-to-petal}, the following is natural.
\begin{definition}[Exponential and $h$--tangent blocks]
\label{def:piezas-elementales}
Consider $\dominio{U}{{\tt a}_\alpha}{{\tt a}_\beta} \subset \CW_z$
a Jordan domain, 
and ${\tt a}_\alpha \neq  {\tt a}_\beta \in \CW_w$.
Let 

\centerline{
$w(z): \dominio{U}{{\tt a}_\alpha}{{\tt a}_\beta} \longrightarrow \CW_w, $
}

\noindent  
be a meromorphic function
with an essential singularity at 
$\mathfrak{e}=\infty\in \partial ( \dominio{U}{{\tt a}_\alpha}{{\tt a}_\beta} )$ 
and exactly two distinct 
asymptotic values ${\tt a}_\alpha, {\tt a}_\beta \in\CW_w$.

\noindent
1.\ 
An \emph{exponential block $\ent{\infty}{{\tt a}_\beta}$ } 
is a function  
$w(z)$ on $\dominio{U}{\infty }{{\tt a}_\beta}$
right--left equivalent to 
$\e^z: \overline{\HH} \longrightarrow \CW$, 
\emph{i.e.}\ there are
biholomorphisms, as in Equation \eqref{eq:right-equiv}, 

\begin{enumerate}[label=\roman*),leftmargin=*]
\item 
$\phi(z):  \dominio{\overline{U}}{\infty }{{\tt a}_\beta}   \longrightarrow \overline{ \HH }$
taking $\mathfrak{e}$ to $\infty$, and

\item 
$\varphi(w):\CW_{w} \longrightarrow \CW_w$ taking 
$0$ to ${\tt a}_\beta$ and 
$\infty$ to ${\tt a}_\alpha $.
\end{enumerate}

\noindent 
The \emph{exponential block $\ent{{\tt a}_\alpha}{\infty}$ } 
is defined in an analogous way.

\noindent
2.\ An \emph{$h$--tangent block 
$\htan{{\tt a}_\alpha }{{\tt a}_\beta }$} is 
a function $w(z)$ on $\dominio{U}{{\tt a}_\alpha }{{\tt a}_\beta }$
right--left equivalent to 
$\tanh (z) : \overline{\HH} \longrightarrow \CW$, 
\emph{i.e.}\ there are
biholomorphisms, as in Equation \eqref{eq:right-equiv},

\begin{enumerate}[label=\roman*),leftmargin=*]
\item 
$\phi(z):  \dominio{\overline{U}}{{\tt a}_\alpha }{{\tt a}_\beta }   \longrightarrow \overline{ \HH }$
taking $\mathfrak{e}$ to $\infty$, and

\item 
$\varphi(w):\CW_{w} \longrightarrow \CW_w$ taking 
$-1$ to ${\tt a}_\alpha $ and 
$1$ to ${\tt a}_\beta  $.
\end{enumerate}

\end{definition}

The elementary blocks 
$\ent{{\tt a}_\alpha }{{\tt a}_\beta }$ and  
$\htan{{\tt a}_\alpha }{{\tt a}_\beta }$
can be easily understood with 
the following commutative diagram

\begin{center}
\begin{picture}(180,55)(-30, 0)

\put(-181,20){\vbox{\begin{equation}\label{dia:piezas-elementales}\end{equation}}}

\put(-15,42){$\dominio{\overline{U}}{{\tt a}_\alpha }{{\tt a}_\beta }\subset\CW_z$}

\put(120,42){\ $\CW_w$}

\put(38,43){\vector(1,0){77}}
\put(60,50){$w(z)$}

\put(10,37){\vector(0,-1){23}}
\put(0,25){$\phi$}

\put(128,14){\vector(0,1){23}}
\put(130,25){$\varphi$}

\put(6,0){$\overline{\HH}$}

\put(120,0){\ $\CW_w$\,.}

\put(30,3){\vector(1,0){85}}
\put(40, 10){$\e^\zeta$ \ or \ $\tanh(\zeta)$}

\end{picture}
\end{center}

\begin{remark}[Using vector fields to distinguishing between topologically equivalent functions]
\label{rem:vec-fields-distinguish-elementary-blocks}
Note that the elementary blocks 
$\ent{{\tt a}_\alpha }{{\tt a}_\beta }$ and  
$\htan{{\tt a}_\alpha }{{\tt a}_\beta }$
are all
right--left equivalent functions;
in particular topologically indistinguishable
(their underlying tessellations, and/or pre Speiser graphs, 
are equivalent under homeomorphisms). 
However, as meromorphic functions, 
they are very different: the exponential block 
$\ent{{\tt a}_\alpha }{{\tt a}_\beta }$ has 
one finite and one infinite asymptotic values defining it, 
but the $h$--tangent block  $\htan{{\tt a}_\alpha }{{\tt a}_\beta }$ 
has two finite asymptotic values defining 
it, and is strictly meromorphic in the interior of its domain.
The use of the associated canonical 
vector fields\footnote{We will drop the subindices 
when those are not essential, 
thus $X_{\ent{}{}}(z)$, $X_{\htan{}{}}(z)$.
}

\centerline{
$X_{\ent{{\tt a}_\alpha }{{\tt a}_\beta }} (z) \doteq \frac{1}{ \ent{{\tt a}_\alpha }{{\tt a}_\beta } ^{\prime} (z) }
\del{}{z}$ 
\quad and \quad
$X_{\htan{{\tt a}_\alpha}{{\tt a}_\beta }  } (z) \doteq \frac{1}{ \htan{{\tt a}_\alpha }{{\tt a}_\beta } ^{\prime} (z) }
\del{}{z}$ }

\noindent
see \cite{AlvarezMucino3}\,prop.\,2.5,
allows us to very easily distinguish between them
by considering their phase portraits. 
See Figure \ref{fig:piezas-elementales}, and
Example \ref{example:N-exp-tanh}.c for further details.
\end{remark}

The next definition is illustrated in
Figure \ref{fig:tower-construction}.

\begin{definition}
\label{def:logarithmic-tower-for-Riemann-surface}
Let $\R_{w(z)}$ be the Speiser Riemann surface associated to a Speiser function $w(z)$
provided with a cyclic order $\mathcal{W}_{\tt q}$ for its $\tt q$ singular values.
By Remark \ref{rem:gamma-poligonal-y-grafica},

\centerline{$\gamma= 
\overline{{\tt a}_{\alpha} {\tt a}_{\beta}} \cup
\overline{{\tt a}_{\beta} {\tt a}_{\alpha}}
\subset \CW_w
$ }

\noindent 
is a Jordan path,
 the union of two geodesic polygonals, 
as in Definition \ref{def:orden-ciclico-y-losetas}, 
that runs through $\mathcal{W}_{\tt q}$. 

\begin{enumerate}[label=\arabic*),leftmargin=*] 

\item The
\emph{closed hemispheres $\mathfrak{H}^\pm$} 
associated to $\gamma$ satisfy 

\centerline{$
\mathfrak{H}^+ \cap \mathfrak{H}^- = \gamma$, \ and\ \ 
$\CW_w = \mathfrak{H}^+ \cup \mathfrak{H}^-$,
}

\noindent
with $ \mathfrak{H}^+ $  on the left hand side of $\gamma$.
Let $\Upxi=\{
\overline{ {\tt w}_{{\tt j}( \msigma )} 
{\tt w}_{{\tt j}( \mrho )} } 
\ \vert \
{\tt j}( \msigma ) \neq {\tt j}( \mrho )
\} 
\subsetneqq \gamma$
be a non--empty collection of polygonal branch cuts.
A \emph{positive half--sheet $\mathfrak{L}^+_\Upxi$} is 

\centerline{$\mathfrak{H}^+ \backslash \Upxi$.}

\noindent
Analogously, a \emph{negative half--sheet $\mathfrak{L}^-_\Upxi$} is 
$\mathfrak{H}^- \backslash \Upxi$.

\item
Let $\overline{{\tt a}_{\alpha} {\tt a}_{\beta}}$
be a polygonal branch cut,
a \emph{logarithmic tower\footnote{
See also 
\cite{AlvarezMucino1}\,p.\,152 and
\cite{AlvarezMucino3}\,p.\,22, where the term `semi--infinite helicoid' is used. }
$\mathcal{T}( {\tt a}_{\alpha}, {\tt a}_{\beta} )$
of $\R_{w(z)}$ }
is a Riemann surface 
\begin{enumerate}[label=\roman*),leftmargin=*] 

\item
associated to an exponential or an $h$--tangent block on 
$\dominio{U}{{\tt a}_\alpha}{{\tt a}_\beta} $ of $w(z)$,

\item 
whose boundary is an element of 
$\pi_2^{-1} (\overline{{\tt a}_{\alpha} {\tt a}_{\beta}}) $.
\end{enumerate}

\item
A $\mathcal{T} ( {\tt a}_{\alpha}, {\tt a}_{\beta} )$ is a 
\emph{maximal logarithmic tower in $\R_{w(z)}$} if 
given any logarithmic tower 
$\widehat{\mathcal{T}} ( {\tt a}_{\alpha}, {\tt a}_{\beta} )$ such that 
$\mathcal{T} ({\tt a}_{\alpha} , {\tt a}_{\beta} )
\subset 
\widehat{\mathcal{T}} ( {\tt a}_{\alpha} ,{\tt a}_{\beta} )$,
then 
$\mathcal{T} ( {\tt a}_{\alpha},{\tt a}_{\beta} ) 
=
\widehat{\mathcal{T}} ({\tt a}_{\alpha} ,{\tt a}_{\beta} )$.
\end{enumerate}
\end{definition}

\begin{figure}[htbp!]
\begin{center}
\includegraphics[width=0.83\textwidth]{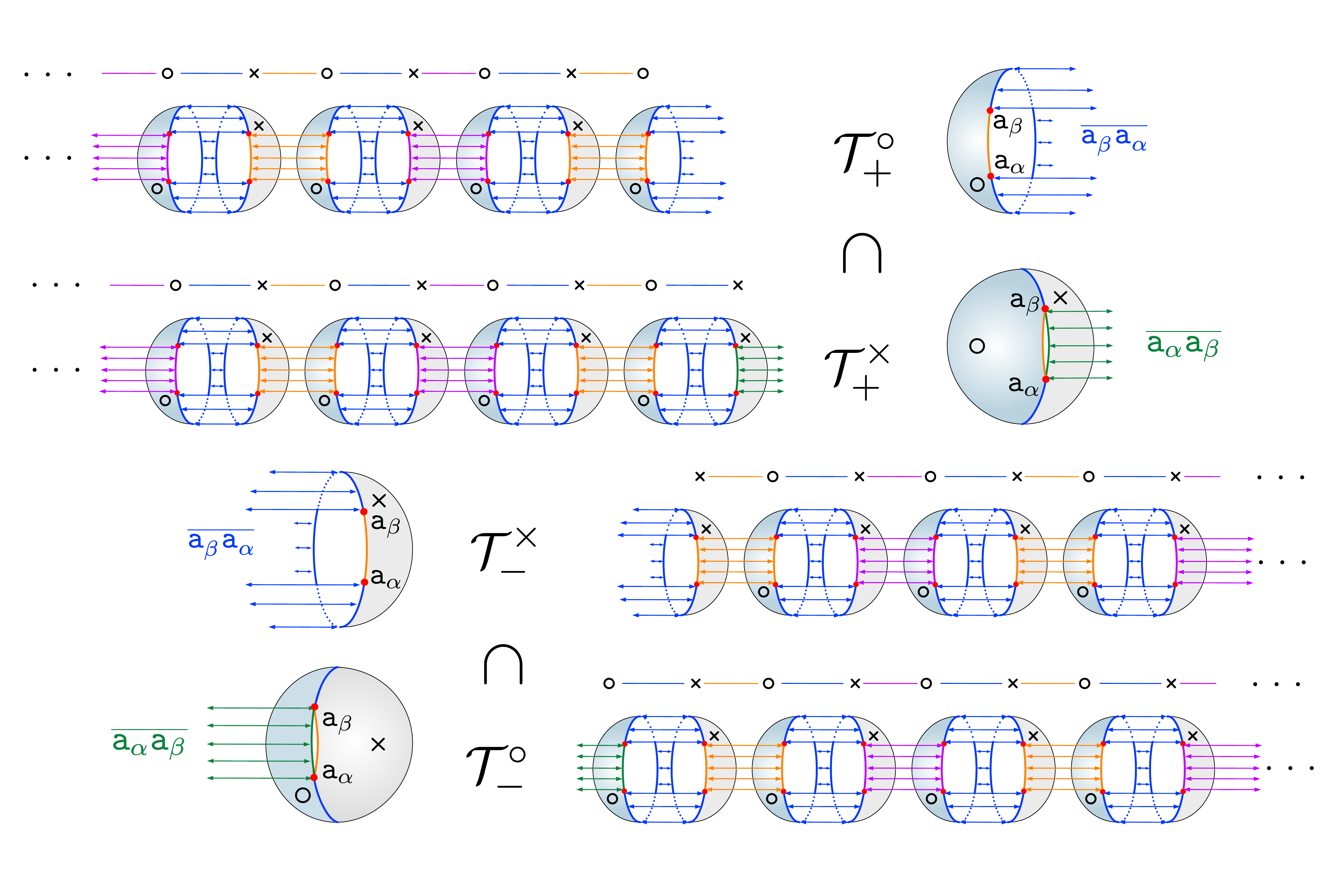}
\caption{
Accurate 
construction of 
logarithmic towers by surgery,
four abstract and qualitatively distinct
cases appear.
Consider $\gamma = 
\overline{{\tt a}_\alpha {\tt a}_\beta} \cup
\overline{{\tt a}_\beta {\tt a}_\alpha}$, 
the blue (resp. gray) hemisphere
is on the left (resp. right)
side of $\overline{{\tt a}_\alpha {\tt a}_\beta}$.
The gluing  is denoted by the same horizontal colored lines.
The towers $\mathcal{T}^\circ_-$ and $\mathcal{T}^\times_+$ have boundary 
$\overline{{\tt a}_{\alpha} {\tt a}_{\beta}}$ (in green), 
while the towers $\mathcal{T}^\times_-$ and $\mathcal{T}^\circ_+$ have boundary 
$\overline{{\tt a}_{\beta} {\tt a}_{\alpha}} \doteq \gamma \backslash\overline{{\tt a}_{\alpha} {\tt a}_{\beta}}$ (in blue).
}
\label{fig:tower-construction}
\end{center}
\end{figure}

\begin{remark}[Construction and notation for logarithmic towers]
1.\ 
A logarithmic tower $\mathcal{T}({\tt a}_{\alpha}, {\tt a}_{\beta})$ is determined by
a pair $({\tt a}_{\alpha}, {\tt a}_{\beta})$ 
of asymptotic values of $w(z)$,
however, not all pairs of asymptotic values of $w(z)$ determine logarithmic towers.
However,
four abstract and qualitatively distinct cases appear:
\begin{equation}
\label{fig:4-torres}
\begin{array}{l}
\mathcal{T}^{\circ}_{+} ( {\tt a}_{\alpha}, {\tt a}_{\beta} )=
\left[ 
(\mathfrak{H}^+ \backslash  \overline{{\tt a}_{\alpha} {\tt a}_{\beta}} )
\ \cup \left(
\bigcup\limits_{\vartheta=2}^{\infty} \big( 
\mathfrak{H}^- \backslash  \overline{{\tt a}_{\alpha} {\tt a}_{\beta}} 
\cup
\mathfrak{H}^+ \backslash  \overline{{\tt a}_{\alpha} {\tt a}_{\beta}} 
\big)_{\vartheta} 
\right)
\right] / \sim ,
\\[10pt]
\mathcal{T}^\times_{+} ( {\tt a}_{\alpha}, {\tt a}_{\beta} )= 
\left[
\bigcup\limits_{\vartheta=1}^{\infty} \big( 
\mathfrak{H}^- \backslash  \overline{{\tt a}_{\alpha} {\tt a}_{\beta}} 
\cup
\mathfrak{H}^+ \backslash  \overline{{\tt a}_{\alpha} {\tt a}_{\beta}} 
\big)_{\vartheta} 
\right] / \sim ,
\\[10pt]
\mathcal{T}^\times_{-} ( {\tt a}_{\alpha}, {\tt a}_{\beta} )= 
\left[ 
(\mathfrak{H}^- \backslash  \overline{{\tt a}_{\alpha} {\tt a}_{\beta}} )
\ \cup \left(
\bigcup\limits_{\vartheta=2}^{\infty} \big( 
\mathfrak{H}^+ \backslash  \overline{{\tt a}_{\alpha} {\tt a}_{\beta}} 
\cup
\mathfrak{H}^- \backslash  \overline{{\tt a}_{\alpha} {\tt a}_{\beta}} 
\big)_{\vartheta} 
\right)
\right] / \sim ,
\\[10pt]
\mathcal{T}^\circ_{-} ( {\tt a}_{\alpha}, {\tt a}_{\beta} )= 
\left[
\bigcup\limits_{\vartheta=1}^{\infty} \big( 
\mathfrak{H}^+ \backslash  \overline{{\tt a}_{\alpha} {\tt a}_{\beta}} 
\cup
\mathfrak{H}^- \backslash  \overline{{\tt a}_{\alpha} {\tt a}_{\beta}} 
\big)_{\vartheta} 
\right]/ \sim .
\end{array}
\end{equation}

\noindent 
The super index $\circ$ or $\times$ indicates that the
first hemisphere of the tower is  blue $\mathfrak{H}^+$ or gray 
$\mathfrak{H}^-$.
As for notation, when the context is clear, we shall sometimes drop the pair of asymptotic values and 
simply write 
$\mathcal{T}^{\circ}_{+}$ instead of 
$\mathcal{T}^{\circ}_{+} ( {\tt a}_{\alpha}, {\tt a}_{\beta} )$, the other cases are analogous.

\noindent
2.\
Recalling Definitions \ref{def:geodesic-segments} and \ref{def:sheet-branchcut},
together with Proposition \ref{prop:Rw-is-a-union-of-sheets}, 
some useful features of 
the construction of
logarithmic towers, by gluing, are the following.
The use of Figure \ref{fig:tower-construction} is recommended.

\begin{enumerate}[label=\roman*),leftmargin=*]

\item 
The isometric glueing\footnote{
Recall Corollary \ref{cor:pegado-isometrico}, 
this technique also appeared in
\cite{AlvarezMucino3} pp.~60 and 61.
} 
for 

\centerline{
$\mathcal{T}^{\circ}_{-} 
=\left[
\bigcup\limits_{\vartheta=1}^{\infty} \big( 
\mathfrak{H}^+ \backslash  \overline{{\tt a}_{\alpha} {\tt a}_{\beta}} 
\cup
\mathfrak{H}^- \backslash  \overline{{\tt a}_{\alpha} {\tt a}_{\beta}} 
\big)_{\vartheta} 
\right]/ \sim $}

\noindent
is as follows.

\begin{enumerate}[label=\alph*),leftmargin=*]
\item
First consider $( \mathfrak{H}^+ \backslash  \overline{{\tt a}_{\alpha} {\tt a}_{\beta}} 
\cup
\mathfrak{H}^- \backslash  \overline{{\tt a}_{\alpha} {\tt a}_{\beta}} )_{\vartheta}$
for fixed $\vartheta\geq 1$,
and glue the half sheets
$\mathfrak{H}^+ \backslash  \overline{{\tt a}_{\alpha} {\tt a}_{\beta}} $
to $\mathfrak{H}^- \backslash  \overline{{\tt a}_{\alpha} {\tt a}_{\beta}} $
along their common boundary 
$\overline{{\tt a}_{\beta} {\tt a}_{\alpha}}
\doteq
\gamma\backslash \overline{{\tt a}_{\alpha} {\tt a}_{\beta}}$.

\item
Clearly, after gluing as in (a), $( \mathfrak{H}^+ \backslash  \overline{{\tt a}_{\alpha} {\tt a}_{\beta}} 
\cup
\mathfrak{H}^- \backslash  \overline{{\tt a}_{\alpha} {\tt a}_{\beta}} )_{\vartheta} / \sim$ 
is a sheet\footnote{
Note that Definition \ref{def:sheet-branchcut} 
of sheet with branch cuts satisfies  
$\mathfrak{L}_\Upxi = 
\CW \backslash \Upxi =
(\mathfrak{H}^+ \cup \mathfrak{H}^-)\backslash \Upxi
=\mathfrak{L}^+_\Upxi  \cup \mathfrak{L}^-_\Upxi $.}
$\big( \CW_w \backslash  \overline{{\tt a}_{\alpha} {\tt a}_{\beta}} \big)_{\vartheta}$,
with boundary consisting of two copies
of the polygonal geodesic segment
$\overline{{\tt a}_{\alpha} {\tt a}_{\beta}}$, namely
$\overline{{\tt a}_{\alpha} {\tt a}_{\beta}}_+$ and $\overline{{\tt a}_{\alpha} {\tt a}_{\beta}}_-$.

\item
Next, 
for each $j\geq 1$,
glue one copy, say $\overline{{\tt a}_{\alpha} {\tt a}_{\beta}}_+$ from the $j$--th sheet
$\big( \CW_w \backslash  \overline{{\tt a}_{\alpha} 
{\tt a}_{\beta}
} \big)_{\vartheta=j}$, 
to a copy
$\overline{{\tt a}_{\alpha} {\tt a}_{\beta}}_-$ from the $(j+1)$--th sheet
$\big( \CW_w \backslash  \overline{{\tt a}_{\alpha} 
{\tt a}_{\beta}
} \big)_{\vartheta=j+1}$;
this leaves the geodesic segment 
$\overline{{\tt a}_{\alpha} {\tt a}_{\beta}}_-$ as the boundary of the sheet 
$\big( \CW_w \backslash  \overline{{\tt a}_{\alpha} 
{\tt a}_{\beta}
} \big)_{\vartheta=1}$; no other boundaries are left.

\item
Thus, the boundary of $\mathcal{T}^{\circ}_{-}$
is the open polygonal 
$\overline{{\tt a}_{\alpha} {\tt a}_{\beta}}_-$ 
coming from the 1--st sheet
$\big( \CW_w \backslash  \overline{{\tt a}_{\alpha} 
{\tt a}_{\beta}
} \big)_{\vartheta=1}$.
\end{enumerate}

\item
The construction of the logarithmic tower 

\centerline{
$\mathcal{T}^{\times}_{-} =
\left[ 
(\mathfrak{H}^- \backslash  \overline{{\tt a}_{\alpha} {\tt a}_{\beta}} )
\ \cup \left(
\bigcup\limits_{\vartheta=2}^{\infty} \big( 
\mathfrak{H}^+ \backslash  \overline{{\tt a}_{\alpha} {\tt a}_{\beta}} 
\cup
\mathfrak{H}^- \backslash  \overline{{\tt a}_{\alpha} {\tt a}_{\beta}} 
\big)_{\vartheta} 
\right)
\right] / \sim $}

\noindent
is similar with the following modifications:
Steps (a)--(d) are the same as above, with $\vartheta\geq 2$;
obtaining a logarithmic tower with boundary 
$\overline{{\tt a}_{\alpha} {\tt a}_{\beta}}_-$ 
coming from the 2--nd sheet
$\big( \CW_w \backslash  \overline{{\tt a}_{\alpha} 
{\tt a}_{\beta}
} \big)_{\vartheta=2}$.
\begin{enumerate}[label=\alph*),leftmargin=*]
\item[d)]
Finally, glue a half sheet 
$(\mathfrak{H}^- \backslash  \overline{{\tt a}_{\alpha} {\tt a}_{\beta}} )$ along the boundary
$\overline{{\tt a}_{\alpha} {\tt a}_{\beta}}_+$
leaving the boundary $\overline{{\tt a}_{\beta} {\tt a}_{\alpha}}$ as the only boundary of
$\mathcal{T}^\times_-$.

\end{enumerate}

\item The construction/glueing of the logarithmic towers $\mathcal{T}^\times_+$ and $\mathcal{T}^\circ_+$ is 
analogous to the above, recall Equation \eqref{fig:4-torres}.
\end{enumerate}

\noindent
3.\
It is clear from the above construction and from Figure \ref{fig:tower-construction}, that
given a pair $({\tt a}_{\beta}, {\tt a}_{\alpha})$ of asymptotic values:
\begin{enumerate}[label=\alph*),leftmargin=*]
\item
there are an infinite number of logarithmic towers 
in $\R_{w(z)}$ of each type,

\item
containment, and thus maximality as in Definition \ref{def:logarithmic-tower-for-Riemann-surface}.3, 
only makes sense when the logarithmic towers are considered 
in $\R_{w(z)}$,
see Examples \ref{example:N-exp-tanh}.b and \ref{example:tessellation-Airy}.b.
\end{enumerate}
\end{remark}

\begin{definition}
\label{def:soul-of-surface}
The \emph{soul $\mathfrak{N}_{w(z)}$,  
of a Riemann surface $\R_{w(z)}$},
is the subset obtained as the complement, in $\R_{w(z)}$, of the
maximal logarithmic towers of $\R_{w(z)}$.
\end{definition}

\begin{remark}
\label{nomenclatura-logarithmic-tower}
The above concepts appear as ``logarithmic end'' and ``nucleus'' 
in the classic literature (recall Remark \ref{rem:historical-origin-log-ends}),
both for the combinatorial objects as for 
their corresponding Riemann surfaces analogous,
surely because of the bijection between them 
(recall Definition \ref{def:Log-end-for-Speiser-graph}.1--2 and see
Lemma \ref{lem:correspondence-log-towers} below).
However, the term ``kernel'' is used by \cite{GoldbergOstrovskii} instead of
nucleus when speaking of the complement of the logarithmic towers on 
the Riemann surface.
We prefer to make the distinction clear and thus use

\noindent 
$\bigcdot$
``logarithmic end'' and ``nucleus'' when considering the combinatorial
objects (Speiser graphs), and 

\noindent 
$\bigcdot$
``logarithmic towers'' and ``soul''
when considering the analytic objects
(Riemann surfaces).
\end{remark}

\begin{example}[Two elementary $N$--functions]
\label{example:N-exp-tanh}
\hfill
\begin{enumerate}[label=\alph*),leftmargin=*]
\item
The functions 

\centerline{
$\e^z  \doteq  \ent{}{} (z)$ 
\ and \
$\tanh(z) \doteq \htan{}{}  (z)$,
}

\noindent
have asymptotic values 

\centerline{$\mathcal{AV}_{\ent{}{}} = \{ {\tt a}_1, {\tt a}_2 \} = \{0,\infty\}$ 
\ and \ 
$\mathcal{AV}_{\htan{}{} } = \{ {\tt a}_1, {\tt a}_2 \} = \{-1,1\}$}

\noindent 
respectively,
as their only singular values, compare also with Definition \ref{def:piezas-elementales}.
Further note that
$\tanh(z) = \frac{\e^{2z} - 1}{\e^{2z} + 1}$,
thus they are right--left $Aut(\CW)$--equivalent,
as in Definition \ref{def:holo-right-left-equiv},
so both have constant
Schwarzian derivative 

\centerline{
$Sw\{ \e^z, z \} = -\frac{1}{2}, 
\quad 
Sw\{ \tanh(z) , z \} = -2$.
}

\noindent 
Thus, they are $N$--functions.

\item
Their Riemann surfaces 
$\R_{w(z)}$, 
have two infinitely ramified branch points,
by Remark \ref{rem:ramindex-subindex}.2 they are denoted as,

\centerline{
$\circled{1}=(\infty_1, {\tt a}_1, \infty)$ 
\ and \ 
$\circled{2}=(\infty_2, {\tt a}_2, \infty)$.
}

\noindent  
For $\R_{\ent{}{}(z)}$ one of the branch points lies over $\infty\in\CW_w$, 
the other over the finite asymptotic value $0$.
For $\R_{\htan{}{} (z)}$ the branch points lie over the finite asymptotic values
$\{-1,1\}$.
The diagonals\footnote{
Recall Definition \ref{def:diagonal}.
In particular that if two diagonals have the same $\vartheta$, then their 
corresponding  branch points share the same sheet.
}
are $\{ \Delta_{\vartheta\, 0\, \infty} = [0,+\infty]_\vartheta \}_{\vartheta\in\ZZ}$ and 
$\{ \Delta_{\vartheta\, -1\, 1} = [-1,1]_\vartheta \}_{\vartheta\in\ZZ}$ 
for $\R_{\ent{}{} (z)}$ and $\R_{\htan{}{} (z)}$,
respectively.
\\
According to Proposition \ref{prop:Rw-is-a-union-of-sheets},
the decomposition in sheets (maximal domains of single--valuedness) is
$$
\R_{w(z)} = 
\bigg(
\bigcup_{\vartheta_1=-\infty}^{\infty} \big( \CW_w \backslash  
\overline{{\tt a}_{1} {\tt a}_{2} } 
\big)_{\vartheta_1}  
\bigg) / \sim ,
$$

\noindent
where we can immediately recognize $\R_{w(z)}$ as the union of two logarithmic towers.
In fact,
considering the cyclic order $\mathcal{W}_2=[{\tt a}_1,{\tt a}_2]$,
there are basically two choices for the decomposition:

\centerline{
$\R_{w(z)} = \mathcal{T}^\times_+ ({\tt a}_{1}, {\tt a}_{2}) 
\cup \mathcal{T}^\circ_- ({\tt a}_{1}, {\tt a}_{2})$
}

\noindent
or

\centerline{
$\R_{w(z)} = \mathcal{T}^\circ_+ ({\tt a}_{1}, {\tt a}_{2}) 
\cup \mathcal{T}^\times_- ({\tt a}_{1}, {\tt a}_{2})$.
}

\noindent 
Note that, both
$\R_{\ent{}{} (z)}$, $\R_{\htan{}{} (z)}$,
are exceptional
in the sense that there are no maximal logarithmic towers, and hence
the soul is empty.

\item For the tessellation, with the cyclic order $\mathcal{W}_2$, 
it follows that $\gamma=\RR\cup\{\infty\}$ and
the Speiser $2$--tessellation

\centerline{
$\big(
(\CC_z \cup \{ \infty_{{\tt a}_1}, \, 
\infty_{{\tt a}_2} \} ) \backslash w(z)^*\gamma, 
w(z)^*\mathcal{L}_\gamma \big)$
}

\noindent 
is shown in Figure \ref{fig:Tessellation-Exp}.a--b. 
Note that the difference between the case for 
$\e^z$ and $\tanh(z)$
is the actual choice of the asymptotic values.
The two vertices of infinite valence of 
the graph $w(z)^*\gamma$
are the points $\infty_{{\tt a}_1}$, $\infty_{{\tt a}_2}$ in the non Hausdorff compactification 

\centerline{$\CC_z \cup \{ 
\infty_{{\tt a}_1}, \, 
\infty_{{\tt a}_2}
\}$}

\noindent 
determined by the two asymptotic values 
$\{ {\tt a}_1, \, {\tt a}_2 \}$.

\item
Their analytic Speiser graph of index ${\tt q}=2$ is
drawn in Figure \ref{fig:Tessellation-Exp}.c.
Note that the tessellations and Speiser graphs for the two functions 
$\e^z$ and $\tanh(z)$
only differ in the choice of 
singular values, \emph{i.e.}\ the cyclically ordered asymptotic values $\mathcal{W}_2$.

\begin{figure}[htbp!]
\begin{center}
{a)}
\includegraphics[width=0.3\textwidth]{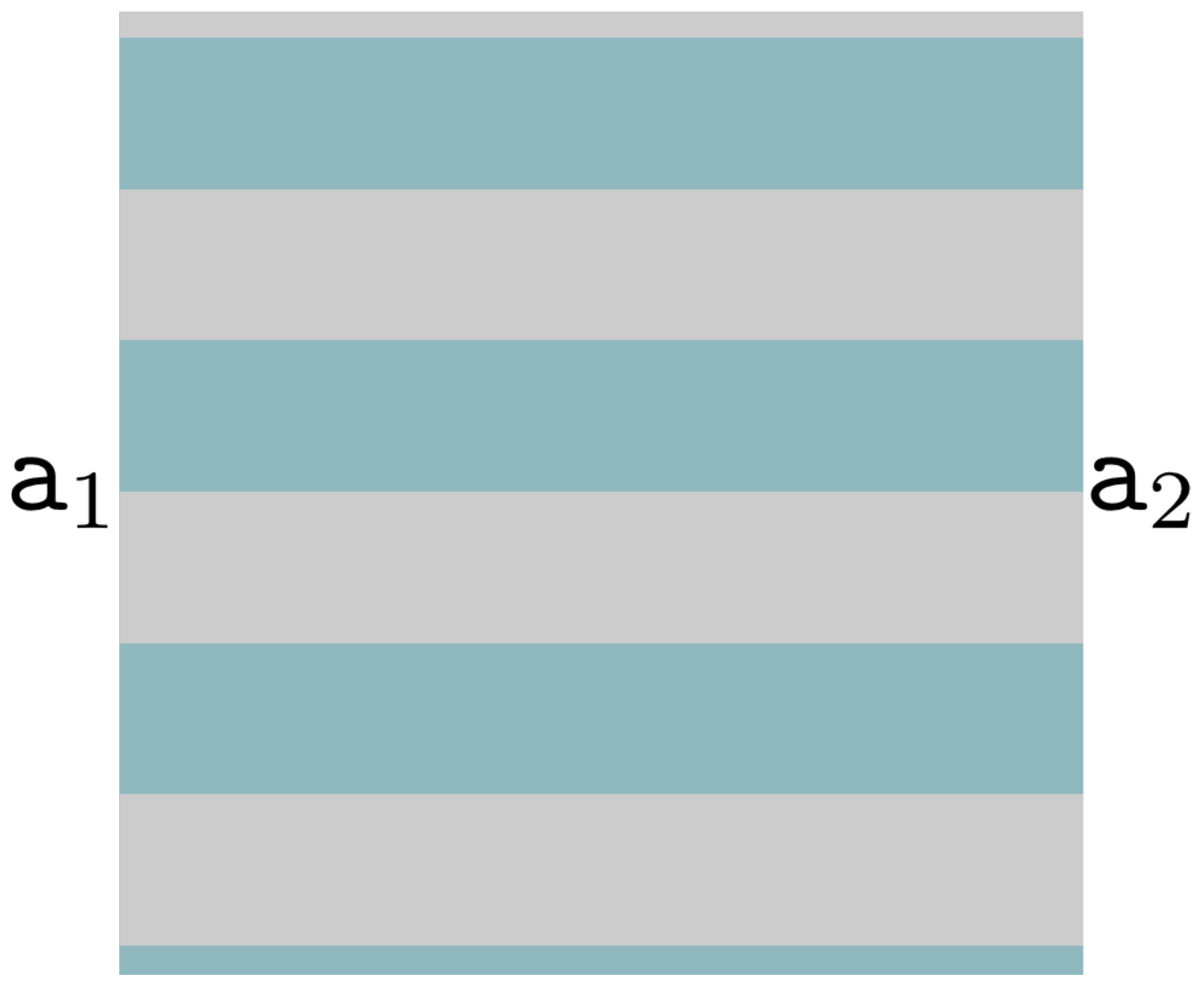}
{c)}
\includegraphics[width=0.25\textwidth]{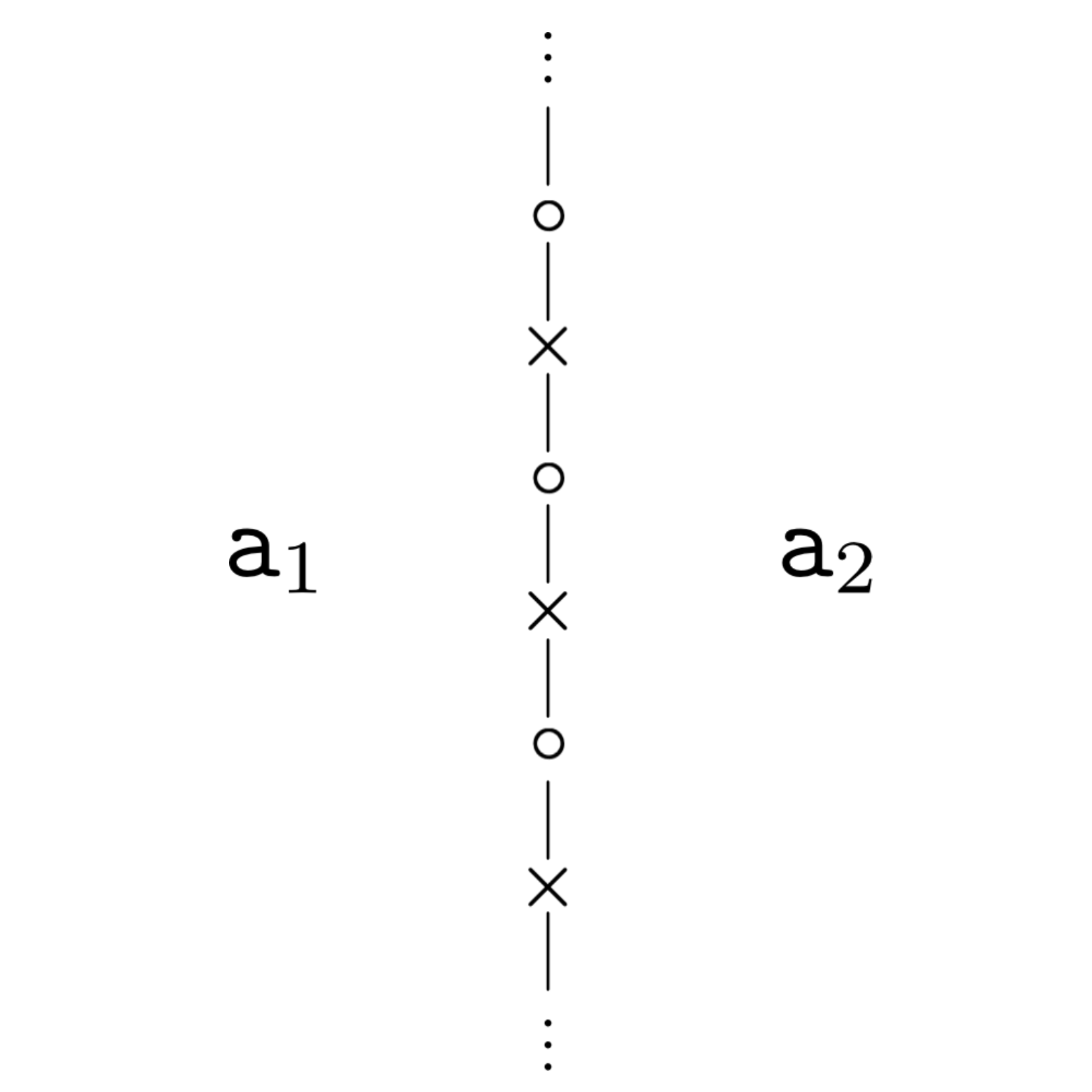}
\\
\vspace{0.3cm}
{b)}
\includegraphics[width=0.3\textwidth]{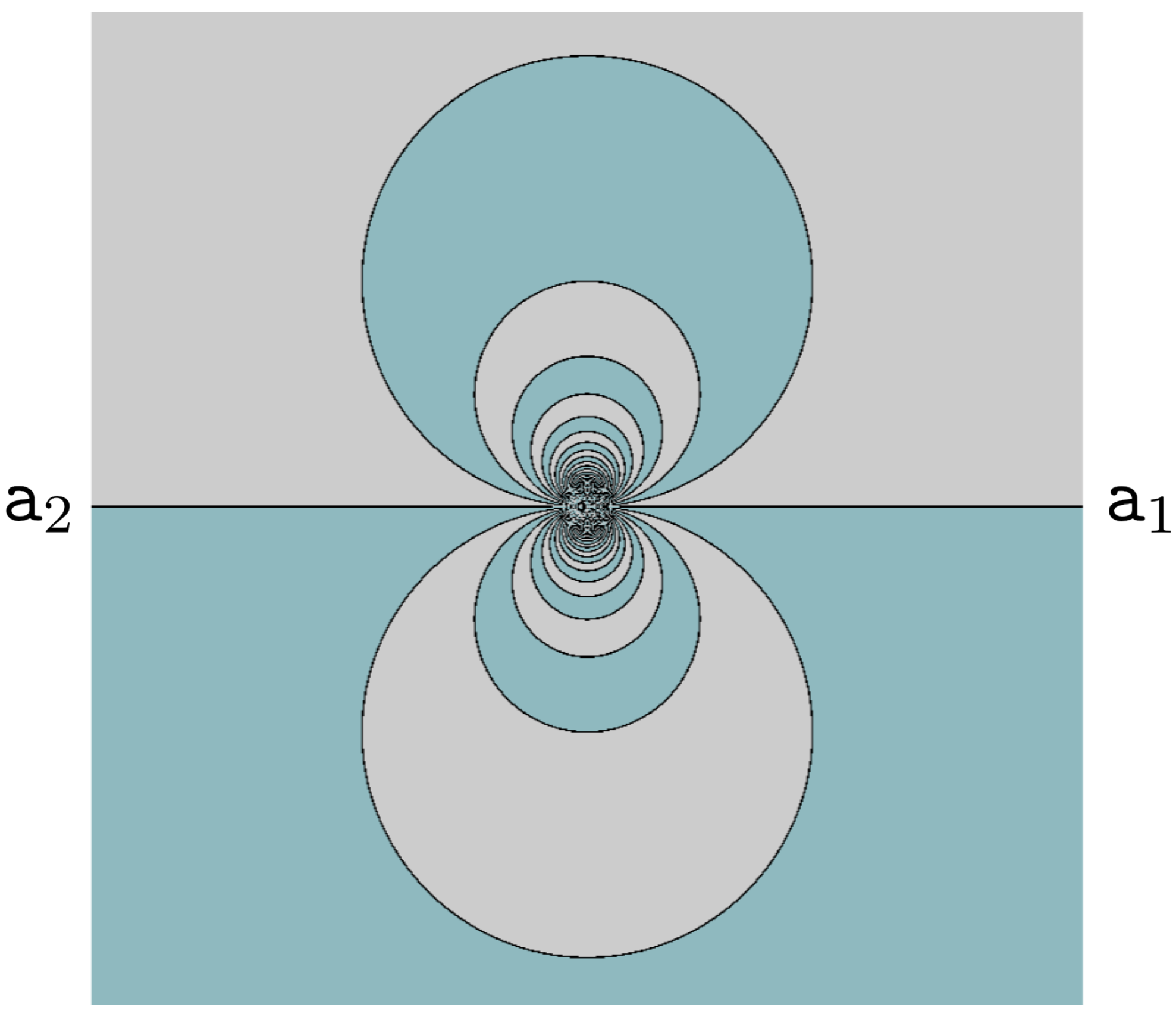}
\caption{
Consider  
$\big(
(\CC_z \cup \{ \infty_{{\tt a}_1}, \, 
\infty_{{\tt a}_2} \} ) \backslash w(z)^*\gamma, 
w(z)^*\mathcal{L}_\gamma \big)$ 
the tessellation
corresponding to the
$N$--functions 
$w(z)=\e^{z}$ 
(when  ${\tt a}_1 = 0, {\tt a}_2 = \infty$)
and 
$w(z)=\tanh(z)$ 
(when  ${\tt a}_1 = -1, {\tt a}_2 = 1$),
for $\gamma = \RR\cup\{\infty\}$. 
(a) The tessellation near the origin, 
(b) near the essential singularity at $\infty\in\CW_z$.
(c) The corresponding Speiser graph of index $2$.
}
\label{fig:Tessellation-Exp}
\end{center}
\end{figure}

\item 
Moreover, by considering the canonical vector fields 

\centerline{
$X_{\ent{}{} }(z) \doteq \frac{1}{ \ent{}{}^{\prime} (z) } \del{}{z} = \e^{-z} \del{}{z}$ 
\quad and \quad
$X_{ \htan{}{}  } (z) \doteq \frac{1}{ \htan{}{}^{\prime} (z) } \del{}{z} = \cosh ^2(z) \del{}{z}$ }

\noindent
associated to 
$\e^z$ and $\tanh(z)$
respectively, we can observe
a correspondence\footnote{
See \cite{AlvarezMucino3}\,proposition\,2.5 for the complete correspondence.
} between 
the elementary blocks (left column) and the pairs (domain, vector field):
\begin{equation}
\begin{aligned}
\label{eq:mini-correspondencia}
\big(\overline{\HH}, \ent{}{} (z) \big) \ &\longleftrightarrow \ \big(\overline{\HH}, X_{\ent{}{} }(z) \big) \\
\big(\overline{\HH}, \htan{}{} (z) \big)  \ &\longleftrightarrow \ \big(\overline{\HH}, X_{\htan{}{} }(z) \big).
\end{aligned}
\end{equation}
\noindent
With this correspondence, we can now visualize the phase portrait of $\Re{X_{\ent{}{} }}(z)$
and $\Re{X_{\htan{}{} }}(z)$ as shown in Figure \ref{fig:piezas-elementales}.
In order to distinguish the two elementary blocks, we abuse notation and use 
the visualization of the corresponding vector fields instead of the usual 
(indistinguishable) tessellations.

\begin{figure}[htbp!]
\begin{center}
\includegraphics[width=0.385\textwidth]{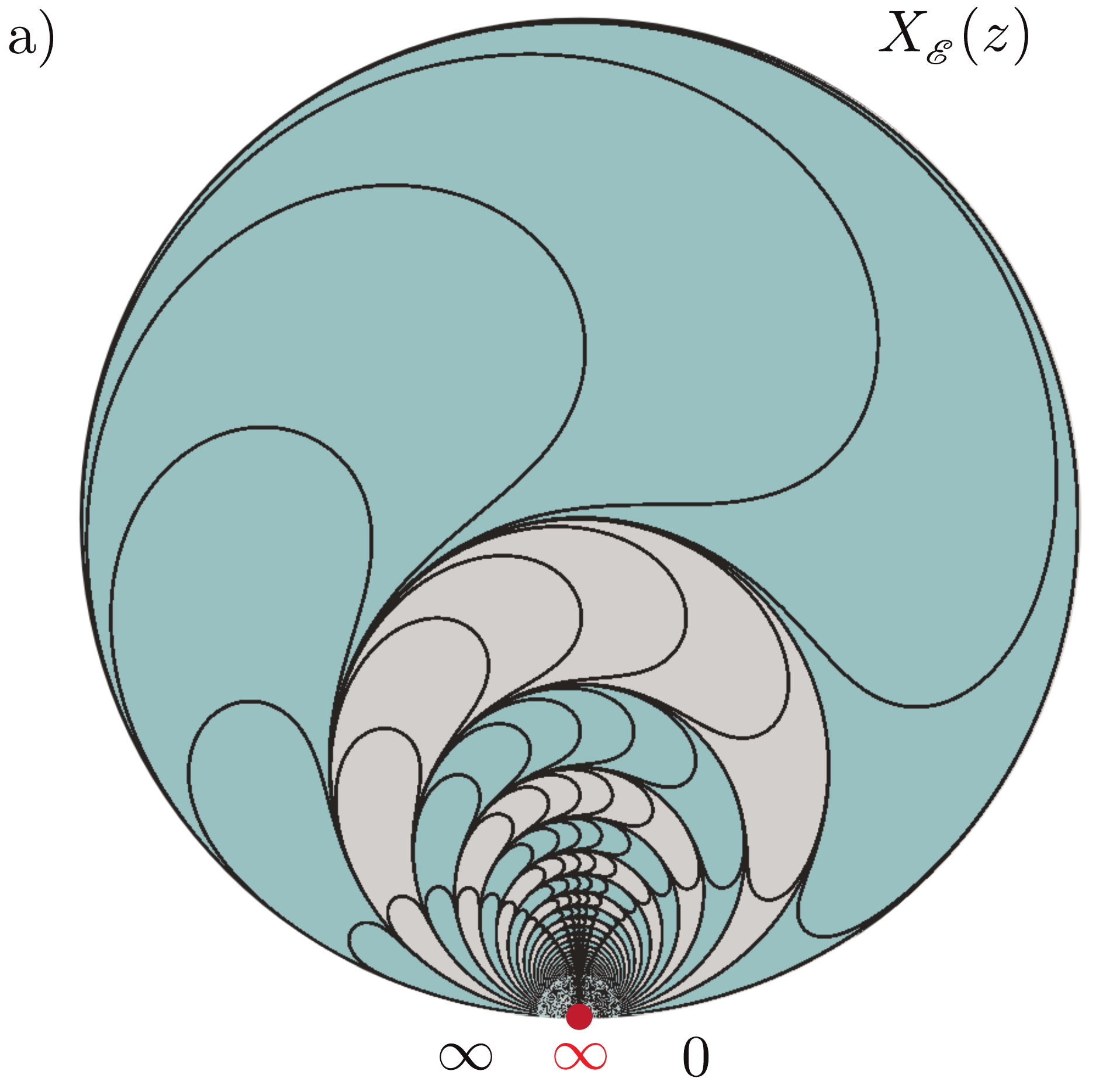}
\hspace{10 pt}
\includegraphics[width=0.38\textwidth]{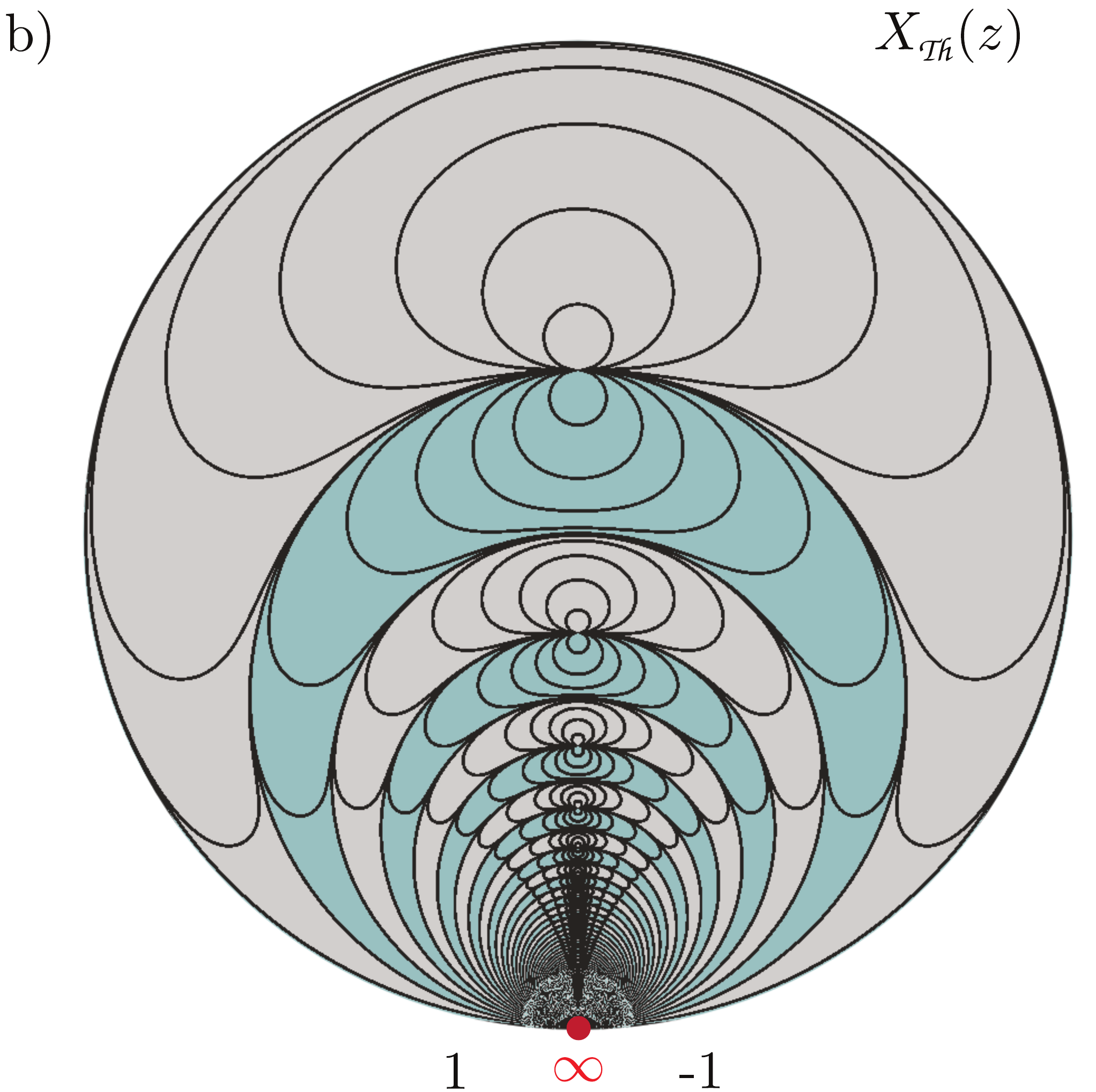}
\caption{
Two elementary blocks. 
(a) Exponential block
$\ent{0}{\infty}$ corresponding to $\ent{}{} (z) = \e^{z}:\HH\longrightarrow\CW_w$.
(b) $h$--tangent block
$\htan{-1}{1} $  corresponding to 
$\htan{}{} (z) = \tanh(z):\HH\longrightarrow\CW_w$.
Visualized using the associated
vector fields $X_{\ent{}{} } (z) = \e^{-z} \del{}{z}$ and
$X_{\htan{}{}  } (z) = \cosh ^2(z) \del{}{z}$.
Note that since the Riemann surfaces
$\R_{\ent{}{} (z)}$ and $\R_{\htan{}{} (z)}$,
do not have maximal logarithmic towers, there is no first blue or gray
region; the blue--gray coloring is not canonical.
The asymptotic values are indicated 
on each side of the essential singularity 
(marked red) at $\infty\in\partial \overline{\HH}$.
The ``dipoles'' in (b) correspond to the simple poles
of $\tanh(z)$.
}
\label{fig:piezas-elementales}
\end{center}
\end{figure}
\end{enumerate}
\end{example}

Recalling the decomposition of a sheet into half sheets, 

\centerline{
$\mathfrak{L}_\Upxi = 
\CW \backslash \Upxi =
(\mathfrak{H}^+ \cup \mathfrak{H}^-)\backslash \Upxi
=\mathfrak{L}^+_\Upxi  \cup \mathfrak{L}^-_\Upxi $,
}

\noindent
we can now decompose $\R_{w(z)}$ in a different way than that of 
Proposition \ref{prop:Rw-is-a-union-of-sheets}.

\begin{theorem}[Decomposition of $\R_{w(z)}$ into the soul and maximal logarithmic towers]
\label{theo:decomposition-soul-logtowers}
Let $w(z)$ be a  Speiser function 
provided with a cyclic order $\mathcal{W}_{\tt q}$ 
for its ${\tt q}\geq 2$ distinct singular values.
\begin{enumerate}[label=\arabic*),leftmargin=*]

\item 
The Riemann surface $\R_{w(z)}$ associated to $w(z)$ can be constructed by 
isometric glueing of 
half sheets 
$\mathfrak{H}^+\backslash\Upxi$ and 
$\mathfrak{H}^-\backslash\Upxi^\prime$,
denoted by $\sim$,  as follows
\begin{multline}
\label{eq:descomposicion-Rw}
\R_{w(z)} = 
\Bigg(
\underbrace{
\bigcup_{{\vartheta_+} =1}^{n_+} 
\bigg[
\mathfrak{H}^+\backslash \Big(\bigcup\limits_{
{\tt j}( \msigma ) , {\tt j}( \mrho ) \in  \Upxi_{\vartheta_+} 
} 
\overline{ {\tt w}_{{\tt j}( \msigma )} 
{\tt w}_{{\tt j}( \mrho )} } \Big)
\bigg]
\bigcup_{{\vartheta_-} =1}^{n_-}
\bigg[
\mathfrak{H}^-\backslash \Big(\bigcup\limits_{
{\tt j}( \msigma ) , {\tt j}( \mrho ) \in  \Upxi_{\vartheta_-} 
} 
\overline{ {\tt w}_{{\tt j}( \msigma )} 
{\tt w}_{{\tt j}( \mrho )} } \Big)
\bigg]
}_{\text{soul}}
\cup
\\ 
\underbrace{
\mathcal{T}({\tt a}_{\alpha_1}, {\tt a}_{\beta_1})
}_{1^{\text{st}} \text{ logarithmic tower}} 
\cup
\ldots 
\cup
\underbrace{ 
\mathcal{T} ( {\tt a}_{\alpha_{\tt p}}, {\tt a}_{\beta_{\tt p}} )
}_{ {\tt p}^{\text{th}} \text{ logarithmic tower}} 
\Bigg) \, \Bigg/ \sim .
\end{multline}

\noindent
In the above expression the following conventions are observed.

\noindent
$\bigcdot$
The singular values 
of $w(z)$
are denoted by $ \{{\tt w}_{{\tt j} (\iota) } \}_{\iota=1}^\delta $, and as usual
the asymptotic values are $\{ {\tt a}_{\alpha_\sigma} , {\tt a}_{\beta_\sigma} \}_{\sigma=1}^{\tt p}$,
recall
Equation \eqref{eq:enumerating-branch-points}.

\noindent
$\bigcdot$
The number of half sheets in the soul is 
$(n_+ + n_-) \in \NN \cup\{\infty\}
$, 
the half sheets $\mathfrak{H}^\pm \backslash\Upxi_{\vartheta_\pm}$ 
are indexed by $\vartheta_\pm$.

\noindent
$\bigcdot$
The $\tt p$ maximal logarithmic towers 
of $\R_{w(z)}$ 
are distinguished by suitable pairs of 
asymptotic values 

\centerline{
$\{ 
\underbrace{
\overline{{\tt a}_{\alpha_1} {\tt a}_{\beta_1}  } , 
\ldots  ,
\overline{{\tt a}_{\alpha_{\tt p} } {\tt a}_{\beta_{\tt p} } }
}_{\star\mbox{--boundary} }
\} 
\subseteq 
\{ \overline{{\tt w}_{{\tt j}( \msigma )} {\tt w}_{{\tt j}( \mrho )} }  \}$,
}

\noindent 
where $\star$--boundary coincides with the boundary of the 
maximal logarithmic towers.

\item 
On each maximal logarithmic tower 

\centerline{
$\mathcal{T}( {\tt a}_{\alpha_{\tt j}} , {\tt a}_{\beta_{\tt j}} )$,
\quad
${\tt j}\in\{1,\ldots,{\tt p} \}$, 
}

\noindent 
the function $w(z)$ 
is an exponential or an $h$--tangent block, 
\emph{i.e.}\ 
is right--left equivalent to 
\begin{enumerate}[label=\alph*),leftmargin=*]

\item 
the function $\exp(z): \overline{\HH} \longrightarrow \CW$ 
when one of the asymptotic values 
$ {\tt a}_{\alpha_{\tt j}} , {\tt a}_{\beta_{\tt j}} \in\CW_w$ 
is infinite,

\item 
the function $\tanh (z): \overline{\HH} \longrightarrow \CW$ 
when both asymptotic values 
$ {\tt a}_{\alpha_{\tt j}}, {\tt a}_{\beta_{\tt j}} \in\CC_w$ are finite.

\end{enumerate}

\item 
The soul is a flat $\tt p $--gon,
with at most 
$0 \leq {\tt r} + {\tt p} \leq \infty$
critical points. 
Furthermore, the soul determines the Riemann surface $\R_{w(z)}$.

\item 
Assume in addition that
$w(z)$ is a finite Speiser function. 
\begin{enumerate}[label=\alph*),leftmargin=*]
\item
The number $\tt p$ of maximal logarithmic towers is zero 
if and only if $w(z)$ is a rational function on $\CW_z$,
otherwise $2 \leq {\tt p} < \infty$.

\item
The soul is a flat $\tt p $--gon, 
with at most 
$0 \leq {\tt r} + {\tt p} < \infty$
critical points,
\emph{i.e.} the function $w(z)$ is right--left 
equivalent to a rational function 
$R(z)$ restricted to a certain Jordan domain 
$\mathscr{P}\subset\CW_z$.
\end{enumerate}
\end{enumerate}
\end{theorem}

\begin{proof}[Proof of Theorem \ref{theo:decomposition-soul-logtowers}]
The proof is basically a direct application of Theorem \ref{th:main-theorem}, 
specifically (4) $\Longleftrightarrow$ (2).
Thus the converse: a Riemann surface as in \eqref{eq:descomposicion-Rw} produces a
Speiser function $w(z)$, is also immediate.

In full detail,
recall Diagram \eqref{diagramaRX}, particularly that $\pi_1$ is a biholomorphism.
Thus the analytic Speiser graph $(\mathfrak{S}_{\tt q},\mathcal{L}_{\mathcal{W}_{\tt q}})$,
together with the cell decomposition $\Omega_z\backslash\mathfrak{S}_{w(z)}$, 
and Proposition \ref{prop:speiser-graph-implies-speiser-function}, 
provide an accurate representation of $\R_{w(z)}$:
\begin{enumerate}[label=\roman*),leftmargin=*]
\item  
each vertex $v=\circ$,  of 
the Speiser graph $(\mathfrak{S}_{w(z)},\mathcal{W}_{\tt q})$,
corresponds to a half sheet 
$\big( \mathfrak{H}^+ \backslash \overline{ {\tt a}_{\alpha} {\tt a}_{\beta} } \big) \subset \R_{w(z)}$,

\item
each vertex $v=\times$,  of 
the Speiser graph $(\mathfrak{S}_{w(z)},\mathcal{W}_{\tt q})$,
corresponds to a half sheet 
$\big( \mathfrak{H}^- \backslash \overline{ {\tt a}_{\alpha} {\tt a}_{\beta} } \big) \subset \R_{w(z)}$,

\item
each edge of the Speiser graph indicates the gluing of the corresponding half sheets,

\item
each ${\tt w}_{\tt j}$--face of
$\Omega_z\backslash\mathfrak{S}_{w(z)}$
represents a branch point of the surface $\R_{w(z)}$, with ramification
index half the number of sides of the ${\tt w}_{\tt j}$--face.
\end{enumerate}
 
\smallskip
Denote the set of asymptotic values (counted with multiplicity) by 
$\{{\tt a}_\beta \} \subset \{ {\tt w}_{{\tt j}(\iota)} \}_{\iota=1}^\delta$; 
necessarily its cardinality is $0 \leq {\tt p}\leq\delta < \infty$.

Given (an arbitrary) Speiser function $w(z)$, 
the following result relates logarithmic ends 
of the Speiser graph 
$(\mathfrak{S}_{w(z)},\mathcal{W}_{\tt q})$, see Definition \ref{def:Log-end-for-Speiser-graph},
to the
logarithmic towers of the Riemann surface $\R_{w(z)}$, 
see Definition \ref{def:logarithmic-tower-for-Riemann-surface}.

\begin{lemma}
\label{lem:correspondence-log-towers}
Let $w(z):\Omega_z \longrightarrow \CW_w$ be a Speiser function with ${\tt q }\geq 2$ singular values.

\begin{enumerate}[label=\arabic*),leftmargin=*]
\item
There is a bijection between:

\begin{enumerate}[label=\roman*),leftmargin=*]
\item
the logarithmic ends $\mathcal{T}$ of the analytic Speiser graph
$(\mathfrak{S}_{w(z)},\mathcal{W}_{\tt q})$, 
determined by  ${\tt a}_\alpha$ and ${\tt a}_\beta$, the asymptotic values corresponding to 
the unbounded ${\tt a}_\alpha$  and ${\tt a}_\beta$--faces 
of $\mathcal{T}$,

\item
the maximal logarithmic towers of $\R_{w(z)}$, 
$\mathcal{T} ({\tt a}_\alpha, {\tt a}_\beta ) $,   
determined by the pair $({\tt a}_\alpha {\tt a}_\beta)$ of asymptotic values of $w(z)$.
\end{enumerate}

\item
The analogous statement applies for the nucleus of $(\mathfrak{S}_{w(z)},\mathcal{W}_{\tt q})$
and the soul of $\R_{w(z)}$.
\end{enumerate}
\end{lemma}

\begin{proof}
By Proposition \ref{prop:speiser-graph-implies-speiser-function}.2.c, 
$\mathcal{T}$ has associated to itself
asymptotic values ${\tt a}_\alpha$ and ${\tt a}_\beta$, 
where ${\tt a}_\alpha\neq{\tt a}_\beta$. 
Furthermore, it is easy to check that the edge bundle in $\mathfrak{S}_{w(z)}$ between the pair of vertices  
$(v_{2\uptau -1}, v_{2\uptau})$ and the pair $(v_{2(\uptau +1)-1}, v_{2(\uptau +1)})$ corresponds to 
the isometric glueing of the sheets
$\big( \CW\backslash \overline{ {\tt a}_{\alpha} {\tt a}_{\beta} } \big)_{\vartheta}$ 
and $\big( \CW\backslash \overline{ {\tt a}_{\alpha} {\tt a}_{\beta} } \big)_{\vartheta +1}$ 
in $\R_{w(z)}$, 
with $\vartheta = \uptau$,
recall Corollary \ref{cor:pegado-isometrico}.
Thus, each pair $ (v_{2\uptau -1}, v_{2 \uptau} )$ of adjacent vertices, represents
a sheet
$\big( \CW\backslash \overline{ {\tt a}_{\alpha} {\tt a}_{\beta} } \big)_{\vartheta }$,
with $\vartheta = \uptau$.

\noindent
The rest of the proof is left to the reader.
\end{proof}

\noindent
Recalling Example \ref{example:ends-not-logarithmic}, 
note that the set of logarithmic towers may be empty.

\noindent
From Definition \ref{def:soul-of-surface}, the complement of the $\tt p$ maximal logarithmic towers 
in the decomposition \eqref{eq:Rw-decomposition-in-sheets}, is the soul, proving statement (1). 

\smallskip

Statement (2) follows immediately from Lemma 
\ref{lem:extention-to-petal}, 
where the pair of logarithmic singularities can be 
recognized
as an exponential or $h$--tangent block. 

\smallskip

Statement (3) follows from the definition of the soul as the complement of the maximal logarithmic towers,
and an accurate interpretation of Lemma \ref{lem:extention-to-petal}.

\smallskip

For statement (4), 
since $w(z)$ is a finite Speiser function, its Speiser graph has a finite number
of bounded and unbounded faces, thus it follows that the ends are in fact logarithmic ends, say $0\leq{\tt p}<\infty$, 
${\tt p}\neq 1$ of them (see below).

\noindent
A simple argument, 
shows that the nucleus $\mathfrak{N}$ is connected
and has $\tt p$ logarithmic ends as its boundary.
Hence by Remark \ref{nomenclatura-logarithmic-tower}
and Lemma \ref{lem:correspondence-log-towers},
the soul is connected and 
can be recognized as a flat $\tt p$--gon, 
thus Lemma \ref{lem:polygon-rational-block} proves
statement (4.b).

\smallskip

Finally, staying in the finite Speiser case, let us describe the Riemann surfaces 
$\R_{w(z)}$ appearing for different values of $\tt p$.

\noindent 
$\bigcdot$ 
Case ${\tt p}=0$. In this case, 
the cell decomposition arising from the Speiser graph $\mathfrak{S}_{w(z)}$
consists of a finite number of bounded $\tt w_j$--faces. 
Thus $\mathfrak{S}_{w(z)}$ is finite and we conclude that $w(z)$ is a rational function.
Moreover, $\R_{w(z)}$ is the soul.

\smallskip

\noindent 
$\bigcdot$ 
Case ${\tt p}=1$ does not appear. Suppose the contrary, hence $\R_{w(z)}$ 
has exactly one infinitely ramified branch point, thus $\R_{w(z)}$ has an 
infinite number of sheets, and hence an infinite number of finitely ramified 
branch points. 
A contradiction, since $w(z)$ is a finite Speiser function.

\smallskip

\noindent 
$\bigcdot$ 
Case ${\tt p}=2$ and no bounded faces other than digons.
By Lemma \ref{lemma-q-2-3}
the cell decomposition arising from the 
Speiser graph $\mathfrak{S}_{w(z)}$, has two unbounded faces, 
as in Figure \ref{fig:Tessellation-Exp}.c.
Note that the nucleus $\mathfrak{N}_\mathfrak{S}$ of $\mathfrak{S}_{w(z)}$ is empty; 
thus the soul $\mathfrak{N}_{w(z)}$ of $\R_{w(z)}$ is also empty.
Example \ref{example:N-exp-tanh} describes explicitly the possible
functions $w(z)$. 

\smallskip

\noindent 
$\bigcdot$ 
Case ${\tt p}\geq 3$.
The proof for the generic cases uses the arguments presented before 
Lemma \ref{lem:correspondence-log-towers} and is as follows.

\noindent
Step 1.\
The Speiser graph
has exactly $\tt p$ logarithmic ends.
Thus, from the fact that there are only a finite number of singularities of $w^{-1}(z)$
and Proposition \ref{prop:speiser-graph-implies-speiser-function}.2, 
the number of bounded faces that are not digons, 
and the number of unbounded faces, are both finite.
From this, it is easy to see that the only way that the cell decomposition arising from a 
Speiser graph $\mathfrak{S}_{w(z)}$ has $\tt p$ unbounded faces, 
is that it has exactly $\tt p$ logarithmic ends.

\noindent
Step 2.\
By Lemma \ref{lem:correspondence-log-towers}, 
we see that $\R_{w(z)}$ has 
exactly $\tt p$ maximal logarithmic towers $\mathcal{T} ({\tt a}_\alpha, {\tt a}_\beta )$ of $\R_{w(z)}$.
\end{proof}

\subsection{Characterization of finite Speiser functions on $\Omega_z=\CC_z, \, \CW_z$}
\label{sec:caract-Speiser-finitas}
Recall that a \emph{finite Speiser function} is a Speiser function with a finite
number $\delta = {\tt p} + {\tt r}$ of singularities of $w^{-1}(z)$.
It is a classical
result of R. Nevanlinna that only $\Omega_z = \CC_z$ appears. See \cite{Nevanlinna1} \S8 
and \cite{Nevanlinna2} p.\ 301.

Considering Theorem \ref{th:main-theorem},
Theorem \ref{theo:decomposition-soul-logtowers}, 
Lemma \ref{lem:extention-to-petal} and Definition \ref{def:piezas-elementales},
we have proved.

\begin{corollary}[Characterization of finite Speiser functions on
$\Omega_z=\CC_z, \, \CW_z$]
\label{Cor:caract-Speiser-finitas}
The following objects are equivalent.
\begin{enumerate}[label=\arabic*),leftmargin=*]

\item 
A finite Speiser function $w(z):\Omega_z\longrightarrow \CW_w$.

\item 
A meromorphic function $w(z):\Omega_z\longrightarrow \CW_w$, 
constructed by surgery of: 

\begin{enumerate}[label=\alph*),leftmargin=*]
\vspace{0.1cm}
\item 
a rational block 
$R(z) : \overline{ \mathscr{P}} \subset \CW_z 
\longrightarrow \CW_w$, and 

\item a finite number of 

\noindent 
$\bigcdot$
exponential blocks, 
$\exp(z):\overline{\HH}  \longrightarrow \CW_w$,

\noindent 
$\bigcdot$
$h$--tangent blocks, 
$\tanh(z):\overline{\HH} \longrightarrow \CW_w$.
\end{enumerate}

\item
A flat ${\tt p}$--gon $\big(\overline{\mathcal{P}}, \, \mathpzc{w} (\zeta) \big)$,
with $2 \leq {\tt p} < \infty$,
whose function $\mathpzc{w} (\zeta)$ has
a finite number of critical points
in the interior of $\overline{\mathcal{P}}$,
with an exponential or $h$--tangent blocks glued 
to each side.

\item
A Speiser Riemann surface $\R_{w(z)}$ with a finite number of branch points.

\item
A Speiser graph of index $\tt q$ with only a finite number of faces that are not digons.

\item 
A Speiser $\tt q$--tessellation with a finite number of vertices of valence 
greater than or equal to 4.
\hfill\qed
\end{enumerate}
\end{corollary}

The above result extends the notion of
\emph{structurally finite entire functions} considered
by M. Taniguchi \cite{Taniguchi1}, \cite{Taniguchi2}.

\begin{remark}
1.\ Note that if the set of logarithmic towers is empty,
\emph{i.e.}\ ${\tt p}=0$,
the soul of $\R_{w(z)}$ is itself.
Moroever, Statement (3) is empty (there is no flat $\tt p$--gon).
In particular, if $w(z)$ has no asymptotic values (for instance if $w(z)$ is rational) then 
the set of logarithmic towers is empty.

\noindent
2.\
An immediate consequence of Corollary \ref{Cor:caract-Speiser-finitas},
is that $N$--functions, can be constructed via surgery of:

\noindent
$\bigcdot$
a rational block, without interior singular points, and

\noindent
$\bigcdot$
a finite number $2\leq {\tt p} < \infty$ of 
exponential and $h$--tangent blocks.

\noindent
3.\ 
If we consider \emph{non finite Speiser functions},
other cases appear: 

\noindent
$\bigcdot$ 
$w(z)= \wp(z)$ on $\CC_z$, 
with ${\tt p}=0$, ${\tt r}=\infty$ and ${\tt q}=4$,
its associated Riemann surface has no logarithmic towers 
(yet it has an infinite number of 
algebraic branch points, see \cite{AlvarezGutierrezMucino} example 5.1 for the tessellation),

\noindent
$\bigcdot$
$w(z)=\cos\sqrt{z}$,
with ${\tt p}=1$, ${\tt r}=\infty$ and ${\tt q}=3$, 
recall Example \ref{ex:examples-finite-non-finite-Speiser-fns}.4.b,

\noindent
$\bigcdot$ 
$w(z)=\sin(z)$,
with ${\tt p}=2$, ${\tt r}=\infty$ and ${\tt q}=3$, 

\noindent
$\bigcdot$
$w(z)=\sin^2(z)$, 
with ${\tt p}=4$, ${\tt r}=\infty$ and ${\tt q}=3$,

\noindent
their Speiser graphs have no logarithmic ends, 
see  \cite{GoldbergOstrovskii}\,p.~360.

\noindent
4.\  
Speiser functions $w(z)$ on $\Omega_z=\Delta$, always have an
infinite number of singularities of $w^{-1}(z)$; a characterization similar
to Corollary \ref{Cor:caract-Speiser-finitas} is an open question.
\end{remark}


\section{Examples}
\label{sec:examples}

\begin{example}[$N$--function with ${\tt q}=3$]
\label{example:tessellation-Airy}
Recall the Schwarzian differential equation 
\eqref{eq:Schwarzian-diff-eqn},
$$
Sw\{ w, z \} = -2z.
$$
\hfill
\begin{enumerate}[label=\alph*),leftmargin=*]
\item
Up to M\"obius transformations, 
the solutions $w(z)$ are quotients of two Airy functions,
in particular

\centerline{
$w_{\Ai} (z) = \dfrac{ \Bi(z)}{\Ai(z)}$,
}

\noindent 
has the asymptotic values 

\centerline{$
\{ {\tt a}_1, {\tt a}_2, {\tt a}_3 \} = \{ -i, i, \infty \}$ }

\noindent 
as its only singular values.
Thus, $w_{\Ai}(z)$ is an $N$--function.

\item 
The branch points in the Riemann surface $\R_{w_{\Ai}(z)}$ are 

\centerline{$
\circled{\text{-}i}=(\infty_1, -i, \infty ), \ 
\circled{i}=(\infty_2, i, \infty ), \
\circled{\infty}=(\infty_3, \infty , \infty ).
$}

\noindent
Amongst the different possible sheets, consider the following 4 different types,
with indices $\Upxi=\{1,2,3,4\}$
\begin{equation}
\label{eq:eleccion-hojas}
\begin{array}{lrl}
\mathfrak{L}_{1} = \CW_w \backslash  \overline{-i\, i}, & &
\mathfrak{L}_{2} = \CW_w \backslash  \overline{i \, \infty},
\\
\mathfrak{L}_{3} = \CW_w \backslash \overline{-i \, i\, \infty}, & &
\mathfrak{L}_4 = \CW_w \backslash 
\big( \overline{-i \, i} \cup  \overline{i \,  \infty} \cup 
\overline{-i \, i\, \infty} \big) .
\end{array}
\end{equation}

\noindent
The sheets $\mathfrak{L}_3$ and $\mathfrak{L}_4$ are different in 
the following sense: 

\noindent
$\bigcdot$
the boundary of $\mathfrak{L}_3$ consists of two copies of the segment 
$\overline{-i \, i\, \infty}$, 

\noindent
$\bigcdot$
the boundary of $\mathfrak{L}_4$ consists of a copy of the segment
$\overline{-i \, i\, \infty}$, 
and a copy of the segments $\overline{-i \, i}$ and $\overline{i \,  \infty}$.

\noindent
In both cases, $i$ is a co--singular value.

\noindent
The diagonals in $\R_{w_{\Ai}(z)}$ are: 

$\Delta_{\vartheta_1 \, -i \,  i}$, with $\vartheta_1 \in \NN$,
note that that $\pi_2 (\Delta_{\vartheta_1 \, -i \, i} ) = \overline{-i \, i}$,

$\Delta_{\vartheta_2 \,  i \, \infty}$, with $\vartheta_2 \in \NN$,
note that that $\pi_2 (\Delta_{\vartheta_2 \, i \, \infty} ) = \overline{i \, \infty}$,

$\Delta_{\vartheta_3 \, -i\, \infty }$, with $\vartheta_3 \in \NN$,
note that that $\pi_2 (\Delta_{\vartheta_3 \, -i\, \infty } ) = \overline{-i\, i\, \infty }$,

$\Delta_{\vartheta_4 \, -i \, i}  \cup 
\Delta_{\vartheta_4 \, i \, \infty}  \cup
\Delta_{\vartheta_4 \, -i \, \infty } $, with $\vartheta_4=1$.

\noindent
Thus, the actual sheets that appear in $\R_{w_{\Ai}(z)}$ are:

$\mathfrak{L}_{1,\vartheta_1} = \CW_w \backslash \big( \pi_2 (\Delta_{\vartheta_1 \, -i \, i} ) \big)
= \big( \CW_w \backslash  \overline{-i \, i} \big)_{\vartheta_1}$,
with $\vartheta_1\in\NN$,

$\mathfrak{L}_{2,\vartheta_2} = \CW_w \backslash \big( \pi_2 (\Delta_{\vartheta_2 \, i \, \infty} ) \big)
= \big( \CW_w \backslash  \overline{i \, \infty} \big)_{\vartheta_2}$,
with $\vartheta_2\in\NN$,

$\mathfrak{L}_{3,\vartheta_3} = \CW_w \backslash \big( \pi_2 (\Delta_{\vartheta_3 \, -i\, \infty } ) \big)
= \big( \CW_w \backslash  \overline{-i\, i\ \infty } \big)_{\vartheta_3}$,
with $\vartheta_3\in\NN$,

$\mathfrak{L}_{4,\vartheta_4} = \CW_w \backslash 
\big( 
\pi_2 (\Delta_{\vartheta_4 \, -i \, i} ) \cup 
\pi_2 (\Delta_{\vartheta_4 \, i \, \infty} ) \cup
\pi_2 (\Delta_{\vartheta_4  \, \infty \, -i} ) 
\big)
$

$
\quad \quad  =\Big( \CW_w \backslash 
\big( \overline{-i \, i} \cup  \overline{i \, \infty} \cup 
\overline{-i\, i\, \infty } \big) \Big)_{\vartheta_4}
$, 
with $\vartheta_4=1$.

\noindent
Thus by gluing $\mathfrak{L}_{4,1}$ to $\mathfrak{L}_{1,1}$,
$\mathfrak{L}_{2,1}$ and $\mathfrak{L}_{3,1}$, along their common boundaries, 
we obtain a
decomposition of $\R_{w_{\Ai}(z)}$ into maximal domains of 
single--valuedness,
as in Proposition \ref{prop:Rw-is-a-union-of-sheets},
namely
\begin{equation}
\label{eq:decomposicion-dominios-maximales-Ai}
\R_{w_{\Ai}(z)} = 
\Bigg(
\mathfrak{L}_{4,1} 
\ \ \cup
\bigcup_{\vartheta_1= 1}^\infty 
\mathfrak{L}_{1, \vartheta_1}
\cup
\bigcup_{\vartheta_2= 1}^\infty 
\mathfrak{L}_{2, \vartheta_2 }
\cup
\bigcup_{\vartheta_3= 1}^\infty 
\mathfrak{L}_{3, \vartheta_3 }
\Bigg) \, \Big/ \sim \, .
\end{equation}

\noindent
Clearly, we could have chosen different types of sheets in Equation \eqref{eq:eleccion-hojas}; 
for instance by choosing $\mathfrak{L}_3 =  \CW_w \backslash \overline{\infty\, -i}$, 
$\mathfrak{L}_1 = \CW_w \backslash \overline{i\, \infty\, -i}$ and 
$\mathfrak{L}_4 = \CW_w \backslash 
\big( \overline{i\, \infty\, -i} \cup  \overline{i \,  \infty} \cup 
\overline{\infty\, -i} \big)$, 
the decomposition \eqref{eq:decomposicion-dominios-maximales-Ai} into 
maximal domains of single--valuedness would be different.

On the other hand,
considering the cyclic order 

\centerline{$\mathcal{W}_3 = [-i,i,\infty]$,}

\noindent
the decomposition of $\R_{w_\Ai(z)}$ into ${\tt p}=3$
maximal logarithmic towers 
and the unique soul is 
\begin{equation*}
\R_{w_\Ai(z)} = \Big[
\underbrace{\mathfrak{H}^+\backslash \big( \overline{-i \, i} \cup  \overline{i \, \infty} \cup 
\overline{\infty \, -i} \big) }_{\text{soul}}
\ \cup
\underbrace{ \mathcal{T}^\times_- (-i,i) }_{\text{logarithmic tower}}
\cup
\underbrace{ \mathcal{T}^\times_- (i,\infty) }_{\text{logarithmic tower}}
\cup
\underbrace{ \mathcal{T}^\times_- (\infty,-i) }_{\text{logarithmic tower}} 
\Big] \, \Big/ \sim.
\end{equation*}

\item
For the tessellation, with the cyclic order $\mathcal{W}_3$,
it follows that
$\gamma=i\RR\cup\{\infty\}$ and the Speiser $3$--tessellation 
$\big(\CC_z \backslash  w_{\Ai}(z)^*\gamma, 
w_{\Ai}(z)^*\mathcal{L}_\gamma \big)$ is
drawn in Figure \ref{fig:Tessellation-AiBi}.a--b. 
The tiles are $3$--gons, with vertices of valence two represented by green dots 
(the cosingular points). 
The three vertices of infinite valence of 
the graph $w_{\Ai}(z)^*\gamma$
are the points in the non Hausdorff compactification 

\centerline{$\CC_z \cup \{ 
\infty_{{\tt a}_1}, \, 
\infty_{{\tt a}_2}, \,
\infty_{{\tt a}_3}
\}$}

\noindent 
determined by the three asymptotic values. 

\item
Its Speiser graph of index $3$ is drawn in
Figure \ref{fig:Tessellation-AiBi}.c. 
Note that the nucleus consists of one vertex $v=\circ$ with 
three edges (in red), surrounded by three logarithmic ends (in black).

\begin{figure}[htbp!]
\begin{center}
{a)}
\includegraphics[width=0.25\textwidth]{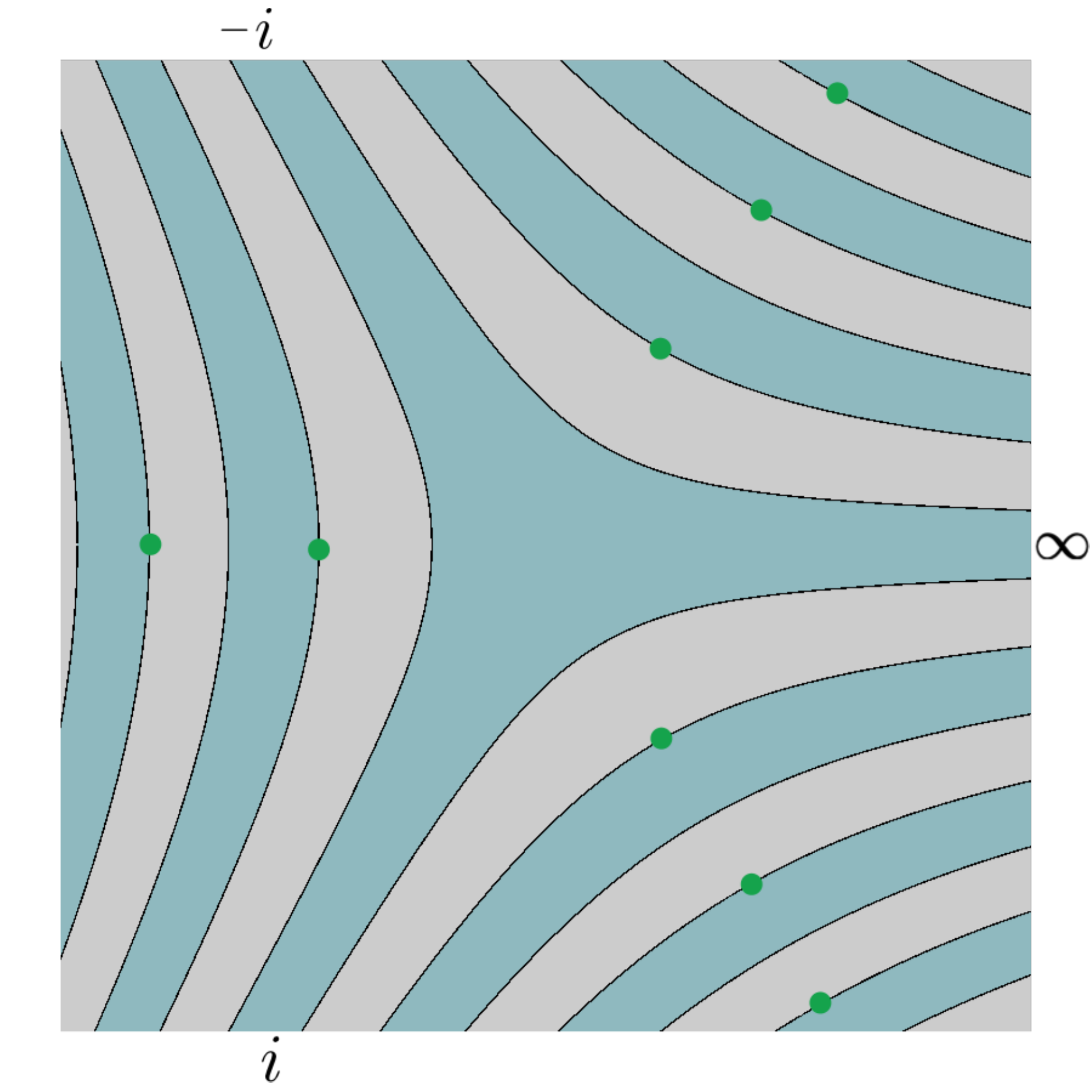}
\hspace{10 pt}
{c)}
\includegraphics[width=0.25\textwidth]{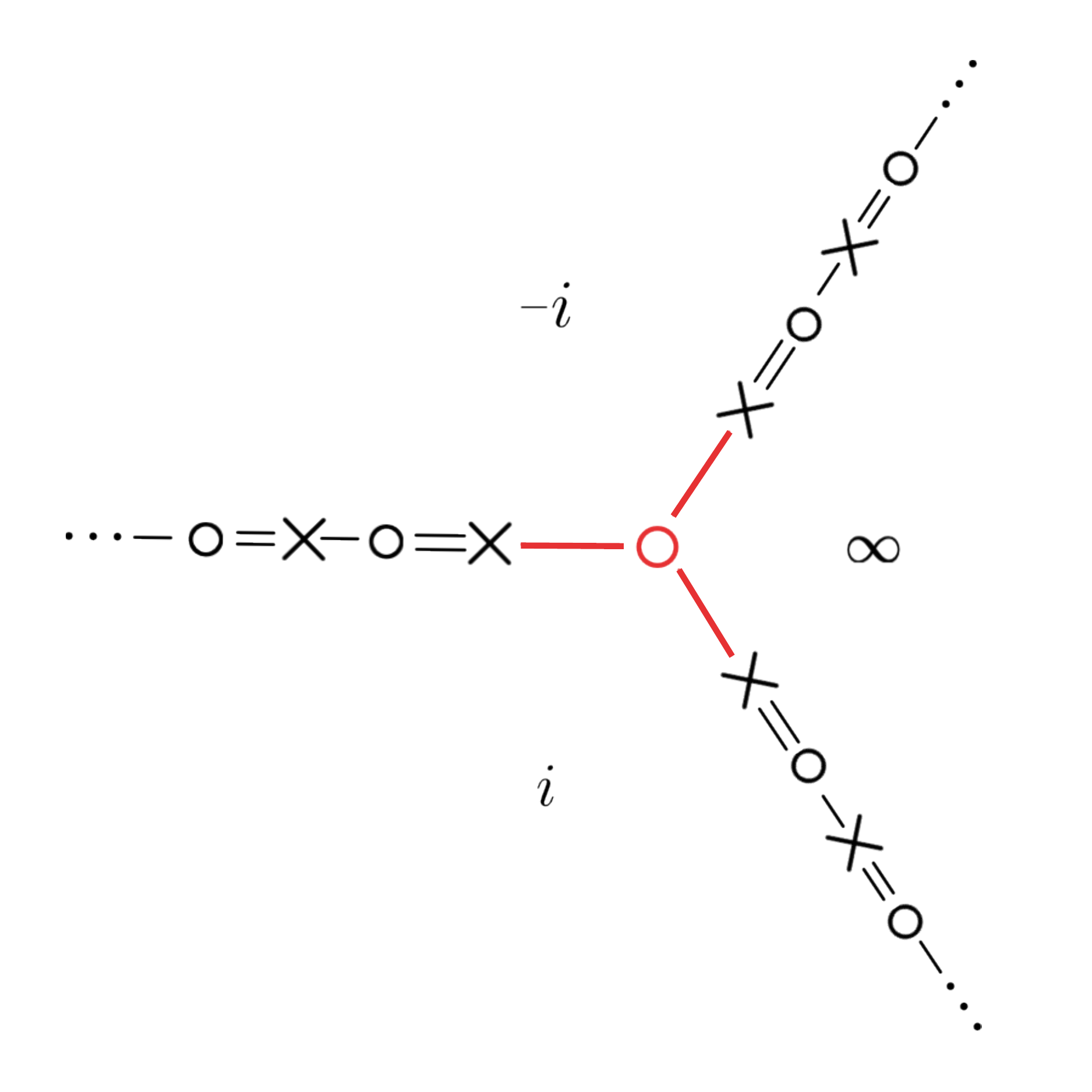}
\\
{b)}
\includegraphics[width=0.28\textwidth]{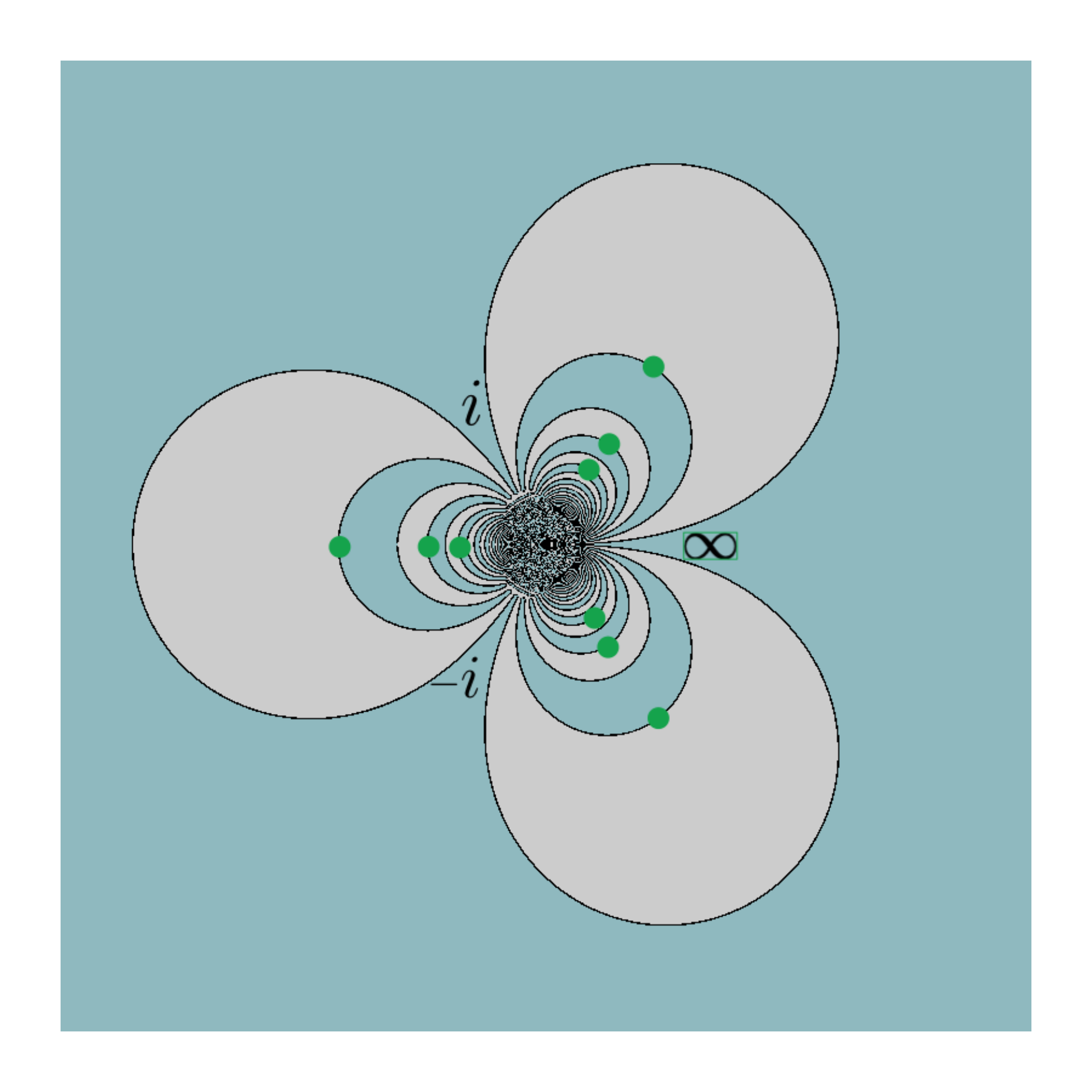}
\caption{
The tessellation 
$\big( \mathscr{T}(w_{\Ai}(z)^*\gamma), w_{\Ai}(z)^*\mathcal{L}_\gamma \big)$ 
corresponding to 
the $N$--function $w_{\Ai}(z) = \Bi(z) / \Ai(z)$
and $\gamma = i\RR\cup\{\infty\}$; 
(a) near the origin, (b) near the essential singularity
at $\infty\in\CW_z$.
(c) The corresponding Speiser graph of index $3$.
The nucleus is colored red and the ${\tt p}$
logarithmic ends are colored black.
}
\label{fig:Tessellation-AiBi}
\end{center}
\end{figure}
\end{enumerate}
\end{example}

\begin{example}
\label{example:function-wM4alt}
Consider the function

\centerline{
$w_4(z) = \dfrac{12 z^4-36 z^2+9}{\sqrt{\pi } \left(4 z^4-12 z^2+3\right) \text{erfi}(z)-2 e^{z^2} z \left(2 z^2-5\right)}$.
}

\begin{enumerate}[label=\alph*),leftmargin=*]
\item
It has the asymptotic values 
$\{ {\tt a}_1, {\tt a}_2, {\tt a}_1, {\tt a}_3 \} =
\left\{0,-\frac{3 i}{\sqrt{\pi }},0,\frac{3 i}{\sqrt{\pi }}\right\}$ as its only singular values.
The Schwarzian derivative is $Sw\{ w_4, z \} = -(z^2 + 9)$,
thus it is an $N$--function with ${\tt q}=3$.

\item 
The Riemann surface $\R_{w_4(z)}$ has 4 infinitely ramified branch points, namely
$\circled{1}=(\infty_1, 0, \infty)$,
$\circled{2}=(\infty_2, -\frac{3i}{\sqrt{pi}}, \infty)$,
$\circled{3}=(\infty_3, 0, \infty)$, and
$\circled{4}=(\infty_4, \frac{3i}{\sqrt{pi}}, \infty)$.

\noindent
There are 6 different types of diagonals:

\centerline{
$\Delta_{\vartheta {\tt a}_1 {\tt a}_2}$,
$\Delta_{\vartheta {\tt a}_2 {\tt a}_1}$,
$\Delta_{\vartheta {\tt a}_1 {\tt a}_3}$,
$\Delta_{\vartheta {\tt a}_3 {\tt a}_1}$,
$\Delta_{\vartheta {\tt a}_2 {\tt a}_3}$, and
$\Delta_{\vartheta {\tt a}_3 {\tt a}_2}$.
}

\noindent
The actual sheets that appear in $\R_{w_4(z)}$ are:

\begin{align*}
\mathfrak{L}_{1,\vartheta_1} &= \CW\backslash \big( \pi_2 (\Delta_{\vartheta_1 {\tt a}_1 {\tt a}_2}  ) \big)
= \big( \CW_w \backslash  \overline{{\tt a}_1 {\tt a}_2} \big)_{\vartheta_1},
\\
\mathfrak{L}_{2,\vartheta_2} &= \CW\backslash \big( \pi_2 (\Delta_{\vartheta_2 {\tt a}_2 {\tt a}_1} ) \big)
= \big( \CW_w \backslash  \overline{{\tt a}_2 {\tt a}_1} \big)_{\vartheta_2},
\\
\mathfrak{L}_{3,\vartheta_3} &= \CW\backslash \big( \pi_2 (\Delta_{\vartheta_3 {\tt a}_1 {\tt a}_3} ) \big)
= \big( \CW_w \backslash  \overline{{\tt a}_1 {\tt a}_3} \big)_{\vartheta_3},
\\
\mathfrak{L}_{4,\vartheta_4} &= \CW\backslash \big( \pi_2 (\Delta_{\vartheta_4 {\tt a}_3 {\tt a}_1} ) \big)
= \big( \CW_w \backslash  \overline{{\tt a}_3 {\tt a}_1} \big)_{\vartheta_4}, 
\\
\mathfrak{L}_{5,\vartheta_5} &= \CW\backslash \big( 
\pi_2 (\Delta_{\vartheta_5 {\tt a}_2 {\tt a}_3} ) \cup
\pi_2 (\Delta_{\vartheta_5 {\tt a}_3 {\tt a}_2} )
\big) 
\\
&= \Big( \CW_w \backslash \big( 
\overline{{\tt a}_2 {\tt a}_1 {\tt a}_3} \cup
\overline{{\tt a}_3 {\tt a}_1 {\tt a}_2}
\big)
\Big)_{\vartheta_5},
\\
\mathfrak{L}_{6,\vartheta_6} &= \CW\backslash \big( 
\pi_2 (\Delta_{\vartheta_6 {\tt a}_1 {\tt a}_2} ) \cup 
\pi_2 (\Delta_{\vartheta_6 {\tt a}_2 {\tt a}_3} ) \cup
\pi_2 (\Delta_{\vartheta_6 {\tt a}_3 {\tt a}_1} ) 
\big)
\\
&= \Big( \CW_w \backslash  \big(
\overline{{\tt a}_1 {\tt a}_2} \cup
\overline{{\tt a}_2 {\tt a}_1 {\tt a}_3} \cup
\overline{{\tt a}_3 {\tt a}_1} 
\big)
\Big)_{\vartheta_6},  \text{ with } \vartheta_6=1, \text{ and }
\\
\mathfrak{L}_{7,\vartheta_7} &= \CW\backslash \big( 
\pi_2 (\Delta_{\vartheta_7 {\tt a}_2 {\tt a}_1} ) \cup 
\pi_2 (\Delta_{\vartheta_7 {\tt a}_1 {\tt a}_3} ) \cup
\pi_2 (\Delta_{\vartheta_7 {\tt a}_3 {\tt a}_2} ) 
\big)
\\
&= \Big( \CW_w \backslash  \big(
\overline{{\tt a}_2 {\tt a}_1} \cup
\overline{{\tt a}_1 {\tt a}_3} \cup
\overline{{\tt a}_3 {\tt a}_1 {\tt a}_2 } 
\big)
\Big)_{\vartheta_7}, \text{ with } \vartheta_7=1.
\end{align*}

\noindent
The decomposition of $\R_{w_4(z)}$ into the maximal domains of single--valuedness is
\begin{equation*}
\R_{w_4(z)} = \Bigg(
\mathfrak{L}_{6,1} \cup 
\mathfrak{L}_{5,1} \cup
\mathfrak{L}_{5,2} \cup
\mathfrak{L}_{5,3} \cup
\mathfrak{L}_{7,1} 
\ \cup
\bigcup_{\vartheta_1= 1}^\infty 
\mathfrak{L}_{1, \vartheta_1}
\cup
\bigcup_{\vartheta_2= 1}^\infty 
\mathfrak{L}_{2, \vartheta_2 }
\cup
\bigcup_{\vartheta_3= 1}^\infty 
\mathfrak{L}_{3, \vartheta_3 }
\cup
\bigcup_{\vartheta_4= 1}^\infty 
\mathfrak{L}_{4, \vartheta_4 }
\Bigg) \, \Big/ \sim \, .
\end{equation*}

On the other hand, considering the cyclic order 

\centerline{$\mathcal{W}_3 = [ {\tt a}_2, {\tt a}_1, {\tt a}_3 ] =
[-\frac{3 i}{\sqrt{\pi }},0,\frac{3 i}{\sqrt{\pi }} ]$,}

\noindent  
the decomposition of $\R_{w(z)}$ into ${\tt p}=4$ maximal logarithmic towers 
and the unique soul is
\begin{multline*}
\R_{w_4(z)} = \Big[
\underbrace{
\mathfrak{H}^-\backslash ( \overline{{\tt a}_1 {\tt a}_3} \cup  \overline{{\tt a}_2 {\tt a}_1}
\cup  \overline{{\tt a}_3 {\tt a}_2} ) 
\cup
\bigcup_{\vartheta=1}^4 \big( 
\mathfrak{H}^+\backslash \overline{{\tt a}_3 {\tt a}_2} 
\cup
\mathfrak{H}^-\backslash \overline{{\tt a}_3 {\tt a}_2} 
\big)_\vartheta
\cup 
\mathfrak{H}^+\backslash ( \overline{{\tt a}_1 {\tt a}_3} \cup  \overline{{\tt a}_2 {\tt a}_1}
\cup  \overline{{\tt a}_3 {\tt a}_2} )
}_{\text{soul}}
\\ 
\cup
\underbrace{ \mathcal{T}^\circ_- ({\tt a}_1, {\tt a}_3) }_{\text{logarithmic tower}}
\cup
\underbrace{ \mathcal{T}^\circ_- ({\tt a}_2, {\tt a}_1) }_{\text{logarithmic tower}}
\cup
\underbrace{ \mathcal{T}^\times_+ ({\tt a}_2, {\tt a}_1) }_{\text{logarithmic tower}} 
\cup
\underbrace{ \mathcal{T}^\times_+ ({\tt a}_1, {\tt a}_3) }_{\text{logarithmic tower}}
\Big] \, \Big/ \sim.
\end{multline*}

\item
For the tessellation, with the cyclic order $\mathcal{W}_3$, it follows that
$\gamma=i\RR\cup\{\infty\}$ and the Speiser $3$--tessellation 
$\big(\CC_z \backslash w_4(z)^*\gamma, w_4(z)^*\mathcal{L}_\gamma \big)$ is
shown in Figure \ref{fig:Tessellation-p=4}.a--b.
The tiles are $3$--gons, with the vertices of valence two (the cosingular points) 
represented by green dots, only a finite number of them are shown. 
The four vertices of infinite valence of 
the graph $w_4(z)^*\gamma$
are the points in the non Hausdorff compactification 

\centerline{$\CC_z \cup \{ 
\infty_{1}, \, 
\infty_{2}, \, 
\infty_{3}, \, 
\infty_{4}
\}$}

\noindent 
determined by the four asymptotic values (with multiplicity)
$\left\{  0, \,  -\frac{3 i}{\sqrt{\pi }}, 
\,  0, \, \frac{3 i}{\sqrt{\pi }} \right\}$.

\item
Its analytic Speiser graph $\mathfrak{S}_{w_4(z)}$ of index $3$ is
drawn in Figure \ref{fig:Tessellation-p=4}.c. 
Note that the nucleus consists of ten vertices and four ``loose'' edges (in red), 
surrounded by four logarithmic ends (in black).

\begin{figure}[h!tbp]
\begin{center}
{ a)}
\includegraphics[width=0.25\textwidth]{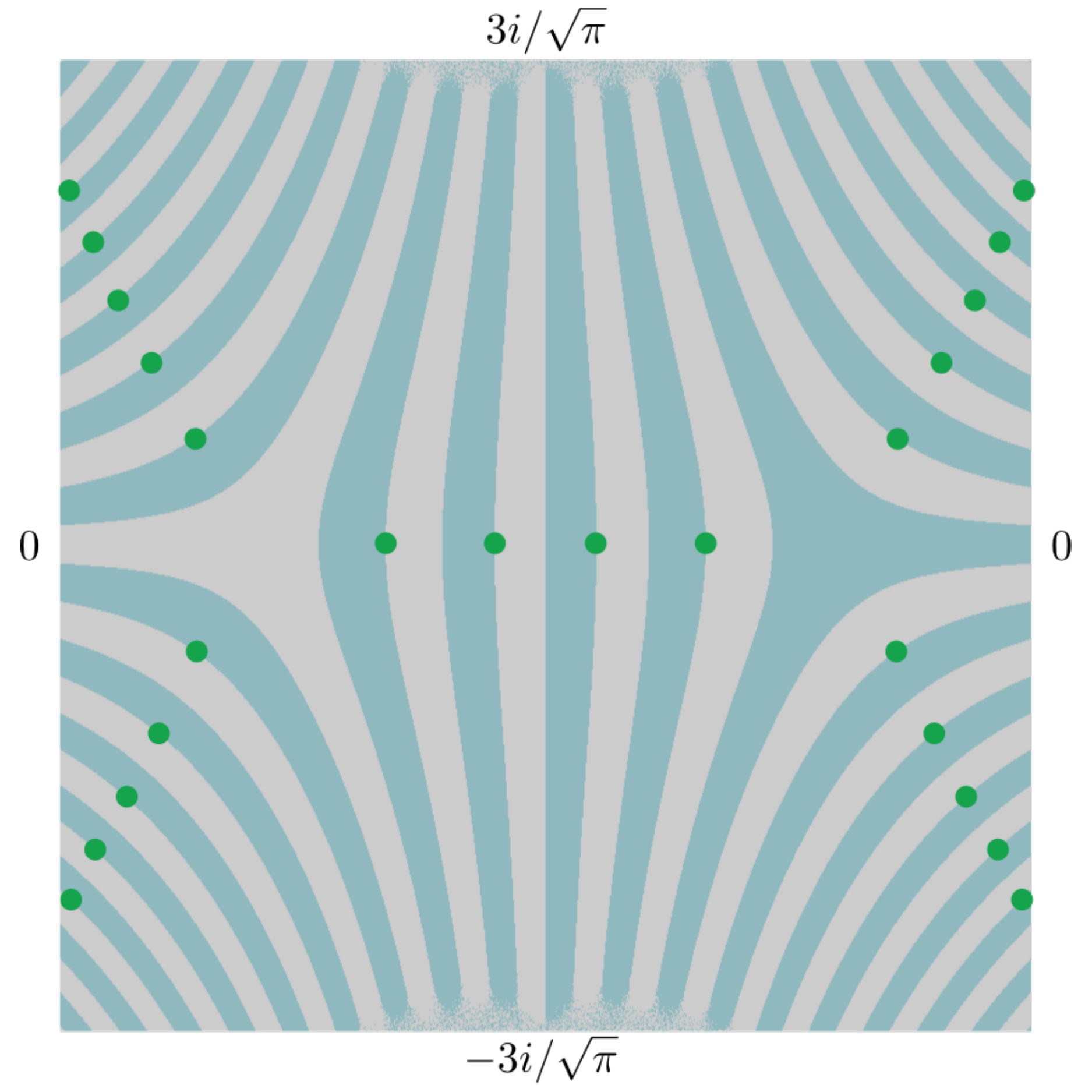}
\hspace{5 pt}
{ c)}
\includegraphics[width=0.30\textwidth]{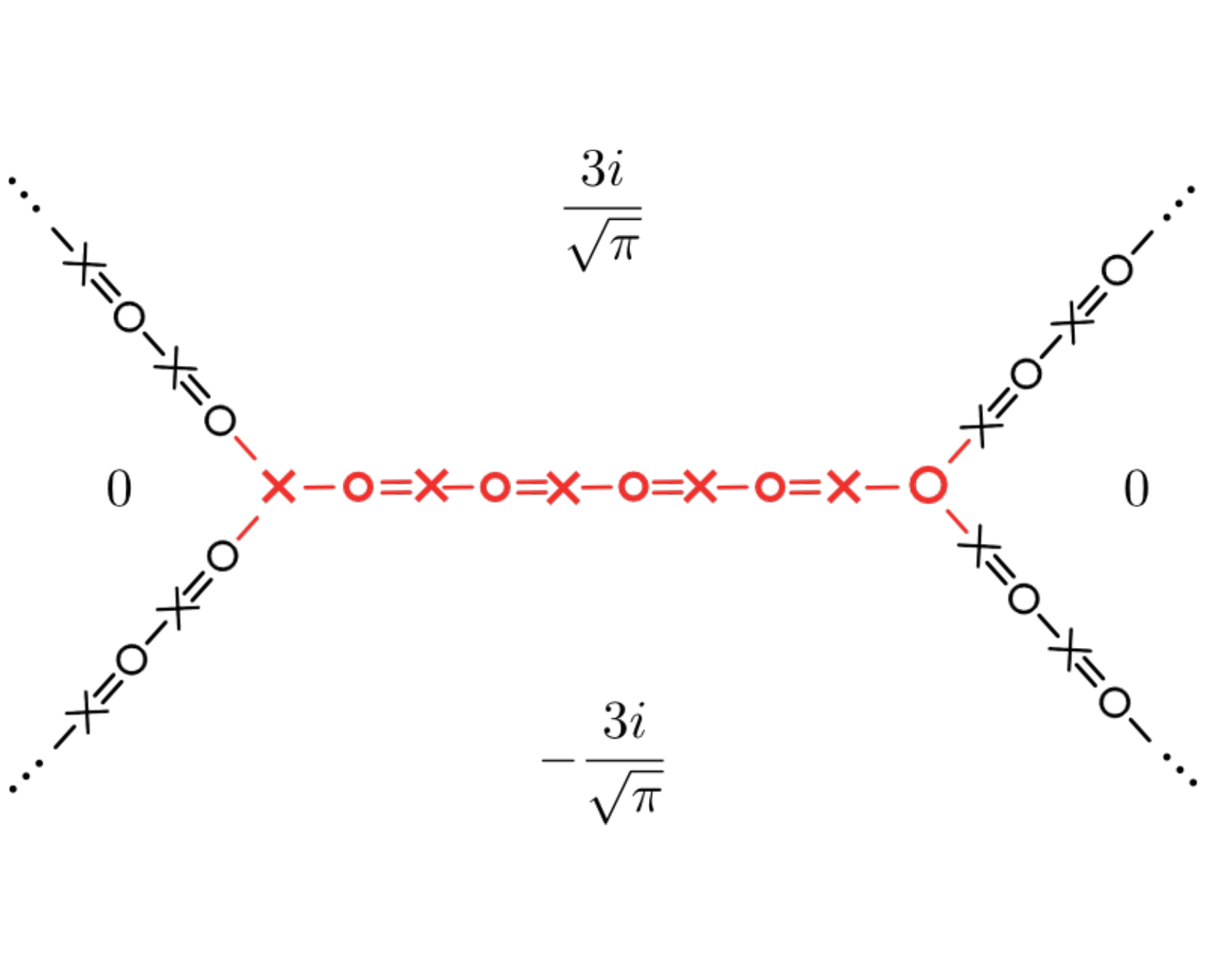}
\\
{ b)}
\includegraphics[width=0.28\textwidth]{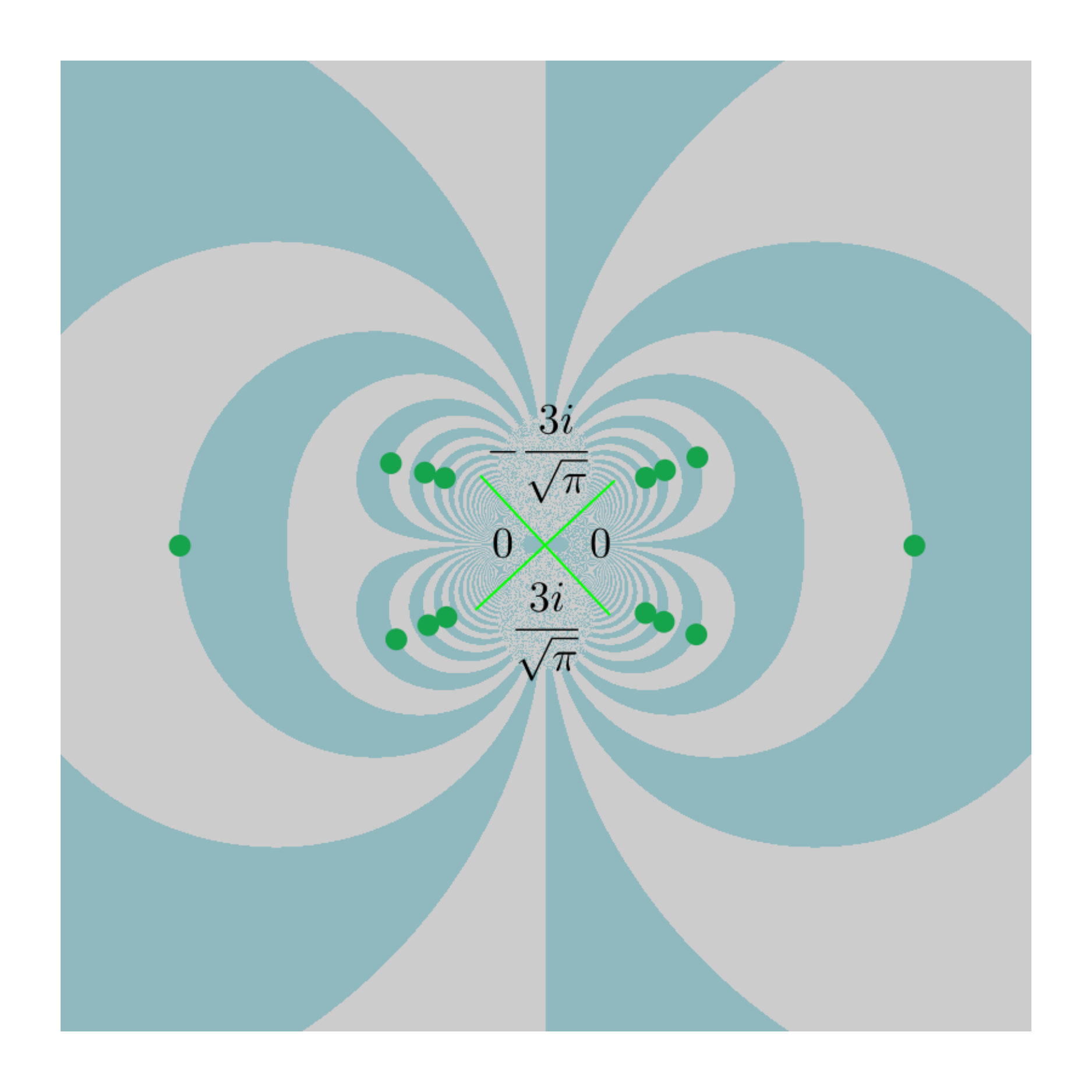}
\caption{
The tessellation $\big( \CC_z \backslash w_4(z)^*\gamma, 
w_4(z)^*\mathcal{L}_\gamma \big)$ 
corresponding to $\gamma = i\RR\cup\{\infty\}$ and the 
$N$--function $w_4(z)$, the green dots indicate the vertices of 
valence 2;
(a) near the origin, (b) near the essential singularity
at $\infty\in\CW_z$; in (b)
the green lines at the center of the drawing have been added to indicate 
the location of the four vertices of infinite valence.
(c) The corresponding Speiser graph of index ${\tt p}= 4$.
The nucleus is colored red and the $4$
logarithmic ends are colored black.
}
\label{fig:Tessellation-p=4}
\end{center}
\end{figure}
\end{enumerate}
\end{example}

\begin{example}
\label{example:exp-exp}
Consider the function 

\centerline{$w(z)=\exp(\exp(z))$.}

\begin{enumerate}[label=\alph*),leftmargin=*]
\item 
The singular values are $\mathcal{SV}_w = \{ {\tt a}_1, {\tt a}_2, {\tt a}_3 \} =\{0,1,\infty\}$; 
note that all of them are asymptotic values.
Thus it is a Speiser function with $\tt q = 3$.
It is not a finite Speiser function since it has an infinite number of logarithmic singularities of $w^{-1}(z)$:

\noindent
$\bigcdot$
A singularity $ U_1$ over the asymptotic value $1$, 
with asymptotic path $\alpha_1(\tau) = -\tau$, for $\tau\in(0,\infty)$. 

\noindent
$\bigcdot$
An infinite number of singularities, 
$\{ U_{0, \sigma} \}_{\sigma \in \ZZ}$, over the asymptotic value $0$, 
with asymptotic paths $\alpha_{0, \sigma}(\tau) = (2\sigma+1) \pi i + \tau$, for $\tau\in(0,\infty)$. 

\noindent
$\bigcdot$
An infinite number of singularities, 
$\{ U_{\infty, \sigma} \}_{\sigma \in \ZZ}$, over the asymptotic value $\infty$,
with asymptotic paths $\alpha_{\infty, \sigma}(\tau) = 2\sigma \pi i + \tau$, for $\tau\in(0,\infty)$.

\item
The Riemann surface $\R_{w(z)}$
has an infinite number of infinitely ramified branch points, namely

\centerline{
$\circled{1}=(\infty_{\widehat{1}}, 1,\infty)$, 
$\Big\{\circled{0}_\sigma = (\infty_{2\sigma+1}, 0, \infty) \Big\}_{\sigma \in \ZZ}$, and
$\Big\{\circled{\infty}_\sigma = (\infty_{2\sigma}, \infty, \infty) \Big\}_{\sigma \in \ZZ}$.}

\noindent
The actual sheets that appear in $\R_{w(z)}$ are:
\begin{align*}
\mathfrak{L}_{1,\vartheta_1} &= \CW\backslash \big( \pi_2 (\Delta_{\vartheta_1 {\tt a}_1 {\tt a}_2}  )
\cup \pi_2 (\Delta_{\vartheta_1 {\tt a}_3 {\tt a}_1}  )
 \big)
= \Big( \CW_w \backslash \big( \overline{01} \cup \overline{ \infty 0} \big) \Big)_{\vartheta_1},
\\
\mathfrak{L}_{2,\vartheta_2} &= \CW\backslash \big( \pi_2 (\Delta_{\vartheta_2 {\tt a}_3 {\tt a}_1} ) \big)
= \big( \CW_w \backslash  \overline{\infty 0} \big)_{\vartheta_2}.
\end{align*}

\noindent
The decomposition of $\R_{w(z)}$ into maximal domains of single--valuedness is
\begin{equation*}
\R_{w(z)} = 
\Bigg[
\bigcup_{\vartheta_1 \in\ZZ} \mathfrak{L}_{1,\vartheta_1}
\ \cup
\bigcup_{\vartheta_3 \in\ZZ}
\bigg(
\bigcup_{\vartheta_2= 1}^\infty 
\mathfrak{L}_{2, \vartheta_2}
\bigg)_{\vartheta_3}
\Bigg]
\, \Bigg/ \sim \, .
\end{equation*}

Considering the cyclic order 

\centerline{$\mathcal{W}_3=[0,1,\infty]$,}

\noindent 
the decomposition
of $\R_{w(z)}$ into an infinite number of maximal logarithmic towers 
and the unique soul is
\begin{equation*}
\R_{w(z)} = \bigg[
\underbrace{
\bigcup_{\vartheta=-\infty}^\infty \Big( 
\mathfrak{H}^+\backslash ( \overline{0 1} \cup  \overline{1 \infty}
\cup  \overline{\infty 0} ) 
\cup
\mathfrak{H}^-\backslash  ( \overline{0 1} \cup  \overline{1 \infty}
\cup  \overline{\infty 0} ) 
\Big)_\vartheta
}_{\text{soul}}
\quad
\bigcup_{\vartheta=-\infty}^\infty \Big( 
\underbrace{ 
\mathcal{T}^\circ_- (\infty, 0) 
}_{\text{logarithmic tower}}
\cup
\underbrace{ 
\mathcal{T}^\times_+ (\infty, 0) 
}_{\text{logarithmic tower}}
\Big)_\vartheta
\bigg] \, \Big/ \sim.
\end{equation*}

\noindent
Clearly, the soul is not a rational block.

\item
For the tessellation, with the cyclic order $\mathcal{W}_3$, it follows that
$\gamma=\RR\cup\{\infty\}$ and the Speiser $3$--tessellation is 
$\big( \mathscr{T}(w(z)^* \gamma),w(z)^* \mathcal{L}_\gamma \big)$, as shown in 
Figure \ref{fig:mosaico-exp-exp}.a. 
The tiles are topological $3$--gons, 
with vertices of valence two represented by 
green dots (the cosingular points, with singular values ${\tt a}_2=1$), 
only two ``columns'' are drawn, however there are an infinite number of them. 
There are an infinite number of vertices of infinite valence of 
the graph $w(z)^*\gamma$, these
are points in the non Hausdorff compactification 

\centerline{$\CC_z \cup \{ 
\infty_{\widehat{1}}, \} \cup 
\{  \infty_{2\sigma}, \infty_{2\sigma+1}  \}_{\sigma\in\ZZ}$}

\noindent 
determined by the asymptotic values.

\item
Its analytic Speiser graph of index $3$ is drawn in Figure \ref{fig:mosaico-exp-exp}.b.
In this case the Speiser graph has 
an unbounded 1--face corresponding to the logarithmic singularity over the asymptotic value 1, 
an infinite number of unbounded 0--faces 
and $\infty$--faces, which correspond to logarithmic singularities over the 
asymptotic values 0 and $\infty$, respectively. 
It also has an infinite number of digons corresponding to the coasymptotic value 1.
Note that the nucleus consists of an infinite number of vertices and an infinite number of ``loose'' edges (in red), 
surrounded by an infinite number of logarithmic ends (in black).

\end{enumerate}
\begin{figure}[h!tbp]
\begin{center}
{a)} \includegraphics[width=0.3\textwidth]{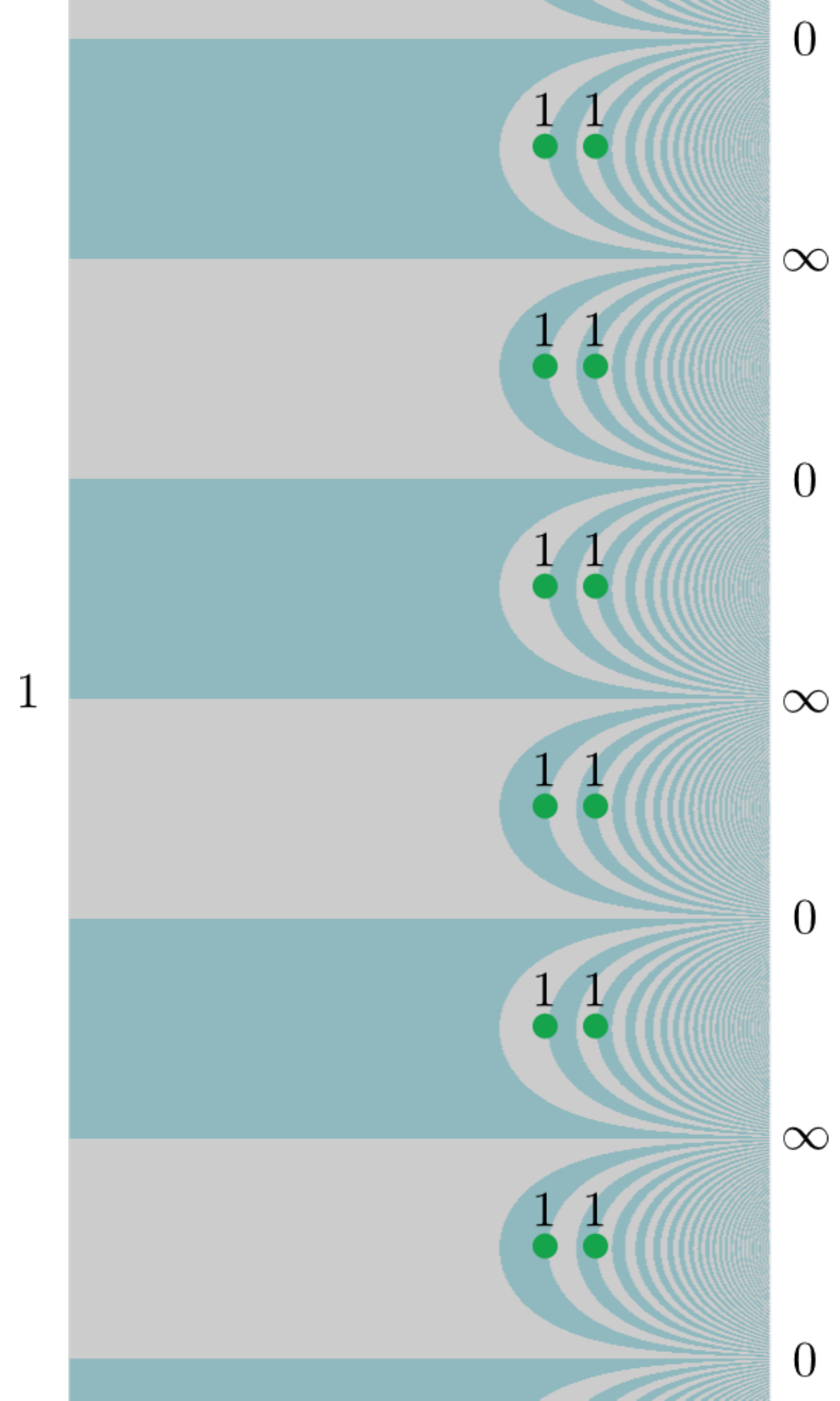}
\hspace{10 pt}
{b)} \includegraphics[width=0.3\textwidth]{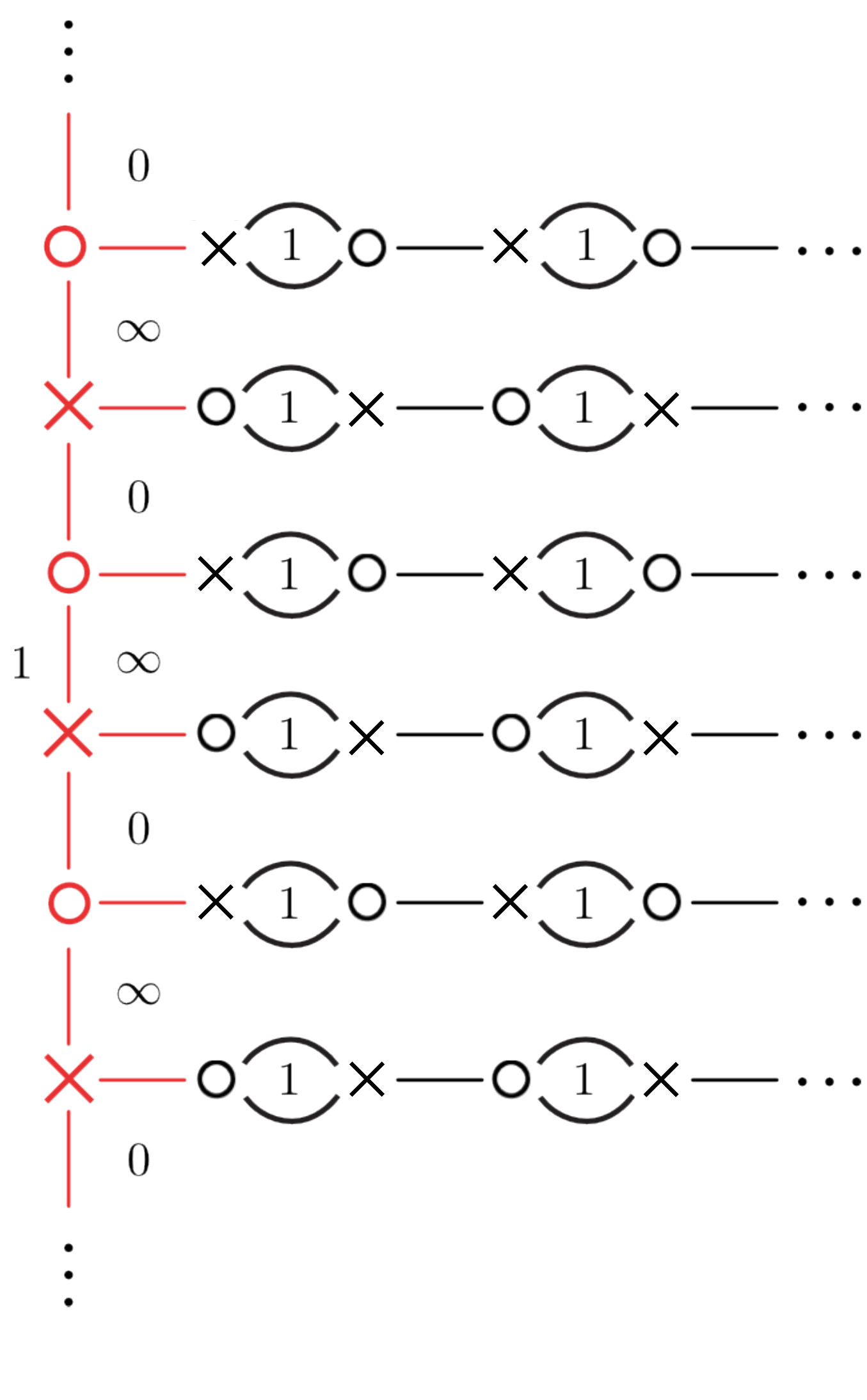}
\caption{
Consider the Speiser function $w(z)=\exp(\exp(z))$ whose singular values are 
$\{0,1,\infty\}$, note that all of them are asymptotic values.
a) Speiser 3--tessellation $\mathscr{T}_\gamma(w(z))$, here 
$\gamma = \RR \cup \{\infty\}$. The green dots are the cosingular points
labelled with the corresponding cosingular value `1'; these
continue indefinitely to the right 
as suggested.
b) Speiser graph $\mathfrak{S}_{w(z)}$ of index ${\tt q}=3$; the digons are labelled `1' since
they are the dual of the cosingular points.
The nucleus is colored red and the 
logarithmic ends are colored black.
}
\label{fig:mosaico-exp-exp}
\end{center}
\end{figure}
\end{example}

\begin{example}
\label{example:tessellation-expsin}
Consider the function 

\centerline{
$w(z)=\e^{\sin(z)}$.
}

\begin{enumerate}[label=\alph*),leftmargin=*]
\item
The singular values are 
$\mathcal{SV}_w=\{{\tt w}_1, {\tt w}_2, {\tt a}_1, {\tt a}_2 \}=\{\e, \e^{-1},0,\infty\}$.
Thus it is a Speiser function with $\tt q = 4$.
It is not a finite Speiser function since:

\noindent
$\bigcdot$ 
It has an infinite number of critical points
$\mathcal{CP}_w \doteq \{z_k = \frac{2k+1}{2}\pi \ \vert \ k\in\ZZ\}$ 
corresponding to the critical values $\mathcal{CV}_w \doteq \{\e, \e^{-1}\}$.

\noindent
$\bigcdot$
It has an infinite number of logarithmic singularities 
$\{ U_{0, \sigma \pm} \}_{\sigma \in \ZZ}$ over the asymptotic value $0$, and
an infinite number of logarithmic singularities 
$\{ U_{\infty, \sigma \pm} \}_{\sigma \in \ZZ}$ over the asymptotic value $\infty$.
The asymptotic paths are 
$\alpha_{{\tt a}_{\sigma \pm}}( \tau )=(2\sigma+1)\frac{\pi}{2} \pm i \tau $, 
for $\sigma\in\ZZ$, $\tau\in (0,\infty)$ associated to the asymptotic values

\centerline{
${\tt a}_{\sigma\pm}=\begin{cases}
  0_{\sigma\pm}=0, \text{ for odd }\sigma,\\
  \infty_{\sigma\pm}=\infty, \text{ for even }\sigma.
\end{cases}$
}

\item The Riemann surface has an infinite number of finitely ramified branch points

\centerline{
$\circled{z_k}= \begin{cases} (z_k, \e^{-1}, 2), \text{ for odd } k,\\
(z_k, \e, 2), \text{ for even } k, \end{cases}$ $k\in \ZZ,$
}

\noindent
and an infinite number of infinitely ramified branch points 

\centerline{
$\circled{\sigma\pm }= \begin{cases} (\infty_{\sigma\pm}, 0, \infty), \text{ for odd } \sigma,\\
(\infty_{\sigma\pm}, \infty, \infty), \text{ for even } \sigma, \end{cases}$ $\sigma \in \ZZ.$
}

\noindent
The actual sheets that appear in $\R_{w(z)}$ are:
\begin{align*}
\mathfrak{L}_{1,\vartheta_1} &= \CW\backslash \big( \pi_2 (\Delta_{\vartheta_1 {\tt a}_2 {\tt a}_1}  ) \big)
= \big( \CW_w \backslash  \overline{\infty 0} \big)_{\vartheta_1},
\\
\mathfrak{L}_{2,\vartheta_2} &= \CW\backslash \big( \pi_2 (\Delta_{\vartheta_2 {\tt a}_2 {\tt a}_1} ) \big)
= \big( \CW_w \backslash  \overline{\infty 0} \big)_{\vartheta_2},
\\
\mathfrak{L}_{3,\vartheta_3} &= \CW\backslash \big( 
\pi_2 (\Delta_{\vartheta_3 {\tt a}_1 {\tt w}_1} ) \cup 
\pi_2 (\Delta_{\vartheta_3 {\tt w}_2 {\tt a}_2} ) \cup
\pi_2 (\Delta_{\vartheta_3 {\tt a}_2 {\tt a}_1} ) 
\big)
\\
&= \Big( \CW_w \backslash  \big(
\overline{0 \e^{-1} } \cup
\overline{\e \infty } \cup
\overline{\infty 0 } 
\big)
\Big)_{\vartheta_3}.
\end{align*}

\noindent
The decomposition of $\R_{w(z)}$ into maximal domains of single--valuedness is
\begin{equation*}
\R_{w(z)} = 
\Bigg[
\bigcup_{\vartheta_3 \in\ZZ}
\mathfrak{L}_{3,\vartheta_3}
\ \cup
\bigcup_{\vartheta_3 \in\ZZ}
\bigg(
\bigcup_{\vartheta_1= 1}^\infty 
\mathfrak{L}_{1, \vartheta_1}
\cup
\bigcup_{\vartheta_2= 1}^\infty 
\mathfrak{L}_{2, \vartheta_2 }
\bigg)_{\vartheta_3}
\Bigg]
\, \Big/ \sim \, .
\end{equation*}

Considering the cyclic order 

\centerline{
$\mathcal{W}_4=[{\tt a}_1, {\tt w}_2, {\tt w}_1, {\tt a}_2]=[0, \e^{-1},\e,\infty ]$,}

\noindent
the decomposition of $\R_{w(z)}$ into an infinite number of 
maximal logarithmic towers and the unique soul is
\begin{multline*}
\R_{w(z)} = 
\bigg[
\underbrace{
\bigcup_{\vartheta=-\infty}^\infty \Big( 
\mathfrak{H}^+\backslash ( \overline{\infty 0} \cup  \overline{\e \infty}
\cup \overline{\e^{-1},\e} \cup \overline{0 \e^{-1}} ) 
\cup
\mathfrak{H}^-\backslash ( \overline{\infty 0} \cup  \overline{\e \infty}
\cup \overline{\e^{-1},\e} \cup \overline{0 \e^{-1}} ) 
\Big)_\vartheta
}_{\text{soul}}
\\ 
\bigcup_{\vartheta=-\infty}^\infty \Big( 
\underbrace{ 
\mathcal{T}^\circ_- (\infty, 0) 
}_{\text{logarithmic tower}}
\cup
\underbrace{ 
\mathcal{T}^\times_+ (\infty, 0) 
}_{\text{logarithmic tower}}
\Big)_\vartheta
\bigg] 
\, \Big/ \sim.
\end{multline*}
\begin{figure}[h!]
\begin{center}
{ a)}\includegraphics[width=0.4\textwidth]{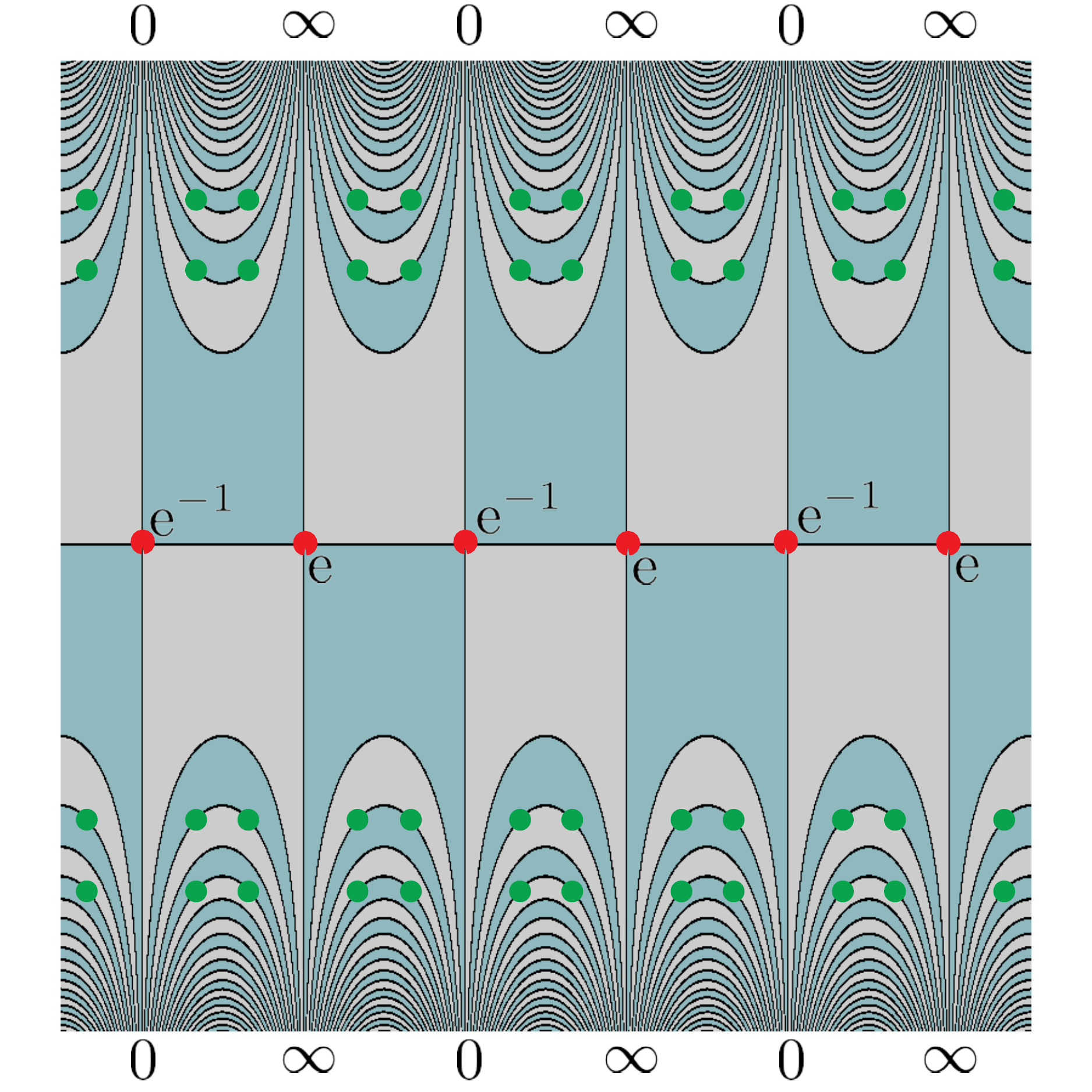}
{ b)}\includegraphics[width=0.5\textwidth]{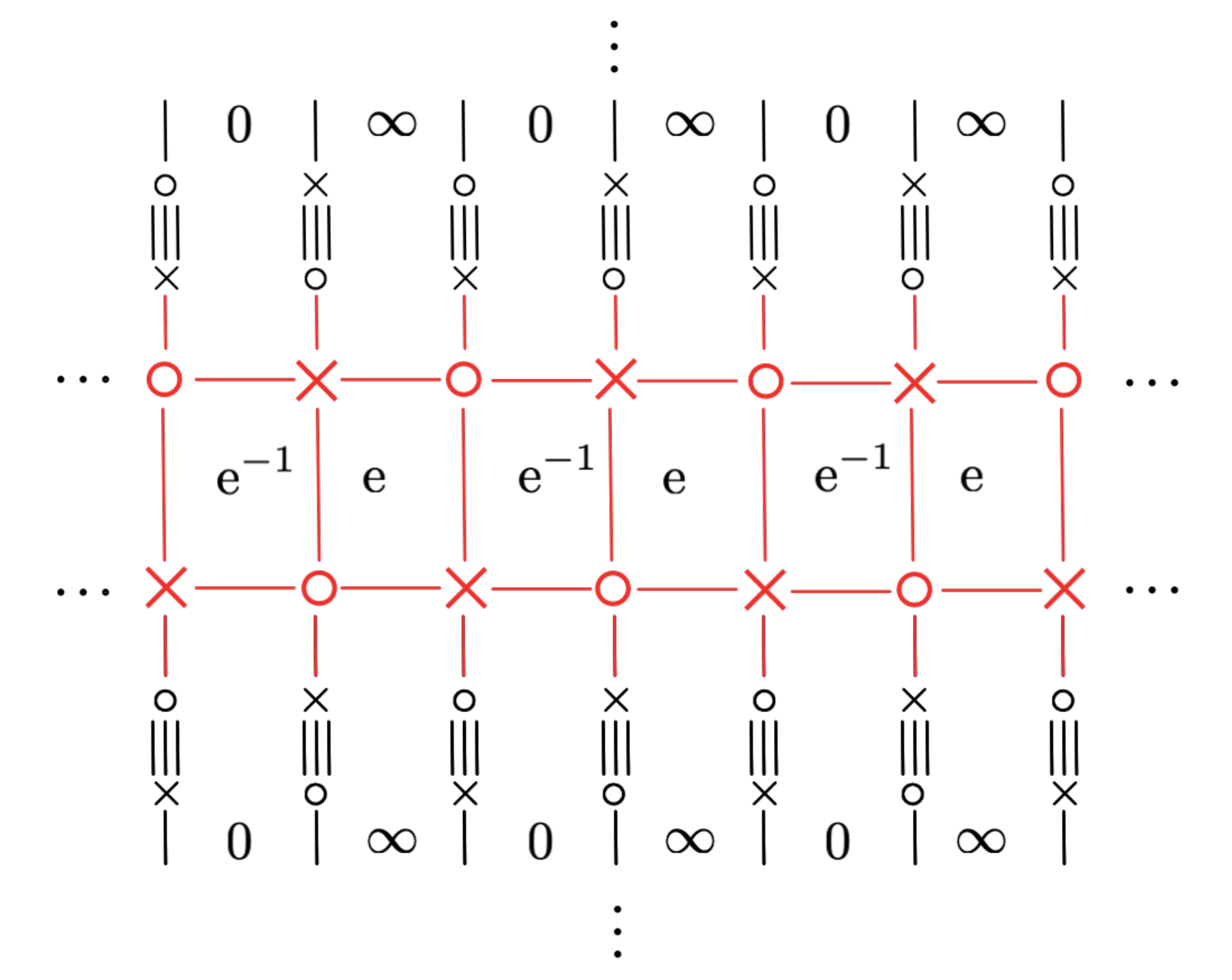}
\caption{
Speiser function $w(z)= \e^{\sin(z)}$.
(a) Speiser 4--tessellation $\big( \mathscr{T}(w(z)^*\gamma), 
w(z)^*\mathcal{L}_\gamma \big)$ corresponding to 
$\gamma=\RR\cup\{\infty\}$ and $w(z)$, 
the green dots indicate the vertices of valence 2.
(b) The corresponding Speiser graph of index ${\tt q}= 4$.  
The nucleus is colored red and the 
logarithmic ends are colored black.
}
\label{figExpSin}
\end{center}
\end{figure}

\item
For the tessellation with the cyclic order 
$\mathcal{W}_4$, it follows that
$\gamma=\RR\cup\{\infty\}$ and its Speiser $4$--tessellation is
$\big( \mathscr{T}(w(z)^* \gamma),w(z)^* \mathcal{L}_\gamma \big)$ which is shown in 
Figure \ref{figExpSin}.a. 
As can be appreciated, the tiles are topological $4$--gons, 
with vertices of valence two represented by 
green dots (the cosingular points, with singular values ${\tt w}_1=\e$ and ${\tt w}_2=\e^{-1}$), 
only four ``rows'' are drawn, however there are an infinite number of them. 
There are an infinite number of vertices of infinite valence of 
the graph $w(z)^*\gamma$. 
These are the ideal points in the non Hausdorff compactification 

\centerline{$\CC_z \cup \{ 
\infty_{{\tt a}_{1\pm}}, \, 
\infty_{{\tt a}_{2\pm}}, \, 
\ldots, 
\infty_{{\tt a}_{\sigma\pm}},\,
\ldots
\}$}

\noindent 
determined by the asymptotic values ${\tt a}_{\sigma\pm}=0$ and 
${\tt a}_{\sigma\pm}=\infty$, for odd or even $\sigma$ respectively.

\item
Its analytic Speiser graph of index $4$ is drawn in Figure \ref{figExpSin}.b.
In this case the Speiser graph has an infinite number of unbounded 0--faces 
and $\infty$--faces, which correspond to logarithmic singularities over the 
asymptotic values 0 and $\infty$, respectively. 
Moreover, it also has an infinite number of bounded $\rm e$--faces and ${\rm e}^{-1}$--faces. 
The bounded faces are $4$--gons, thus they correspond to 
finitely ramified branch points of ramification index 2.
Note that the nucleus consists of an infinite number of squares arranged in a horizontal line 
and an infinite number of ``loose'' edges (in red), 
surrounded by an infinite number of logarithmic ends (in black).

\end{enumerate}
\end{example}

\begin{example}
\label{example:sinexpsin}
Consider the function 

\centerline{$w(z)= \sin (z) \exp (\sin (z))$.}

\begin{enumerate}[label=\alph*),leftmargin=*]
\item
The singular values are 
$\mathcal{SV}_w=\{ {\tt w_1}, {\tt a}_1, {\tt w}_2, {\tt a}_2 \} = \{ -\e^{-1}, 0, \e, \infty \}$,
where $\{ 0, \infty \}$ are asymptotic values and $\{ -\e^{-1}, \e \}$ are critical values.
It is thus a Speiser function, but not a finite Speiser function,
since it has an infinite number 
of singularities of $w^{-1}(z)$:

\noindent
$\bigcdot$
It has an infinite number of critical points
$\mathcal{CP}_w \doteq \{z_k = \frac{2k+1}{2}\pi \ \vert \ k\in\ZZ\}$ 
corresponding to the critical values $\mathcal{CV}_w \doteq \{ {\tt w}_1, {\tt w}_2 \}$;
for odd $\sigma$ the critical values are ${\tt w}_1=-\e^{-1}$ with ramification index 4, 
and for even $\sigma$ the critical values are ${\tt w}_2=\e$ with ramification index 2.

\noindent
$\bigcdot$
It has an infinite number of logarithmic singularities 
$\{ U_{0, \sigma \pm} \}_{\sigma \in \ZZ}$ over the asymptotic value $0$, and
an infinite number of logarithmic singularities 
$\{ U_{\infty, \sigma \pm} \}_{\sigma \in \ZZ}$ over the asymptotic value $\infty$.
The asymptotic paths are 
$\alpha_{{\tt a}_{\sigma \pm}}( \tau )=(2\sigma+1)\frac{\pi}{2} \pm i \tau $, 
for $\sigma\in\ZZ$, $\tau\in (0,\infty)$ associated to the asymptotic values

\centerline{
${\tt a}_{\sigma\pm}=\begin{cases}
  0_{\sigma\pm}=0, \text{ for odd }\sigma,\\
  \infty_{\sigma\pm}=\infty, \text{ for even }\sigma.
\end{cases}$
}

\noindent
Moreover the points $\{ k \pi \}_{k\in\ZZ}$ are cocritical points with cocritical value 0. 
There are many more cocritical as will shortly be seen.

\item
The Riemann surface has an infinite number of finitely ramified branch points

\centerline{
$\circled{z_k}= \begin{cases} (z_k, -\e^{-1}, 4), \text{ for odd } k,\\
(z_k, \e, 2), \text{ for even } k, \end{cases}$ $k\in \ZZ,$
}

\noindent
and an infinite number of infinitely ramified branch points 

\centerline{
$\circled{\sigma\pm }= \begin{cases} (\infty_{\sigma\pm}, 0, \infty), \text{ for odd } \sigma,\\
(\infty_{\sigma\pm}, \infty, \infty), \text{ for even } \sigma, \end{cases}$ $\sigma \in \ZZ.$
}

\noindent
The actual sheets that appear in $\R_{w(z)}$ are:
\begin{align*}
\mathfrak{L}_{1,\vartheta_j} &= \CW\backslash \big( \pi_2 (\Delta_{\vartheta_j {\tt a}_1 {\tt a}_2}  ) \big)
= \big( \CW_w \backslash  \overline{0 \infty} \big)_{\vartheta_j}, \quad j=1,2,3,4,
\\
\mathfrak{L}_{2,\vartheta_k} &= \CW\backslash \big( 
\pi_2 (\Delta_{\vartheta_k {\tt w}_1 {\tt a}_1} ) \cup 
\pi_2 (\Delta_{\vartheta_k {\tt a}_1{\tt a}_2} ) \cup
\pi_2 (\Delta_{\vartheta_k {\tt w}_1 {\tt w}_2} ) \cup
\pi_2 (\Delta_{\vartheta_k {\tt w}_2 {\tt a}_2} ) 
\big),
\\
&= \Big( \CW_w \backslash  \big(
\overline{-\e^{-1} 0 } \cup
\overline{0 \e \infty } \cup
\overline{-\e^{-1} 0 \e } \cup 
\overline{\e \infty }
\big)
\Big)_{\vartheta_k} , \quad \vartheta_k=1,2,3,4.
\end{align*}

\noindent
The decomposition of $\R_{w(z)}$ into maximal domains of single--valuedness is 
\begin{equation*}
\R_{w(z)} = 
\Bigg[
\bigcup_{\vartheta_5 \in\ZZ} \bigg( 
\mathfrak{L}_{2,1} \cup \mathfrak{L}_{2,2}
\cup \mathfrak{L}_{2,3} \cup \mathfrak{L}_{2,4}
\cup
\bigcup_{\vartheta_1= 1}^\infty 
\mathfrak{L}_{1, \vartheta_1}
\cup
\bigcup_{\vartheta_2= 1}^\infty 
\mathfrak{L}_{1, \vartheta_2 }
\cup
\bigcup_{\vartheta_3= 1}^\infty 
\mathfrak{L}_{1, \vartheta_3}
\cup
\bigcup_{\vartheta_4= 1}^\infty 
\mathfrak{L}_{1, \vartheta_4 }
\bigg)_{\vartheta_5}
\Bigg]
\, \Bigg/ \sim \, .
\end{equation*}

Considering the cyclic order 

\centerline{$\mathcal{W}_4 = [ -\e^{-1}, 0, \e, \infty ]$,}

\noindent  
the decomposition
of $\R_{w(z)}$ into an infinite number of maximal logarithmic towers 
and the unique soul is
\begin{multline*}
\R_{w(z)} = 
\Bigg[
\underbrace{
\bigcup_{\vartheta_5 \in\ZZ} \big( 
\mathfrak{L}_{2,1} \cup \mathfrak{L}_{2,2}
\cup \mathfrak{L}_{2,3} \cup \mathfrak{L}_{2,4}
\big)_{\vartheta_5}
}_{\text{soul}} 
\cup
\\
\bigcup_{\vartheta_5 \in\ZZ}
\bigg(
\underbrace{
\mathcal{T}^\times_+(0,\infty)
}_{\substack{\text{left upper}\\ \text{logarithmic tower}} }
\cup
\underbrace{
\mathcal{T}^\circ_-(0,\infty)
}_{\substack{\text{left lower}\\ \text{logarithmic tower}} }
\cup
\underbrace{
\mathcal{T}^\circ_-(0,\infty)
}_{\substack{\text{right upper}\\ \text{logarithmic tower}} }
\cup
\underbrace{
\mathcal{T}^\times_+(0,\infty)
}_{\substack{\text{right lower}\\ \text{logarithmic tower}} }
\bigg)_{\vartheta_5}
\Bigg]
\, \Bigg/ \sim \, ,
\end{multline*}

\noindent
where for $\vartheta_k=1,2,3,4$
the sheet $\mathfrak{L}_{2,\vartheta_k}$ decomposes into 
half sheets

\centerline{
$\mathfrak{L}_{2,\vartheta_k} = 
\Big( \mathfrak{H}^+ 
\backslash  \big(
\overline{-\e^{-1} 0 } \cup
\overline{0 \infty } \cup
\overline{\infty -\e^{-1} } 
\big)
\cup
\mathfrak{H}^- 
\backslash  \big(
\overline{-\e^{-1} \e } \cup 
\overline{\e \infty } \cup
\overline{\infty -\e^{-1} }
\big)
\Big)_{\vartheta_k}$,
}

\noindent
glued along the common boundary 
$\overline{\infty -\e^{-1} }$.

\begin{figure}[h!tbp]
\begin{center}
{ a)}\includegraphics[width=0.8\textwidth]{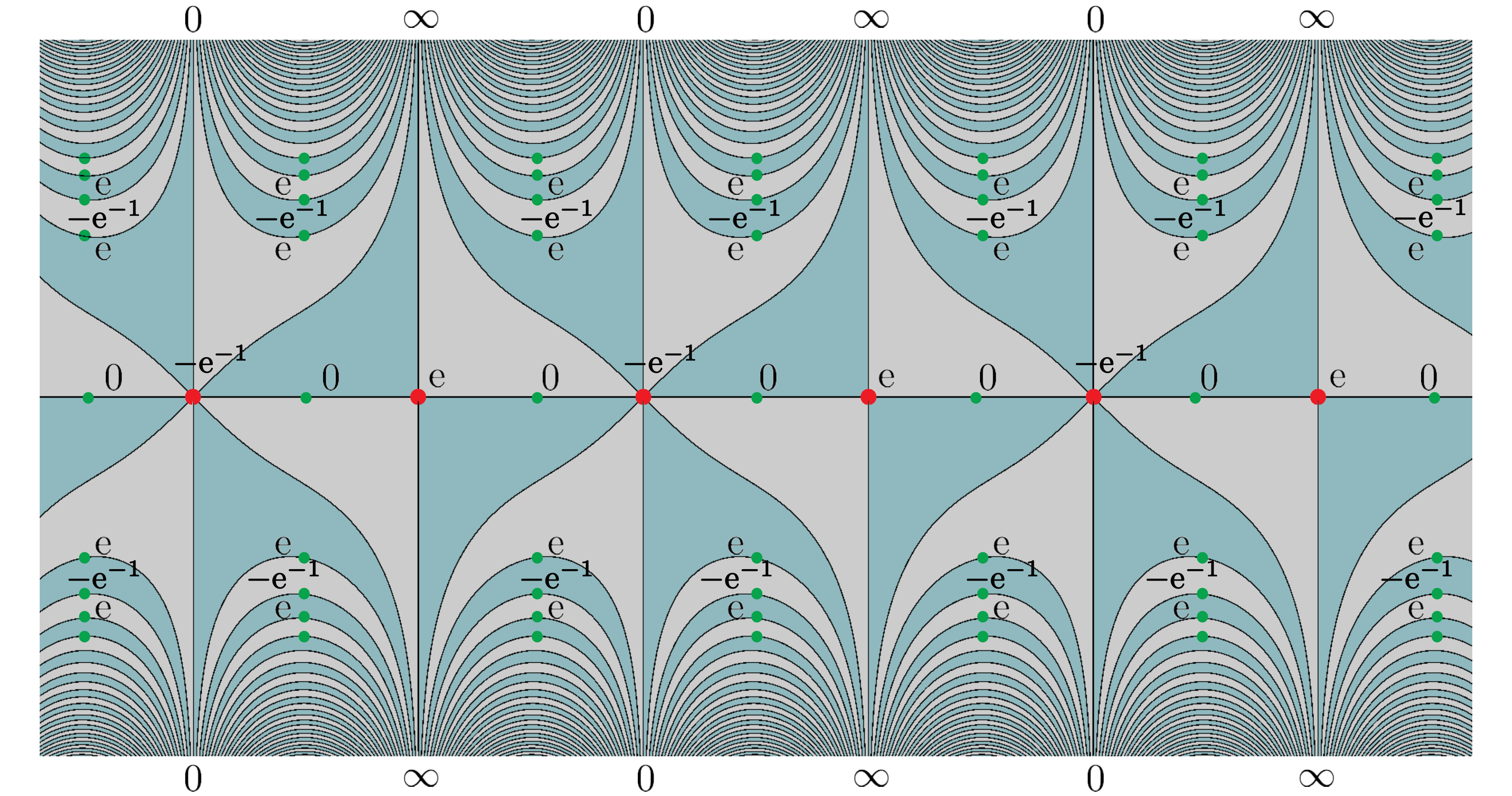}
\\[10pt]
{ b)}\includegraphics[width=0.8\textwidth]{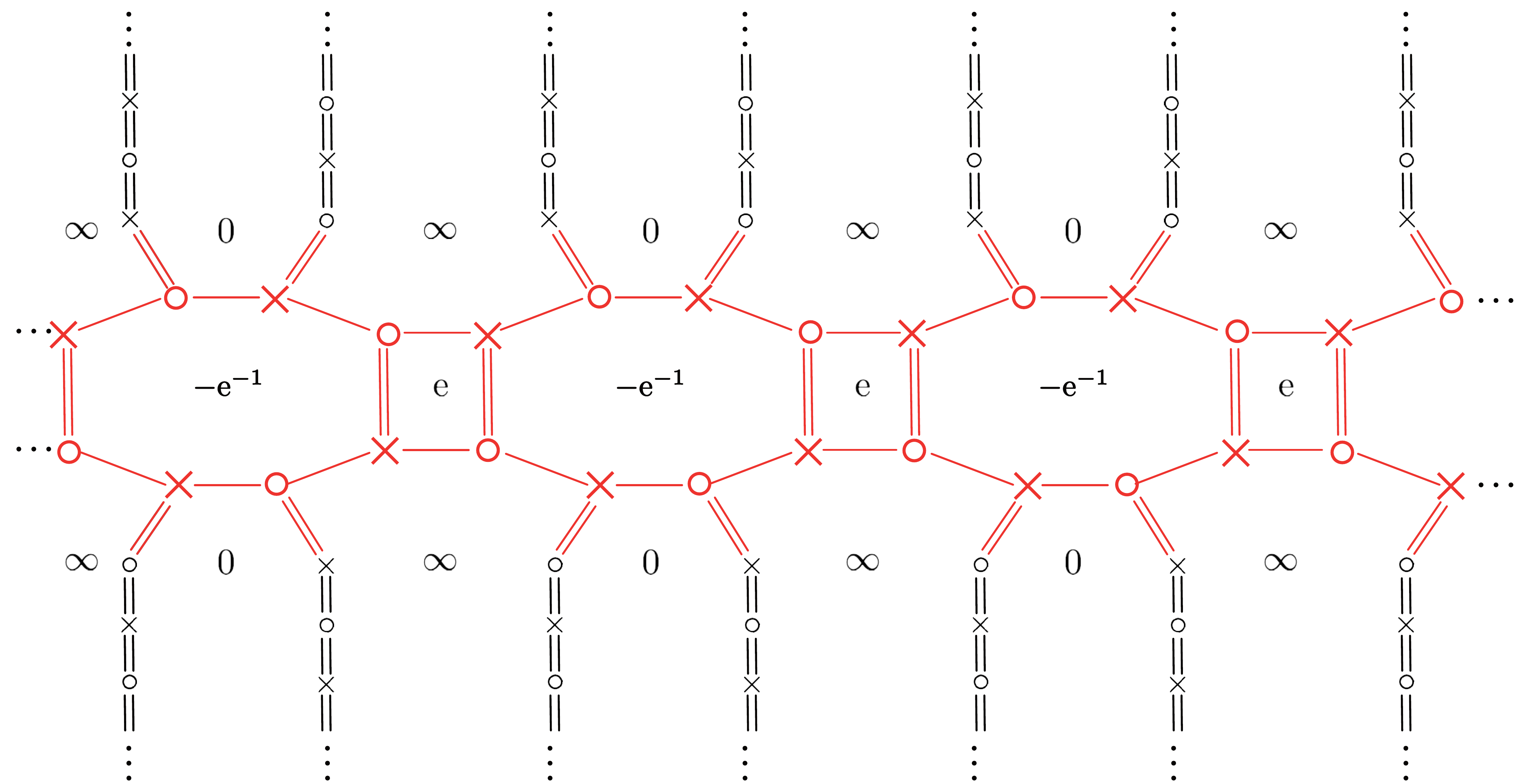}
\caption{ 
Speiser 4--tessellation and analytic Speiser graph of index 
${\tt q}=4$
for $w(z)= \sin (z) \exp (\sin (z))$ with cyclic order 
$\mathcal{W}_4 = [ -\e^{-1}, 0, \e, \infty ]$.
The nucleus is colored red and the 
logarithmic ends are colored black.
}
\label{fig:SinExpSin}
\end{center}
\end{figure}

\item
For the tessellation, with the cyclic order $\mathcal{W}_4$, it follows that 
$\gamma=\RR\cup\{\infty\}$.
Its Speiser $4$--tessellation 
$\big( \mathscr{T}(w(z)^* \gamma),w(z)^* \mathcal{L}_\gamma \big)$ is shown in 
Figure \ref{fig:SinExpSin}.a. 
As can be appreciated, the tiles are topological $4$--gons, 
with vertices of valence two represented by 
green dots (the cosingular points, with singular values ${\tt a}_1=0$ on the real axis, 
${\tt w}_1=\e$, and ${\tt w}_2=\e^{-1}$ on alternating rows symmetric withrespect to the
real axis),
only a couple of ``rows'' are drawn, however there are an infinite number of them. 

There are an infinite number of vertices of infinite valence of 
the graph $w(z)^*\gamma$.
These are the ideal points in the non Hausdorff compactification 

\centerline{$\CC_z \cup \{ 
\infty_{{\tt a}_{1\pm}}, \, 
\infty_{{\tt a}_{2\pm}}, \, 
\ldots, 
\infty_{{\tt a}_{\sigma\pm}},\,
\ldots
\}$}

\noindent 
determined by the asymptotic values ${\tt a}_{\sigma\pm}=0$ and 
${\tt a}_{\sigma\pm}=\infty$, for odd or even $\sigma$ respectively.

\item
Its Speiser graph of index $4$ is drawn in Figure \ref{fig:SinExpSin}.b.
In this case the Speiser graph has an infinite number of unbounded 0--faces 
and $\infty$--faces, which correspond to logarithmic singularities over the 
asymptotic values 0 and $\infty$, respectively. 
Moreover, it also has an infinite number of bounded $\e$--faces and $-\e^{-1}$--faces. 
The bounded $-\e^{-1}$--faces are $8$--gons, thus they correspond to 
finitely ramified branch points of ramification index 4, 
while the bounded $\e$--faces are $4$--gons, thus they correspond to 
finitely ramified branch points of ramification index 2.
Note that the nucleus consists of an infinite number of octagons and squares arranged in a horizontal line 
and an infinite number of ``loose'' edge bundles (in red), 
surrounded by an infinite number of logarithmic ends (in black).
\end{enumerate}
\end{example}

\begin{example}\label{example:Speiser3-types}
With Speiser graphs one can easily specify functions with ``strange'' behavior.
For instance, consider the Speiser graph $\mathfrak{S}_{EPH}$ of index $4$ drawn in Figure \ref{fig:Speiser3-types}.
\begin{enumerate}[label=\arabic*),leftmargin=*]
\item
On the 
`right subgraph' 
of the Speiser graph $\mathfrak{S}_{EPH}$, 
we note a behavior similar to that of the Weirstrass $\wp$--function: 
a lattice 
with an infinite number of bounded faces that are not digons. 
These correspond to an infinite number of critical points of ramification index 2 with 
4 distinct critical values 
$\{ {\tt w}_{1}, {\tt w}_{2}, {\tt w}_{3}, {\tt w}_{4} \}$.
This part of the graph has ``elliptic conformal type behavior''.

\item
On the 
`middle subgraph' 
of the Speiser graph $\mathfrak{S}_{EPH}$, 
we observe a logarithmic end of the Speiser graph 
$\mathfrak{S}_{EPH}$ consisting of an infinite sequence of digons with 
alternating labels $1$ and $3$. The logarithmic end is delimited by two unbounded 
faces labelled $2$ and $4$. 
This part of the graph has ``parabolic conformal type behavior''.

\item
On the 
`left subgraph' 
of the Speiser graph $\mathfrak{S}_{EPH}$,
we observe a tree structure with 4 edges on each vertex. 
Every face, of this part of the Speiser graph, is an unbounded face;
thus we have an infinite number of unbounded faces.
This part of the Speiser graph has ``hyperbolic conformal type behavior''.
\end{enumerate}

\noindent
Of course, the actual conformal type of the associated function is hyperbolic.

\noindent
Moreover, there are three ``special'' unbounded faces: 
\begin{itemize}[label=$\bigcdot$,leftmargin=*]

\item
the first one labelled $2$, is delimited by the \emph{bounded faces 
with labels $3$ and $1$}, of the lattice in (1)
and of the logarithmic end in (2); 

\item
the second one labelled $4$, delimited by the \emph{bounded faces 
with labels $3$ and $1$}, of
the logarithmic end in (2), and by the \emph{unbounded faces 
with labels $3$ and $1$}, of the tree in (3);

\item
the third one labelled $1$, 
is delimited by the \emph{unbounded faces 
with labels $2$ and $4$}, of
the tree in (3), and by the \emph{bounded faces 
with labels $2$ and $4$}, of
the lattice in (1).

\item
Note that the nucleus (colored red) consists of all the Speiser graph minus the one logarithmic end 
(colored black) described in (2) above.

\end{itemize}
\begin{figure}[h!]
\begin{center}
\includegraphics[width=0.75\textwidth]{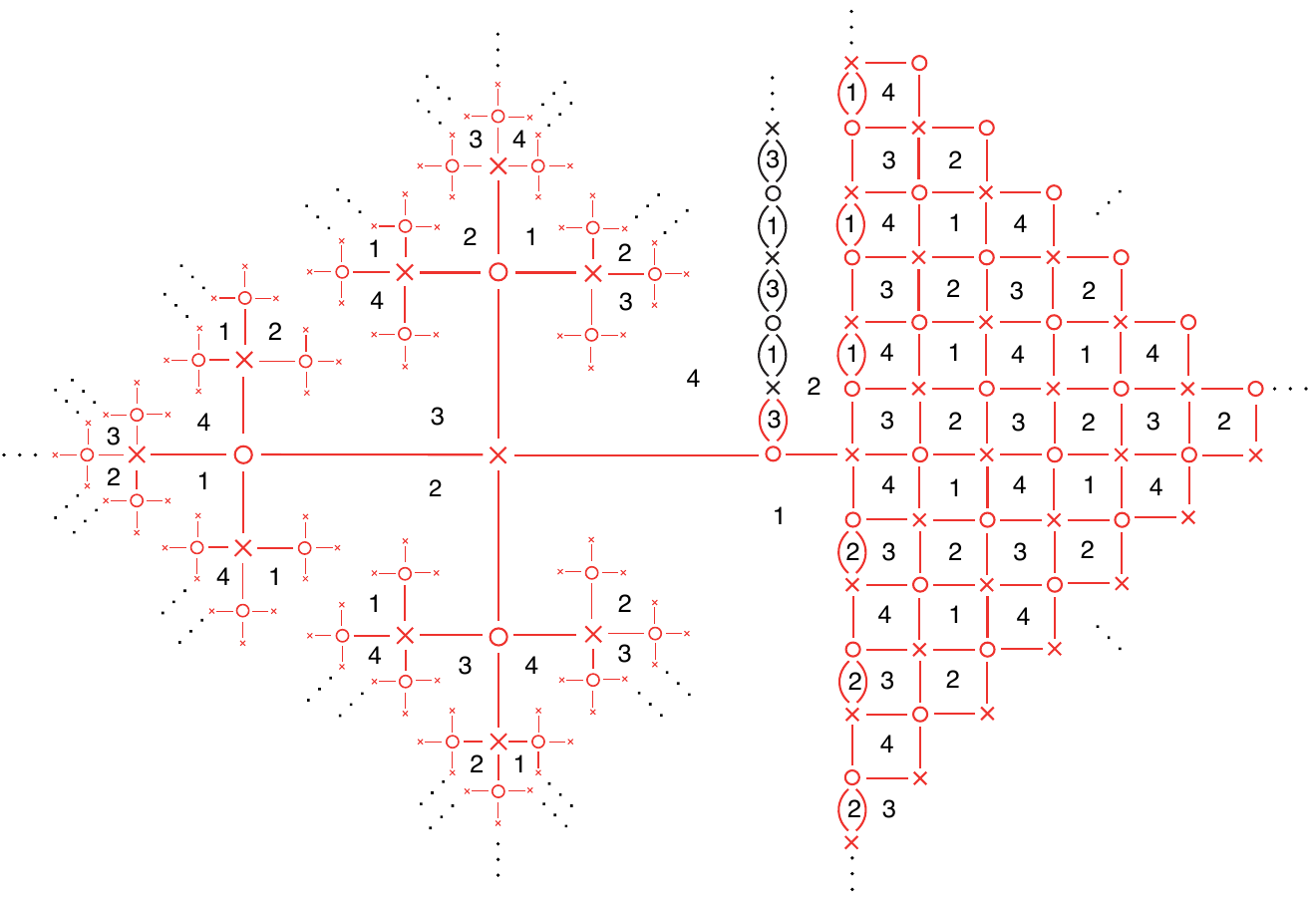}
\caption{
Speiser graph $\mathfrak{S}_{EPH}$
of a function $w(z)$ 
with ${\tt q}=4$ distinct singular values that has 
an infinite number of bounded faces that are not digons,
an infinite number of unbounded faces, and only one logarithmic tower.
The function $w(z)$ exhibits ``behavior'' associated to the three conformal types: elliptic,
parabolic and hyperbolic. 
The nucleus is colored red and the unique
logarithmic end is colored black.
}
\label{fig:Speiser3-types}
\end{center}
\end{figure}

\end{example}


\end{document}